\documentclass[a4paper,10pt]{amsart}

\usepackage{pifont}
\usepackage[normalem]{ulem}

\def\mylabelonoff{off}
\def\allowdisbrkyesno{yes}
\def\numberingtheoremsectionyesno{yes}
\def\numberingequationsectionyesno{yes}
\def\pagesizeextendednormal{extended}

\def\reportudemathyesno{no}
\def\reportudemathnumber{SM-UDE-809}
\def\reportudemathyear{2016}
\def\reportudematheingang{\mydate}

\def\mytitle{
Solution Theory, Variational Formulations,
and Functional A Posteriori Error Estimates\\
for General First Order Systems
with Applications to Electro-Magneto-Statics
and More}
\def\mytitlerepude{Solution Theory and Functional A Posteriori Error Estimates\\
for General First Order Systems\\ with Applications to Electro-Magneto-Statics}
\def\myshorttitle{General First Order Systems}
\def\myauthorone{Dirk Pauly}
\def\myauthors{\myauthorone}
\def\myaddressone{Fakult\"at f\"ur Mathematik,
Universit\"at Duisburg-Essen, Campus Essen, Germany}
\def\myemailone{dirk.pauly@uni-due.de}
\def\mykeywords{general first order systems, functional a posteriori error estimates,
electro-magneto statics, mixed boundary conditions}
\def\mysubjclass{35F05, 35J46, 47A05, 47A50, 65N15, 78A25, 78A30}
\def\mydate{\today}

\usepackage[mathscr]{eucal}
\usepackage[english]{babel}
\usepackage{a4,exscale,ifthen,amsfonts,amssymb,amsmath,amscd,graphicx,color}
\usepackage{nicefrac,tikz,fancyhdr,caption,array,multirow,multicol,booktabs,algorithm,algorithmic}
\usepackage[all]{xy}

\ifthenelse{\equal{\mylabelonoff}{on}}
{\newcommand{\mylabel}[1]{\label{#1}\fbox{{\sf #1}}}}
{\newcommand{\mylabel}[1]{\label{#1}}}
\ifthenelse{\equal{\allowdisbrkyesno}{yes}}
{\allowdisplaybreaks}
{}

\ifthenelse{\equal{\pagesizeextendednormal}{extended}}
{\setlength{\textwidth}{16cm}
\setlength{\textheight}{22cm}
\setlength{\oddsidemargin}{-0.2cm}
\setlength{\evensidemargin}{-0.2cm}}
{}
\ifthenelse{\equal{\numberingequationsectionyesno}{yes}}
{\numberwithin{equation}{section}}
{}

\newcommand{\ti}[1]{\tilde{#1}}

\newcommand{\ovl}[1]{\overline{#1}}
\newcommand{\unl}[1]{\underline{#1}}

\DeclareMathOperator{\diam}{diam}

\newcommand{\set}[2]{\{#1\,:\,#2\}}
\newcommand{\setb}[2]{\big\{#1\,:\,#2\big\}}

\ifthenelse{\equal{\numberingtheoremsectionyesno}{yes}}
{\newtheorem{lem}{Lemma}[section]}
{\newtheorem{lem}{Lemma}}

\newtheorem{theo}[lem]{Theorem}
\newtheorem{cor}[lem]{Corollary}
\newtheorem{rem}[lem]{Remark}

\newtheorem{genass}[lem]{General Assumption}

\newtheorem{algo}[lem]{Algorithm}

\newcommand{\om}{\Omega}
\newcommand{\omb}{\ovl{\om}}

\newcommand{\ga}{\Gamma}
\newcommand{\gat}{\ga_{\mathsf{t}}}
\newcommand{\gan}{\ga_{\mathsf{n}}}

\newcommand{\eps}{\epsilon}

\newcommand{\calF}{\mathcal{F}}

\newcommand{\reals}{\mathbb{R}}
\newcommand{\n}{\mathbb{N}}

\newcommand{\rt}{\reals^{3}}
\newcommand{\rN}{\reals^{N}}

\newcommand{\rttt}{\reals^{3\times3}}

\newcommand{\foh}{\frac{1}{2}}
\newcommand{\oh}{\nicefrac{1}{2}}
\newcommand{\moh}{-\oh}

\newcommand{\impl}{\Rightarrow}
\newcommand{\lpmi}{\Leftarrow}

\newcommand{\equi}{\Leftrightarrow}

\newcommand{\qqequi}{\qquad\equi\qquad}

\DeclareMathOperator{\id}{id}
\DeclareMathOperator{\sym}{sym}

\DeclareMathOperator{\dev}{dev}

\DeclareMathOperator{\supp}{supp}
\DeclareMathOperator{\dist}{dist}
\DeclareMathOperator{\A}{A}

\DeclareMathOperator{\Almo}{\A_{\ell-1}}
\DeclareMathOperator{\Al}{\A_{\ell}}
\DeclareMathOperator{\Alpo}{\A_{\ell+1}}

\DeclareMathOperator{\Aslmo}{\A_{\ell-1}^{*}}
\DeclareMathOperator{\Asl}{\A_{\ell}^{*}}
\DeclareMathOperator{\Aslpo}{\A_{\ell+1}^{*}}

\DeclareMathOperator{\cA}{\mathcal{A}}

\DeclareMathOperator{\cAlmo}{\cA_{\ell-1}}
\DeclareMathOperator{\cAl}{\cA_{\ell}}

\DeclareMathOperator{\cAslmo}{\cA_{\ell-1}^{*}}
\DeclareMathOperator{\cAsl}{\cA_{\ell}^{*}}

\DeclareMathOperator{\cM}{\mathcal{M}}

\DeclareMathOperator{\pil}{\pi_{\ell}}

\DeclareMathOperator{\p}{\partial}

\DeclareMathOperator{\rot}{rot}

\DeclareMathOperator{\divergence}{div}
\renewcommand{\div}{\divergence}

\DeclareMathOperator{\ed}{d}
\DeclareMathOperator{\cd}{\delta}

\newcommand{\csymbol}{\mathsf{C}}
\newcommand{\cgen}[3]{\overset{#1}{\csymbol}{}^{#2}_{#3}}

\newcommand{\ci}{\cgen{}{\infty}{}}

\newcommand{\ciomb}{\ci(\omb)}

\newcommand{\lsymbol}{\mathsf{L}}
\newcommand{\lgen}[3]{\overset{#1}{\lsymbol}{}^{#2}_{#3}}

\newcommand{\lt}{\lgen{}{2}{}}

\newcommand{\li}{\lgen{}{\infty}{}}

\newcommand{\ltbot}{\lgen{}{2}{\bot}}
\newcommand{\lteps}{\lgen{}{2}{\eps}}
\newcommand{\ltmu}{\lgen{}{2}{\mu}}

\newcommand{\ltom}{\lt(\om)}

\newcommand{\ltbotom}{\ltbot(\om)}
\newcommand{\ltepsom}{\lteps(\om)}
\newcommand{\ltmuom}{\ltmu(\om)}

\newcommand{\hsymbol}{\mathsf{H}}

\newcommand{\hgen}[3]{\overset{#1}{\hsymbol}{}^{#2}_{#3}}

\newcommand{\ho}{\hgen{}{1}{}}

\newcommand{\hl}{\hgen{}{\ell}{}}

\newcommand{\hobot}{\hgen{}{1}{\bot}}

\newcommand{\hoga}{\hgen{}{1}{\ga}}

\newcommand{\hoom}{\ho(\om)}

\newcommand{\hlom}{\hl(\om)}

\newcommand{\hobotom}{\hobot(\om)}

\newcommand{\hogaom}{\hoga(\om)}

\newcommand{\hoganom}{\hogan(\om)}

\newcommand{\rsymbol}{\mathsf{R}}

\newcommand{\rgen}[3]{\overset{#1}{\rsymbol}{}^{#2}_{#3}}

\renewcommand{\r}{\rgen{}{}{}}

\newcommand{\rz}{\rgen{}{}{0}}

\newcommand{\rga}{\rgen{}{}{\ga}}
\newcommand{\rgat}{\rgen{}{}{\gat}}
\newcommand{\rgan}{\rgen{}{}{\gan}}

\newcommand{\rgaz}{\rgen{}{}{\ga,0}}
\newcommand{\rgatz}{\rgen{}{}{\gat,0}}
\newcommand{\rganz}{\rgen{}{}{\gan,0}}

\newcommand{\rom}{\r(\om)}

\newcommand{\rzom}{\rz(\om)}

\newcommand{\rgaom}{\rga(\om)}

\newcommand{\rgatom}{\rgat(\om)}
\newcommand{\rganom}{\rgan(\om)}

\newcommand{\rgatzom}{\rgatz(\om)}
\newcommand{\rganzom}{\rganz(\om)}

\newcommand{\rgazom}{\rgaz(\om)}

\newcommand{\dsymbol}{\mathsf{D}}

\newcommand{\dgen}[3]{\overset{#1}{\dsymbol}{}^{#2}_{#3}}

\renewcommand{\d}{\dgen{}{}{}}

\newcommand{\dz}{\dgen{}{}{0}}

\newcommand{\dga}{\dgen{}{}{\ga}}

\newcommand{\dgat}{\dgen{}{}{\gat}}
\newcommand{\dgan}{\dgen{}{}{\gan}}
\newcommand{\dgatz}{\dgen{}{}{\gat,0}}
\newcommand{\dganz}{\dgen{}{}{\gan,0}}

\newcommand{\dgaz}{\dgen{}{}{\ga,0}}

\newcommand{\dom}{\d(\om)}

\newcommand{\dzom}{\dz(\om)}

\newcommand{\dgatom}{\dgat(\om)}
\newcommand{\dganom}{\dgan(\om)}

\newcommand{\dgatzom}{\dgatz(\om)}
\newcommand{\dganzom}{\dganz(\om)}

\newcommand{\dgaom}{\dga(\om)}

\newcommand{\dgazom}{\dgaz(\om)}

\newcommand{\harmsymbol}{\mathcal{H}}
\newcommand{\harmgen}[3]{\overset{#1}{\harmsymbol}{}^{#2}_{#3}}

\newcommand{\harm}{\harmgen{}{}{}}

\newcommand{\norm}[1]{|#1|}

\newcommand{\bnorm}[1]{\big|#1\big|}

\newcommand{\normltom}[1]{\norm{#1}_{\ltom}}

\newcommand{\normltepsom}[1]{\norm{#1}_{\ltepsom}}
\newcommand{\normltmuom}[1]{\norm{#1}_{\ltmuom}}

\newcommand{\scp}[2]{\langle#1,#2\rangle}

\newcommand{\bscp}[2]{\big\langle#1,#2\big\rangle}

\newcommand{\scpltom}[2]{\scp{#1}{#2}_{\ltom}}

\newcommand{\preprintudemath}[5]{
\thispagestyle{empty}
\Large
\begin{center}SCHRIFTENREIHE DER FAKULT\"AT F\"UR MATHEMATIK\end{center}
\vspace*{5mm}
\begin{center}#1\end{center}
\vspace*{5mm}
\begin{center}by\end{center}
\begin{center}#2\end{center}
\vspace*{5mm}
\begin{center}#3\hspace{80mm}#4\end{center}
\newpage
\thispagestyle{empty}
\vspace*{210mm}
Received: #5
\newpage
\addtocounter{page}{-2}
\normalsize}

\title[\sc\myshorttitle]{\Large\sf\mytitle}
\author{\myauthorone}
\address{\myaddressone}
\email[\myauthorone]{\myemailone}
\subjclass{\mysubjclass}
\date{\mydate}

\DeclareMathOperator{\Hilbert}{\mathsf{H}}
\DeclareMathOperator{\Hiz}{\Hilbert_{0}}
\DeclareMathOperator{\Hio}{\Hilbert_{1}}
\DeclareMathOperator{\Hit}{\Hilbert_{2}}
\DeclareMathOperator{\Hith}{\Hilbert_{3}}
\DeclareMathOperator{\Hif}{\Hilbert_{4}}
\DeclareMathOperator{\Hil}{\Hilbert_{\ell}}
\DeclareMathOperator{\Hilmo}{\Hilbert_{\ell-1}}
\DeclareMathOperator{\Hilpo}{\Hilbert_{\ell+1}}

\DeclareMathOperator{\T}{T}
\DeclareMathOperator{\Az}{A_{0}}
\DeclareMathOperator{\Ao}{A_{1}}
\DeclareMathOperator{\At}{A_{2}}
\DeclareMathOperator{\Ath}{A_{3}}
\DeclareMathOperator{\cAz}{\cA_{0}}
\DeclareMathOperator{\cAo}{\cA_{1}}
\DeclareMathOperator{\cAt}{\cA_{2}}
\DeclareMathOperator{\cAth}{\cA_{3}}
\DeclareMathOperator{\Azs}{A_{0}^{*}}
\DeclareMathOperator{\Aos}{A_{1}^{*}}
\DeclareMathOperator{\Ats}{A_{2}^{*}}
\DeclareMathOperator{\Aths}{A_{3}^{*}}
\newcommand{\Als}{\Asl}
\newcommand{\Almos}{\Aslmo}
\newcommand{\Alpos}{\Aslpo}
\DeclareMathOperator{\cAzs}{\cA_{0}^{*}}
\DeclareMathOperator{\cAos}{\cA_{1}^{*}}
\DeclareMathOperator{\cAts}{\cA_{2}^{*}}
\DeclareMathOperator{\cAths}{\cA_{3}^{*}}
\newcommand{\cAls}{\cAsl}
\newcommand{\cAlmos}{\cAslmo}

\DeclareMathOperator{\pio}{\pi_{1}}
\DeclareMathOperator{\pit}{\pi_{2}}
\DeclareMathOperator{\pith}{\pi_{3}}

\newcommand{\dhookrightarrow}{\hookrightarrow}

\newcommand{\hogatom}{\hgen{}{1}{\gat}(\om)}
\renewcommand{\hoganom}{\hgen{}{1}{\gan}(\om)}

\DeclareMathOperator{\grad}{grad}

\renewcommand{\eps}{\varepsilon}

\newcommand{\harmgantom}{\harmsymbol_{\mathsf{n,t}}(\om)}
\newcommand{\harmgatnepsom}{\harmsymbol_{\mathsf{t,n},\eps}(\om)}

\begin{document}


\ifthenelse{\equal{\reportudemathyesno}{yes}}
{\preprintudemath{\mytitlerepude}{\myauthors}{\reportudemathnumber}{\reportudemathyear}{\reportudematheingang}}
{}


\begin{abstract}
We prove a comprehensive solution theory using tools from functional analysis,
show corresponding variational formulations, and 
present functional a posteriori error estimates
for general linear first order systems of type
\begin{align*}
\At x&=f,\\
\Aos x&=g,
\end{align*}
for two densely defined and closed (possibly unbounded) linear operators $\Ao$ and $\At$
having the complex property $\At\Ao=0$.
As a prototypical application we will discuss
the system of electro-magneto statics in 3D with mixed tangential and normal boundary conditions
\begin{align*}
\rot E&=F,\\
-\div\eps E&=g.
\end{align*}
Our theory covers a lot more applications in 2D, 3D, and ND, such as general differential forms
and all kind of systems arising, e.g., in general relativity, biharmonic problems, Stokes equations,
or linear elasticity, to mention just a few, for example
\begin{align*}
\ed E&=F,
&
\mathrm{Rot}_{\mathbb{S}}M&=F,
&
\mathrm{Div}_{\mathbb{T}}T&=F,
&
\mathrm{Rot\,Rot}_{\mathbb{S}}^{\top}S&=F,\\
-\cd\eps E&=G,
&
\div\mathrm{Div}_{\mathbb{S}}\,\eps M&=G,
&
\sym\mathrm{Rot}_{\mathbb{T}}\,\eps T&=G,
&
-\mathrm{Div}_{\mathbb{S}}\,\eps S&=G,
\end{align*}
all with possibly mixed boundary conditions
of generalized tangential and normal type.
Second order systems of types
\begin{align*}
\Ats\At x&=f,
&
\Ats\At x&=f,\\
\Aos x&=g,
&
\Ao\Aos x&=g
\end{align*}
will be considered as well using the same techniques.
\end{abstract}


\vspace*{-10mm}
\maketitle
\tableofcontents


\section{Introduction}

Throughout this paper we assume the following:
For $\ell\in\mathbb{Z}$ 
let $\Hil$ be Hilbert spaces. Moreover, let
$$\Al:D(\Al)\subset\Hil\to\Hilpo$$ 
be densely defined and closed (possibly unbounded) linear operators.
In applications, often almost all operators $\Al$ will be zero,
i.e., only finitely many $\Al$ are different from zero,
typically $\Az$, $\Ao$, $\At$, $\Ath$, $\A_{4}$ in 3D PDE applications
or $\Az$, $\Ao$, \dots, $\A_{N}$, $\A_{N+1}$ in ND PDE applications.
Here, $D(\A)$ denotes the domain of definition of a linear operator $\A$
and we introduce by $N(\A)$ and $R(\A)$ its kernel and range, respectively.
Inner product, norm, orthogonality, orthogonal sum and difference of (or in)
an Hilbert space $\Hilbert$ will be denoted by
$\scp{\,\cdot\,}{\,\cdot\,}_{\Hilbert}$, $\norm{\,\cdot\,}_{\Hilbert}$,
$\bot_{\Hilbert}$, and $\oplus_{\Hilbert}$, $\ominus_{\Hilbert}$, respectively.
We note that $D(\A)$, equipped with the graph inner product, is a Hilbert space itself.
Moreover, we assume that the operators $\Al$ satisfy the sequence or complex property, 
this is for all $\ell$ 
\begin{align}
\mylabel{exseqAl}
R(\Al)\subset N(\Alpo)
\end{align}
or equivalently $\Alpo\Al\subset0$.
Then 
the (Hilbert space) adjoint operators
$$\Als:D(\Als)\subset\Hilpo\to\Hil$$ 
defined by the relation
$$\forall\,x\in D(\Al)\quad
\forall\,y\in D(\Als)\qquad
\scp{\Al x}{y}_{\Hilpo}=\scp{x}{\Als y}_{\Hil}$$
satisfy also the sequence or complex property, i.e., for all $\ell$
\begin{align}
\mylabel{exseqAls}
R(\Alpos)\subset N(\Als)
\end{align}
or equivalently $\Als\Alpos\subset0$.
We note $\A_{\ell}^{**}=\ovl{\Al}=\Al$, i.e.,
$(\Al,\Als)$ are dual pairs. 
The complex
\begin{align}
\mylabel{seqAl}
\begin{CD}
\cdots @> \Almo >>
D(\Al) @> \Al >>
D(\Alpo) @> \Alpo >>
\cdots
\end{CD}
\end{align}
is called closed, if all ranges 
$R(\Al)$ are closed,
and called exact, if $R(\Al)=N(\Alpo)$ holds for all $\ell$. By the closed range theorem,
\eqref{seqAl} is closed resp. exact, if and only if the adjoint complex
\begin{align}
\mylabel{seqAls}
\begin{CD}
\cdots @< \Almos <<
D(\Almos) @< \Als <<
D(\Als) @< \Alpos <<
\cdots
\end{CD}
\end{align}
is closed resp. exact. 
For all $\ell$
and by the projection theorem the Helmholtz type decompositions
\begin{align*}
\Hil&=N(\Al)\oplus_{\Hil}\ovl{R(\Als)},
&
N(\Al)&=R(\Als)^{\bot_{\Hil}},\\
\Hil
&=\ovl{R(\Almo)}\oplus_{\Hil}N(\Almos),
&
N(\Almos)&=R(\Almo)^{\bot_{\Hil}}
\end{align*}
hold. Moreover, the complex properties \eqref{exseqAl}-\eqref{exseqAls} show
$$N(\Al)=\ovl{R(\Almo)}\oplus_{\Hil}K_{\ell},\qquad
N(\Almos)=K_{\ell}\oplus_{\Hil}\ovl{R(\Als)},$$
where we introduce 
the cohomology groups
$$K_{\ell}:=N(\Al)\cap N(\Almos).$$
Therefore, we obtain the refined Helmholtz type decompositions
$$\Hil=\ovl{R(\Almo)}\oplus_{\Hil}K_{\ell}\oplus_{\Hil}\ovl{R(\Als)}.$$
Note that, if $\Al$ has closed range then 
by the closed range theorem and the projection theorem
$$R(\Al)=N(\Als)^{\bot_{\Hilpo}},\qquad
R(\Als)=N(\Al)^{\bot_{\Hil}}.$$
Finally, we define for all $\ell$
the domains of definitions for our mixed problems
$$D_{\ell}:=D(\Al)\cap D(\Almos).$$

\subsection{Aims and Main Results}

The central aim of this paper is to prove functional a posteriori error estimates
in the spirit of Sergey Repin, see, e.g., 
\cite{repinapostvarprob,repintsestelleq,NeittaanmakiRepin2004,MaliRepinNeittaanmaki2014,repinbookone},
for the linear system
\begin{align}
\mylabel{Aprob}
\begin{split}
\At x&=f,\\
\Aos x&=g,\\
\pit x&=k
\end{split}
\end{align}
with 
$$x\in D_{2}=D(\At)\cap D(\Aos),$$
where $\pi_{2}:\Hit\to K_{2}=N(\At)\cap N(\Aos)$ denotes 
the orthonormal projector onto the cohomology group or kernel $K_{2}$.
We recall the complex property $\At\Ao=0$, and hence also $\Aos\Ats=0$.
Obviously, $f\in R(\At)$, $g\in R(\Aos)$, and $k\in K_{2}$ are necessary 
for solvability of \eqref{Aprob} and there exists at most one solution to \eqref{Aprob}.
A proper solution theory for \eqref{Aprob}, i.e., 
existence of a solution of \eqref{Aprob} depending continuously on the data, 
will be given in the next section. 
The main result for this is Theorem \ref{soltheofos} and reads as follows:\\

\noindent{\bf Theorem I} (Theorem \ref{soltheofos})
{\it Let $R(\Ao)$ and $R(\At)$ be closed.
Then \eqref{Aprob} is uniquely solvable in $D_{2}$, if and only if
$f\in R(\At)$, $g\in R(\Aos)$, and $k\in K_{2}$.
The solution $x\in D_{2}$ depends linearly and continuously on the data, i.e.,
$\norm{x}_{\Hit}\leq c_{2}\norm{f}_{\Hith}+c_{1}\norm{g}_{\Hio}+\norm{k}_{\Hit}$.}\\

\noindent{\bf Remark II} (Lemma \ref{poincarerange}, Lemma \ref{cptembtoolbox}, \eqref{classup})
{\it 
\begin{itemize}
\item[\bf(i)]
By the closed range theorem,
$R(\Ao)$ resp. $R(\At)$ is closed,
if and only if $R(\Aos)$ resp. $R(\Ats)$ is closed.
Moreover, $R(\Ao)$ and $R(\At)$ are closed,
if, e.g., $D_{2}\dhookrightarrow\Hit$ is compact,
see Lemma \ref{cptembtoolbox},
in which case $K_{2}$ is also finite dimensional,
see General Assumption \ref{genass} and Remark \ref{genassrem}.
\item[\bf(ii)]
By the closed graph theorem the following assertions are equivalent: 
\begin{itemize}
\item[$\bullet$]
The range $R(\Ao)$ is closed in $\Hit$.
\item[$\bullet$]
There exists $0<c<\infty$ such that for all
$\phi\in D(\Ao)\cap N(\Ao)^{\bot_{\Hio}}$ it holds
$\norm{\phi}_{\Hio}\leq c\norm{\Ao\phi}_{\Hit}$.
\item[$\bullet$]
The inverse $\cA_{1}^{-1}:R(\Ao)\to D(\Ao)\cap N(\Ao)^{\bot_{\Hio}}$ is continuous,
where $\cAo$ is the corresponding reduced operator of $\Ao$, i.e.,
the restriction of $\Ao$ to $D(\Ao)\cap N(\Ao)^{\bot_{\Hio}}$.
\end{itemize}
\item[\bf(iii)]
If $R(\Ao)$ is closed, then $c_{1}$ is defined as the best possible constant in (ii) 
and hence equals the norm of the inverse
$\cA_{1}^{-1}$ regarded as operator from $R(\Ao)$ to $N(\Ao)^{\bot_{\Hio}}$.
Moreover, $c_{1}$ is also given by the Rayleigh quotient
$$\inf_{0\neq\phi\in D(\Ao)\cap N(\Ao)^{\bot_{\Hio}}}\frac{\norm{\Ao\phi}_{\Hit}^2}{\norm{\phi}_{\Hio}^2}
=\frac{1}{c_{1}^2}=\lambda_{1},$$
which defines the smallest positive eigenvalue $\lambda_{1}$ 
of the selfadjoint operator\footnote{Thus 
$\lambda_{1}$ is also the smallest positive eigenvalue
of the selfadjoint operator $\Ao\Aos$.} 
$\Aos\Ao$.
\item[\bf(iii)]
Similar results and definitions as in (ii) and (iii) hold for the constant $c_{2}$
provided that $R(\At)$ is closed.
\item[\bf(iv)]
The unique solution $x\in D_{2}$ in Theorem I is simply given by $x=\cA_{2}^{-1}f+(\cAos)^{-1}g+k$.\\
\end{itemize}}

Although the solution theory is based on pure functional analysis and operator theory,
we shall give a few variational (multiple) saddle point formulations as well
propose methods for computing the exact solution $x\in D_{2}$. 
These formulations are not only alternatives to prove Theorem I, 
but also suggestions for possible numerical methods in future applications,
and will be discussed extensively, see, e.g.,
Theorem \ref{soltheofosvarform}, Theorem \ref{soltheofosvarformtogether}, Theorem \ref{soltheofosvarformtogetherwithK2},
Theorem \ref{soltheofosvarformpartsolfull}, and Theorem \ref{soltheofosvarformtogetherwithK2andmore}.
One of these results reads as follows:\\

\noindent{\bf Theorem III} (Theorem \ref{soltheofosvarformtogetherwithK2})
{\it Let $R(\Ao)$ and $R(\At)$ be closed.
Moreover, let $f\in R(\At)$ and $g\in R(\Aos)$.
The unique solution $x\in D_{2}$ in Theorem I
can be found by the following two variational double saddle point formulations: 
\begin{itemize}
\item[\bf(i)]
There exists a unique tripple $(\hat{x},z,h)\in D(\At)\times\big(D(\Ao)\cap R(\Aos)\big)\times K_{2}$,
such that for all triples $(\xi,\varphi,\kappa)\in D(\At)\times\big(D(\Ao)\cap R(\Aos)\big)\times K_{2}$
\begin{align*}
\scp{\At\hat{x}}{\At\xi}_{\Hith}
+\scp{\Ao z}{\xi}_{\Hit}
+\scp{h}{\xi}_{\Hit}
&=\scp{f}{\At\xi}_{\Hith},\\
\scp{\hat{x}}{\Ao\varphi}_{\Hit}
&=\scp{g}{\varphi}_{\Hio},\\
\scp{\hat{x}}{\kappa}_{\Hit}
&=\scp{k}{\kappa}_{\Hit}.
\end{align*}
It holds $z=0$ and $h=0$ as well as
$\At\hat{x}=f$ and $\pit\hat{x}=k$.
Moreover, the variational formulation holds for all $\varphi\in D(\Ao)$,
and thus $\hat{x}\in D(\Aos)$ with $\Aos\hat{x}=g$.
Finally, $\hat{x}=x$ from Theorem I.
\item[\bf(ii)]
There exists a unique triple $(\hat{x},y,h)\in D(\Aos)\times\big(D(\Ats)\cap R(\At)\big)\times K_{2}$,
such that for all triples $(\zeta,\phi,\kappa)\in D(\Aos)\times\big(D(\Ats)\cap R(\At)\big)\times K_{2}$
\begin{align*}
\scp{\Aos\hat{x}}{\Aos\zeta}_{\Hio}
+\scp{\Ats y}{\zeta}_{\Hit}
+\scp{h}{\zeta}_{\Hit}
&=\scp{g}{\Aos\zeta}_{\Hio},\\
\scp{\hat{x}}{\Ats\phi}_{\Hit}
&=\scp{f}{\phi}_{\Hith},\\
\scp{\hat{x}}{\kappa}_{\Hit}
&=\scp{k}{\kappa}_{\Hit}.
\end{align*}
It holds $y=0$ and $h=0$ as well as $\Aos\hat{x}=g$ and $\pit\hat{x}=k$.
The variational formulation holds for all $\phi\in D(\Ats)$,
and thus $\hat{x}\in D(\At)$ with $\At\hat{x}=f$.
Finally, $\hat{x}=x$ from Theorem I.\\
\end{itemize}}

Theorem III (i) resp. (ii) is a weak formulation of 
\begin{align*}
\Ats\At\hat{x}+\Ao z+h&=\Ats f,
&
\Aos\hat{x}&=g,
&
\pit\hat{x}&=k,
\intertext{resp.}
\Ao\Aos\hat{x}+\Ats y+h&=\Ao g,
&
\At\hat{x}&=f,
&
\pit\hat{x}&=k,
\end{align*}
i.e., in formal matrix notation
$$\begin{bmatrix}
\Ats\At & \Ao & \iota_{K_{2}}\\
\Aos & 0 & 0\\
\pit=\iota_{K_{2}}^{*} & 0 & 0
\end{bmatrix}
\begin{bmatrix}
\hat{x}\\
z\\
h
\end{bmatrix}
=\begin{bmatrix}
\Ats f\\
g\\
k
\end{bmatrix},\qquad
\begin{bmatrix}
\Ao\Aos & \Ats & \iota_{K_{2}}\\
\At & 0 & 0\\
\pit=\iota_{K_{2}}^{*} & 0 & 0
\end{bmatrix}
\begin{bmatrix}
\hat{x}\\
y\\
h
\end{bmatrix}
=\begin{bmatrix}
\Ao g\\
f\\
k
\end{bmatrix},$$
respectively, where $\iota_{K_{2}}$ is the canonical embedding of $K_{2}$ into $\Hit$.
Note $z=0$, $h=0$ resp. $y=0$, $h=0$.
Often the additional condition $z\in R(\Aos)$ resp. $y\in R(\At)$ is unpleasant,
especially for possible numerical applications,
and hence the saddle point idea has to be repeated until 
$D(\cA_{\ell})=D(\A_{\ell})$, i.e., $R(\A_{\ell}^{*})=\Hil$
resp. $D(\cA_{\ell}^{*})=D(\A_{\ell}^{*})$, i.e., $R(\A_{\ell})=\Hilpo$
holds for some $\ell$.
In 3D we typically have only the operators $\Az$, $\Ao$, $\At$, $\Ath$, $\A_{4}$
with adjoints $\Azs$, $\Aos$, $\Ats$, $\Aths$, $\A_{4}^{*}$
forming the Hilbert complexes and it holds $R(\Azs)=\hsymbol_{0}$ and $R(\A_{4})=\hsymbol_{5}$.
Hence the biggest system in 3D arising for $\At$ resp. $\Aos$ as leading operator 
to compute $\hat{x}=x$ is
$$\begin{bmatrix}
\Ats\At & \Ao & 0 & \iota_{K_{2}} & 0\\
\Aos & 0 & \Az & 0 & \iota_{K_{1}}\\
0 & \Azs & 0 & 0 & 0\\
\pit=\iota_{K_{2}}^{*} & 0 & 0 & 0 & 0\\
0 & \pio=\iota_{K_{1}}^{*} & 0 & 0 & 0
\end{bmatrix}
\begin{bmatrix}
\hat{x}\\
z\\
u\\
h_{2}\\
h_{1}
\end{bmatrix}
=\begin{bmatrix}
\Ats f\\
g\\
0\\
k\\
0
\end{bmatrix}$$
resp.
$$\begin{bmatrix}
\Ao\Aos & \Ats & 0 & 0 & \iota_{K_{2}} & 0 & 0\\
\At & 0 & \Aths & 0 & 0 & \iota_{K_{3}}& 0 \\
0 & \Ath & 0 & \A_{4}^{*} & 0 & 0 & \iota_{K_{4}}\\
0 & 0 & \A_{4} & 0 & 0 & 0 & 0\\
\pit=\iota_{K_{2}}^{*} & 0 & 0 & 0 & 0 & 0 & 0\\
0 & \pith=\iota_{K_{3}}^{*} & 0 & 0 & 0 & 0 & 0 \\
0 & 0 & \pi_{4}=\iota_{K_{4}}^{*} & 0 & 0 & 0 & 0
\end{bmatrix}
\begin{bmatrix}
\hat{x}\\
y\\
v\\
w\\
h_{2}\\
h_{3}\\
h_{4}
\end{bmatrix}
=\begin{bmatrix}
\Ao g\\
f\\
0\\
0\\
k\\
0\\
0
\end{bmatrix}.$$
Note $z=0$, $u=0$, $h_{2}=0$, $h_{1}=0$
resp. $y=0$, $v=0$, $w=0$, $h_{2}=0$, $h_{3}=0$, $h_{4}=0$.\\

\noindent{\bf Remark IV}
{\it Particularly interesting cases are those for which
$R(\Aos)$ in Theorem III (i) or $R(\At)$ in Theorem III (ii)
already have finite co-dimension, i.e., in the best cases $R(\Aos)=\Hio$ or $R(\At)=\Hith$,
i.e., $N(\Ao)=\{0\}$ or $N(\Ats)=\{0\}$. Fortunately,
these situations are typical in many applications,
as we will see at the end of the introduction or in more detail in the 
Application Section \ref{secappl}.
Indeed, typically $N(\Ao)=\{0\}$ or at least $\dim N(\Ao)<\infty$
and $N(\Aths)=\{0\}$ or at least $\dim N(\Aths)<\infty$.\\}

Let $\ti{x}\in\Hit$ and let us consider $\ti{x}$ 
as a possibly (very) non-conforming\footnote{A conforming
``approximation'' $\ti{x}$ would belong to $D_{2}$.} 
``approximation'' for the exact solution 
$$x\in D_{2}=D(\At)\cap D(\Aos)$$
of \eqref{Aprob}. Proving functional a posteriori error estimates,
also called a posteriori error estimates of functional type,
for the linear problem \eqref{Aprob} means, 
that we will present two-sided estimates for the error 
$$e:=x-\ti{x}\in\Hit$$ 
with the following properties:
\begin{enumerate}
\item[\ding{172}]
There exist two functionals $\cM_{\mp}$, a lower and an upper bound, such that
\begin{align}
\mylabel{lowupbdM}
\forall\,z_{i},y_{j}\qquad
\cM_{-}(z_{1},\ldots,z_{I};\ti{x},f,g,k)
\leq\norm{e}_{\Hit}
\leq\cM_{+}(y_{1},\ldots,y_{J};\ti{x},f,g,k),
\end{align}
were the $z_{i}$ and the $y_{j}$ belong to some suitable Hilbert spaces.
The functionals $\cM_{\mp}$ are guaranteed lower and upper bounds for 
the norm of the error $\norm{e}_{\Hit}$
and explicitly computable 
as long as at least upper bounds  for the natural Friedrichs/Poincar\'e type constants $c_{1}$ and $c_{2}$
for the operators $\Ao$ and $\At$ are known\footnote{Just needed for the upper bound $\cM_{+}$.}. 
The bounds $\cM_{\mp}$ do not depend on the possibly and generally unknown exact solution $x$,
but only on the data, the approximation $\ti{x}$, 
and the ``free'' vectors $z_{i}$, $y_{j}$.
\item[\ding{173}] 
The lower and upper bound $\cM_{\mp}$ are sharp, i.e.,
\begin{align}
\mylabel{lowupbdMinfsup}
\max_{z_{1},\ldots,z_{I}}\cM_{-}(z_{1},\ldots,z_{I};\ti{x},f,g,k)
=\norm{e}_{\Hit}
=\min_{y_{1},\ldots,y_{J}}\cM_{+}(y_{1},\ldots,y_{J};\ti{x},f,g,k).
\end{align}
\item[\ding{174}]
The minimization over $z_{i}$ and $y_{j}$ is ``simple'',
typically a minimization of quadratic functionals.
\item[\ding{175}]
The bounds $\cM_{\mp}$ are general in the sense that
they do not depend on any specific numerical method
which might be used in some possible application.
\end{enumerate}

Concerning the error estimates the main result of this contribution 
is Corollary \ref{apostestfoscortsb}, which summarizes 
Theorem \ref{apostestfos}, Theorem \ref{apostestfoslowbd},
and the corresponding corollaries and reads as follows:\\

\noindent{\bf Theorem V} (Corollary \ref{apostestfoscortsb})
{\it Let $R(\Ao)$ and $R(\At)$ be closed.
Moreover, let $x\in D_{2}$ be the exact solution of \eqref{Aprob} 
and let $\ti{x}\in\Hit$, regarded as non-conforming approximation of $x$.
Then the error $e:=x-\ti{x}$ decomposes orthogonally, i.e.,
$$e=e_{\Ao}+e_{K_{2}}+e_{\Ats}
\in R(\Ao)\oplus_{\Hit}K_{2}\oplus_{\Hit}R(\Ats),\qquad
\norm{e}_{\Hit}^2
=\norm{e_{\Ao}}_{\Hit}^2
+\norm{e_{K_{2}}}_{\Hit}^2
+\norm{e_{\Ats}}_{\Hit}^2,$$
and the following a posteriori error estimates 
for the respective error parts hold:
\begin{itemize}
\item[\bf(i)]
The projection $e_{\Ao}\in R(\Ao)$ satisfies
$$\max_{\varphi\in D(\Ao)}
\cM_{-,\Ao}(\varphi;\ti{x},g)
=\norm{e_{\Ao}}_{\Hit}^2
=\min_{\zeta\in D(\Aos)}
\cM_{+,\Ao}^2(\zeta;\ti{x},g),$$
$$\cM_{-,\Ao}(\varphi;\ti{x},g)
:=2\scp{g}{\varphi}_{\Hio}
-\scp{2\ti{x}+\Ao\varphi}{\Ao\varphi}_{\Hit},\qquad
\cM_{+,\Ao}(\zeta;\ti{x},g)
:=c_{1}\norm{\Aos\zeta-g}_{\Hio}
+\norm{\zeta-\ti{x}}_{\Hit}.$$
The maximum is attained at any $\hat{\varphi}\in D(\Ao)$ with $\Ao\hat{\varphi}=e_{\Ao}$
and $\hat{\zeta}:=e_{\Ao}+\ti{x}\in D(\Aos)$ gives the minimum.
It holds $\Aos\hat{\zeta}=\Aos x=g$.
\item[\bf(ii)]
The projection $e_{\Ats}\in R(\Ats)$ satisfies
$$\max_{\phi\in D(\Ats)}
\cM_{-,\Ats}(\phi;\ti{x},f)
=\norm{e_{\Ats}}_{\Hit}^2
=\min_{\xi\in D(\At)}
\cM_{+,\Ats}^2(\xi;\ti{x},f),$$
$$\cM_{-,\Ats}(\phi;\ti{x},f)
:=2\scp{f}{\phi}_{\Hith}
-\scp{2\ti{x}+\Ats\phi}{\Ats\phi}_{\Hit},\qquad
\cM_{+,\Ats}(\xi;\ti{x},f)
:=c_{2}\norm{\At\xi-f}_{\Hith}
+\norm{\xi-\ti{x}}_{\Hit}.$$
The maximum is attained at any $\hat{\phi}\in D(\Ats)$ with $\Ats\hat{\phi}=e_{\Ats}$
and $\hat{\xi}:=e_{\Ats}+\ti{x}\in D(\At)$ gives the minimum.
It holds $\At\hat{\xi}=\At x=f$.
\item[\bf(iii)]
The projection $e_{K_{2}}=\pit e=k-\pit\ti{x}\in K_{2}$ satisfies
$$\max_{\theta\in K_{2}}
\cM_{-,K_{2}}(\theta;\ti{x},k)
=\norm{e_{K_{2}}}_{\Hit}^2
=\min_{\substack{\varphi\in D(\Ao),\\\phi\in D(\Ats)}}
\cM_{+,K_{2}}^2(\varphi,\phi;\ti{x},k)$$
$$\cM_{-,K_{2}}(\theta;\ti{x},k)
:=\bscp{2(k-\ti{x})-\theta}{\theta}_{\Hit},\qquad
\cM_{+,K_{2}}(\varphi,\phi;\ti{x},k)
:=\norm{k-\ti{x}+\Ao\varphi+\Ats\phi}_{\Hit}.$$
The maximum is attained at $\hat{\theta}:=e_{K_{2}}\in K_{2}$
and the minimum at any pair $(\hat{\varphi},\hat{\phi})\in D(\Ao)\times D(\Ats)$
with $\Ao\hat{\varphi}+\Ats\hat{\phi}=(1-\pit)\ti{x}$.\\
\end{itemize}}

\noindent{\bf Remark VI} (Corollary \ref{apostestfoscortsb} continued, Section \ref{seccomperrfunc})
{\it 
\begin{itemize}
\item[\bf(i)]
In applications, often $\ti{x}:=k+\ti{x}_{\bot}$ holds
with some $\ti{x}_{\bot}\in K_{2}^{\bot_{\Hit}}$.
In this case $e_{K_{2}}=0$ and 
in Theorem V (i) and Theorem V (ii) $\ti{x}$ can be replaced by $\ti{x}_{\bot}$.
Moreover,
$\hat{\zeta}_{\bot}:=e_{\Ao}+\ti{x}_{\bot}\in D(\Aos)$ and
$\hat{\xi}_{\bot}:=e_{\Ats}+\ti{x}_{\bot}\in D(\At)$ holds for the attaining minima.
\item[\bf(ii)]
Differentiating the lower bound
$\cM_{-,\Ao}(\varphi;\ti{x},g)$ with respect to $\varphi$
shows that a possible maximizer $\hat{\varphi}\in D(\Ao)$ 
of the maximum in Theorem V (i) solves the variational formulation
\begin{align}
\mylabel{varformMmAointro}
\forall\,\varphi\in D(\Ao)\qquad
\scp{\Ao\hat{\varphi}}{\Ao\varphi}_{\Hit}
=\scp{g}{\varphi}_{\Hio}
-\scp{\ti{x}}{\Ao\varphi}_{\Hit},
\end{align}
which implies $\Ao\hat{\varphi}+\ti{x}\in D(\Aos)$
with $\Aos(\Ao\hat{\varphi}+\ti{x})=g$
and presents a weak formulation\footnote{Thus
$\Ao\hat{\varphi}-e_{\Ao}\in N(\Aos)\cap R(\Ao)=N(\Aos)\cap N(\Aos)^{\bot_{\Hit}}=\{0\}$.} of
$$\Aos\Ao\hat{\varphi}
=g-\Aos\ti{x}
=\Aos e
=\Aos e_{\Ao}.$$
By Remark II (ii) $\Ao$ is strictly positive over $D(\Ao)\cap N(\Ao)^{\bot_{\Hio}}$
and hence \eqref{varformMmAointro} admits a unique solution
$\hat{\varphi}\in D(\Ao)\cap N(\Ao)^{\bot_{\Hio}}$.
A particularly simple case is again given if $N(\Ao)$ is finite dimensional or even $N(\Ao)=\{0\}$,
which occurs in many applications. 
\item[\bf(ii')]
On the other hand, considering the minimum in Theorem V (i)
we can roughly estimate the upper bound by, e.g.,
$$\cM_{+,\Ao}^2(\zeta;\ti{x},g)
\leq2c_{1}^2\norm{\Aos\zeta-g}_{\Hio}^2
+2\norm{\zeta-\ti{x}}_{\Hit}^2.$$
Differentiating the right hand side with respect to $\zeta$ 
shows that the minimizer $\hat{\zeta}\in D(\Aos)$ solves the variational formulation
\begin{align}
\mylabel{varformMpAointro}
\begin{split}
\forall\,\zeta\in D(\Aos)\qquad
c_{1}^2\scp{\Aos\hat{\zeta}}{\Aos\zeta}_{\Hio}
+\scp{\hat{\zeta}}{\zeta}_{\Hit}
=c_{1}^2\scp{g}{\Aos\zeta}_{\Hio}
+\scp{\ti{x}}{\zeta}_{\Hit},
\end{split}
\end{align}
which implies $\Aos\hat{\zeta}-g\in D(\Ao)$
and $c_{1}^2\Ao(\Aos\hat{\zeta}-g)=(\ti{x}-\hat{\zeta})$
and presents a weak formulation of
$$c_{1}^2\Ao\Aos\hat{\zeta}+\hat{\zeta}
=c_{1}^2\Ao g+\ti{x}.$$
Unique solvability of \eqref{varformMpAointro} in $D(\Aos)$ is trivial
as the variational formulation reproduces a graph inner product of $D(\Aos)$. 
An optimized minimization process using a more careful estimate
is explained in some detail in Section \ref{seccomperrfunc}.
\item[\bf(iii)]
Similar arguments and formulations hold for Theorem V (ii) and (iii) as well.\\
\end{itemize}
}

We shall also present a full theory, in particular functional a posteriori error estimates,
for linear second order systems such as 
\begin{align}
\mylabel{AsAprob}
\begin{split}
\Ats\At x&=f,\\
\Aos x&=g,\\
\pit x&=k
\end{split}
\end{align}
with $x\in D_{2}$ such that $\At x\in D(\Ats)$, i.e.,
$x\in D(\Aos)\cap D(\Ats\At)$.
This will follow immediately by the theory developed 
for the first order system \eqref{Aprob},
since the solution pair 
$$(x,y)\in\big(D(\At)\cap D(\Aos)\big)\times\big(D(\Ath)\cap D(\Ats)\big)$$
defined by $y:=\At x\in D(\Ats)\cap R(\At)$
solves the system of first order systems
\begin{align*}
\At x&=y,
&
\Ath y&=0,\\
\Aos x&=g,
&
\Ats y&=f,\\
\pit x&=k,
&
\pith y&=0.
\end{align*}

Analogously, we can treat problems such as 
\begin{align}
\mylabel{AsAAsAprob}
\begin{split}
\Ats\At x&=f,\\
\Ao\Aos x&=g,\\
\pit x&=k
\end{split}
\end{align}
as well, which are strongly related to the generalized Hodge-Helmholtz decomposition
of $f+g+k\in\Hit$.

\subsection{Applications}

Our main applications will be the linear first order systems
of electro-magneto statics as well as related second order $\rot\rot$ systems
and, as a very simple example, the Laplacian, see Section \ref{secappl},
especially Theorem \ref{apostestemsys}.
In this paper, we only discuss homogeneous boundary conditions,
noting that the canonical extension to inhomogeneous boundary conditions is straight forward.
As we shall give a detailed description of more applications fitting our general theory 
for the linear systems \eqref{Aprob}, \eqref{AsAprob}, \eqref{AsAAsAprob}
and for the general complexes \eqref{seqAl}, \eqref{seqAls}
in Section \ref{moreappsec}, we just indicate a few applications by
listing some interesting and important underlying complexes 
arising in, e.g., general electro-magneto statics, for differential forms on Riemannian manifolds,
in problems of linear elasticity, Stokes equations, biharmonic theory, 
general relativity, $\rot\rot\rot\rot$-operators, to mention just a few examples.
Although all these systems are allowed to have mixed generalized tangential and normal boundary conditions
and inhomogeneous and anisotropic material properties,
see Section \ref{moreappsec}, we will just present the cases of full boundary conditions
and homogeneous and isotropic material parameters here in this introductory part.
For this let $\om\subset\rt$ or $\om\subset\rN$, $N\geq2$,
be a bounded weak Lipschitz domain.
\begin{itemize}
\item
electro-magnetics
$$\begin{CD}
\{0\} @> \Az=\iota_{\{0\}} >>
\hogaom @> \Ao=\grad_{\ga} >>
\rgaom @> \At=\rot_{\ga} >>
\dgaom @> \Ath=\div_{\ga} >>
\ltom @> \A_{4}=\pi_{\reals} >>
\reals
\end{CD}$$
$$\begin{CD}
\{0\} @< \Azs=\pi_{\{0\}} <<
\ltom @< \Aos=-\div <<
\dom @< \Ats=\rot <<
\rom @< \Aths=-\grad <<
\hoom @< \A_{4}^{*}=\iota_{\reals} <<
\reals
\end{CD}$$
A typical system for a vector field $E$ is
$$\rot_{\ga}E=F,\qquad
-\div E=g.$$
\item
generalized electro-magnetics (differential forms)
{\tiny$$\begin{CD}
\{0\} @> \Az=\iota_{\{0\}} >>
\dgen{}{0}{\ga}(\om) @> \A_{1}=\ed_{\ga} >>
\cdots @> \A_{q-1}=\ed_{\ga} >>
\dgen{}{q-1}{\ga}(\om) @> \A_{q}=\ed_{\ga} >>
\dgen{}{q}{\ga}(\om) @> \A_{q+1}=\ed_{\ga} >>
\cdots @> \A_{N}=\ed_{\ga} >>
\lgen{}{2,N}{}(\om) @> \A_{N+1}=\pi_{\reals} >>
\reals
\end{CD}$$
$$\begin{CD}
\{0\} @< \A_{0}^{*}=\pi_{\{0\}} <<
\lgen{}{2,0}{}(\om) @< \A_{1}^{*}=-\cd <<
\cdots @< \A_{q-1}^{*}=-\cd <<
\Delta^{q-1}(\om) @< \A_{q}^{*}=-\cd <<
\Delta^{q}(\om) @< \A_{q+1}^{*}=-\cd <<
\cdots @< \A_{N}^{*}=-\cd <<
\dgen{}{N}{}(\om) @< \A_{N+1}^{*}=\iota_{\reals} <<
\reals
\end{CD}$$}
A typical system for a differential form $E$ is
$$\ed_{\ga}E=F,\qquad
-\cd E=G.$$
\item
biharmonic problems, Stokes problems, and general relativity
{\small$$\begin{CD}
\{0\} @> \Az=\iota_{\{0\}} >>
\hgen{}{2}{\ga}(\om) @> \Ao=\mathrm{Grad}\grad_{\ga} >>
\rgen{}{}{\ga}(\om;\mathbb{S}) @> \At=\mathrm{Rot}_{\mathbb{S},\ga} >>
\dgen{}{}{\ga}(\om;\mathbb{T}) @> \Ath=\mathrm{Div}_{\mathbb{T},\ga} >>
\ltom @> \A_{4}=\pi_{\mathsf{RT}} >>
\mathsf{RT}
\end{CD}$$
$$\begin{CD}
\{0\} @< \Azs=\pi_{\{0\}}<<
\ltom @< \Aos=\div\mathrm{Div}_{\mathbb{S}} <<
\dgen{}{}{}\dgen{}{}{}(\om;\mathbb{S}) @< \Ats=\sym\mathrm{Rot}_{\mathbb{T}} <<
\rgen{}{}{\sym}(\om;\mathbb{T}) @< \Aths=-\dev\mathrm{Grad} <<
\hoom @< \A_{4}^{*}=\iota_{\mathsf{RT}} <<
\mathsf{RT}
\end{CD}$$}
A typical system for a symmetric tensor field $S$ resp. a deviatoric (trace free) tensor field $T$ is
$$\mathrm{Rot}_{\mathbb{S},\ga}S=F,\qquad
\div\mathrm{Div}_{\mathbb{S}}S=g
\qquad\text{resp.}\qquad
\mathrm{Div}_{\mathbb{T},\ga}T=F,\qquad
\sym\mathrm{Rot}_{\mathbb{T}}T=G.$$
\item
linear elasticity
{\small$$\begin{CD}
\{0\} @> \Az=\iota_{\{0\}} >>
\hogaom @> \Ao=\sym\mathrm{Grad}_{\ga} >>
\rgen{}{}{}\rgen{}{\top}{\ga}(\om;\mathbb{S}) @> \At=\mathrm{Rot}\mathrm{Rot}^{\top}_{\mathbb{S},\ga} >>
\dgen{}{}{\ga}(\om;\mathbb{S}) @> \Ath=\mathrm{Div}_{\mathbb{S},\ga} >>
\ltom @> \A_{4}=\pi_{\mathsf{RM}} >>
\mathsf{RM}
\end{CD}$$
$$\begin{CD}
\{0\} @< \Azs=\pi_{\{0\}} <<
\ltom @< \Aos=-\mathrm{Div}_{\mathbb{S}} <<
\dgen{}{}{}(\om;\mathbb{S}) @< \Ats=\mathrm{Rot}\mathrm{Rot}^{\top}_{\mathbb{S}} <<
\rgen{}{}{}\rgen{}{\top}{}(\om;\mathbb{S}) @< \Aths=-\sym\mathrm{Grad} <<
\hoom @< \A_{4}^{*}=\iota_{\mathsf{RM}} <<
\mathsf{RM}
\end{CD}$$}
A typical system for a symmetric tensor field $S$ is
$$\mathrm{Rot}\mathrm{Rot}^{\top}_{\mathbb{S},\ga}S=F,\qquad
-\mathrm{Div}_{\mathbb{S}}S=G.$$
\end{itemize}
Here we denote the rigid motions and the global Raviart-Thomas fields of $\om$ by
\begin{align*}
\mathsf{RM}&:=\set{P|_{\om}}{P(x)=Qx+b,\,Q\in\rttt\text{ skew-symmetric},\,b\in\rt},\\
\mathsf{RT}&:=\set{P|_{\om}}{P(x)=a\,x+b,\,a\in\reals,\,b\in\rt}.
\end{align*}

\section{Functional Analysis Tool Box}
\label{sectoolbox}

Let $\ell\in\mathbb{Z}$. 
By the projection theorem the Helmholtz type decompositions
\begin{align}
\mylabel{helm}
\Hil&=N(\Al)\oplus_{\Hil}\ovl{R(\Als)},
&
\Hilpo&=N(\Als)\oplus_{\Hilpo}\ovl{R(\Al)}
\end{align}
hold 
and define in a natural way the reduced operators
\begin{align*}
\cAl&:=\Al|_{\ovl{R(\Als)}}:D(\cAl)\subset\ovl{R(\Als)}\to\ovl{R(\Al)},&
D(\cAl)&:=D(\Al)\cap\ovl{R(\Als)}=D(\Al)\cap N(\Al)^{\bot_{\Hil}},\\
\cAsl&:=\Als|_{\ovl{R(\Al)}}:D(\cAsl)\subset\ovl{R(\Al)}\to\ovl{R(\Als)},&
D(\cAsl)&:=D(\Als)\cap\ovl{R(\Al)}=D(\Als)\cap N(\Als)^{\bot_{\Hilpo}},
\end{align*}
which are also densely defined and closed linear operators.
We note that $\cAl$ and $\cAsl$ are indeed adjoint to each other, i.e.,
$(\cAl,\cAsl)$ is a dual pair as well. Now the inverse operators 
$$\cA_\ell^{-1}:R(\Al)\to D(\cAl),\qquad
(\cAsl)^{-1}:R(\Als)\to D(\cAsl)$$
exist, since $\cAl$ and $\cAsl$ are injective by definition,
and they are bijective, as, e.g., for $x\in D(\cAl)$
and $y:=\Al x\in R(\Al)$ we get $\cA_\ell^{-1}y=x$
by the injectivity of $\cAl$.
Furthermore, by the Helmholtz type decompositions \eqref{helm} we have
\begin{align}
\label{DacA}
D(\Al)&=N(\Al)\oplus_{\Hil}D(\cAl),&
D(\Als)&=N(\Als)\oplus_{\Hil}D(\cAsl)
\intertext{and thus we obtain for the ranges}
\label{RacA}
R(\Al)&=R(\cAl),&
R(\Als)&=R(\cAsl).
\end{align}

By the closed range and closed graph theorem we get immediately the following lemma.

\begin{lem}
\label{poincarerange}
The following assertions are equivalent:
\begin{itemize}
\item[\bf(i)] 
$\exists\,c_{\ell}\in(0,\infty)$ \quad 
$\forall\,x\in D(\cAl)$ \qquad 
$\norm{x}_{\Hil}\leq c_{\ell}\norm{\Al x}_{\Hilpo}$
\item[\bf(i${}^{*}$)] 
$\exists\,c^{*}_{\ell}\in(0,\infty)$ \quad 
$\forall\,y\in D(\cAsl)$ \qquad 
$\norm{y}_{\Hilpo}\leq c^{*}_{\ell}\norm{\Als y}_{\Hil}$
\item[\bf(ii)] 
$R(\Al)=R(\cAl)$ is closed in $\Hilpo$.
\item[\bf(ii${}^{*}$)] 
$R(\Als)=R(\cAsl)$ is closed in $\Hil$.
\item[\bf(iii)] 
$\cA_\ell^{-1}:R(\Al)\to D(\cAl)$ is continuous and bijective
with norm bounded by $(1+c_{\ell}^2)^{\oh}$.
\item[\bf(iii${}^{*}$)] 
$(\cAsl)^{-1}:R(\Als)\to D(\cAsl)$ is continuous and bijective
with norm bounded by $(1+c_{\ell}^{*}{}^2)^{\oh}$.
\end{itemize}
\end{lem}

\begin{proof}
Note that by the closed range theorem (ii) $\equi$ (ii${}^{*}$) holds.
Hence, by symmetry it is sufficient to show (i) $\equi$ (ii) $\equi$ (iii).
\begin{itemize}
\item[(i)$\impl$(ii)]
Pick a sequence $(y_{n})\subset R(\Al)$ converging to $y\in\Hilpo$ in $\Hilpo$.
By \eqref{RacA} there exists a sequence $(x_{n})\subset D(\cAl)$ with $y_{n}=\Al x_{n}$.
(i) implies that $(x_{n})$ is a Cauchy sequence in $\Hil$
and hence there exists some $x\in\Hil$ with $x_{n}\to x$ in $\Hil$.
As $\Al$ is closed, we get $x\in D(\Al)$ and $\Al x=y\in R(\Al)$.
\item[(ii)$\impl$(iii)]
Note that $\cA_\ell^{-1}:R(\Al)\to D(\cAl)$ is a densely defined and closed linear operator.
By (ii), $R(\Al)$ is closed and hence itself a Hilbert space.
By the closed graph theorem $\cA_\ell^{-1}$ is continuous.
\item[(iii)$\impl$(i)]
For $x\in D(\cAl)$ let $y:=\Al x\in R(\Al)$.
Then $x=\cA_\ell^{-1}y$ as $\cAl$ is injective\footnote{It holds 
$\Al\big(x-\cA_\ell^{-1}y\big)=0$ and thus $x=\cA_\ell^{-1}y$.}.
Therefore, 
$$\norm{x}_{\Hil}
=\bnorm{\cA_\ell^{-1}y}_{\Hil}
\leq\bnorm{\cA_\ell^{-1}}_{R(\Al),R(\Als)}\norm{y}_{\Hilpo}
=c_{\ell}\norm{\Al x}_{\Hilpo}$$
with $c_{\ell}:=\bnorm{\cA_\ell^{-1}}_{R(\Al),R(\Als)}$.
\end{itemize}
If (i) holds we have for $y\in R(\Al)$ and $x:=\cA_\ell^{-1}y\in D(\cAl)$
$$\bnorm{\cA_\ell^{-1}y}_{\Hil}
\leq c_{\ell}\norm{\Al x}_{\Hilpo}
=c_{\ell}\norm{y}_{\Hilpo}$$
and hence
\begin{align*}
\bnorm{\cA_\ell^{-1}}_{R(\Al),R(\Als)}
&=\sup_{0\neq y\in R(\Al)}\frac{\bnorm{\cA_\ell^{-1}y}_{\Hil}}{\norm{y}_{\Hilpo}}
\leq c_{\ell},\\
\bnorm{\cA_\ell^{-1}}_{R(\Al),D(\cAl)}^2
&=\sup_{0\neq y\in R(\Al)}\frac{\bnorm{\cA_\ell^{-1}y}_{D(\Al)}^2}{\norm{y}_{\Hilpo}^2}
=\sup_{0\neq y\in R(\Al)}\frac{\bnorm{\cA_\ell^{-1}y}_{\Hil}^2+\norm{y}_{\Hilpo}^2}{\norm{y}_{\Hilpo}^2}
\leq c_{\ell}^2+1,
\end{align*}
finishing the proof.
\end{proof}

From now on we assume that we always choose the best 
Friedrichs/Poincar\'e type constants $c_{\ell},c^{*}_{\ell}$, if they exist in $(0,\infty)$, i.e.,
$c_{\ell}$ and $c^{*}_{\ell}$ are given by the Rayleigh quotients
$$\frac{1}{c_{\ell}}
:=\inf_{0\neq x\in D(\cAl)}\frac{\norm{\Al x}_{\Hilpo}}{\norm{x}_{\Hil}},\qquad
\frac{1}{c^{*}_{\ell}}
:=\inf_{0\neq y\in D(\cAsl)}\frac{\norm{\Als y}_{\Hil}}{\norm{y}_{\Hilpo}}.$$
Moreover, we see
\begin{align}
\mylabel{classup}
c_{\ell}
&=\sup_{0\neq x\in D(\cAl)}\frac{\norm{x}_{\Hil}}{\norm{\Al x}_{\Hilpo}}
=\sup_{0\neq y\in R(\Al)}\frac{\bnorm{\cA_\ell^{-1}y}_{\Hil}}{\norm{y}_{\Hilpo}}
=\bnorm{\cA_\ell^{-1}}_{R(\Al),R(\Als)},
\intertext{as $0\neq x\in D(\cAl)$ implies $0\neq\Al x$
and for $y:=\Al x$ with $x\in D(\cAl)$ we have
$\cA_\ell^{-1}y=x$, both by the injectivity of $\cAl$.
Analogously, we get}
\mylabel{clsassup}
c^{*}_{\ell}
&=\sup_{0\neq y\in D(\cAsl)}\frac{\norm{y}_{\Hilpo}}{\norm{\Als y}_{\Hil}}
=\sup_{0\neq x\in R(\Als)}\frac{\bnorm{(\cAsl)^{-1}x}_{\Hilpo}}{\norm{x}_{\Hil}}
=\bnorm{(\cAsl)^{-1}}_{R(\Als),R(\Al)}.
\end{align}

\begin{lem}
\label{consteq}
Assume that $c_{\ell}\in(0,\infty)$ or $c^{*}_{\ell}\in(0,\infty)$ exists. Then $c_{\ell}=c_{\ell}^{*}$.
\end{lem}

We note that also in the case $c_{\ell}=\infty$ or $c_{\ell}^{*}=\infty$
we have $c_{\ell}=c_{\ell}^{*}=\infty$.

\begin{proof}
Let, e.g., $c^{*}_{\ell}$ exist in $(0,\infty)$. 
By Lemma \ref{poincarerange} also $c_{\ell}$ exists in $(0,\infty)$
and the ranges $R(\Al)=R(\cAl)$ and $R(\Als)=R(\cAls)$ are closed. 
Then for $x\in D(\cAl)=D(\Al)\cap R(\Als)$ 
there is $y\in D(\cAsl)$ with $x=\Als y$. 
More precisely, $y:=(\cAsl)^{-1}x\in D(\cAsl)$ is uniquely determined and we have
$\norm{y}_{\Hilpo}\leq c^{*}_{\ell}\norm{\Als y}_{\Hil}$. But then
$$\norm{x}_{\Hil}^2
=\scp{x}{\Als y}_{\Hil}
=\scp{\Al x}{y}_{\Hilpo}
\leq \norm{\Al x}_{\Hilpo}\norm{y}_{\Hilpo}
\leq c^{*}_{\ell}\norm{\Al x}_{\Hilpo}\norm{\Als y}_{\Hil},$$
yielding $\norm{x}_{\Hil}\leq c^{*}_{\ell}\norm{\Al x}_{\Hilpo}$.
Therefore, $c_{\ell}\leq c^{*}_{\ell}$ 
and by symmetry we obtain $c_{\ell}=c^{*}_{\ell}$.
\end{proof}

A standard indirect argument shows the following lemma.

\begin{lem}
\label{cptembtoolbox}
Let $D(\cAl)=D(\Al)\cap\ovl{R(\Als)}\dhookrightarrow\Hil$ be compact. 
Then the assertions of Lemma \ref{poincarerange} and Lemma \ref{consteq} hold.
Moreover, the inverse operators 
$$\cA_\ell^{-1}:R(\Al)\to R(\Als),\quad
(\cAls)^{-1}:R(\Als)\to R(\Al)$$
are compact with norms 
$\bnorm{\cA_\ell^{-1}}_{R(\Al),R(\Als)}
=\bnorm{(\cAls)^{-1}}_{R(\Als),R(\Al)}
=c_{\ell}$.
\end{lem}

\begin{proof}
If, e.g., Lemma \ref{poincarerange} (i) was wrong,
there exists a sequence $(x_{n})\subset D(\cAl)$ with $\norm{x_{n}}_{\Hil}=1$
and $\Al x_{n}\to0$. As $(x_{n})$ is bounded in $D(\cAl)$
we can extract a subsequence, again denoted by $(x_{n})$,
with $x_{n}\to x\in\Hil$ in $\Hil$. Since $\Al$ is closed,
we have $x\in D(\Al)$ and $\Al x=0$. Hence $x\in N(\Al)$.
On the other hand, $(x_{n})\subset D(\cAl)\subset\ovl{R(\Als)}=N(\Al)^{\bot}$
implies $x\in N(\Al)^{\bot}$. Thus $x=0$, in contradiction to $1=\norm{x_{n}}_{\Hil}\to\norm{x}_{\Hil}=0$.
\end{proof}

\begin{lem}
\label{cptemb}
The embedding $D(\cAl)\dhookrightarrow\Hil$ is compact,
if and only if the embedding $D(\cAsl)\dhookrightarrow\Hilpo$ is compact.
In this case all assertions of Lemma \ref{poincarerange} and Lemma \ref{consteq} are valid.
\end{lem}

\begin{proof}
By symmetry it is enough to show one direction.
Let, e.g., the embedding $D(\cAl)\dhookrightarrow\Hil$ be compact.
By Lemma \ref{poincarerange} and Lemma \ref{cptembtoolbox}, 
especially $R(\Al)=R(\cAl)$ and $R(\Als)=R(\cAsl)$ are closed.
Let $(y_{n})\subset D(\cAsl)=D(\Als)\cap R(\Al)$ be a $D(\Als)$-bounded sequence.
We pick a sequence $(x_{n})\subset D(\cAl)$ with $y_{n}=\Al x_{n}$, i.e., 
$x_{n}=\cA_\ell^{-1}y_{n}$. As $\cA_\ell^{-1}:R(\Al)\to D(\cAl)$ is continuous,
$(x_{n})$ is bounded in $D(\cAl)$ and thus contains a subsequence, again denoted by $(x_{n})$,
converging in $\Hil$ to some $x\in\Hil$. Now
\begin{align*}
\norm{y_{n}-y_{m}}_{\Hilpo}^2
=\bscp{y_{n}-y_{m}}{\Al(x_{n}-x_{m})}_{\Hilpo}
=\bscp{\Als(y_{n}-y_{m})}{x_{n}-x_{m}}_{\Hil}
\leq c\,\norm{x_{n}-x_{m}}_{\Hil}
\end{align*}
as $(y_{n})$ is $D(\Als)$-bounded.
Finally, we see that $(y_{n})$ is a Cauchy sequence in $\Hilpo$.
\end{proof}

Let us summarize:

\begin{cor}
\label{cortoolboxone}
Let $R(\Al)$ be closed. Then
$$\frac{1}{c_{\ell}}
=\inf_{0\neq x\in D(\cAl)}\frac{\norm{\Al x}_{\Hilpo}}{\norm{x}_{\Hil}}
=\inf_{y\in D(\cAsl)}\frac{\norm{\Als y}_{\Hil}}{\norm{y}_{\Hilpo}}$$
exists in $(0,\infty)$. Furthermore:
\begin{itemize}
\item[\bf(i)]
The Poincar\'e type estimates
\begin{align*}
\forall\,x&\in D(\cAl)&\norm{x}_{\Hil}&\leq c_{\ell}\norm{\Al x}_{\Hilpo},\\
\forall\,y&\in D(\cAsl)&\norm{y}_{\Hilpo}&\leq c_{\ell}\norm{\Als y}_{\Hil}
\end{align*}
hold.
\item[\bf(ii)]
The ranges $R(\Al)=R(\cAl)$ and $R(\Als)=R(\cAsl)$ are closed.
Moreover,
$D(\cAl)=D(\Al)\cap R(\Als)$ and $D(\cAsl)=D(\Als)\cap R(\Al)$ with
$$\cAl:D(\cAl)\subset R(\Als)\to R(\Al),\quad
\cAsl:D(\cAsl)\subset R(\Al)\to R(\Als).$$
\item[\bf(iii)]
The Helmholtz type decompositions
\begin{align*}
\Hil
&=N(\Al)\oplus_{\Hil}R(\Als),
&
\Hilpo
&=N(\Als)\oplus_{\Hilpo}R(\Al),\\
D(\Al)
&=N(\Al)\oplus_{\Hil}D(\cAl),
&
D(\Als)
&=N(\Als)\oplus_{\Hilpo}D(\cAsl)
\end{align*}
hold.
\item[\bf(iv)]
The inverse operators 
$$\cA_\ell^{-1}:R(\Al)\to D(\cAl),\quad
(\cAsl)^{-1}:R(\Als)\to D(\cAsl)$$
are continuous and bijective with norms 
$\bnorm{\cA_\ell^{-1}}_{R(\Al),D(\cAl)}
=\bnorm{(\cAsl)^{-1}}_{R(\Als),D(\cAsl)}
=(1+c_{\ell}^2)^{\oh}$
and 
$\bnorm{\cA_\ell^{-1}}_{R(\Al),R(\Als)}
=\bnorm{(\cAsl)^{-1}}_{R(\Als),R(\Al)}
=c_{\ell}$.
\end{itemize} 
\end{cor}

\begin{cor}
\label{cortoolboxtwo}
Let $D(\cAl)\dhookrightarrow\Hil$ be compact. 
Then $R(\Al)$ is closed and the assertions of Corollary \ref{cortoolboxone} hold.
Moreover, the inverse operators 
$$\cA_\ell^{-1}:R(\Al)\to R(\Als),\quad
(\cAsl)^{-1}:R(\Als)\to R(\Al)$$
are compact. 
\end{cor}

So far, we did not use the complex property \eqref{exseqAl}.
Hence Lemma \ref{poincarerange}, Lemma \ref{consteq}, Lemma \ref{cptembtoolbox}, Lemma \ref{cptemb},
and Corollary \ref{cortoolboxone}, Corollary \ref{cortoolboxtwo}
hold without the complex property \eqref{exseqAl}.
Now the complex property \eqref{exseqAl} enters the theory.
Recall the Helmholtz type decompositions \eqref{helm} in the form
\begin{align*}
\Hil=N(\Al)\oplus_{\Hil}\ovl{R(\Als)}
=\ovl{R(\Almo)}\oplus_{\Hil}N(\Almos)
\end{align*}
hold. Then the complex properties \eqref{exseqAl}-\eqref{exseqAls} yield
$$N(\Al)=\ovl{R(\Almo)}\oplus_{\Hil}K_{\ell},\quad
N(\Almos)=K_{\ell}\oplus_{\Hil}\ovl{R(\Als)},\qquad
K_{\ell}=N(\Al)\cap N(\Almos).$$
Therefore, we get the refined Helmholtz type decomposition
\begin{align}
\mylabel{helmref}
\Hil=\ovl{R(\Almo)}\oplus_{\Hil}K_{\ell}\oplus_{\Hil}\ovl{R(\Als)}.
\end{align}

\begin{lem}
\label{lemtoolboxexseq}
The refined Helmholtz type decompositions
\begin{align*}
\Hil
&=\ovl{R(\Almo)}\oplus_{\Hil}K_{\ell}\oplus_{\Hil}\ovl{R(\Als)},
&
K_{\ell}
&=N(\Al)\cap N(\Almos),\\
N(\Al)
&=\ovl{R(\Almo)}\oplus_{\Hil}K_{\ell},
&
N(\Almos)
&=K_{\ell}\oplus_{\Hil}\ovl{R(\Als)},\\
\ovl{R(\cAlmo)}=\ovl{R(\Almo)}
&=N(\Al)\ominus_{\Hil}K_{\ell},
&
\ovl{R(\cAsl)}=\ovl{R(\Als)}
&=N(\Almos)\ominus_{\Hil}K_{\ell},\\
D(\Al)
&=\ovl{R(\Almo)}\oplus_{\Hil}K_{\ell}\oplus_{\Hil}D(\cAl),
&
D(\Almos)
&=D(\cAslmo)\oplus_{\Hil}K_{\ell}\oplus_{\Hil}\ovl{R(\Als)},\\
D_{\ell}
&=D(\cAslmo)\oplus_{\Hil}K_{\ell}\oplus_{\Hil}D(\cAl),
&
D_{\ell}
&=D(\Al)\cap D(\Almos)
\end{align*}
hold. If the range $R(\Almo)$ or $R(\Al)$ is closed, 
the respective closure bars can be dropped
and the assertions of Corollary \ref{cortoolboxone} are valid.
Especially, if $R(\Almo)$ and $R(\Al)$ are closed, 
the assertions of Corollary \ref{cortoolboxone} and
the refined Helmholtz type decompositions
\begin{align*}
\Hil
&=R(\Almo)\oplus_{\Hil}K_{\ell}\oplus_{\Hil}R(\Als),
&
K_{\ell}
&=N(\Al)\cap N(\Almos),\\
N(\Al)
&=R(\Almo)\oplus_{\Hil}K_{\ell},
&
N(\Almos)
&=K_{\ell}\oplus_{\Hil}R(\Als),\\
R(\cAlmo)=R(\Almo)
&=N(\Al)\ominus_{\Hil}K_{\ell},
&
R(\cAsl)=R(\Als)
&=N(\Almos)\ominus_{\Hil}K_{\ell},\\
D(\Al)
&=R(\Almo)\oplus_{\Hil}K_{\ell}\oplus_{\Hil}D(\cAl),
&
D(\Almos)
&=D(\cAslmo)\oplus_{\Hil}K_{\ell}\oplus_{\Hil}R(\Als),\\
D_{\ell}
&=D(\cAslmo)\oplus_{\Hil}K_{\ell}\oplus_{\Hil}D(\cAl),
&
D_{\ell}
&=D(\Al)\cap D(\Almos)
\end{align*}
hold.
\end{lem}

Observe that
\begin{align}
\begin{split}
\label{DAlAlmosDl}
D(\cAl)
&=D(\Al)\cap\ovl{R(\Als)}
\subset D(\Al)\cap N(\Almos)
\subset D(\Al)\cap D(\Almos)
=D_{\ell},\\
D(\cAlmos)
&=D(\Almos)\cap\ovl{R(\Almo)}
\subset D(\Almos)\cap N(\Al)
\subset D(\Almos)\cap D(\Al)
=D_{\ell}.
\end{split}
\end{align}

\begin{lem}
\label{lemcptAlAslmo}
The embeddings $D(\cAl)\dhookrightarrow\Hil$, $D(\cAlmo)\dhookrightarrow\Hilmo$,
and $K_{\ell}\dhookrightarrow\Hil$ are compact,
if and only if the embedding $D_{\ell}\dhookrightarrow\Hil$ is compact.
In this case, $K_{\ell}$ has finite dimension. 
\end{lem}

\begin{proof}
Note that, by Lemma \ref{cptemb}, $D(\cAlmo)\dhookrightarrow\Hilmo$ is compact, if and only if
$D(\cAslmo)\dhookrightarrow\Hil$ is compact.
\begin{itemize}
\item[$\impl$:]
Let $(x_{n})\subset D_{\ell}$ be a $D_{\ell}$-bounded sequence. 
By the refined Helmholtz type decomposition 
of Lemma \ref{lemtoolboxexseq} we decompose
$$x_{n}=a^{*}_{n}+k_{n}+a_{n}\in D(\cAslmo)\oplus_{\Hil}K_{\ell}\oplus_{\Hil}D(\cAl).$$
with $\Al x_{n}=\Al a_{n}$ and $\Almos x_{n}=\Almos a^{*}_{n}$.
Hence $(a_{n})$ is bounded in $D(\cAl)$ and $(a^{*}_{n})$ is bounded in $D(\cAslmo)$
and we can extract $\Hil$-converging subsequences of $(a_{n})$, $(a^{*}_{n})$, and $(k_{n})$.
\item[$\lpmi$:]
If $D_{\ell}\dhookrightarrow\Hil$ is compact, so is $K_{\ell}\dhookrightarrow\Hil$.
Moreover, by \eqref{DAlAlmosDl}
$$D(\cAl)\subset D_{\ell}\dhookrightarrow\Hil,\qquad
D(\cAslmo)\subset D_{\ell}\dhookrightarrow\Hil.$$
\end{itemize}
Finally, if $K_{\ell}\dhookrightarrow\Hil$ is compact,
the unit ball in $K_{\ell}$ is compact, showing that $K_{\ell}$ has finite dimension.
\end{proof}

Lemma \ref{lemcptAlAslmo} implies immediately the following result.

\begin{cor}
\label{cortoolboxexseq}
Let $D_{\ell}\dhookrightarrow\Hil$ be compact.
Then $R(\Almo)$ and $R(\Al)$ are closed, and, besides the assertions of Corollary \ref{cortoolboxtwo},
the refined Helmholtz type decompositions of Lemma \ref{lemtoolboxexseq} hold
and the cohomology group $K_{\ell}$ is finite dimensional.
\end{cor}

\begin{rem}
\label{remDl}
Under the assumption that the embedding $D_{\ell}\dhookrightarrow\Hil$ is compact,
all the assertions of this section hold. Especially, the (short) complex
$$\begin{CD}
D(\Almo) @> \Almo >>
D(\Al) @> \Al >>
\Hilpo
\end{CD}$$
together with its adjoint complex
$$\begin{CD}
\Hilmo @< \Almos <<
D(\Almos) @< \Als <<
D(\Als)
\end{CD}$$
is closed. These complexes are even exact, if additionally $K_{\ell}=\{0\}$.
\end{rem}

Defining and recalling the orthonormal projectors
\begin{align}
\mylabel{orthoporjdef}
\pi_{\Almo}:=\pi_{\ovl{R(\Almo)}}:\Hil
&\to\ovl{R(\Almo)},
&
\pi_{\Als}:=\pi_{\ovl{R(\Als)}}:\Hil
&\to\ovl{R(\Als)},
&
\pil:\Hil
&\to K_{\ell},
\end{align}
we have $\pil=1-\pi_{\Almo}-\pi_{\Als}$ as well as
\begin{align*}
\pi_{\Almo}\Hil
&=\pi_{\Almo}D(\Al)
=\pi_{\Almo}N(\Al)
=\ovl{R(\Almo)}
=\ovl{R(\cAlmo)},\\
\pi_{\Als}\Hil
&=\pi_{\Als}D(\Almos)
=\pi_{\Als}N(\Almos)
=\ovl{R(\Als)}
=\ovl{R(\cAsl)}
\end{align*}
and
\begin{align*}
\pi_{\Almo}D(\Almos)
&=\pi_{\Almo}D_{\ell}
=D(\cAslmo),
&
\pi_{\Als}D(\Al)
&=\pi_{\Als}D_{\ell}
=D(\cAl).
\end{align*}
Moreover
\begin{align*}
\forall\;\xi
&\in D(\Almos)
&
\pi_{\Almo}\xi
&\in D(\cAslmo)
&
&\wedge
&
\Almos\pi_{\Almo}\xi
&=\Almos\xi,\\
\forall\;\zeta
&\in D(\Al)
&
\pi_{\Als}\zeta
&\in D(\cAl)
&
&\wedge
&
\Al\pi_{\Als}\zeta
&=\Al\zeta.
\end{align*}
We also introduce the orthogonal projectors onto the kernels
\begin{align*}
\pi_{N(\Almos)}:=1-\pi_{\Almo}:\Hil
&\to N(\Almos),
&
\pi_{N(\Al)}:=1-\pi_{\Als}:\Hil
&\to N(\Al).
\end{align*}

\section{Solution Theory and Variational Formulations}
\mylabel{secsoltheovarform}

From now on and throughout this paper we suppose the following.

\begin{genass}
\label{genass}
$R(\Ao)$ and $R(\At)$ are closed and $K_{2}$ is finite dimensional.
\end{genass}

\begin{rem}
\label{genassrem}
The General Assumption \ref{genass} is satisfied, 
if, e.g., $D_{2}\dhookrightarrow\Hit$ is compact.
The finite dimension of the cohomology group $K_{2}$ may be dropped.
\end{rem}

\subsection{First Order Systems}

We recall the linear first order system \eqref{Aprob} from the introduction:
Find $x\in D_{2}=D(\At)\cap D(\Aos)$ such that
\begin{align}
\begin{split}
\label{Aprobsoltheo}
\At x&=f,\\
\Aos x&=g,\\
\pit x&=k.
\end{split}
\end{align}

\begin{theo}
\label{soltheofos}
\eqref{Aprobsoltheo} is uniquely solvable in $D_{2}$, if and only if
$f\in R(\At)$, $g\in R(\Aos)$, and $k\in K_{2}$. 
The unique solution $x\in D_{2}$ is given by 
\begin{align*}
x:=x_{f}+x_{g}+k
&\in D(\cAt)\oplus_{\Hit}D(\cAos)\oplus_{\Hit}K_{2}=D_{2},\\
x_{f}:=\cA_{2}^{-1}f
&\in D(\cAt)=D(\cAt)\cap D_{2},\\
x_{g}:=(\cAos)^{-1}g
&\in D(\cAos)=D(\cAos)\cap D_{2}
\end{align*}
and depends continuously on the data, i.e.,
$\norm{x}_{\Hit}
\leq c_{2}\norm{f}_{\Hith}
+c_{1}\norm{g}_{\Hio}
+\norm{k}_{\Hit}$,
as 
$$\norm{x_{f}}_{\Hit}\leq c_{2}\norm{f}_{\Hith},\qquad
\norm{x_{g}}_{\Hit}\leq c_{1}\norm{g}_{\Hio}.$$
It holds 
$$\pi_{\Ats}x=x_{f},\qquad
\pi_{\Ao}x=x_{g},\qquad
\pit x=k,\qquad
\norm{x}_{\Hit}^2=\norm{x_{f}}_{\Hit}^2+\norm{x_{g}}_{\Hit}^2+\norm{k}_{\Hit}^2.$$
\end{theo}

\begin{proof}
As pointed out in the introduction, we just need to show existence.
We use the results of Section \ref{sectoolbox}.
Let $f\in R(\At)$, $g\in R(\Aos)$, $k\in K_{2}$ and define
$x$, $x_{f}$, and $x_{g}$ according to the theorem.
For the orthogonality we refer to Lemma \ref{lemtoolboxexseq}.
Moreover, $x_{f}$, $x_{g}$, and $k$ solve the linear systems
\begin{align*}
\At x_{f}&=f,
&
\At x_{g}&=0,
&
\At k&=0,\\
\Aos x_{f}&=0,
&
\Aos x_{g}&=g,
&
\Aos k&=0,\\
\pit x_{f}&=0,
&
\pit x_{g}&=0,
&
\pit k&=k.
\end{align*}
Thus $x$ solves \eqref{Aprobsoltheo}
and we have by Corollary \ref{cortoolboxone} 
$\norm{x_{f}}_{\Hit}\leq c_{2}\norm{f}_{\Hith}$
and $\norm{x_{g}}_{\Hit}\leq c_{1}\norm{g}_{\Hio}$,
which completes the proof of the solution theory.
\end{proof} 

\begin{rem}
\label{remsoltheofos}
By orthogonality 
and with $\At x=\At x_{f}=f$
and $\Aos x=\Aos x_{g}=g$
we even have
\begin{align*}
\norm{x}_{\Hit}^2
&=\bnorm{x_{f}}_{\Hit}^2
+\bnorm{x_{g}}_{\Hit}^2
+\norm{k}_{\Hit}^2
\leq c_{2}^2\norm{f}_{\Hith}^2
+c_{1}^2{}\norm{g}_{\Hio}^2
+\norm{k}_{\Hit}^2,\\
\norm{x}_{D_{2}}^2
&=\bnorm{x_{f}}_{\Hit}^2
+\norm{f}_{\Hith}^2
+\bnorm{x_{g}}_{\Hit}^2
+\norm{g}_{\Hio}^2
+\norm{k}_{\Hit}^2
\leq(1+c_{2}^2)\norm{f}_{\Hith}^2
+(1+c_{1}^2)\norm{g}_{\Hio}^2
+\norm{k}_{\Hit}^2.
\end{align*}
\end{rem}

\subsubsection{Variational Formulations}

Recall the partial solutions
\begin{align}
\label{recalldefxfxg}
\begin{split}
x_{f}:=\cA_{2}^{-1}f
&\in D(\cAt)=D(\At)\cap R(\Ats)=D(\At)\cap N(\Aos)\cap K_{2}^{\bot_{\Hit}},\\
x_{g}:=(\cAos)^{-1}g
&\in D(\cAos)=D(\Aos)\cap R(\Ao)=D(\Aos)\cap N(\At)\cap K_{2}^{\bot_{\Hit}}.
\end{split}
\end{align}
There are at least two obvious ways to get variational formulations for finding 
each of the partial solutions $x_{f}$ and $x_{g}$.
Looking at $x_{f}\in D(\cAt)$ the first idea is to multiply 
the equation $\At x_{f}=f$
by $\At\xi$ with some $\xi\in D(\cAt)$ leading to
\begin{align*}
\forall\,\xi\in D(\cAt)\qquad
\scp{\At x_{f}}{\At\xi}_{\Hith}
=\scp{f}{\At\xi}_{\Hith},
\end{align*}
which is a weak formulation of the second order equation
$$\Ats\At x_{f}=\Ats f,$$
more precisely of $\Ats(\At x_{f}-f)=0$.
While the latter choice was straight forward to find $x_{f}$ itself, 
the next choice searches for a potential $y_{f}$,
e.g., $y_{f}:=(\cAts)^{-1}x_{f}\in D(\cAts)$, of 
$$x_{f}=\Ats y_{f}\in D(\cAt)=D(\At)\cap R(\Ats)=D(\At)\cap R(\cAts),$$
see Remark \ref{remsoltheofos}. Multiplying by $\Ats\phi$ with some $\phi\in D(\cAts)$ gives
\begin{align*}
\forall\,\phi\in D(\cAts)\qquad
\scp{\Ats y_{f}}{\Ats\phi}_{\Hit}
=\scp{x_{f}}{\Ats\phi}_{\Hit}
=\scp{\At x_{f}}{\phi}_{\Hith}
=\scp{f}{\phi}_{\Hith},
\end{align*}
which is a weak formulation of the second order equation
$$\At\Ats y_{f}=f.$$

Similar ideas apply to find corresponding variational formulations for $x_{g}$ as well.

\begin{theo}
\label{soltheofosvarform}
The partial solutions $x_{f}$ and $x_{g}$ in Theorem \ref{soltheofos} 
can be found by the following four variational formulations: 
\begin{itemize}
\item[\bf(i)]
There exists a unique $\ti{x}_{f}\in D(\cAt)$ such that
\begin{align}
\mylabel{varxf}
\forall\,\xi\in D(\cAt)\qquad
\scp{\At\ti{x}_{f}}{\At\xi}_{\Hith}
=\scp{f}{\At\xi}_{\Hith}.
\end{align}
\eqref{varxf} is even satisfied for all $\xi\in D(\At)$.
Moreover, $\At\ti{x}_{f}=f$ holds if and only if $f\in R(\At)$.
In this case $\ti{x}_{f}=x_{f}$.
\item[\bf(i')]
There exists a unique potential $y_{f}\in D(\cAts)$ such that
\begin{align}
\mylabel{varyf}
\forall\,\phi\in D(\cAts)\qquad
\scp{\Ats y_{f}}{\Ats\phi}_{\Hit}
=\scp{f}{\phi}_{\Hith}.
\end{align}
\eqref{varyf} even holds for all $\phi\in D(\Ats)$ 
if and only if $f\in R(\At)$.
In this case we have 
$$\Ats y_{f}\in D(\At)\cap R(\Ats)=D(\cAt)$$ 
with $\At\Ats y_{f}=f$ and hence $\Ats y_{f}=x_{f}$.
\item[\bf(ii)]
There exists a unique $\ti{x}_{g}\in D(\cAos)$ such that
\begin{align}
\mylabel{varxg}
\forall\,\zeta\in D(\cAos)\qquad
\scp{\Aos\ti{x}_{g}}{\Aos\zeta}_{\Hio}
=\scp{g}{\Aos\zeta}_{\Hio}.
\end{align}
\eqref{varxg} is even satisfied for all $\zeta\in D(\Aos)$.
Moreover, $\Aos\ti{x}_{g}=g$ holds if and only if $g\in R(\Aos)$.
In this case $\ti{x}_{g}=x_{g}$.
\item[\bf(ii')]
There exists a unique potential $z_{g}\in D(\cAo)$ such that
\begin{align}
\mylabel{varzg}
\forall\,\varphi\in D(\cAo)\qquad
\scp{\Ao z_{g}}{\Ao\varphi}_{\Hit}
=\scp{g}{\varphi}_{\Hio}.
\end{align}
\eqref{varzg} even holds for all $\varphi\in D(\Ao)$
if and only if $g\in R(\Aos)$.
In this case we have 
$$\Ao z_{g}\in D(\Aos)\cap R(\Ao)=D(\cAos)$$ 
with $\Aos\Ao z_{g}=g$ and thus $\Ao z_{g}=x_{g}$.
\end{itemize}
\end{theo}

\begin{proof}
\eqref{varxf} is strictly positive (or coercive) over $D(\cAt)$ 
by the Friedrichs/Poincar\'e type estimates of Corollary \ref{cortoolboxone} (i)
and hence a unique $\ti{x}_{f}\in D(\cAt)$ exists
by Riesz' representation theorem (or Lax-Milgram's lemma) solving \eqref{varxf}.
By \eqref{RacA}, i.e., $R(\cAt)=R(\At)$, \eqref{varxf} holds for all $\xi\in D(\At)$.
Hence 
$$\forall\,\xi\in D(\At)\qquad
\scp{\At\ti{x}_{f}-f}{\At\xi}_{\Hith}=0,$$
yielding $\At\ti{x}_{f}-f\in R(\At)^{\bot_{\Hith}}$.
Thus, if $f\in R(\At)$, we see 
$\At\ti{x}_{f}-f\in R(\At)\cap R(\At)^{\bot_{\Hith}}=\{0\}$, i.e., $\At\ti{x}_{f}=f$.
As $\ti{x}_{f}\in D(\cAt)$ conclude $\ti{x}_{f}=x_{f}$ by the injectivity of $\cAt$,
which completes the proof of (i).

\eqref{varyf} is strictly positive over $D(\cAts)$ 
by Corollary \ref{cortoolboxone} (i)
and thus a unique $y_{f}\in D(\cAts)$ exists
by Riesz' representation theorem solving \eqref{varyf}.
Using Corollary \ref{cortoolboxone} (iii) or Lemma \ref{lemtoolboxexseq}
we can split any $\phi\in D(\Ats)=N(\Ats)\oplus_{\Hith}D(\cAts)$
into $\phi=\phi_{N}+\phi_{R}$ (null space and range) 
with $\phi_{N}\in N(\Ats)$, $\phi_{R}\in D(\cAts)$,
and $\Ats\phi=\Ats\phi_{R}$. 
Let $f\in R(\At)$. Utilizing \eqref{varyf} for $\phi_{R}$
and orthogonality, i.e., $f\in R(\At)=N(\Ats)^{\bot_{\Hith}}$, we get
$$\scp{\Ats y_{f}}{\Ats\phi}_{\Hit}
=\scp{\Ats y_{f}}{\Ats\phi_{R}}_{\Hit}
=\scp{f}{\phi_{R}}_{\Hith}
=\scp{f}{\phi}_{\Hith}.$$
Therefore, \eqref{varyf} holds for all $\phi\in D(\Ats)$.
On the other hand, if \eqref{varyf} holds for all $\phi\in D(\Ats)$,
then $\Ats y_{f}\in D(\At)$ with $\At\Ats y_{f}=f$.
Hence\footnote{Another proof is the following: 
Pick $\phi\in N(\Ats)$ and get by \eqref{varyf} directly $\scp{f}{\phi}_{\Hith}=0$.
Thus $f\in N(\Ats)^{\bot_{\Hith}}=R(\At)$.} 
$f\in R(\At)$.
Therefore, if $f$ belongs to $R(\At)$, we obtain 
$\Ats y_{f}\in D(\At)\cap R(\Ats)=D(\cAt)$
with $\At\Ats y_{f}=f$ and hence $\Ats y_{f}=x_{f}$,
again by the injectivity of $\cAt$.


Analogously, we prove (ii) and (ii').
\end{proof}

\begin{rem}
\label{soltheofosvarformremmatrix}
Note that
\begin{align*}
x_{f}=\cA_{2}^{-1}f&\in D(\cAt),
&
x_{g}=(\cAos)^{-1}g&\in D(\cAos),\\
y_{f}=(\cAts)^{-1}x_{f}=(\cAts)^{-1}\cA_{2}^{-1}f&\in D(\cAts),
&
z_{g}=\cA_{1}^{-1}x_{g}=\cA_{1}^{-1}(\cAos)^{-1}g&\in D(\cAo)
\end{align*}
hold with $\At\Ats y_{f}=f$ and $\Aos\Ao z_{g}=g$.
Hence $x_{f}$, $x_{g}$, $k$, and $y_{f}$, $z_{g}$ solve 
the first resp. second order systems
\begin{align*}
\At x_{f}&=f,
&
\At x_{g}&=0,
&
\At k&=0,
&
\Ats y_{f}&=x_{f},
&
\At\Ats y_{f}&=f,
&
\Ao z_{g}&=x_{g},
&
\Aos\Ao z_{g}&=g,\\
\Aos x_{f}&=0,
&
\Aos x_{g}&=g,
&
\Aos k&=0,
&
\Ath y_{f}&=0,
&
\Ath y_{f}&=0,
&
\Azs z_{g}&=0,
&
\Azs z_{g}&=0,\\
\pit x_{f}&=0,
&
\pit x_{g}&=0,
&
\pit k&=k,
&
\pith y_{f}&=0,
&
\pith y_{f}&=0,
&
\pio z_{g}&=0,
&
\pio z_{g}&=0.
\end{align*}
Moreover:
\begin{itemize}
\item[\bf(i)]
\eqref{varxf} is a weak formulation of 
$$\Ats\At\ti{x}_{f}=\Ats f,\qquad
\Aos\ti{x}_{f}=0,\qquad
\pit\ti{x}_{f}=0,$$
i.e., in formal matrix notation
$$\begin{bmatrix}
\Ats\At\\
\Aos\\
\pit
\end{bmatrix}
\begin{bmatrix}
\ti{x}_{f}
\end{bmatrix}
=\begin{bmatrix}
\Ats f\\
0\\
0
\end{bmatrix}.$$
\item[\bf(i')]
\eqref{varyf} is a weak formulation of
$$\At\Ats y_{f}=f,\qquad
\Ath y_{f}=0,\qquad
\pith y_{f}=0,$$
i.e., in formal matrix notation
$$\begin{bmatrix}
\At\Ats\\
\Ath\\
\pith
\end{bmatrix}
\begin{bmatrix}
y_{f}
\end{bmatrix}
=\begin{bmatrix}
f\\
0\\
0
\end{bmatrix}.$$
\item[\bf(ii)]
\eqref{varxg} is a weak formulation of
$$\Ao\Aos\ti{x}_{g}=\Ao g,\qquad
\At\ti{x}_{g}=0,\qquad
\pit\ti{x}_{g}=0,$$
i.e., in formal matrix notation
$$\begin{bmatrix}
\Ao\Aos\\
\At\\
\pit
\end{bmatrix}
\begin{bmatrix}
\ti{x}_{g}
\end{bmatrix}
=\begin{bmatrix}
\Ao g\\
0\\
0
\end{bmatrix}.$$
\item[\bf(ii')]
\eqref{varzg} is a weak formulation of
$$\Aos\Ao z_{g}=g,\qquad
\Azs z_{g}=0,\qquad
\pio z_{g}=0,$$
i.e., in formal matrix notation
$$\begin{bmatrix}
\Aos\Ao\\
\Azs\\
\pio
\end{bmatrix}
\begin{bmatrix}
z_{g}
\end{bmatrix}
=\begin{bmatrix}
g\\
0\\
0
\end{bmatrix}.$$
\end{itemize}
\end{rem}

We also emphasize that the variational formulations \eqref{varxf}-\eqref{varzg}
have a saddle point structure. 
We have already seen that, provided $f\in R(\At)$ and $g\in R(\Aos)$,
the formulations \eqref{varxf}-\eqref{varzg}
are equivalent to the following four problems:
Find $\ti{x}_{f}\in D(\cAt)$, $y_{f}\in D(\cAts)$, 
$\ti{x}_{g}\in D(\cAos)$, and $z_{g}\in D(\cAo)$, such that
\begin{align}
\mylabel{varxftwo}
\forall\,\xi&\in D(\At)
&
\scp{\At\ti{x}_{f}}{\At\xi}_{\Hith}
&=\scp{f}{\At\xi}_{\Hith},\\
\mylabel{varyftwo}
\forall\,\phi&\in D(\Ats)
&
\scp{\Ats y_{f}}{\Ats\phi}_{\Hit}
&=\scp{f}{\phi}_{\Hith},\\
\mylabel{varxgtwo}
\forall\,\zeta&\in D(\Aos)
&
\scp{\Aos\ti{x}_{g}}{\Aos\zeta}_{\Hio}
&=\scp{g}{\Aos\zeta}_{\Hio},\\
\mylabel{varzgtwo}
\forall\,\varphi&\in D(\Ao)
&
\scp{\Ao z_{g}}{\Ao\varphi}_{\Hit}
&=\scp{g}{\varphi}_{\Hio}.
\end{align} 
Note that in the end one needs only two out of these four formulations for computing
$$x_{f}=\ti{x}_{f}=\Ats y_{f},\qquad
x_{g}=\ti{x}_{g}=\Ao z_{g}.$$
Moreover, 
\begin{align*}
\ti{x}_{f}&\in D(\cAt)=D(\At)\cap R(\Ats)
&
&\qqequi
&
\ti{x}_{f}&\in D(\At)
&
&\wedge
&
\ti{x}_{f}&\in R(\Ats)=N(\At)^{\bot_{\Hit}},\\
y_{f}&\in D(\cAts)=D(\Ats)\cap R(\At)
&
&\qqequi
&
y_{f}&\in D(\Ats)
&
&\wedge
&
y_{f}&\in R(\At)=N(\Ats)^{\bot_{\Hith}},\\
\ti{x}_{g}&\in D(\cAos)=D(\Aos)\cap R(\Ao)
&
&\qqequi
&
\ti{x}_{g}&\in D(\Aos)
&
&\wedge
&
\ti{x}_{g}&\in R(\Ao)=N(\Aos)^{\bot_{\Hit}},\\
z_{g}&\in D(\cAo)=D(\Ao)\cap R(\Aos)
&
&\qqequi
&
z_{g}&\in D(\Ao)
&
&\wedge
&
z_{g}&\in R(\Aos)=N(\Ao)^{\bot_{\Hio}}.
\end{align*}
Therefore, the variational formulations \eqref{varxftwo}-\eqref{varzgtwo}
are equivalent to the following four saddle point problems:
Find $\ti{x}_{f}\in D(\At)$, $y_{f}\in D(\Ats)$, 
$\ti{x}_{g}\in D(\Aos)$, and $z_{g}\in D(\Ao)$, such that
\begin{align}
\mylabel{saddlexfone}
\forall\,\xi&\in D(\At)
&
\scp{\At\ti{x}_{f}}{\At\xi}_{\Hith}
&=\scp{f}{\At\xi}_{\Hith}
&
&\wedge
&
\forall\,\kappa&\in N(\At)
&
\scp{\ti{x}_{f}}{\kappa}_{\Hit}
&=0,\\
\mylabel{saddleyfone}
\forall\,\phi&\in D(\Ats)
&
\scp{\Ats y_{f}}{\Ats\phi}_{\Hit}
&=\scp{f}{\phi}_{\Hith}
&
&\wedge
&
\forall\,\theta&\in N(\Ats)
&
\scp{y_{f}}{\theta}_{\Hith}
&=0,\\
\mylabel{saddlexgone}
\forall\,\zeta&\in D(\Aos)
&
\scp{\Aos\ti{x}_{g}}{\Aos\zeta}_{\Hio}
&=\scp{g}{\Aos\zeta}_{\Hio}
&
&\wedge
&
\forall\,\lambda&\in N(\Aos)
&
\scp{\ti{x}_{g}}{\lambda}_{\Hit}
&=0,\\
\mylabel{saddlezgone}
\forall\,\varphi&\in D(\Ao)
&
\scp{\Ao z_{g}}{\Ao\varphi}_{\Hit}
&=\scp{g}{\varphi}_{\Hio}
&
&\wedge
&
\forall\,\psi&\in N(\Ao)
&
\scp{z_{g}}{\psi}_{\Hio}
&=0.
\end{align} 
Let us additionally assume that $R(\Az)$ and $R(\Ath)$ are closed.
Using Lemma \ref{lemtoolboxexseq}, i.e.,
\begin{align}
\mylabel{NAoNaos}
N(\Ao)
&=R(\Az)\oplus_{\Hio}K_{1},
&
N(\Aos)
&=R(\Ats)\oplus_{\Hit}K_{2},\\
\mylabel{NAtNats}
N(\At)
&=R(\Ao)\oplus_{\Hit}K_{2},
&
N(\Ats)
&=R(\Aths)\oplus_{\Hith}K_{3},
\end{align}
the systems \eqref{saddlexfone}-\eqref{saddlezgone} may be further refined to the following 
four double saddle point formulations:
Find $\ti{x}_{f}\in D(\At)$, $y_{f}\in D(\Ats)$, 
$\ti{x}_{g}\in D(\Aos)$, and $z_{g}\in D(\Ao)$, such that
\begin{align}
\mylabel{dsaddlexfone}
\forall\,\xi&\in D(\At)
&
\scp{\At\ti{x}_{f}}{\At\xi}_{\Hith}
&=\scp{f}{\At\xi}_{\Hith}
&
&\wedge
&
\forall\,\alpha&\in D(\Ao)
&
\scp{\ti{x}_{f}}{\Ao\alpha}_{\Hit}
&=0\\
\nonumber
&&&&
&\wedge
&
\forall\,\kappa&\in K_{2}
&
\scp{\ti{x}_{f}}{\kappa}_{\Hit}
&=0,\\
\mylabel{dsaddleyfone}
\forall\,\phi&\in D(\Ats)
&
\scp{\Ats y_{f}}{\Ats\phi}_{\Hit}
&=\scp{f}{\phi}_{\Hith}
&
&\wedge
&
\forall\,\beta&\in D(\Aths)
&
\scp{y_{f}}{\Aths\beta}_{\Hith}
&=0\\
\nonumber
&&&&
&\wedge
&
\forall\,\theta&\in K_{3}
&
\scp{y_{f}}{\theta}_{\Hith}
&=0,\\
\mylabel{dsaddlexgone}
\forall\,\zeta&\in D(\Aos)
&
\scp{\Aos\ti{x}_{g}}{\Aos\zeta}_{\Hio}
&=\scp{g}{\Aos\zeta}_{\Hio}
&
&\wedge
&
\forall\,\gamma&\in D(\Ats)
&
\scp{\ti{x}_{g}}{\Ats\gamma}_{\Hit}
&=0\\
\nonumber
&&&&
&\wedge
&
\forall\,\lambda&\in K_{2}
&
\scp{\ti{x}_{g}}{\lambda}_{\Hit}
&=0,\\
\mylabel{dsaddlezgone}
\forall\,\varphi&\in D(\Ao)
&
\scp{\Ao z_{g}}{\Ao\varphi}_{\Hit}
&=\scp{g}{\varphi}_{\Hio}
&
&\wedge
&
\forall\,\delta&\in D(\Az)
&
\scp{z_{g}}{\Az\delta}_{\Hio}
&=0\\
\nonumber
&&&&
&\wedge
&
\forall\,\psi&\in K_{1}
&
\scp{z_{g}}{\psi}_{\Hio}
&=0.
\end{align} 

\begin{rem}
\label{solremfosvarform}
For possible numerical purposes and applications
let us mention a few observations:
\begin{itemize}
\item[\bf(i)]
Using the variational formulation \eqref{saddlexfone} or \eqref{dsaddlexfone}
corresponding to $x_{f}=\ti{x}_{f}\in D(\At)$
for finding a numerical (discrete) approximation $x_{f,h}$ of
$x_{f}$ proposes a $D(\At)$-conforming method 
in some finite dimensional (discrete) subspace $D_{h}(\At)$ of $D(\At)$
giving also a $D(\At)$-conforming discrete solution $x_{f,h}\in D_{h}(\At)\subset D(\At)$.
\item[\bf(ii)]
Using the variational formulation \eqref{saddleyfone} or \eqref{dsaddleyfone}
corresponding to $x_{f}=\Ats y_{f}\in R(\Ats)$ 
for finding a discrete approximation $x_{f,h}=\Ats y_{f,h}$ of $x_{f}$
proposes a $D(\Ats)$-conforming method
in some discrete subspace $D_{h}(\Ats)$ of $D(\Ats)$
giving a $D(\Ats)$-conforming discrete potential $y_{f,h}\in D_{h}(\Ats)\subset D(\Ats)$,
but yielding a $D(\Aos)$-conforming solution as
$$x_{f,h}=\Ats y_{f,h}\in R(\Ats)=N(\Aos)\cap K_{2}^{\bot_{\Hit}}\subset D(\Aos).$$
\item[\bf(ii')]
A possible discrete solution $x_{f,h}=\Ats y_{f,h}$ from (ii) satisfies 
automatically the side conditions 
$$\Aos x_{f,h}=0,\qquad
\pit x_{f,h}=0,$$
i.e., even on the discrete level there is no error in the side conditions.
The other option from (i) yields a discrete solution $x_{f,h}$, which
most probably has got errors in the side conditions.
\item[\bf(iii)]
Similar observations hold for \eqref{saddlexgone} or \eqref{dsaddlexgone} with $D(\Aos)$-conforming discrete solutions
and \eqref{saddlezgone} or \eqref{dsaddlezgone} with $D(\Ao)$- resp. $D(\At)$-conforming discrete solutions.
\end{itemize}
\end{rem}

\begin{rem}
\label{solremfos}
The finite dimensionality of $K_{2}$ may be dropped.
Then all assertions of Theorem \ref{soltheofos} 
and Theorem \ref{soltheofosvarform} and all
variational and saddle point formulations remain valid.
Note that $R(\Ao)$ and $R(\At)$ are closed, 
if $D(\cAo)\dhookrightarrow\Hio$ and $D(\cAt)\dhookrightarrow\Hit$ are compact.
Moreover, by Lemma \ref{lemcptAlAslmo}
$D(\cAo)\dhookrightarrow\Hio$ and $D(\cAt)\dhookrightarrow\Hit$ are compact
and $K_{2}$ is finite dimensional if and only if $D_{2}\dhookrightarrow\Hit$ is compact. 
\end{rem}

\subsubsection{Trivial Cohomology Groups}

The double saddle point formulations \eqref{dsaddlexfone}-\eqref{dsaddlezgone} 
can be simplified if some assumptions on the cohomology groups are imposed.
For this, let additionally to our General Assumption \ref{genass}
the two ranges $R(\Az)$ and $R(\Ath)$ be closed as well
and let the cohomology groups $K_{1}$ and $K_{3}$ be trivial.
Thus all ranges $R(\Az)$, $R(\Ao)$, $R(\At)$, and $R(\Ath)$ are assumed to be closed
and we have
$$K_{1}=\{0\},\qquad
K_{3}=\{0\}.$$
Recalling \eqref{NAoNaos} and \eqref{NAtNats} we see
$$N(\Ao)=R(\Az),\qquad
N(\Ats)=R(\Aths).$$
If we now focus on the two double saddle point problems
\eqref{dsaddleyfone} and \eqref{dsaddlezgone} we get the following simplified versions:
Find $y_{f}\in D(\Ats)$ and $z_{g}\in D(\Ao)$, such that
\begin{align}
\mylabel{saddlexftwo}
\forall\,\phi&\in D(\Ats)
&
\scp{\Ats y_{f}}{\Ats\phi}_{\Hit}
&=\scp{f}{\phi}_{\Hith}
&
&\wedge
&
\forall\,\beta&\in D(\Aths)
&
\scp{y_{f}}{\Aths\beta}_{\Hith}
&=0,\\
\mylabel{saddlexgtwo}
\forall\,\varphi&\in D(\Ao)
&
\scp{\Ao z_{g}}{\Ao\varphi}_{\Hit}
&=\scp{g}{\varphi}_{\Hio}
&
&\wedge
&
\forall\,\delta&\in D(\Az)
&
\scp{z_{g}}{\Az\delta}_{\Hio}
&=0.
\end{align} 
Let us consider the following modified system: Find 
$$(y_{f},v_{f})\in D(\Ats)\times D(\cAths),\qquad
(z_{g},w_{g})\in D(\Ao)\times D(\cAz),$$ 
such that
\begin{align}
\mylabel{saddlexfthree}
\forall\,(\phi,\beta)&\in D(\Ats)\times D(\cAths)
&
\scp{\Ats y_{f}}{\Ats\phi}_{\Hit}
+\scp{\Aths v_{f}}{\phi}_{\Hith}
&=\scp{f}{\phi}_{\Hith}
&
&\wedge
&
\scp{y_{f}}{\Aths\beta}_{\Hith}
&=0,\\
\mylabel{saddlexgthree}
\forall\,(\varphi,\delta)&\in D(\Ao)\times D(\cAz)
&
\scp{\Ao z_{g}}{\Ao\varphi}_{\Hit}
+\scp{\Az w_{g}}{\varphi}_{\Hio}
&=\scp{g}{\varphi}_{\Hio}
&
&\wedge
&
\scp{z_{g}}{\Az\delta}_{\Hio}
&=0.
\end{align} 
The unique solutions $y_{f}$, $z_{g}$ of \eqref{saddlexftwo}-\eqref{saddlexgtwo} yield solutions
$(y_{f},0)$, $(z_{g},0)$ of \eqref{saddlexfthree}-\eqref{saddlexgthree}.
On the other hand, for any solutions $(y_{f},v_{f})$, $(z_{g},w_{g})$
of \eqref{saddlexfthree}-\eqref{saddlexgthree} we get 
$\Aths v_{f}=0$ and $\Az w_{g}=0$
by testing with $\phi:=\Aths v_{f}\in R(\Aths)=N(\Ats)\subset D(\Ats)$
and $\varphi:=\Az w_{g}\in R(\Az)=N(\Ao)\subset D(\Ao)$
since $f\in R(\At)\bot_{\Hith}N(\Ats)$ and $g\in R(\Aos)\bot_{\Hio}N(\Ao)$, respectively.
Hence, as $v_{f}\in D(\cAths)$ and $w_{g}\in D(\cAz)$ we see $v_{f}=0$ and $w_{g}=0$.
Thus, $y_{f}$, $z_{g}$ are the unique solutions of \eqref{saddlexftwo}-\eqref{saddlexgtwo}.
The latter arguments show that \eqref{saddlexftwo}-\eqref{saddlexgtwo} 
and \eqref{saddlexfthree}-\eqref{saddlexgthree}
are equivalent and both are uniquely solvable.

\begin{rem}
\mylabel{infsuprem}
The saddle point formulations \eqref{saddlexfthree}-\eqref{saddlexgthree} 
are also accessible by the standard inf-sup-theory, 
which is widely used in the numerical community.
For this, let us note that
the bilinear forms $\scp{\Ats\,\cdot\,}{\Ats\,\cdot\,}_{\Hit}$
and $\scp{\Ao\,\cdot\,}{\Ao\,\cdot\,}_{\Hit}$ are strictly positive (coercive) 
over the respective kernels given by each of the second forms, 
which are by assumption
\begin{align*}
N(\Ath)
&=\big(N(\Ath)^{\bot}\big)^{\bot}
=R(\Aths)^{\bot}
=N(\Ats)^{\bot}
=R(\At),\\
N(\Azs)
&=\big(N(\Azs)^{\bot}\big)^{\bot}
=R(\Az)^{\bot}
=N(\Ao)^{\bot}
=R(\Aos),
\end{align*}
i.e., over $D(\cAts)$ and $D(\cAo)$.
Moreover, the inf-sup-conditions are satisfied.
For this, we compute by choosing $\phi:=\Aths\beta\in R(\Aths)=N(\Ats)$
and $\varphi:=\Az\delta\in R(\Az)=N(\Ao)$ for some given 
$0\neq\beta\in D(\cAths)$ and $0\neq\delta\in D(\cAz)$
\begin{align*}
\frac{\norm{\Aths\beta}_{\Hith}}{\norm{\beta}_{D(\Aths)}}
&\leq\sup_{0\neq\phi\in D(\Ats)}
\frac{\scp{\Aths\beta}{\phi}_{\Hith}}{\norm{\beta}_{D(\Aths)}\norm{\phi}_{D(\Ats)}}
\leq\frac{\norm{\Aths\beta}_{\Hith}}{\norm{\beta}_{D(\Aths)}}
\leq1,\\
\frac{\norm{\Az\delta}_{\Hio}}{\norm{\delta}_{D(\Az)}}
&\leq\sup_{0\neq\varphi\in D(\Ao)}
\frac{\scp{\Az\delta}{\varphi}_{\Hio}}{\norm{\delta}_{D(\Az)}\norm{\varphi}_{D(\Ao)}}
\leq\frac{\norm{\Az\delta}_{\Hio}}{\norm{\delta}_{D(\Az)}}
\leq1,
\end{align*}
which shows that actually equality holds.
Thus, the inf-sup-conditions follow\footnote{Note that
by \eqref{classup}, \eqref{clsassup}, Lemma \ref{classup}, and Corollary \ref{cortoolboxone} (iv)
\begin{align*}
\inf_{0\neq\beta\in D(\cAths)}
\frac{\norm{\Aths\beta}_{\Hith}^2}{\norm{\beta}_{D(\Aths)}^2}
&=\Big(\sup_{0\neq\beta\in D(\cAths)}
\frac{\norm{\beta}_{\Hif}^2+\norm{\Aths\beta}_{\Hith}^2}{\norm{\Aths\beta}_{\Hith}^2}\Big)^{-1}
=\Big(1+\sup_{0\neq\beta\in D(\cAths)}
\frac{\norm{\beta}_{\Hif}^2}{\norm{\Aths\beta}_{\Hith}^2}\Big)^{-1}
=\frac{1}{1+c_{3}^2}
=\bnorm{\cA_{3}^{-1}}_{R(\Ath),D(\cAth)}^{-2},\\
\inf_{0\neq\delta\in D(\cAz)}
\frac{\norm{\Az\delta}_{\Hio}^2}{\norm{\delta}_{D(\Az)}^2}
&=\Big(\sup_{0\neq\delta\in D(\cAz)}
\frac{\norm{\delta}_{\Hiz}^2+\norm{\Az\delta}_{\Hio}^2}{\norm{\Az\delta}_{\Hio}^2}\Big)^{-1}
=\Big(1+\sup_{0\neq\delta\in D(\cAz)}
\frac{\norm{\delta}_{\Hiz}^2}{\norm{\Az\delta}_{\Hio}^2}\Big)^{-1}
=\frac{1}{c_{0}^2+1}
=\bnorm{\cA_{0}^{-1}}_{R(\Az),D(\cAz)}^{-2}
\end{align*}
hold.}
\begin{align*}
1\geq\inf_{0\neq\beta\in D(\cAths)}
\sup_{0\neq\phi\in D(\Ats)}
\frac{\scp{\Aths\beta}{\phi}_{\Hith}}{\norm{\beta}_{D(\Aths)}\norm{\phi}_{D(\Ats)}}
&=\inf_{0\neq\beta\in D(\cAths)}
\frac{\norm{\Aths\beta}_{\Hith}}{\norm{\beta}_{D(\Aths)}}
=(c_{3}^2+1)^{\moh}
=\bnorm{(\cAths)^{-1}}_{R(\Aths),D(\cAths)}^{-1},\\
1\geq\inf_{0\neq\delta\in D(\cAz)}
\sup_{0\neq\varphi\in D(\Ao)}
\frac{\scp{\Az\delta}{\varphi}_{\Hio}}{\norm{\delta}_{D(\Az)}\norm{\varphi}_{D(\Ao)}}
&=\inf_{0\neq\delta\in D(\cAz)}
\frac{\norm{\Az\delta}_{\Hio}}{\norm{\delta}_{D(\Az)}}
=(c_{0}^2+1)^{\moh}
=\bnorm{\cA_{0}^{-1}}_{R(\Az),D(\cAz)}^{-1},
\end{align*}
which are actually nothing else than the boundedness of the norms 
of the respective inverse operators
$\bnorm{(\cAths)^{-1}}_{R(\Aths),D(\cAths)}=\bnorm{\cA_{3}^{-1}}_{R(\Ath),D(\cAth)}$
and $\bnorm{\cA_{0}^{-1}}_{R(\Az),D(\cAz)}=\bnorm{(\cAzs)^{-1}}_{R(\Azs),D(\cAzs)}$, 
i.e., the boundedness of the respective inverse operators
$\cA_{3}^{-1}$, $(\cAths)^{-1}$, $\cA_{0}^{-1}$, $(\cAzs)^{-1}$, itself.
\end{rem}

Now, if $D(\cAths)$ and $D(\cAz)$ are still not suitable
and provided that the respective cohomology groups are trivial,
we can repeat the procedure to obtain additional saddle point formulations 
for $v_{f}$ and $w_{g}$. Note that \eqref{saddlexfthree}-\eqref{saddlexgthree} is equivalent to find 
$(y_{f},v_{f},z_{g},w_{g})\in D(\Ats)\times D(\cAths)\times D(\Ao)\times D(\cAz)$, such that
for all $(\phi,\beta,\varphi,\delta)\in D(\Ats)\times D(\cAths)\times D(\Ao)\times D(\cAz)$
\begin{align}
\mylabel{saddlexffour}
\begin{split}
&\scp{\Ats y_{f}}{\Ats\phi}_{\Hit}
+\scp{\Aths v_{f}}{\phi}_{\Hith}
+\scp{y_{f}}{\Aths\beta}_{\Hith}
+\scp{\Ao z_{g}}{\Ao\varphi}_{\Hit}
+\scp{\Az w_{g}}{\varphi}_{\Hio}
+\scp{z_{g}}{\Az\delta}_{\Hio}\\
&\qquad\qquad=\scp{f}{\phi}_{\Hith}
+\scp{g}{\varphi}_{\Hio}.
\end{split}
\end{align}

\subsubsection{More Variational Formulations}

Another idea is to compute the two partial solutions $x_{f}$ and $x_{g}$ 
from Theorem \ref{soltheofos} together in just one
variational formulation for the sum $x_{f}+x_{g}$. 
For this, let $f\in R(\At)$ and $g\in R(\Aos)$. 
Recall that $x_{f}\in D(\cAt)$
and $x_{g}\in D(\cAos)$ are given by the variational formulations
in Theorem \ref{soltheofosvarform} (i) and (ii), i.e.,
\begin{align}
\mylabel{varxftogetherone}
\forall\,\xi&\in D(\At)
&
\scp{\At x_{f}}{\At\xi}_{\Hith}
&=\scp{f}{\At\xi}_{\Hith},\\
\mylabel{varxgtogetherone}
\forall\,\zeta&\in D(\Aos)
&
\scp{\Aos x_{g}}{\Aos\zeta}_{\Hio}
&=\scp{g}{\Aos\zeta}_{\Hio},
\end{align}
respectively, compare also to the variational formulations 
\eqref{dsaddlexfone} and \eqref{dsaddlexgone}.
As $\Aos x_{f}=\Aos k=0$ and $\At x_{g}=\At k=0$,
these latter two formulations hold for $x=x_{f}+x_{g}+k$ as well, i.e.,
\begin{align}
\mylabel{varxftogethertwo}
\forall\,\xi&\in D(\At)
&
\scp{\At x}{\At\xi}_{\Hith}
&=\scp{f}{\At\xi}_{\Hith},\\
\mylabel{varxgtogethertwo}
\forall\,\zeta&\in D(\Aos)
&
\scp{\Aos x}{\Aos\zeta}_{\Hio}
&=\scp{g}{\Aos\zeta}_{\Hio}.
\end{align}
The first option is to use \eqref{varxftogethertwo} together 
with a weak version of $\Aos x=g$, i.e.,
\begin{align*}
\forall\,\varphi&\in D(\Ao)
&
\scp{g}{\varphi}_{\Hio}
&=\scp{\Aos x}{\varphi}_{\Hio}
=\scp{x}{\Ao\varphi}_{\Hit}.
\intertext{The second option is to use \eqref{varxgtogethertwo} together 
with a weak version of $\At x=f$, i.e.,}
\forall\,\phi&\in D(\Ats)
&
\scp{f}{\phi}_{\Hith}
&=\scp{\At x}{\phi}_{\Hith}
=\scp{x}{\Ats\phi}_{\Hit}.
\end{align*}

For simplicity, let us assume that the cohomology group $K_{2}$ is trivial.

\begin{theo}
\label{soltheofosvarformtogether}
Let $K_{2}=\{0\}$.
The unique solution $x=x_{f}+x_{g}\in D_{2}$ in Theorem \ref{soltheofos} 
can be found by the following two variational saddle point formulations: 
\begin{itemize}
\item[\bf(i)]
There exists a unique pair $(\ti{x},z)\in D(\At)\times D(\cAo)$ such that
\begin{align}
\mylabel{varxzoneone}
\forall\,(\xi,\varphi)&\in D(\At)\times D(\cAo)
&
\scp{\At\ti{x}}{\At\xi}_{\Hith}
+\scp{\Ao z}{\xi}_{\Hit}
&=\scp{f}{\At\xi}_{\Hith},\\
\mylabel{varxzonetwo}
&&
\scp{\ti{x}}{\Ao\varphi}_{\Hit}
&=\scp{g}{\varphi}_{\Hio}.
\intertext{It holds $z=0$ as well as}
\mylabel{varxztwoone}
\forall\,(\xi,\varphi)&\in D(\At)\times D(\cAo)
&
\scp{\At\ti{x}}{\At\xi}_{\Hith}
&=\scp{f}{\At\xi}_{\Hith},\\
\mylabel{varxztwotwo}
&&
\scp{\ti{x}}{\Ao\varphi}_{\Hit}
&=\scp{g}{\varphi}_{\Hio}.
\end{align}
Moreover, $\At\ti{x}=f$ if and only if $f\in R(\At)$.
\eqref{varxzonetwo}, \eqref{varxztwotwo} hold for all $\varphi\in D(\Ao)$
if and only if $g\in R(\Aos)$ if and only if $\ti{x}\in D(\Aos)$ and $\Aos\ti{x}=g$.
In this case, i.e., $f\in R(\At)$ and $g\in R(\Aos)$, we have $\ti{x}=x$ 
from Theorem \ref{soltheofos}.
\item[\bf(ii)]
There exists a unique pair $(\hat{x},y)\in D(\Aos)\times D(\cAts)$ such that
\begin{align}
\mylabel{varxyoneone}
\forall\,(\zeta,\phi)&\in D(\Aos)\times D(\cAts)
&
\scp{\Aos\hat{x}}{\Aos\zeta}_{\Hio}
+\scp{\Ats y}{\zeta}_{\Hit}
&=\scp{g}{\Aos\zeta}_{\Hio},\\
\mylabel{varxyonetwo}
&&
\scp{\hat{x}}{\Ats\phi}_{\Hit}
&=\scp{f}{\phi}_{\Hith}.
\intertext{It holds $y=0$ as well as}
\mylabel{varxytwoone}
\forall\,(\zeta,\phi)&\in D(\Aos)\times D(\cAts)
&
\scp{\Aos\hat{x}}{\Aos\zeta}_{\Hio}
&=\scp{g}{\Aos\zeta}_{\Hio},\\
\mylabel{varxytwotwo}
&&
\scp{\hat{x}}{\Ats\phi}_{\Hit}
&=\scp{f}{\phi}_{\Hith}.
\end{align}
Moreover, $\Aos\hat{x}=g$ if and only if $g\in R(\Aos)$.
\eqref{varxyonetwo}, \eqref{varxytwotwo} hold for all $\phi\in D(\Ats)$
if and only if $f\in R(\At)$ if and only if $\hat{x}\in D(\At)$ and $\At\hat{x}=f$.
In this case, i.e., $f\in R(\At)$ and $g\in R(\Aos)$, we have $\hat{x}=x$ 
from Theorem \ref{soltheofos}.
\end{itemize}
\end{theo}

\begin{proof}
We prove unique solvability by standard saddle point theory.
By Corollary \ref{cortoolboxone} (i) 
the principal part of \eqref{varxzoneone} is strictly positive over the kernel
of \eqref{varxzonetwo}, which is
$$D(\At)\cap N(\Aos)=D(\At)\cap R(\Ats)=D(\cAt),$$
as $K_{2}=\{0\}$. Moreover, we have for $0\neq\varphi\in D(\cAo)$
\begin{align*}
\frac{\norm{\Ao\varphi}_{\Hit}}{\norm{\varphi}_{D(\Ao)}}
\leq\sup_{0\neq\xi\in D(\At)}
\frac{\scp{\Ao\varphi}{\xi}_{\Hit}}{\norm{\varphi}_{D(\Ao)}\norm{\xi}_{D(\At)}}
\leq\frac{\norm{\Ao\varphi}_{\Hit}}{\norm{\varphi}_{D(\Ao)}}
\leq1
\end{align*}
by choosing $\xi:=\Ao\varphi\in R(\Ao)=N(\At)$,
which shows that actually equality holds. Hence
$$1\geq\inf_{0\neq\varphi\in D(\cAo)}\sup_{0\neq\xi\in D(\At)}
\frac{\scp{\Ao\varphi}{\xi}_{\Hit}}{\norm{\varphi}_{D(\Ao)}\norm{\xi}_{D(\At)}}
=\inf_{0\neq\varphi\in D(\cAo)}\frac{\norm{\Ao\varphi}_{\Hit}}{\norm{\varphi}_{D(\Ao)}}
\geq(c_{1}^2+1)^{\moh}
=\bnorm{\cA_{1}^{-1}}_{R(\Ao),D(\cAo)}^{-1},$$
which shows that the inf-sup-condition is satisfied.
Therefore, \eqref{varxzoneone}-\eqref{varxzonetwo} admits a unique solution.
Picking $\xi:=\Ao z\in R(\Ao)=N(\At)$ in \eqref{varxzoneone} yields
$\norm{\Ao z}_{\Hit}^2=0$ and hence $z=0$ as $z\in D(\cAo)$. 
Since $\Ao z=0$ even \eqref{varxztwoone}-\eqref{varxztwotwo} are valid.
By \eqref{varxztwoone} we see $\At\ti{x}-f\in R(\At)^{\bot_{\Hith}}$,
showing $\At\ti{x}=f$ if and only if $f\in R(\At)$.
Using the orthonormal projector $\pi_{\Aos}$ and
by \eqref{varxztwotwo} we see for all $\varphi\in D(\Ao)$ as
$\pi_{\Aos}\varphi\in D(\cAo)$
$$\scp{\ti{x}}{\Ao\varphi}_{\Hit}
=\scp{\ti{x}}{\Ao\pi_{\Aos}\varphi}_{\Hit}
=\scp{g}{\pi_{\Aos}\varphi}_{\Hio}
=\scp{\pi_{\Aos}g}{\varphi}_{\Hio}
=\scp{g}{\varphi}_{\Hio},$$
if $g\in R(\Aos)$. On the other hand, if \eqref{varxztwotwo} 
holds for all $\varphi\in D(\Ao)$,
then $\ti{x}\in D(\Aos)$ with $\Aos\ti{x}=g$, especially $g\in R(\Aos)$.
Therefore, if $f\in R(\At)$ and $g\in R(\Aos)$, we have 
$\ti{x}\in D(\At)\cap D(\Aos)=D_{2}$ with
$\At\ti{x}=f$ and $\Aos\ti{x}=g$, finally showing $\ti{x}=x$ 
by the unique solvability of \eqref{Aprobsoltheo} from Theorem \ref{soltheofos}.
Analogously we prove (ii).
\end{proof}

\begin{rem}
Let us note the following:
\begin{itemize}
\item[\bf(i)]
\eqref{varxzoneone}-\eqref{varxzonetwo} is a weak formulation of 
$$\Ats\At\ti{x}+\Ao z=\Ats f,\qquad
\Aos\ti{x}=g,$$
i.e., in formal matrix notation
$$\begin{bmatrix}
\Ats\At & \Ao\\
\Aos & 0
\end{bmatrix}
\begin{bmatrix}
\ti{x}\\
z
\end{bmatrix}
=\begin{bmatrix}
\Ats f\\
g
\end{bmatrix}.$$
Note $z=0$.
\item[\bf(ii)]
\eqref{varxyoneone}-\eqref{varxyonetwo} is a weak formulation of
$$\Ao\Aos\hat{x}+\Ats y=\Ao g,\qquad
\At\hat{x}=f,$$
i.e., in formal matrix notation
$$\begin{bmatrix}
\Ao\Aos & \Ats\\
\At & 0
\end{bmatrix}
\begin{bmatrix}
\hat{x}\\
y
\end{bmatrix}
=\begin{bmatrix}
\Ao g\\
f
\end{bmatrix}.$$
Note $y=0$.
\end{itemize}
\end{rem}

The restriction $K_{2}=\{0\}$ can easily be removed from Theorem \ref{soltheofosvarformtogether}
leading to double saddle point formulations as in \eqref{dsaddlexfone}-\eqref{dsaddlezgone}.

\begin{theo}
\label{soltheofosvarformtogetherwithK2}
The unique solution $x=x_{f}+x_{g}+k\in D_{2}$ in Theorem \ref{soltheofos} 
can be found by the following two variational double saddle point formulations: 
\begin{itemize}
\item[\bf(i)]
There exists a unique tripple $(\ti{x},z,h)\in D(\At)\times D(\cAo)\times K_{2}$ such that
\begin{align}
\nonumber
\forall\,(\xi,\varphi,\kappa)&\in D(\At)\times D(\cAo)\times K_{2}
&
\scp{\At\ti{x}}{\At\xi}_{\Hith}
+\scp{\Ao z}{\xi}_{\Hit}
+\scp{h}{\xi}_{\Hit}
&=\scp{f}{\At\xi}_{\Hith},\\
\mylabel{doublesaddlepointfullsystemoneone}
&&
\scp{\ti{x}}{\Ao\varphi}_{\Hit}
&=\scp{g}{\varphi}_{\Hio},\\
\nonumber
&&
\scp{\ti{x}}{\kappa}_{\Hit}
&=\scp{k}{\kappa}_{\Hit}.
\intertext{It holds $z=0$ and $h=0$ as well as}
\nonumber
\forall\,(\xi,\varphi,\kappa)&\in D(\At)\times D(\cAo)\times K_{2}
&
\scp{\At\ti{x}}{\At\xi}_{\Hith}
&=\scp{f}{\At\xi}_{\Hith},\\
\mylabel{doublesaddlepointfullsystemonetwo}
&&
\scp{\ti{x}}{\Ao\varphi}_{\Hit}
&=\scp{g}{\varphi}_{\Hio},\\
\nonumber
&&
\scp{\ti{x}}{\kappa}_{\Hit}
&=\scp{k}{\kappa}_{\Hit}.
\end{align}
Moreover, $\At\ti{x}=f$ if and only if $f\in R(\At)$.
The second equations of 
\eqref{doublesaddlepointfullsystemoneone},
\eqref{doublesaddlepointfullsystemonetwo} hold for all $\varphi\in D(\Ao)$
if and only if $g\in R(\Aos)$ if and only if $\ti{x}\in D(\Aos)$ and $\Aos\ti{x}=g$.
Furthermore, $\pit\ti{x}=k$. 
In this case, i.e., $f\in R(\At)$ and $g\in R(\Aos)$, we have $\ti{x}=x$ 
from Theorem \ref{soltheofos}.
\item[\bf(ii)]
There exists a unique triple $(\hat{x},y,h)\in D(\Aos)\times D(\cAts)\times K_{2}$ such that
\begin{align}
\nonumber
\forall\,(\zeta,\phi,\kappa)&\in D(\Aos)\times D(\cAts)\times K_{2}
&
\scp{\Aos\hat{x}}{\Aos\zeta}_{\Hio}
+\scp{\Ats y}{\zeta}_{\Hit}
+\scp{h}{\zeta}_{\Hit}
&=\scp{g}{\Aos\zeta}_{\Hio},\\
\mylabel{doublesaddlepointfullsystemtwoone}
&&
\scp{\hat{x}}{\Ats\phi}_{\Hit}
&=\scp{f}{\phi}_{\Hith},\\
\nonumber
&&
\scp{\hat{x}}{\kappa}_{\Hit}
&=\scp{k}{\kappa}_{\Hit}.
\intertext{It holds $y=0$ and $h=0$ as well as}
\nonumber
\forall\,(\zeta,\phi,\kappa)&\in D(\Aos)\times D(\cAts)\times K_{2}
&
\scp{\Aos\hat{x}}{\Aos\zeta}_{\Hio}
&=\scp{g}{\Aos\zeta}_{\Hio},\\
\mylabel{doublesaddlepointfullsystemtwotwo}
&&
\scp{\hat{x}}{\Ats\phi}_{\Hit}
&=\scp{f}{\phi}_{\Hith},\\
\nonumber
&&
\scp{\hat{x}}{\kappa}_{\Hit}
&=\scp{k}{\kappa}_{\Hit}.
\end{align}
Moreover, $\Aos\hat{x}=g$ if and only if $g\in R(\Aos)$.
The second equations of \eqref{doublesaddlepointfullsystemtwoone}, 
\eqref{doublesaddlepointfullsystemtwotwo} hold for all $\phi\in D(\Ats)$
if and only if $f\in R(\At)$ if and only if $\hat{x}\in D(\At)$ and $\At\hat{x}=f$.
Furthermore, $\pit\hat{x}=k$. 
In this case, i.e., $f\in R(\At)$ and $g\in R(\Aos)$, we have $\hat{x}=x$ 
from Theorem \ref{soltheofos}.
\end{itemize}
\end{theo}

\begin{proof}
Again we prove unique solvability by standard (double) saddle point theory.
The kernels of the operators encoded in the last two equations 
of \eqref{doublesaddlepointfullsystemoneone}
are $N(\Aos)$ and $K_{2}^{\bot_{\Hit}}$. Hence by Corollary \ref{cortoolboxone} (i)  
the principal part of the first equation in \eqref{doublesaddlepointfullsystemoneone}
is strictly positive over the latter kernels, i.e., over
$$D(\At)\cap N(\Aos)\cap K_{2}^{\bot_{\Hit}}=D(\At)\cap R(\Ats)=D(\cAt).$$
Moreover, we have for $\varphi\in D(\cAo)$ and $\kappa\in K_{2}$
with $(\varphi,\kappa)\neq0$
\begin{align*}
\frac{(\norm{\Ao\varphi}_{\Hit}^2+\norm{\kappa}_{\Hit}^2)^{\oh}}
{(\norm{\varphi}_{D(\Ao)}^2+\norm{\kappa}_{\Hit}^2)^{\oh}}
\leq\sup_{0\neq\xi\in D(\At)}
\frac{\scp{\Ao\varphi}{\xi}_{\Hit}+\scp{\kappa}{\xi}_{\Hit}}
{(\norm{\varphi}_{D(\Ao)}^2+\norm{\kappa}_{\Hit}^2)^{\oh}\norm{\xi}_{D(\At)}}
\leq\frac{\norm{\Ao\varphi}_{\Hit}+\norm{\kappa}_{\Hit}}
{(\norm{\varphi}_{D(\Ao)}^2+\norm{\kappa}_{\Hit}^2)^{\oh}}
\leq\sqrt{2}
\end{align*}
by choosing $\xi:=\Ao\varphi+\kappa\in R(\Ao)\oplus K_{2}=N(\At)$. 
Hence by Corollary \ref{cortoolboxone} (i)  
\begin{align*}
\sqrt{2}&\geq\inf_{\substack{\varphi\in D(\cAo)\\\kappa\in K_{2}\\(\varphi,\kappa)\neq0}}
\sup_{0\neq\xi\in D(\At)}
\frac{\scp{\Ao\varphi}{\xi}_{\Hit}+\scp{\kappa}{\xi}_{\Hit}}
{(\norm{\varphi}_{D(\Ao)}^2+\norm{\kappa}_{\Hit}^2)^{\oh}\norm{\xi}_{D(\At)}}
\geq\inf_{\substack{\varphi\in D(\cAo)\\\kappa\in K_{2}\\(\varphi,\kappa)\neq0}}
\frac{(\norm{\Ao\varphi}_{\Hit}^2+\norm{\kappa}_{\Hit}^2)^{\oh}}
{(\norm{\varphi}_{D(\Ao)}^2+\norm{\kappa}_{\Hit}^2)^{\oh}}\\
&\geq\inf_{\substack{\varphi\in D(\cAo)\\\kappa\in K_{2}\\(\varphi,\kappa)\neq0}}
\frac{(\norm{\Ao\varphi}_{\Hit}^2+\norm{\kappa}_{\Hit}^2)^{\oh}}
{((c_{1}^2+1)\norm{\Ao\varphi}_{\Hit}^2+\norm{\kappa}_{\Hit}^2)^{\oh}}
\geq(c_{1}^2+1)^{\moh}
=\bnorm{\cA_{1}^{-1}}_{R(\Ao),D(\cAo)}^{-1},
\end{align*}
which shows that the inf-sup-condition is satisfied.
Therefore, \eqref{doublesaddlepointfullsystemoneone} admits a unique solution.
Picking\footnote{We can test directly by
$\xi:=\Ao z+h\in R(\Ao)+K_{2}=N(\At)$ in \eqref{doublesaddlepointfullsystemoneone} as well,
since orthogonality shows immediately
$0=\scp{\Ao z}{\xi}_{\Hit}+\scp{h}{\xi}_{\Hit}
=\norm{\Ao z}_{\Hit}^2+\norm{h}_{\Hit}^2$.}
$\xi:=\Ao z\in R(\Ao)=N(\At)\cap K_{2}^{\bot_{\Hit}}$ 
in \eqref{doublesaddlepointfullsystemoneone} yields
$\norm{\Ao z}_{\Hit}^2=0$ and hence $z=0$ as $z\in D(\cAo)$. 
Choosing $\xi:=h\in K_{2}=N(\At)\cap R(\Ao)^{\bot_{\Hit}}$ 
in \eqref{doublesaddlepointfullsystemoneone} shows
$\norm{h}_{\Hit}^2=0$. Since $\Ao z=h=0$ 
even \eqref{doublesaddlepointfullsystemonetwo} is valid.
By the first equation of \eqref{doublesaddlepointfullsystemonetwo} 
we see $\At\ti{x}-f\in R(\At)^{\bot_{\Hith}}$,
showing $\At\ti{x}=f$ if and only if $f\in R(\At)$.
Using the orthonormal projector $\pi_{\Aos}$ and
by the second equation of \eqref{doublesaddlepointfullsystemonetwo}
we get for all $\varphi\in D(\Ao)$ as
$\pi_{\Aos}\varphi\in D(\cAo)$
$$\scp{\ti{x}}{\Ao\varphi}_{\Hit}
=\scp{\ti{x}}{\Ao\pi_{\Aos}\varphi}_{\Hit}
=\scp{g}{\pi_{\Aos}\varphi}_{\Hio}
=\scp{\pi_{\Aos}g}{\varphi}_{\Hio}
=\scp{g}{\varphi}_{\Hio},$$
if $g\in R(\Aos)$. On the other hand, 
if the second equation of \eqref{doublesaddlepointfullsystemonetwo}
holds for all $\varphi\in D(\Ao)$,
then $\ti{x}\in D(\Aos)$ with $\Aos\ti{x}=g$, especially $g\in R(\Aos)$.
Therefore, if $f\in R(\At)$ and $g\in R(\Aos)$, we have 
$\ti{x}\in D(\At)\cap D(\Aos)=D_{2}$ with
$\At\ti{x}=f$ and $\Aos\ti{x}=g$.
The third equation of \eqref{doublesaddlepointfullsystemonetwo} implies
for all $\kappa\in K_{2}$
$$0
=\scp{\ti{x}-k}{\kappa}_{\Hit}
=\scp{\ti{x}-k}{\pit\kappa}_{\Hit}
=\scp{\pit\ti{x}-k}{\kappa}_{\Hit},$$
i.e., $\pit\ti{x}=k$. 
Therefore, $\ti{x}=x$ 
by the unique solvability of \eqref{Aprobsoltheo} from Theorem \ref{soltheofos},
which completes the proof of (i). Analogously we prove (ii).
\end{proof}

\begin{rem}
\label{soltheofosvarformtogetherrem}
Let us note the following:
\begin{itemize}
\item[\bf(i)]
Using the saddle point formulation in Theorem \ref{soltheofosvarformtogether} (i) 
or Theorem \ref{soltheofosvarformtogetherwithK2} (i) 
for finding a numerical approximation $x_{h}$ of $x$
provides a $D(\At)$-conforming approximation $x_{h}\in D(\At)$ of \eqref{Aprobsoltheo}, 
whereas using the saddle point formulation in Theorem \ref{soltheofosvarformtogether} (ii) 
or Theorem \ref{soltheofosvarformtogetherwithK2} (ii) 
for finding a numerical approximation $x_{h}$ of $x$
provides a $D(\Aos)$-conforming approximation $x_{h}\in D(\Aos)$ of \eqref{Aprobsoltheo}.
\item[\bf(ii)]
The variational formulations in Theorem \ref{soltheofosvarformtogether} (i), (ii) 
or Theorem \ref{soltheofosvarformtogetherwithK2} (i), (ii) 
are exactly those from \eqref{dsaddlexfone} and \eqref{dsaddlexgone}
for the special right hand sides $g=0$, $f=0$, and $k=0$, respectively.
\item[\bf(iii)]
\eqref{doublesaddlepointfullsystemoneone} is a weak formulation of 
$$\Ats\At\ti{x}+\Ao z+h=\Ats f,\qquad
\Aos\ti{x}=g,\qquad
\pit\ti{x}=k,$$
i.e., in formal matrix notation
$$\begin{bmatrix}
\Ats\At & \Ao & \iota_{K_{2}}\\
\Aos & 0 & 0\\
\pit=\iota_{K_{2}}^{*} & 0 & 0
\end{bmatrix}
\begin{bmatrix}
\ti{x}\\
z\\
h
\end{bmatrix}
=\begin{bmatrix}
\Ats f\\
g\\
k
\end{bmatrix},$$
where $\iota_{K_{2}}$ is the canonical embedding of $K_{2}$ into $\Hit$.
Note $z=0$ and $h=0$.
\item[\bf(iii')]
\eqref{doublesaddlepointfullsystemtwoone} is a weak formulation of
$$\Ao\Aos\hat{x}+\Ats y+h=\Ao g,\qquad
\At\hat{x}=f,\qquad
\pit\hat{x}=k,$$
i.e., in formal matrix notation
$$\begin{bmatrix}
\Ao\Aos & \Ats & \iota_{K_{2}}\\
\At & 0 & 0\\
\pit=\iota_{K_{2}}^{*} & 0 & 0
\end{bmatrix}
\begin{bmatrix}
\hat{x}\\
y\\
h
\end{bmatrix}
=\begin{bmatrix}
\Ao g\\
f\\
k
\end{bmatrix}.$$
Note $y=0$ and $h=0$.
\end{itemize}
\end{rem}

Finally, we present double saddle point variational formulations for finding the partial solutions
in \eqref{dsaddlexfone}-\eqref{dsaddlezgone} as well.

\begin{theo}
\label{soltheofosvarformpartsolfull}
Let additionally $R(\Az)$ and $R(\Ath)$ be closed.
The partial solutions $x_{f}=\ti{x}_{f}\in D(\cAt)$, $x_{g}=\ti{x}_{g}\in D(\cAos)$,
and their potentials $y_{f}\in D(\cAts)$, $z_{g}\in D(\cAo)$
from Theorem \ref{soltheofos}, Theorem \ref{soltheofosvarform}, \eqref{varxftwo}-\eqref{varzgtwo},
\eqref{saddlexfone}-\eqref{saddlezgone}, and \eqref{dsaddlexfone}-\eqref{dsaddlezgone}
can be found by the following four variational double saddle point formulations: 
\begin{itemize}
\item[\bf(i)]
There exists a unique tripple $(\ti{x}_{f},u,h)\in D(\At)\times D(\cAo)\times K_{2}$ such that
\begin{align}
\nonumber
\forall\,(\xi,\varphi,\kappa)&\in D(\At)\times D(\cAo)\times K_{2}
&
\scp{\At\ti{x}_{f}}{\At\xi}_{\Hith}
+\scp{\Ao u}{\xi}_{\Hit}
+\scp{h}{\xi}_{\Hit}
&=\scp{f}{\At\xi}_{\Hith},\\
\mylabel{doublesaddlepointpartsystemxfoneone}
&&
\scp{\ti{x}_{f}}{\Ao\varphi}_{\Hit}
&=0,\\
\nonumber
&&
\scp{\ti{x}_{f}}{\kappa}_{\Hit}
&=0.
\end{align}
It holds $u=0$ and $h=0$ as well as \eqref{dsaddlexfone}.
Moreover, $\At\ti{x}_{f}=f$ if and only if $f\in R(\At)$.
The second equation of 
\eqref{doublesaddlepointpartsystemxfoneone} holds for all $\varphi\in D(\Ao)$
and thus $\ti{x}_{f}\in N(\Aos)$.
Furthermore, $\pit\ti{x}_{f}=0$. 
Finally, if $f\in R(\At)$, we have $\ti{x}_{f}=x_{f}$ from Theorem \ref{soltheofos},
see Theorem \ref{soltheofosvarform} (i).
\item[\bf(i')]
There exists a unique tripple $(y_{f},v,h)\in D(\Ats)\times D(\cAths)\times K_{3}$ such that
\begin{align}
\nonumber
\forall\,(\phi,\theta,\kappa)&\in D(\Ats)\times D(\cAths)\times K_{3}
&
\scp{\Ats y_{f}}{\Ats\phi}_{\Hit}
+\scp{\Aths v}{\phi}_{\Hith}
+\scp{h}{\phi}_{\Hith}
&=\scp{f}{\phi}_{\Hith},\\
\mylabel{doublesaddlepointpartsystemyfoneone}
&&
\scp{y_{f}}{\Aths\theta}_{\Hith}
&=0,\\
\nonumber
&&
\scp{y_{f}}{\kappa}_{\Hith}
&=0.
\end{align}
It holds $v=0$ if and only if $f\bot_{\Hith}R(\Aths)$ if and only if\footnote{$v=0$ implies
$f-h=\At\Ats y_{f}\in R(\At)\subset N(\Ath)$ and hence $f\in N(\Ath)$.}
$f\in N(\Ath)$.
$h=0$ if and only if\footnote{$h=0$ implies
$f-\Aths v=\At\Ats y_{f}\in R(\At)\,\bot_{\Hith}K_{3}$ and hence $f\,\bot_{\Hith}K_{3}$.}
$f\bot_{\Hith}K_{3}$.
Thus $v=0$ and $h=0$ if and only if $f\in N(\Ath)\cap K_{3}^{\bot_{\Hith}}=R(\At)$.
Furthermore, \eqref{dsaddleyfone} holds.
Moreover, $\Ats y_{f}\in D(\At)$ and $\At\Ats y_{f}=f$ if and only if $f\in R(\At)$.
The second equation of 
\eqref{doublesaddlepointpartsystemyfoneone} holds for all $\theta\in D(\Aths)$
and hence $y_{f}\in N(\Ath)$.
Furthermore, $\pith y_{f}=0$. 
Finally, if $f\in R(\At)$, we have $\Ats y_{f}=x_{f}$ from Theorem \ref{soltheofos},
see Theorem \ref{soltheofosvarform} (i').
\item[\bf(ii)]
There exists a unique triple $(\ti{x}_{g},p,h)\in D(\Aos)\times D(\cAts)\times K_{2}$ such that
\begin{align}
\nonumber
\forall\,(\zeta,\phi,\kappa)&\in D(\Aos)\times D(\cAts)\times K_{2}
&
\scp{\Aos\ti{x}_{g}}{\Aos\zeta}_{\Hio}
+\scp{\Ats p}{\zeta}_{\Hit}
+\scp{h}{\zeta}_{\Hit}
&=\scp{g}{\Aos\zeta}_{\Hio},\\
\mylabel{doublesaddlepointpartsystemxgtwoone}
&&
\scp{\ti{x}_{g}}{\Ats\phi}_{\Hit}
&=0,\\
\nonumber
&&
\scp{\ti{x}_{g}}{\kappa}_{\Hit}
&=0.
\end{align}
It holds $p=0$ and $h=0$ as well as \eqref{dsaddlexgone}.
Moreover, $\Aos\ti{x}_{g}=g$ if and only if $g\in R(\Aos)$.
The second equation of \eqref{doublesaddlepointpartsystemxgtwoone}
holds for all $\phi\in D(\Ats)$ and thus $\ti{x}_{g}\in N(\At)$.
Furthermore, $\pit\ti{x}_{g}=0$. 
Finally, if $g\in R(\Aos)$, we have $\ti{x}_{g}=x_{g}$ from Theorem \ref{soltheofos},
see Theorem \ref{soltheofosvarform} (ii).
\item[\bf(ii')]
There exists a unique triple $(z_{g},q,h)\in D(\Ao)\times D(\cAz)\times K_{1}$ such that
\begin{align}
\nonumber
\forall\,(\varphi,\vartheta,\kappa)&\in D(\Ao)\times D(\cAz)\times K_{1}
&
\scp{\Ao z_{g}}{\Ao\varphi}_{\Hit}
+\scp{\Az q}{\varphi}_{\Hio}
+\scp{h}{\varphi}_{\Hio}
&=\scp{g}{\varphi}_{\Hio},\\
\mylabel{doublesaddlepointpartsystemzgtwoone}
&&
\scp{z_{g}}{\Az\vartheta}_{\Hio}
&=0,\\
\nonumber
&&
\scp{z_{g}}{\kappa}_{\Hio}
&=0.
\end{align}
It holds $q=0$ if and only if $g\bot_{\Hio}R(\Az)$ if and only if\footnote{$q=0$ implies
$g-h=\Aos\Ao z_{g}\in R(\Aos)\subset N(\Azs)$ and hence $g\in N(\Azs)$.}
$g\in N(\Azs)$.
$h=0$ if and only if\footnote{$h=0$ implies
$g-\Az q=\Aos\Ao z_{g}\in R(\Aos)\,\bot_{\Hio}K_{1}$ and hence $g\,\bot_{\Hio}K_{1}$.}
$g\bot_{\Hio}K_{1}$.
Thus $q=0$ and $h=0$ if and only if $g\in N(\Azs)\cap K_{1}^{\bot_{\Hio}}=R(\Aos)$.
Furthermore, \eqref{dsaddlezgone} holds.
Moreover, $\Ao z_{g}\in D(\Aos)$ and $\Aos\Ao z_{g}=g$ if and only if $g\in R(\Aos)$.
The second equation of 
\eqref{doublesaddlepointpartsystemzgtwoone} holds for all $\vartheta\in D(\Az)$
and hence $z_{g}\in N(\Azs)$.
Furthermore, $\pio z_{g}=0$. 
Finally, if $g\in R(\Aos)$, we have $\Ao z_{g}=x_{g}$ from Theorem \ref{soltheofos},
see Theorem \ref{soltheofosvarform} (ii').
\end{itemize}
\end{theo}

\begin{proof}
The proof follows closely the lines of the proof of Theorem \ref{soltheofosvarformtogetherwithK2}.
\end{proof}

\begin{rem}
\label{soltheofosvarformpartsolfullremmatrix}
Again we have formal matrix representations:
\begin{itemize}
\item[\bf(i)]
\eqref{doublesaddlepointpartsystemxfoneone} is a weak formulation of 
$$\Ats\At\ti{x}_{f}+\Ao u+h=\Ats f,\qquad
\Aos\ti{x}_{f}=0,\qquad
\pit\ti{x}_{f}=0,$$
i.e., in formal matrix notation
$$\begin{bmatrix}
\Ats\At & \Ao & \iota_{K_{2}}\\
\Aos & 0 & 0\\
\pit=\iota_{K_{2}}^{*} & 0 & 0
\end{bmatrix}
\begin{bmatrix}
\ti{x}_{f}\\
u\\
h
\end{bmatrix}
=\begin{bmatrix}
\Ats f\\
0\\
0
\end{bmatrix}.$$
Note $u=0$ and $h=0$.
\item[\bf(i')]
\eqref{doublesaddlepointpartsystemyfoneone} is a weak formulation of
$$\At\Ats y_{f}+\Aths v+h=f,\qquad
\Ath y_{f}=0,\qquad
\pith y_{f}=0,$$
i.e., in formal matrix notation
$$\begin{bmatrix}
\At\Ats & \Aths & \iota_{K_{3}}\\
\Ath & 0 & 0\\
\pith=\iota_{K_{3}}^{*} & 0 & 0
\end{bmatrix}
\begin{bmatrix}
y_{f}\\
v\\
h
\end{bmatrix}
=\begin{bmatrix}
f\\
0\\
0
\end{bmatrix}.$$
Note $v=0$ and $h=0$.
\item[\bf(ii)]
\eqref{doublesaddlepointpartsystemxgtwoone} is a weak formulation of
$$\Ao\Aos\ti{x}_{g}+\Ats p+h=\Ao g,\qquad
\At\ti{x}_{g}=0,\qquad
\pit\ti{x}_{g}=0,$$
i.e., in formal matrix notation
$$\begin{bmatrix}
\Ao\Aos & \Ats & \iota_{K_{2}}\\
\At & 0 & 0\\
\pit=\iota_{K_{2}}^{*} & 0 & 0
\end{bmatrix}
\begin{bmatrix}
\ti{x}_{g}\\
p\\
h
\end{bmatrix}
=\begin{bmatrix}
\Ao g\\
0\\
0
\end{bmatrix}.$$
Note $p=0$ and $h=0$.
\item[\bf(ii')]
\eqref{doublesaddlepointpartsystemzgtwoone} is a weak formulation of
$$\Aos\Ao z_{g}+\Az q+h=g,\qquad
\Azs z_{g}=0,\qquad
\pio z_{g}=0,$$
i.e., in formal matrix notation
$$\begin{bmatrix}
\Aos\Ao & \Az & \iota_{K_{1}}\\
\Azs & 0 & 0\\
\pio=\iota_{K_{1}}^{*} & 0 & 0
\end{bmatrix}
\begin{bmatrix}
z_{g}\\
q\\
h
\end{bmatrix}
=\begin{bmatrix}
g\\
0\\
0
\end{bmatrix}.$$
Note $q=0$ and $h=0$.
\end{itemize}
\end{rem}

\subsubsection{Even More Variational Formulations}

In our variational formulations still the unpleasant spaces $D(\cAl)$ and $D(\cAls)$ occur
in the side conditions, see, e.g., Theorem \ref{soltheofosvarformtogetherwithK2}, where 
$$z\in D(\cAo)=D(\Ao)\cap R(\Aos),\qquad
y\in D(\cAts)=D(\Ats)\cap R(\At).$$
We can even go one step further and remove these restrictions
just by applying the same ideas as before. E.g., in Theorem \ref{soltheofosvarformtogetherwithK2}
\begin{align*}
z\in R(\Aos)
&=N(\Azs)\cap K_{1}^{\bot_{\Hio}}
=R(\Az)^{\bot_{\Hio}}\cap K_{1}^{\bot_{\Hio}},\\
y\in R(\At)
&=N(\Ath)\cap K_{3}^{\bot_{\Hith}}
=R(\Aths)^{\bot_{\Hith}}\cap K_{3}^{\bot_{\Hith}}
\end{align*}
can easily be formulated as additional side conditions.
Of course, this procedure can be prolongated ad infinitum depending on the length of the underlying complex.

\begin{rem}
\mylabel{3DAzAfrem}
In 3D applications the cohomology groups $K_{0}$, $K_{1}$ and $K_{4}$, $K_{5}$ 
are typically already trivial, see, e.g., the applications section \ref{emssec}. 
Also the kernels $N(\Az)$ and $N(\A_{4})$ are always trivial.
Moreover, the kernels $N(\Ao)$ and $N(\Aths)$ are typically trivial
or at least finite dimensional. The same applies to the orthogonal complements
of the kernels $N(\Azs)$ and $N(\A_{4})$.
In particular, $D(\cAo)=D(\Ao)$ resp. $D(\cAths)=D(\Aths)$ or at least 
$$D(\cAo)=D(\Ao)\cap R(\Aos)=D(\Ao)\cap N(\Ao)^{\bot_{\Hio}},\qquad
D(\cAths)=D(\Aths)\cap R(\Ath)=D(\Aths)\cap N(\Aths)^{\bot_{\Hif}},$$
respectively, with $N(\Ao)$ resp. $N(\Aths)$ being finite dimensional.
We always have $D(\cAz)=D(\Az)$ and $D(\cA_{4}^{*})=D(\A_{4}^{*})$.
\end{rem}

For example Theorem \ref{soltheofosvarformtogetherwithK2} can be modified as follows:

\begin{theo}
\label{soltheofosvarformtogetherwithK2andmore}
Let additionally $R(\Az)$, $R(\Ath)$, and $R(\A_{4})$ be closed.
Moreover, let $f\in R(\At)$ and $g\in R(\Aos)$.
The unique solution $x=x_{f}+x_{g}+k\in D_{2}$ in Theorem \ref{soltheofos} 
can be found by the following three variational quadruple resp. sextuple saddle point formulations: 
\begin{itemize}
\item[\bf(i)]
There exists a unique five tuple $(\ti{x},z,u,h_{2},h_{1})\in D(\At)\times D(\Ao)\times D(\cAz)\times K_{2}\times K_{1}$ 
such that for all $(\xi,\varphi,\vartheta,\kappa,\lambda)\in D(\At)\times D(\Ao)\times D(\cAz)\times K_{2}\times K_{1}$
\begin{align}
\nonumber
\scp{\At\ti{x}}{\At\xi}_{\Hith}
+\scp{\Ao z}{\xi}_{\Hit}
+\scp{h_{2}}{\xi}_{\Hit}
&=\scp{f}{\At\xi}_{\Hith},\\
\nonumber
\scp{\ti{x}}{\Ao\varphi}_{\Hit}
+\scp{\Az u}{\varphi}_{\Hio}
+\scp{h_{1}}{\varphi}_{\Hio}
&=\scp{g}{\varphi}_{\Hio},\\
\mylabel{doublesaddlepointfullsystemoneoneandmore}
\scp{z}{\Az\vartheta}_{\Hio}
&=0,\\
\nonumber
\scp{\ti{x}}{\kappa}_{\Hit}
&=\scp{k}{\kappa}_{\Hit},\\
\nonumber
\scp{z}{\lambda}_{\Hio}
&=0.
\end{align}
The third equation of 
\eqref{doublesaddlepointfullsystemoneoneandmore} is valid for all $\vartheta\in D(\Az)$.
It holds $z=0$ and $h_{2}=0$ as well as $u=0$ and $h_{1}=0$.
Moreover, $\At\ti{x}=f$ and $\ti{x}\in D(\Aos)$ with $\Aos\ti{x}=g$
as well as $\pit\ti{x}=k$.
Finally, $\ti{x}=x$ from Theorem \ref{soltheofos}.
\item[\bf(ii)]
There exists a unique five tuple $(\hat{x},y,v,h_{2},h_{3})\in D(\Aos)\times D(\Ats)\times D(\cAths)\times K_{2}\times K_{3}$ 
such that for all $(\zeta,\phi,\theta,\kappa,\lambda)\in D(\Aos)\times D(\Ats)\times D(\cAths)\times K_{2}\times K_{3}$ 
\begin{align}
\nonumber
\scp{\Aos\hat{x}}{\Aos\zeta}_{\Hio}
+\scp{\Ats y}{\zeta}_{\Hit}
+\scp{h_{2}}{\zeta}_{\Hit}
&=\scp{g}{\Aos\zeta}_{\Hio},\\
\nonumber
\scp{\hat{x}}{\Ats\phi}_{\Hit}
+\scp{\Aths v}{\phi}_{\Hith}
+\scp{h_{3}}{\phi}_{\Hith}
&=\scp{f}{\phi}_{\Hith},\\
\mylabel{doublesaddlepointfullsystemtwooneandmore}
\scp{y}{\Aths\theta}_{\Hith}
&=0,\\
\nonumber
\scp{\hat{x}}{\kappa}_{\Hit}
&=\scp{k}{\kappa}_{\Hit},\\
\nonumber
\scp{y}{\lambda}_{\Hith}
&=0.
\end{align}
The third equation of 
\eqref{doublesaddlepointfullsystemtwooneandmore} is valid for all $\theta\in D(\Aths)$.
It holds $y=0$ and $h_{2}=0$ as well as $v=0$ and $h_{3}=0$.
Moreover, $\Aos\hat{x}=g$ and $\hat{x}\in D(\At)$ with $\At\hat{x}=f$
as well as $\pit\hat{x}=k$.
Finally, $\hat{x}=x$ from Theorem \ref{soltheofos}.
\item[\bf(ii')]
There exists
$(\hat{x},y,v,w,h_{2},h_{3},h_{4})\in D(\Aos)\times D(\Ats)\times D(\Aths)\times D(\cA_{4}^{*})\times K_{2}\times K_{3}\times K_{4}$,
a unique seven tuple, such that for all 
$(\zeta,\phi,\theta,\sigma,\kappa,\lambda,\nu)\in D(\Aos)\times D(\Ats)\times D(\Aths)\times D(\cA_{4}^{*})\times K_{2}\times K_{3}\times K_{4}$ 
\begin{align}
\nonumber
\scp{\Aos\hat{x}}{\Aos\zeta}_{\Hio}
+\scp{\Ats y}{\zeta}_{\Hit}
+\scp{h_{2}}{\zeta}_{\Hit}
&=\scp{g}{\Aos\zeta}_{\Hio},\\
\nonumber
\scp{\hat{x}}{\Ats\phi}_{\Hit}
+\scp{\Aths v}{\phi}_{\Hith}
+\scp{h_{3}}{\phi}_{\Hith}
&=\scp{f}{\phi}_{\Hith},\\
\nonumber
\scp{y}{\Aths\theta}_{\Hith}
+\scp{\A_{4}^{*}w}{\theta}_{\Hif}
+\scp{h_{4}}{\theta}_{\Hif}
&=0,\\
\mylabel{doublesaddlepointfullsystemthreeoneandmore}
\scp{v}{\A_{4}^{*}\sigma}_{\Hif}
&=0,\\
\nonumber
\scp{\hat{x}}{\kappa}_{\Hit}
&=\scp{k}{\kappa}_{\Hit},\\
\nonumber
\scp{y}{\lambda}_{\Hith}
&=0,\\
\nonumber
\scp{v}{\nu}_{\Hif}
&=0.
\end{align}
The fourth equation of 
\eqref{doublesaddlepointfullsystemthreeoneandmore} is valid for all $\sigma\in D(\A_{4}^{*})$.
It holds $y=0$, $h_{2}=0$ and $v=0$, $h_{3}=0$ as well as $w=0$, $h_{4}=0$.
Moreover, $\Aos\hat{x}=g$ and $\hat{x}\in D(\At)$ with $\At\hat{x}=f$
as well as $\pit\hat{x}=k$.
Finally, $\hat{x}=x$ from Theorem \ref{soltheofos}.
\end{itemize}
\end{theo}

Theorem \ref{soltheofosvarformpartsolfull} can be extended in the same way.

\begin{rem}
\mylabel{soltheofosvarformtogetherwithK2andmorerem}
For (ii') recall $R(\Ath)=N(\A_{4})\cap K_{4}^{\bot_{\Hif}}=R(\A_{4}^{*})^{\bot_{\Hif}}\cap K_{4}^{\bot_{\Hif}}$.
Let us also note that generally the solution and test spaces look like
\begin{align*}
D(\Al)\times D(\Almo)&\times\cdots\times D(\A_{\ell-n+1})\times D(\cA_{\ell-n})
\times K_{\ell}\times K_{\ell-1}\times\cdots\times K_{\ell-n+1},\\
D(\Als)\times D(\Alpos)&\times\cdots\times D(\A_{\ell+n-1}^{*})\times D(\cA_{\ell+n}^{*})
\times K_{\ell+1}\times K_{\ell+2}\times\cdots\times K_{\ell+n}.
\end{align*}
Moreover:
\begin{itemize}
\item[\bf(i)]
\eqref{doublesaddlepointfullsystemoneoneandmore} is a weak formulation of 
$$\Ats\At\ti{x}+\Ao z+h_{2}=\Ats f,\quad
\Aos\ti{x}+\Az u+h_{1}=g,\quad
\Azs z=0,\quad
\pit\ti{x}=k,\quad
\pio z=0,$$
i.e., in formal matrix notation
$$\begin{bmatrix}
\Ats\At & \Ao & 0 & \iota_{K_{2}} & 0\\
\Aos & 0 & \Az & 0 & \iota_{K_{1}}\\
0 & \Azs & 0 & 0 & 0\\
\pit=\iota_{K_{2}}^{*} & 0 & 0 & 0 & 0\\
0 & \pio=\iota_{K_{1}}^{*} & 0 & 0 & 0
\end{bmatrix}
\begin{bmatrix}
\ti{x}\\
z\\
u\\
h_{2}\\
h_{1}
\end{bmatrix}
=\begin{bmatrix}
\Ats f\\
g\\
0\\
k\\
0
\end{bmatrix}.$$
Note $z=0$, $u=0$ and $h_{2}=0$, $h_{1}=0$.
\item[\bf(ii)]
\eqref{doublesaddlepointfullsystemtwooneandmore} is a weak formulation of
$$\Ao\Aos\hat{x}+\Ats y+h_{2}=\Ao g,\quad
\At\hat{x}+\Aths v+h_{3}=f,\quad
\Ath y=0,\quad
\pit\hat{x}=k,\quad
\pith y=0,$$
i.e., in formal matrix notation
$$\begin{bmatrix}
\Ao\Aos & \Ats & 0 & \iota_{K_{2}} & 0\\
\At & 0 & \Aths & 0 & \iota_{K_{3}}\\
0 & \Ath & 0 & 0 & 0\\
\pit=\iota_{K_{2}}^{*} & 0 & 0 & 0 & 0\\
0 & \pith=\iota_{K_{3}}^{*} & 0 & 0 & 0
\end{bmatrix}
\begin{bmatrix}
\hat{x}\\
y\\
v\\
h_{2}\\
h_{3}
\end{bmatrix}
=\begin{bmatrix}
\Ao g\\
f\\
0\\
k\\
0
\end{bmatrix}.$$
Note $y=0$, $v=0$ and $h_{2}=0$, $h_{3}=0$.
\item[\bf(ii')]
\eqref{doublesaddlepointfullsystemthreeoneandmore} is a weak formulation of
$$\Ao\Aos\hat{x}+\Ats y+h_{2}=\Ao g,\quad
\At\hat{x}+\Aths v+h_{3}=f,\quad
\Ath y+\A_{4}^{*}w+h_{4}=0,\quad
\A_{4}v=0,$$
and $\pit\hat{x}=k$, $\pith y=0$, $\pi_{4}v=0$, 
i.e., in formal matrix notation
$$\begin{bmatrix}
\Ao\Aos & \Ats & 0 & 0 & \iota_{K_{2}} & 0 & 0\\
\At & 0 & \Aths & 0 & 0 & \iota_{K_{3}}& 0 \\
0 & \Ath & 0 & \A_{4}^{*} & 0 & 0 & \iota_{K_{4}}\\
0 & 0 & \A_{4} & 0 & 0 & 0 & 0\\
\pit=\iota_{K_{2}}^{*} & 0 & 0 & 0 & 0 & 0 & 0\\
0 & \pith=\iota_{K_{3}}^{*} & 0 & 0 & 0 & 0 & 0 \\
0 & 0 & \pi_{4}=\iota_{K_{4}}^{*} & 0 & 0 & 0 & 0
\end{bmatrix}
\begin{bmatrix}
\hat{x}\\
y\\
v\\
w\\
h_{2}\\
h_{3}\\
h_{4}
\end{bmatrix}
=\begin{bmatrix}
\Ao g\\
f\\
0\\
0\\
k\\
0\\
0
\end{bmatrix}.$$
Note $y=0$, $v=0$, $w=0$ and $h_{2}=0$, $h_{3}=0$, $h_{4}=0$.
\end{itemize}
\end{rem}

\subsection{Second Order Systems}

We recall the linear second order system \eqref{AsAprob}, i.e., find\footnote{We generally 
define $\ti{D}_{\ell}:=\setb{\xi\in D_{\ell}}{\Al\xi\in D(\Als)}=D(\Almos)\cap D(\Als\Al)$ for $\ell=1,\dots,3$.} 
$$x\in\ti{D}_{2}
:=\setb{\xi\in D_{2}}{\At\xi\in D(\Ats)}
=\setb{\xi\in D(\At)\cap D(\Aos)}{\At\xi\in D(\Ats)}
=D(\Aos)\cap D(\Ats\At)$$
such that
\begin{align}
\begin{split}
\label{AsAprobsoltheo}
\Ats\At x&=f,\\
\Aos x&=g,\\
\pit x&=k.
\end{split}
\end{align}

\begin{theo}
\label{soltheosos}
\eqref{AsAprobsoltheo} is uniquely solvable in $\ti{D}_{2}$,
if and only if $f\in R(\Ats)$, $g\in R(\Aos)$, and $k\in K_{2}$. 
The unique solution $x\in\ti{D}_{2}$ is given by 
\begin{align*}
x:=x_{f}+x_{g}+k
&\in\big(D(\cAt)\oplus_{\Hit}D(\cAos)\oplus_{\Hit}K_{2}\big)
\cap\ti{D}_{2}=\ti{D}_{2},\\
x_{f}:=\cA_{2}^{-1}(\cAts)^{-1}f
&\in D(\cAts\cAt)=D(\cAt)\cap\ti{D}_{2},\\
x_{g}:=(\cAos)^{-1}g
&\in D(\cAos)=D(\cAos)\cap\ti{D}_{2}
\end{align*}
and depends continuously on the data, i.e.,
$\norm{x}_{\Hit}
\leq c_{2}^2\norm{f}_{\Hit}
+c_{1}\norm{g}_{\Hio}
+\norm{k}_{\Hit}$, as
$$\norm{x_{f}}_{\Hit}\leq c_{2}^2\norm{f}_{\Hit},\qquad
\norm{x_{g}}_{\Hit}\leq c_{1}\norm{g}_{\Hio}.$$
It holds 
$$\pi_{\Ats}x=x_{f},\qquad
\pi_{\Ao}x=x_{g},\qquad
\pit x=k,\qquad
\norm{x}_{\Hit}^2=\norm{x_{f}}_{\Hit}^2+\norm{x_{g}}_{\Hit}^2+\norm{k}_{\Hit}^2.$$
\end{theo}

\begin{proof}
The necessary conditions are clear.
To show uniqueness, let $x\in\ti{D}_{2}$ solve
\begin{align*}
\Ats\At x&=0,&
\Aos x&=0,&
\pit x&=0.
\end{align*}
Hence $x\in N(\Aos)\cap K_{2}^{\bot_{\Hit}}$ 
and also $x\in N(\At)$ as $\At x\in D(\Ats)$ and
$$\norm{\At x}_{\Hith}^2=\scp{x}{\Ats\At x}_{\Hit}=0,$$
yielding $x\in K_{2}\cap K_{2}^{\bot_{\Hit}}=\{0\}$.
To prove existence, let $f\in R(\Ats)$, $g\in R(\Aos)$, $k\in K_{2}$
and define $x$, $x_{f}$, and $x_{g}$ according to the theorem.
Again the orthogonality follows directly by Lemma \ref{lemtoolboxexseq}.
Moreover, $x_{f}$, $x_{g}$, and $k$ solve the linear systems
\begin{align*}
\Ats\At x_{f}&=f,
&
\At x_{g}&=0,
&
\At k&=0,\\
\Aos x_{f}&=0,
&
\Aos x_{g}&=g,
&
\Aos k&=0,\\
\pit x_{f}&=0,
&
\pit x_{g}&=0,
&
\pit k&=k.
\end{align*}
Thus $x$ solves \eqref{AsAprobsoltheo}
and we have by Corollary \ref{cortoolboxone} 
$\norm{x_{f}}_{\Hit}\leq c_{2}\norm{\At x_{f}}_{\Hith}\leq c_{2}^2\norm{f}_{\Hit}$
and $\norm{x_{g}}_{\Hit}\leq c_{1}\norm{g}_{\Hio}$,
completing the proof of the solution theory.
\end{proof} 

\begin{rem}
\label{remsoltheosos}
By orthogonality 
and with $\At x=(\cAts)^{-1}f$,
$\Ats\At x=f$,
and $\Aos x=g$
we even have
\begin{align*}
\norm{x}_{\Hit}^2
&=\bnorm{x_{f}}_{\Hit}^2
+\bnorm{x_{g}}_{\Hit}^2
+\norm{k}_{\Hit}^2
\leq c_{2}^4\norm{f}_{\Hit}^2
+c_{1}^2{}\norm{g}_{\Hio}^2
+\norm{k}_{\Hit}^2,\\
\norm{x}_{\ti{D}_{2}}^2
&=\bnorm{x_{f}}_{\Hit}^2
+\bnorm{\At x}_{\Hith}^2
+\norm{f}_{\Hit}^2
+\bnorm{x_{g}}_{\Hit}^2
+\norm{g}_{\Hio}^2
+\norm{k}_{\Hit}^2\\
&\leq(1+c_{2}^2+c_{2}^4)\norm{f}_{\Hith}^2
+(1+c_{1}^2)\norm{g}_{\Hio}^2
+\norm{k}_{\Hit}^2.
\end{align*}
\end{rem}

\begin{rem}
\label{soltheosecorderrem}
Since the second order system \eqref{AsAprobsoltheo} decomposes
into the two first order systems of shape \eqref{Aprob} resp. \eqref{Aprobsoltheo}, i.e.,
\begin{align*}
\At x&=y,
&
\Ath y&=0,\\
\Aos x&=g,
&
\Ats y&=f,\\
\pit x&=k,
&
\pith y&=0
\end{align*}
for the pair $(x,y)\in D_{2}\times D_{3}$ 
with $y:=\At x\in D(\Ats)\cap R(\At)=D(\cAts)$,
the solution theory follows directly by Theorem \ref{soltheofos} as well.
One just has to solve and set
\begin{align*}
y&:=(\cAts)^{-1}f\in D(\cAts)\subset R(\At),\\
x&:=\cA_{2}^{-1}y+(\cAos)^{-1}g+k
\in\big(D(\cAt)\oplus_{\Hit}D(\cAos)\oplus_{\Hit}K_{2}\big)
\cap\ti{D}_{2}=\ti{D}_{2}.
\end{align*}
\end{rem}

\subsubsection{Variational Formulations}

We note
\begin{align}
\label{recalldefxfxgsos}
\begin{split}
D(\cAts\cAt)
&=D(\Ats\cAt)
=D(\Ats\At)\cap D(\cAt)
=D(\Ats\At)\cap R(\Ats)
=D(\Ats\At)\cap N(\Aos)\cap K_{2}^{\bot_{\Hit}}\\
&=\ti{D}_{2}\cap D(\cAt)
=\ti{D}_{2}\cap R(\Ats)
=\ti{D}_{2}\cap N(\Aos)\cap K_{2}^{\bot_{\Hit}},\\
D(\cAos)
&=D(\Aos)\cap R(\Ao)
=D(\Aos)\cap N(\At)\cap K_{2}^{\bot_{\Hit}}\\
&=\ti{D}_{2}\cap D(\cAos)
=\ti{D}_{2}\cap R(\cAo)
=\ti{D}_{2}\cap N(\At)\cap K_{2}^{\bot_{\Hit}}
\end{split}
\end{align}
and recall
\begin{align*}
x_{f}=\cA_{2}^{-1}(\cAts)^{-1}f
&\in D(\cAts\cAt)
=D(\Ats\At)\cap R(\Ats)
=D(\Ats\At)\cap N(\Aos)\cap K_{2}^{\bot_{\Hit}},\\
x_{g}=(\cAos)^{-1}g
&\in D(\cAos)
=D(\Aos)\cap R(\Ao)
=D(\Aos)\cap N(\At)\cap K_{2}^{\bot_{\Hit}}.
\end{align*}
As in the corresponding section for the first order systems,
there are several options for variational formulations
for finding each of the partial solutions $x_{f}$ and $x_{g}$,
which all make sense from a functional analytical point of view.
Looking at Remark \ref{soltheosecorderrem} it is clear that all
variational formulations proposed for the first order systems
from the earlier sections are applicable here as well.
Especially for $x_{g}$ we do not observe anything new.
On the other hand, for the second order system related to $x_{f}$
we can do as follows: The first option is to multiply 
the equation $\Ats\At x_{f}=f$ by $\Ats\At\phi$ with some 
$\phi\in D(\cAts\cAt)$
giving the variational formulation
$$\forall\,\phi\in D(\cAts\cAt)\qquad
\scp{\Ats\At x_{f}}{\Ats\At\phi}_{\Hit}
=\scp{f}{\Ats\At\phi}_{\Hit},$$
which is a weak formulation of the fourth order equation 
$$(\Ats\At)^2x_{f}=\Ats\At f,$$
more precisely of $\Ats\At(\Ats\At x_{f}-f)=0$.
Perhaps a more convenient choice is to multiply $\Ats\At x_{f}=f$ 
by some $\xi\in D(\cAt)$
giving the variational formulation
$$\forall\,\xi\in D(\cAt)\qquad
\scp{\At x_{f}}{\At\xi}_{\Hith}
=\scp{f}{\xi}_{\Hit},$$
which is a weak formulation of the second order equation 
$$\Ats\At x_{f}=f.$$
The latter choices are finding straight forward $x_{f}$ itself.
As a third option, we propose a formulation
to find a potential $y_{f}$ for $x_{f}$.
For this we go for a second order potential $y_{f}$
with $\Ats\At y_{f}=x_{f}$,
e.g., $y_{f}:=\cA_{2}^{-1}(\cAts)^{-1}x_{f}\in D(\cAts\cAt)$, of 
$$x_{f}=\Ats\At y_{f}\in 
D(\cAts\cAt)
=D(\Ats\At)\cap R(\Ats)
=D(\Ats\At)\cap R(\cAts).$$
Multiplying by $\Ats\At\tau$ with some $\tau\in D(\cAts\cAt)$ gives
\begin{align*}
\forall\,\tau\in D(\cAts\cAt)\qquad
\scp{\Ats\At y_{f}}{\Ats\At\tau}_{\Hit}
=\scp{x_{f}}{\Ats\At\tau}_{\Hit}
=\scp{\Ats\At x_{f}}{\tau}_{\Hit}
=\scp{f}{\tau}_{\Hit},
\end{align*}
which is a weak formulation of the fourth order equation
$$(\Ats\At)^2y_{f}=f.$$

\begin{theo}
\label{soltheososvarform}
The partial solutions $x_{f}$ and $x_{g}$ in Theorem \ref{soltheosos} 
can be found by the following variational formulations: 
\begin{itemize}
\item[\bf(i)]
There exists a unique $\hat{x}_{f}\in D(\cAts\cAt)$, such that
\begin{align}
\mylabel{varxfsosone}
\forall\,\phi\in D(\cAts\cAt)\qquad
\scp{\Ats\At\hat{x}_{f}}{\Ats\At\phi}_{\Hit}
=\scp{f}{\Ats\At\phi}_{\Hit}.
\end{align}
\eqref{varxfsosone} even holds for all $\phi\in D(\Ats\At)$. 
Moreover, $\Ats\At\hat{x}_{f}=f$ if and only if $f\in R(\Ats)$.
In this case $\hat{x}_{f}=x_{f}$.
\item[\bf(i')]
There exists a unique $\ti{x}_{f}\in D(\cAt)$, such that
\begin{align}
\mylabel{varxfsostwo}
\forall\,\xi\in D(\cAt)\qquad
\scp{\At\ti{x}_{f}}{\At\xi}_{\Hith}
=\scp{f}{\xi}_{\Hit}.
\end{align}
\eqref{varxfsostwo} even holds for all $\xi\in D(\At)$
if and only if $f\in R(\Ats)$. In this case we have
$$\At\ti{x}_{f}\in D(\Ats)\cap R(\At)=D(\cAts)$$ 
with $\Ats\At\ti{x}_{f}=f$ and thus $\ti{x}_{f}=x_{f}$.
\item[\bf(i'')]
There exists a unique potential $y_{f}\in D(\cAts\cAt)$, such that
\begin{align}
\mylabel{varyfsos}
\forall\,\tau\in D(\cAts\cAt)\qquad
\scp{\Ats\At y_{f}}{\Ats\At\tau}_{\Hit}
=\scp{f}{\tau}_{\Hit}.
\end{align}
\eqref{varyfsos} even holds for all $\tau\in D(\Ats\At)$
if and only if $f\in R(\Ats)$. In this case we have
$$\Ats\At y_{f}\in D(\cAts\cAt)$$ 
with $(\Ats\At)^2y_{f}=f$ and hence $\Ats\At y_{f}=x_{f}$.
\item[\bf(ii)]
There exists a unique $\ti{x}_{g}\in D(\cAos)$ such that
\begin{align}
\mylabel{varxgsos}
\forall\,\zeta\in D(\cAos)\qquad
\scp{\Aos\ti{x}_{g}}{\Aos\zeta}_{\Hio}
=\scp{g}{\Aos\zeta}_{\Hio}.
\end{align}
\eqref{varxgsos} is even satisfied for all $\zeta\in D(\Aos)$.
Moreover, $\Aos\ti{x}_{g}=g$ holds if and only if $g\in R(\Aos)$.
In this case $\ti{x}_{g}=x_{g}$.
\item[\bf(ii')]
There exists a unique potential $z_{g}\in D(\cAo)$, such that
\begin{align}
\mylabel{varzgsos}
\forall\,\varphi\in D(\cAo)\qquad
\scp{\Ao z_{g}}{\Ao\varphi}_{\Hit}
=\scp{g}{\varphi}_{\Hio}.
\end{align}
\eqref{varzgsos} even holds for all $\varphi\in D(\Ao)$
if and only if $g\in R(\Aos)$. In this case we have
$$\Ao z_{g}\in D(\Aos)\cap R(\Ao)=D(\cAos)$$ 
with $\Aos\Ao z_{g}=g$ and hence $\Ao z_{g}=x_{g}$.
\end{itemize}
\end{theo}

\begin{proof}
To show (i), let $\phi\in D(\cAts\cAt)$. Then
$\At\phi$ belongs to $D(\cAts)$ and by Corollary \ref{cortoolboxone} (i) we see
$$\norm{\Ats\At\phi}_{\Hit}
\geq\frac{1}{c_{2}}\norm{\At\phi}_{\Hith}
\geq\frac{1}{c_{2}^2}\norm{\phi}_{\Hit}.$$
Hence, the bilinear form in \eqref{varxfsosone} is strictly positive 
over $D(\cAts\cAt)$ and thus 
Riesz' representation theorem yields the unique solvability of \eqref{varxfsosone}.
From $D(\At)=N(\At)\oplus_{\Hit}D(\cAt)$, 
see Corollary \ref{cortoolboxone} (iii) or Lemma \ref{lemtoolboxexseq},
and \eqref{recalldefxfxgsos} we get 
$$D(\Ats\At)=N(\At)\oplus_{\Hit}D(\Ats\cAt)=N(\At)\oplus_{\Hit}D(\cAts\cAt).$$
Therefore $R(\Ats\At)=R(\cAts\cAt)$
and thus \eqref{varxfsosone} holds for all $\phi\in D(\Ats\At)$ as well.
Let $\psi\in D(\Ats)$ and decompose it according to
Corollary \ref{cortoolboxone} (iii) or Lemma \ref{lemtoolboxexseq}
into $\psi=\psi_{N}+\psi_{R}\in D(\Ats)=N(\Ats)\oplus_{\Hith}D(\cAts)$ 
(null space and range) and, as $D(\cAts)=D(\Ats)\cap R(\At)=D(\cAts)\cap R(\cAt)$,
further into\footnote{Here it would be enough to decompose
$$\psi=\psi_{N}+\At\phi_{R}\in D(\Ats)
=N(\Ats)\oplus_{\Hith}\big(D(\Ats)\cap R(\At)\big),\qquad
\phi_{R}\in D(\Ats\At).$$}
$$\psi=\psi_{N}+\At\phi_{R}\in D(\Ats)
=N(\Ats)\oplus_{\Hith}\big(D(\cAts)\cap R(\cAt)\big),\qquad
\phi_{R}\in D(\cAts\cAt).$$
Utilizing the latter decomposition and \eqref{varxfsosone} we obtain
for all $\psi\in D(\Ats)$
\begin{align*}
\scp{\Ats\At\hat{x}_{f}}{\Ats\psi}_{\Hit}
&=\scp{\Ats\At\hat{x}_{f}}{\Ats\At\phi_{R}}_{\Hit}
=\scp{f}{\Ats\At\phi_{R}}_{\Hit}
=\scp{f}{\Ats\psi}_{\Hit},
\end{align*}
which shows $\Ats\At\hat{x}_{f}-f\in N(\At)=R(\Ats)^{\bot_{\Hit}}$.
Thus, $\Ats\At\hat{x}_{f}-f=0$, if and only if $f\in R(\Ats)$.
In this case we have $\Ats\At(\hat{x}_{f}-x_{f})=0$
and the injectivity of $\cAts$ and $\cAt$ shows $\hat{x}_{f}=x_{f}$,
which finishes the proof of (i).

The left hand side of \eqref{varxfsostwo} is strictly positive over $D(\cAt)$ and thus 
Riesz' representation theorem yields the unique solvability of \eqref{varxfsostwo}.
Let us recall that the orthonormal projector $\pi_{\Ats}$ onto $R(\Ats)$ satisfies
$\At\pi_{\Ats}\xi=\At\xi$ and $\pi_{\Ats}\xi\in D(\cAt)$ for $\xi\in D(\At)$ 
and $\pi_{\Ats}f=f$ for $f\in R(\Ats)$.
Therefore, if $f\in R(\Ats)$, then \eqref{varxfsostwo} yields for $\xi\in D(\At)$
\begin{align*}
\scp{\At\ti{x}_{f}}{\At\xi}_{\Hith}
&=\scp{\At\ti{x}_{f}}{\At\pi_{\Ats}\xi}_{\Hith}
=\scp{f}{\pi_{\Ats}\xi}_{\Hit}
=\scp{\pi_{\Ats}f}{\xi}_{\Hit}
=\scp{f}{\xi}_{\Hit},
\end{align*}
i.e., \eqref{varxfsostwo} holds for $\xi\in D(\At)$.
On the other hand, if \eqref{varxfsostwo} holds for $\xi\in D(\At)$, 
then $\At\ti{x}_{f}\in D(\Ats)$ and $\Ats\At\ti{x}_{f}=f$,
especially\footnote{Another proof is the following: 
Pick $\xi\in N(\At)$ and get by \eqref{varxfsostwo} directly $\scp{f}{\xi}_{\Hit}=0$.
Thus $f\in N(\At)^{\bot_{\Hit}}=R(\Ats)$.}  
$f\in R(\Ats)$. As in this case $\ti{x}_{f}\in D(\cAt)$
and $\At\ti{x}_{f}\in D(\cAts)$ with $\Ats\At\ti{x}_{f}=f$, 
we get $\ti{x}_{f}=x_{f}$ by the injectivity of $\cAts$ and $\cAt$,
which shows (i').

In (i'') the unique solvability follows as in (i).
Let $f\in R(\Ats)$. Using the same arguments with the same projector $\pi_{\Ats}$ 
as in (i') we obtain by \eqref{varyfsos} for all $\tau\in D(\Ats\At)$
\begin{align*}
\scp{\Ats\At y_{f}}{\Ats\At\tau}_{\Hit}
&=\scp{\Ats\At y_{f}}{\Ats\At\pi_{\Ats}\tau}_{\Hit}
=\scp{f}{\pi_{\Ats}\tau}_{\Hit}
=\scp{\pi_{\Ats}f}{\tau}_{\Hit}
=\scp{f}{\tau}_{\Hit},
\end{align*}
as $\pi_{\Ats}\tau\in D(\Ats\cAt)=D(\cAts\cAt)$ by \eqref{recalldefxfxgsos}.
Thus \eqref{varyfsos} holds for all $\tau\in D(\Ats\At)$.
On the other hand, if \eqref{varyfsos} holds for all $\tau\in D(\Ats\At)$, 
then we obtain $\scp{f}{\tau}_{\Hit}=0$ for all $\tau\in N(\At)$,
showing $f\in N(\At)^{\bot_{\Hit}}=R(\Ats)$. Now, in this case of $f\in R(\Ats)=R(\cAts)$,
we define $h:=(\cAts)^{-1}f\in D(\cAts)$ and observe with $\Ats h=f$ that
by \eqref{varyfsos} for all $\tau\in D(\Ats\At)$
\begin{align}
\mylabel{fgcAtsmo}
\scp{\Ats\At y_{f}}{\Ats\At\tau}_{\Hit}
&=\scp{f}{\tau}_{\Hit}
=\scp{h}{\At\tau}_{\Hith}
=\scp{h}{\pi_{\At}\At\tau}_{\Hith}.
\end{align}
As in the proof of (i), let $\psi\in D(\Ats)$ and let it be decomposed into
$$\psi=\psi_{N}+\At\tau\in D(\Ats)
=N(\Ats)\oplus_{\Hith}\big(D(\Ats)\cap R(\At)\big),\qquad
\tau\in D(\Ats\At).$$
Using \eqref{fgcAtsmo} and the latter decomposition we see for all $\psi\in D(\Ats)$
$$\scp{\Ats\At y_{f}}{\Ats\psi}_{\Hit}
=\scp{\Ats\At y_{f}}{\Ats\At\tau}_{\Hit}
=\scp{h}{\pi_{\At}\At\tau}_{\Hith}
=\scp{h}{\pi_{\At}\psi}_{\Hith}
=\scp{h}{\psi}_{\Hith},$$
since $h\in D(\cAts)\subset R(\At)$.
Thus $\Ats\At y_{f}\in D(\At)$ and $\At\Ats\At y_{f}=h\in D(\cAts)$, showing 
$$\Ats\At\Ats\At y_{f}=\Ats h=f.$$
Since $(\Ats\At)^2y_{f}=(\cAts\cAt)^2y_{f}$ we get 
$\cAts\cAt(\cAts\cAt y_{f}-x_{f})=0$
and injectivity yields $\cAts\cAt y_{f}=x_{f}$.

(ii) and (ii') are clear from Theorem \ref{soltheofosvarform} (ii), (ii').
\end{proof} 

\begin{rem}
\label{soltheososvarformremmatrix}
Note that
\begin{align*}
\hat{x}_{f}=\ti{x}_{f}=x_{f}=\cA_{2}^{-1}(\cAts)^{-1}f
&\in D(\cAts\cAt),
&
\ti{x}_{g}=x_{g}=(\cAos)^{-1}g
&\in D(\cAos),\\
y_{f}=\cA_{2}^{-1}(\cAts)^{-1}x_{f}=\big(\cA_{2}^{-1}(\cAts)^{-1}\big)^2f
&\in D\big((\cAts\cAt)^2\big),
&
z_{g}=\cA_{1}^{-1}x_{g}=\cA_{1}^{-1}(\cAos)^{-1}g&\in D(\cAos\cAo)
\end{align*}
holds with $\Ats\At x_{f}=f$, $\Ats\At y_{f}=x_{f}$ 
and $\Aos x_{g}=g$, $\Ao z_{g}=x_{g}$.
Hence $x_{f}$, $x_{g}$, and $y_{f}$, $z_{g}$ solve 
the first resp. second order systems
\begin{align*}
\Ats\At x_{f}&=f,
&
\At x_{g}&=0,
&
\Ats\At y_{f}&=x_{f},
&
(\Ats\At)^2y_{f}&=f,
&
\Ao z_{g}&=x_{g},
&
\Aos\Ao z_{g}&=g,\\
\Aos x_{f}&=0,
&
\Aos x_{g}&=g,
&
\Aos y_{f}&=0,
&
\Aos y_{f}&=0,
&
\Azs z_{g}&=0,
&
\Azs z_{g}&=0,\\
\pit x_{f}&=0,
&
\pit x_{g}&=0,
&
\pit y_{f}&=0,
&
\pit y_{f}&=0,
&
\pio z_{g}&=0,
&
\pio z_{g}&=0.
\end{align*}
Moreover:
\begin{itemize}
\item[\bf(i)]
\eqref{varxfsosone} is a weak formulation of 
$$(\Ats\At)^2\hat{x}_{f}=\Ats\At f,\qquad
\Aos\hat{x}_{f}=0,\qquad
\pit\hat{x}_{f}=0,$$
i.e., in formal matrix notation
$$\begin{bmatrix}
(\Ats\At)^2\\
\Aos\\
\pit
\end{bmatrix}
\begin{bmatrix}
\hat{x}_{f}
\end{bmatrix}
=\begin{bmatrix}
\Ats\At f\\
0\\
0
\end{bmatrix}.$$
\item[\bf(i')]
\eqref{varxfsostwo} is a weak formulation of 
$$\Ats\At\ti{x}_{f}=f,\qquad
\Aos\ti{x}_{f}=0,\qquad
\pit\ti{x}_{f}=0,$$
i.e., in formal matrix notation
$$\begin{bmatrix}
\Ats\At\\
\Aos\\
\pit
\end{bmatrix}
\begin{bmatrix}
\ti{x}_{f}
\end{bmatrix}
=\begin{bmatrix}
f\\
0\\
0
\end{bmatrix}.$$
\item[\bf(i'')]
\eqref{varyfsos} is a weak formulation of
$$(\At\Ats)^2y_{f}=f,\qquad
\Aos y_{f}=0,\qquad
\pit y_{f}=0,$$
i.e., in formal matrix notation
$$\begin{bmatrix}
(\At\Ats)^2\\
\Aos\\
\pit
\end{bmatrix}
\begin{bmatrix}
y_{f}
\end{bmatrix}
=\begin{bmatrix}
f\\
0\\
0
\end{bmatrix}.$$
\item[\bf(ii)]
\eqref{varxgsos} is a weak formulation of
$$\Ao\Aos\ti{x}_{g}=\Ao g,\qquad
\At\ti{x}_{g}=0,\qquad
\pit\ti{x}_{g}=0,$$
i.e., in formal matrix notation
$$\begin{bmatrix}
\Ao\Aos\\
\At\\
\pit
\end{bmatrix}
\begin{bmatrix}
\ti{x}_{g}
\end{bmatrix}
=\begin{bmatrix}
\Ao g\\
0\\
0
\end{bmatrix}.$$
\item[\bf(ii')]
\eqref{varzgsos} is a weak formulation of
$$\Aos\Ao z_{g}=g,\qquad
\Azs z_{g}=0,\qquad
\pio z_{g}=0,$$
i.e., in formal matrix notation
$$\begin{bmatrix}
\Aos\Ao\\
\Azs\\
\pio
\end{bmatrix}
\begin{bmatrix}
z_{g}
\end{bmatrix}
=\begin{bmatrix}
g\\
0\\
0
\end{bmatrix}.$$
\end{itemize}
\end{rem}

As before we emphasize that the variational formulations 
\eqref{varxfsosone}-\eqref{varzgsos}
have again saddle point structure. 
Provided $f\in R(\Ats)$ and $g\in R(\Aos)$
the formulations \eqref{varxfsosone}-\eqref{varzgsos}
are equivalent to the following five problems: Find
\begin{align*}
\hat{x}_{f},y_{f}&\in D(\cAts\cAt)=D(\Ats\cAt)
=D(\Ats\At)\cap R(\Ats)=D(\Ats\At)\cap N(\At)^{\bot_{\Hit}},\\
\ti{x}_{f}&\in D(\cAt)=D(\At)\cap R(\Ats)=D(\At)\cap N(\At)^{\bot_{\Hit}},\\
\ti{x}_{g}&\in D(\cAos)=D(\Aos)\cap R(\Ao)=D(\Aos)\cap N(\Aos)^{\bot_{\Hit}},\\
z_{g}&\in D(\cAo)=D(\Ao)\cap R(\Aos)=D(\Ao)\cap N(\Ao)^{\bot_{\Hio}},
\end{align*}
such that
\begin{align}
\mylabel{varxfsostwoone}
\forall\,\phi&\in D(\Ats\At)
&
\scp{\Ats\At\hat{x}_{f}}{\Ats\At\phi}_{\Hit}
&=\scp{f}{\Ats\At\phi}_{\Hit},\\
\mylabel{varxfsostwotwo}
\forall\,\xi&\in D(\At)
&
\scp{\At\ti{x}_{f}}{\At\xi}_{\Hith}
&=\scp{f}{\xi}_{\Hit},\\
\mylabel{varyfsostwo}
\forall\,\tau&\in D(\Ats\At)
&
\scp{\Ats\At y_{f}}{\Ats\At\tau}_{\Hit}
&=\scp{f}{\tau}_{\Hit},\\
\mylabel{varxgsostwo}
\forall\,\zeta&\in D(\Aos)
&
\scp{\Aos\ti{x}_{g}}{\Aos\zeta}_{\Hio}
&=\scp{g}{\Aos\zeta}_{\Hio},\\
\mylabel{varzgsostwo}
\forall\,\varphi&\in D(\Ao)
&
\scp{\Ao z_{g}}{\Ao\varphi}_{\Hit}
&=\scp{g}{\varphi}_{\Hio}.
\end{align}
Similar to the first order case, 
the variational formulations \eqref{varxfsostwoone}-\eqref{varzgsostwo}
are equivalent to the following five saddle point problems:
Find $\hat{x}_{f},y_{f}\in D(\Ats\At)$, $\ti{x}_{f}\in D(\At)$,
$\ti{x}_{g}\in D(\Aos)$, $z_{g}\in D(\Ao)$, such that
\begin{align*}
\forall\,\phi&\in D(\Ats\At)
&
\scp{\Ats\At\hat{x}_{f}}{\Ats\At\phi}_{\Hit}
&=\scp{f}{\Ats\At\phi}_{\Hit}
&
&\wedge
&
\forall\,\theta&\in N(\At)
&
\scp{\hat{x}_{f}}{\theta}_{\Hit}
&=0,\\
\forall\,\xi&\in D(\At)
&
\scp{\At\ti{x}_{f}}{\At\xi}_{\Hith}
&=\scp{f}{\xi}_{\Hit}
&
&\wedge
&
\forall\,\kappa&\in N(\At)
&
\scp{\ti{x}_{f}}{\kappa}_{\Hit}
&=0,\\
\forall\,\tau&\in D(\Ats\At)
&
\scp{\Ats\At y_{f}}{\Ats\At\tau}_{\Hit}
&=\scp{f}{\tau}_{\Hit}
&
&\wedge
&
\forall\,\sigma&\in N(\At)
&
\scp{y_{f}}{\sigma}_{\Hit}
&=0,\\
\forall\,\zeta&\in D(\Aos)
&
\scp{\Aos\ti{x}_{g}}{\Aos\zeta}_{\Hio}
&=\scp{g}{\Aos\zeta}_{\Hio}
&
&\wedge
&
\forall\,\lambda&\in N(\Aos)
&
\scp{\ti{x}_{g}}{\lambda}_{\Hit}
&=0,\\
\forall\,\varphi&\in D(\Ao)
&
\scp{\Ao z_{g}}{\Ao\varphi}_{\Hit}
&=\scp{g}{\varphi}_{\Hio}
&
&\wedge
&
\forall\,\psi&\in N(\Ao)
&
\scp{z_{g}}{\psi}_{\Hio}
&=0.
\end{align*}
At this point, we have followed the corresponding section for the first order problems
up to \eqref{saddlexfone}-\eqref{saddlezgone}.
We emphasize that all considerations after \eqref{saddlexfone}-\eqref{saddlezgone}
can be repeated here, giving similar saddle point formulations 
for the second order problem as well.
As an example we present a corresponding result to Theorem \ref{soltheofosvarformtogether}
for finding the solution $x=x_{f}+x_{g}$ in just one variational saddle point formulation.
For this, let us pick, e.g., the two formulations \eqref{varxfsostwotwo}
and \eqref{varxgsostwo} together 
with the (very) weak versions of $\Aos x=g$ resp. $\Ats\At x=f$.

\begin{theo}
\label{soltheososvarformtogether}
Let $K_{2}=\{0\}$.
The unique solution $x=x_{f}+x_{g}\in\ti{D}_{2}$ in Theorem \ref{soltheosos} 
can be found by the following two variational saddle point formulations: 
\begin{itemize}
\item[\bf(i)]
There exists a unique pair $(\ti{x},z)\in D(\At)\times D(\cAo)$ such that
\begin{align}
\mylabel{varxzthreeone}
\forall\,(\xi,\varphi)&\in D(\At)\times D(\cAo)
&
\scp{\At\ti{x}}{\At\xi}_{\Hith}
+\scp{\Ao z}{\xi}_{\Hit}
&=\scp{f}{\xi}_{\Hit},\\
\mylabel{varxzthreetwo}
&&
\scp{\ti{x}}{\Ao\varphi}_{\Hit}
&=\scp{g}{\varphi}_{\Hio}.
\intertext{It holds $z=0$, if and only if $f\in R(\Ats)$, if and only if
$\At\ti{x}\in D(\Ats)$ with $\Ats\At\ti{x}=f$.
In this case}
\mylabel{varxzfourone}
\forall\,(\xi,\varphi)&\in D(\At)\times D(\cAo)
&
\scp{\At\ti{x}}{\At\xi}_{\Hith}
&=\scp{f}{\xi}_{\Hit},\\
\mylabel{varxzfourtwo}
&&
\scp{\ti{x}}{\Ao\varphi}_{\Hit}
&=\scp{g}{\varphi}_{\Hio}.
\end{align}
\eqref{varxzthreetwo}, \eqref{varxzfourtwo} hold for all $\varphi\in D(\Ao)$
if and only if $g\in R(\Aos)$ if and only if $\ti{x}\in D(\Aos)$ with $\Aos\ti{x}=g$.
Moreover, if $f\in R(\Ats)$ and $g\in R(\Aos)$, we have $\ti{x}=x$ 
from Theorem \ref{soltheosos}.
\item[\bf(ii)]
There exists a unique pair $(\hat{x},y)\in D(\Aos)\times D(\Ats\cAt)$ such that
\begin{align}
\mylabel{varxythreeone}
\forall\,(\zeta,\phi)&\in D(\Aos)\times D(\Ats\cAt)
&
\scp{\Aos\hat{x}}{\Aos\zeta}_{\Hio}
+\scp{\Ats\At y}{\zeta}_{\Hit}
&=\scp{g}{\Aos\zeta}_{\Hio},\\
\mylabel{varxythreetwo}
&&
\scp{\hat{x}}{\Ats\At\phi}_{\Hit}
&=\scp{f}{\phi}_{\Hit}.
\intertext{It holds $y=0$ as well as}
\mylabel{varxyfourone}
\forall\,(\zeta,\phi)&\in D(\Aos)\times D(\Ats\cAt)
&
\scp{\Aos\hat{x}}{\Aos\zeta}_{\Hio}
&=\scp{g}{\Aos\zeta}_{\Hio},\\
\mylabel{varxyfourtwo}
&&
\scp{\hat{x}}{\Ats\At\phi}_{\Hit}
&=\scp{f}{\phi}_{\Hit}.
\end{align}
Moreover, $\Aos\hat{x}=g$ if and only if $g\in R(\Aos)$.
\eqref{varxythreetwo}, \eqref{varxyfourtwo} hold for all $\phi\in D(\Ats\At)$
if and only if $f\in R(\Ats)$ if and only if $\hat{x}\in D(\Ats\At)$ with $\Ats\At\hat{x}=f$.
In this case, i.e., $f\in R(\Ats)$ and $g\in R(\Aos)$, we have $\hat{x}=x$ 
from Theorem \ref{soltheosos}.
\end{itemize}
\end{theo}

\begin{proof}
Unique solvability of (i) follows again by standard saddle point theory
as in Theorem \ref{soltheofosvarformtogether} (i).
Inserting $\xi:=\Ao z\in R(\Ao)=N(\At)=R(\Ats)^{\bot_{\Hit}}$ in \eqref{varxzthreeone} yields
$\norm{\Ao z}_{\Hit}^2=\scp{f}{\Ao z}_{\Hit}$ and hence $\Ao z=0$,
even $z=0$ as $z\in D(\cAo)$,
if $f\in R(\Ats)$. On the other hand, if $\Ao z=0$ then \eqref{varxzthreeone} shows 
$\scp{f}{\xi}_{\Hit}=0$ for all $\xi\in N(\At)$, 
i.e., $f\in N(\At)^{\bot_{\Hit}}=R(\Ats)$. Moreover, if $f\in R(\Ats)$,
then \eqref{varxzfourone}-\eqref{varxzfourtwo} hold. 
Especially \eqref{varxzfourone} yields $\At\ti{x}\in D(\Ats)$ and $\Ats\At\ti{x}=f$.
The assertions related to $g$ follow as in the proof of 
Theorem \ref{soltheofosvarformtogether} (i).
Theorem \ref{soltheosos} yields $\ti{x}=x$, which completes the proof of (i).

For (ii), we pick $\psi\in D(\Ats)$ and decompose it 
as in the proof of Theorem \ref{soltheososvarform} (i) into
$$\psi=\psi_{N}+\At\phi_{R}\in D(\Ats)
=N(\Ats)\oplus_{\Hith}\big(D(\cAts)\cap R(\cAt)\big),\qquad
\phi_{R}\in D(\cAts\cAt).$$
If $f=0$, then using the latter decomposition we see for all $\psi\in D(\Ats)$
\begin{align*}
\scp{\hat{x}}{\Ats\psi}_{\Hit}
&=\scp{\hat{x}}{\Ats\At\phi_{R}}_{\Hit}
=0,
\end{align*}
which holds if and only if $\hat{x}\in N(\At)$. Thus
the kernel of \eqref{varxythreetwo} equals $N(\At)$.
By Corollary \ref{cortoolboxone} (i) 
the principal part of \eqref{varxythreeone} is strictly positive over the kernel
of \eqref{varxythreetwo}, which is 
$$D(\Aos)\cap N(\At)=D(\Aos)\cap R(\Ao)=R(\cAo)$$
as we just have derived and since $K_{2}=\{0\}$. 
Moreover, we have for $0\neq\phi\in D(\Ats\cAt)$
\begin{align*}
\frac{\norm{\Ats\At\phi}_{\Hit}}{\norm{\phi}_{D(\Ats\cAt)}}
\leq\sup_{0\neq\zeta\in D(\Aos)}
\frac{\scp{\Ats\At\phi}{\zeta}_{\Hit}}{\norm{\phi}_{D(\Ats\cAt)}\norm{\zeta}_{D(\Aos)}}
\leq\frac{\norm{\Ats\At\phi}_{\Hit}}{\norm{\phi}_{D(\Ats\cAt)}}
\leq1
\end{align*}
by choosing\footnote{Indeed we can easily see $R(\Ats\cAt)=R(\Ats)$,
since $R(\Ats\cAt)\subset R(\Ats)$ holds and for $\zeta\in R(\Ats)=R(\cAts)$ 
there is $\phi:=\cA_{2}^{-1}(\cAts)^{-1}\zeta\in D(\cAts\cAt)$ with $\zeta=\Ats\At\phi\in R(\cAts\cAt)=R(\Ats\cAt)$.}
$\zeta:=\Ats\At\phi\in R(\Ats\cAt)=R(\Ats)=N(\Aos)$,
which shows that actually equality holds. Hence
\begin{align*}
1
&\geq\inf_{0\neq\phi\in D(\Ats\cAt)}\sup_{0\neq\zeta\in D(\Aos)}
\frac{\scp{\Ats\At\phi}{\zeta}_{\Hit}}{\norm{\phi}_{D(\Ats\cAt)}\norm{\zeta}_{D(\Aos)}}\\
&=\inf_{0\neq\phi\in D(\Ats\cAt)}\frac{\norm{\Ats\At\phi}_{\Hit}}{\norm{\phi}_{D(\cAts\cAt)}}
\geq(c_{2}^4+c_{2}^2+1)^{\moh}
=\bnorm{\cA_{2}^{-1}(\cAts)^{-1}}_{R(\Ats),D(\cAts\cAt)}^{-1},
\end{align*}
which shows that the inf-sup-condition is satisfied.
Therefore, \eqref{varxythreeone}-\eqref{varxythreetwo} admits a unique solution
by the saddle point theory.
Picking $\zeta:=\Ats\At y\in R(\Ats)=N(\Aos)$ in \eqref{varxythreeone} yields
$\norm{\Ats\At y}_{\Hit}^2=0$ and hence $y=0$ as $y\in D(\Ats\cAt)=D(\cAts\cAt)$. 
Since $\Ats\At y=0$ even \eqref{varxyfourone}-\eqref{varxyfourtwo} are valid.
By \eqref{varxyfourone} we see $\Aos\hat{x}-g\in R(\Aos)^{\bot_{\Hio}}$,
showing $\Aos\hat{x}=g$ if and only if $g\in R(\Aos)$.
Using the orthonormal projector $\pi_{\Ats}$ and
by \eqref{varxyfourtwo} we see for all $\phi\in D(\Ats\At)$ as
$\pi_{\Ats}\phi\in D(\Ats\cAt)=D(\cAts\cAt)$
$$\scp{\hat{x}}{\Ats\At\phi}_{\Hit}
=\scp{\hat{x}}{\Ats\At\pi_{\Ats}\phi}_{\Hit}
=\scp{f}{\pi_{\Ats}\phi}_{\Hit}
=\scp{\pi_{\Ats}f}{\phi}_{\Hit}
=\scp{f}{\phi}_{\Hit},$$
if $f\in R(\Ats)$. On the other hand, 
if \eqref{varxyfourtwo} holds for all $\phi\in D(\Ats\At)$,
then $\scp{f}{\phi}_{\Hit}=0$ for all $\phi\in N(\At)$ and hence
$f\in N(\At)^{\bot_{\Hit}}=R(\Ats)$.
Now, following the proof of Theorem \ref{soltheososvarform} (i''), 
let $f\in R(\Ats)=R(\cAts)$ as well as
define $h:=(\cAts)^{-1}f\in D(\cAts)$ and observe with $\Ats h=f$ that
by \eqref{varxyfourtwo} for all $\phi\in D(\Ats\At)$
\begin{align}
\mylabel{fgcAtsmotwo}
\scp{\hat{x}}{\Ats\At\phi}_{\Hit}
&=\scp{f}{\phi}_{\Hit}
=\scp{\Ats h}{\phi}_{\Hit}
=\scp{h}{\At\phi}_{\Hith}
=\scp{h}{\pi_{\At}\At\phi}_{\Hith}.
\end{align}
As before, let $\psi\in D(\Ats)$ and let it be decomposed into
$$\psi=\psi_{N}+\At\phi\in D(\Ats)
=N(\Ats)\oplus_{\Hith}\big(D(\Ats)\cap R(\At)\big),\qquad
\phi\in D(\Ats\At).$$
Using \eqref{fgcAtsmotwo} and the latter decomposition we see for all $\psi\in D(\Ats)$
$$\scp{\hat{x}}{\Ats\psi}_{\Hit}
=\scp{\hat{x}}{\Ats\At\phi}_{\Hit}
=\scp{h}{\pi_{\At}\At\phi}_{\Hith}
=\scp{h}{\pi_{\At}\psi}_{\Hith}
=\scp{h}{\psi}_{\Hith},$$
since $h\in D(\cAts)\subset R(\At)$.
Thus $\hat{x}\in D(\At)$ and $\At\hat{x}=h\in D(\cAts)$, showing $\hat{x}\in D(\Ats\At)$ with
$$\Ats\At\hat{x}=\Ats h=f.$$
Finally, if $f\in R(\Ats)$ and $g\in R(\Aos)$, we have 
$\hat{x}\in D(\Ats\At)\cap D(\Aos)=\ti{D}_{2}$ 
with $\Ats\At\hat{x}=f$ and $\Aos\hat{x}=g$ and thus $\hat{x}=x$
by Theorem \ref{soltheosos}, completing the proof.
\end{proof}


\begin{rem}
Let us note the following:
\begin{itemize}
\item[\bf(i)]
\eqref{varxzthreeone}-\eqref{varxzthreetwo} is a weak formulation of 
$$\Ats\At\ti{x}+\Ao z=f,\qquad
\Aos\ti{x}=g,$$
i.e., in formal matrix notation
$$\begin{bmatrix}
\Ats\At & \Ao\\
\Aos & 0
\end{bmatrix}
\begin{bmatrix}
\ti{x}\\
z
\end{bmatrix}
=\begin{bmatrix}
f\\
g
\end{bmatrix}.$$
Note $z=0$.
\item[\bf(ii)]
\eqref{varxythreeone}-\eqref{varxythreetwo} is a weak formulation of
$$\Ao\Aos\hat{x}+\Ats\At y=\Ao g,\qquad
\Ats\At\hat{x}=f,$$
i.e., in formal matrix notation
$$\begin{bmatrix}
\Ao\Aos & \Ats\At\\
\Ats\At & 0
\end{bmatrix}
\begin{bmatrix}
\hat{x}\\
y
\end{bmatrix}
=\begin{bmatrix}
\Ao g\\
f
\end{bmatrix}.$$
Note $y=0$.
\end{itemize}
\end{rem}

A corresponding result to Theorem \ref{soltheofosvarformtogetherwithK2} 
can be formulated as well, skipping the assumption $K_{2}=\{0\}$ 
in Theorem \ref{soltheososvarformtogether}.

\begin{theo}
\label{soltheososvarformtogetherwithK2}
The unique solution $x=x_{f}+x_{g}+k\in\ti{D}_{2}$ in Theorem \ref{soltheosos} 
can be found by the following two variational double saddle point formulations: 
\begin{itemize}
\item[\bf(i)]
There exists a unique tripple $(\ti{x},z,h)\in D(\At)\times D(\cAo)\times K_{2}$ such that
\begin{align}
\nonumber
\forall\,(\xi,\varphi,\kappa)&\in D(\At)\times D(\cAo)\times K_{2}
&
\scp{\At\ti{x}}{\At\xi}_{\Hith}
+\scp{\Ao z}{\xi}_{\Hit}
+\scp{h}{\xi}_{\Hit}
&=\scp{f}{\xi}_{\Hit},\\
\mylabel{doublesaddlepointfullsystemsosoneone}
&&
\scp{\ti{x}}{\Ao\varphi}_{\Hit}
&=\scp{g}{\varphi}_{\Hio},\\
\nonumber
&&
\scp{\ti{x}}{\kappa}_{\Hit}
&=\scp{k}{\kappa}_{\Hit}.
\intertext{It holds $z=0$ and $h=0$, if and only if $f\in R(\Ats)$, if and only if
$\At\ti{x}\in D(\Ats)$ with $\Ats\At\ti{x}=f$.
In this case}
\nonumber
\forall\,(\xi,\varphi,\kappa)&\in D(\At)\times D(\cAo)\times K_{2}
&
\scp{\At\ti{x}}{\At\xi}_{\Hith}
&=\scp{f}{\xi}_{\Hit},\\
\mylabel{doublesaddlepointfullsystemsosonetwo}
&&
\scp{\ti{x}}{\Ao\varphi}_{\Hit}
&=\scp{g}{\varphi}_{\Hio},\\
\nonumber
&&
\scp{\ti{x}}{\kappa}_{\Hit}
&=\scp{k}{\kappa}_{\Hit}.
\end{align}
\eqref{doublesaddlepointfullsystemsosoneone}, \eqref{doublesaddlepointfullsystemsosonetwo} 
hold for all $\varphi\in D(\Ao)$
if and only if $g\in R(\Aos)$ if and only if $\ti{x}\in D(\Aos)$ with $\Aos\ti{x}=g$.
Furthermore, $\pit\ti{x}=k$.
Moreover, if $f\in R(\Ats)$ and $g\in R(\Aos)$, we have $\ti{x}=x$ 
from Theorem \ref{soltheosos}.
\item[\bf(ii)]
There exists a unique tripple $(\hat{x},y,h)\in D(\Aos)\times D(\Ats\cAt)\times K_{2}$ such that
\begin{align}
\nonumber
\forall\,(\zeta,\phi,\kappa)&\in D(\Aos)\times D(\Ats\cAt)\times K_{2}
&
\scp{\Aos\hat{x}}{\Aos\zeta}_{\Hio}
+\scp{\Ats\At y}{\zeta}_{\Hit}
+\scp{h}{\zeta}_{\Hit}
&=\scp{g}{\Aos\zeta}_{\Hio},\\
\mylabel{doublesaddlepointfullsystemsostwoone}
&&
\scp{\hat{x}}{\Ats\At\phi}_{\Hit}
&=\scp{f}{\phi}_{\Hit},\\
\nonumber
&&
\scp{\hat{x}}{\kappa}_{\Hit}
&=\scp{k}{\kappa}_{\Hit}.
\intertext{It holds $y=0$ and $h=0$ as well as}
\nonumber
\forall\,(\zeta,\phi,\kappa)&\in D(\Aos)\times D(\Ats\cAt)\times K_{2}
&
\scp{\Aos\hat{x}}{\Aos\zeta}_{\Hio}
&=\scp{g}{\Aos\zeta}_{\Hio},\\
\mylabel{doublesaddlepointfullsystemsostwotwo}
&&
\scp{\hat{x}}{\Ats\At\phi}_{\Hit}
&=\scp{f}{\phi}_{\Hit},\\
\nonumber
&&
\scp{\hat{x}}{\kappa}_{\Hit}
&=\scp{k}{\kappa}_{\Hit}.
\end{align}
Moreover, $\Aos\hat{x}=g$ if and only if $g\in R(\Aos)$.
\eqref{doublesaddlepointfullsystemsostwoone}, \eqref{doublesaddlepointfullsystemsostwotwo} 
hold for all $\phi\in D(\Ats\At)$
if and only if $f\in R(\Ats)$ if and only if $\hat{x}\in D(\Ats\At)$ with $\Ats\At\hat{x}=f$.
Furthermore, $\pit\hat{x}=k$.
In this case, i.e., $f\in R(\Ats)$ and $g\in R(\Aos)$, we have $\hat{x}=x$ 
from Theorem \ref{soltheosos}.
\end{itemize}
\end{theo}

\begin{rem}
\label{soltheososvarformtogetherrem}
Let us note the following:
\begin{itemize}
\item[\bf(i)]
Literally, Remark \ref{soltheofosvarformtogetherrem} (i) holds here as well.
\item[\bf(ii)]
\eqref{doublesaddlepointfullsystemsosoneone} is a weak formulation of 
$$\Ats\At\ti{x}+\Ao z+h=f,\qquad
\Aos\ti{x}=g,\qquad
\pit\ti{x}=k,$$
i.e., in formal matrix notation
$$\begin{bmatrix}
\Ats\At & \Ao & \iota_{K_{2}}\\
\Aos & 0 & 0\\
\pit=\iota_{K_{2}}^{*} & 0 & 0
\end{bmatrix}
\begin{bmatrix}
\ti{x}\\
z\\
h
\end{bmatrix}
=\begin{bmatrix}
f\\
g\\
k
\end{bmatrix}.$$
Note $z=0$ and $h=0$.
\item[\bf(ii')]
\eqref{doublesaddlepointfullsystemsostwoone} is a weak formulation of
$$\Ao\Aos\hat{x}+\Ats\At y+h=\Ao g,\qquad
\Ats\At\hat{x}=f,\qquad
\pit\hat{x}=k,$$
i.e., in formal matrix notation
$$\begin{bmatrix}
\Ao\Aos & \Ats\At & \iota_{K_{2}}\\
\Ats\At & 0 & 0\\
\pit=\iota_{K_{2}}^{*} & 0 & 0
\end{bmatrix}
\begin{bmatrix}
\hat{x}\\
y\\
h
\end{bmatrix}
=\begin{bmatrix}
\Ao g\\
f\\
k
\end{bmatrix}.$$
Note $y=0$ and $h=0$.
\end{itemize}
\end{rem}

For the partial solutions and potentials in 
Theorem \ref{soltheosos} and Theorem \ref{soltheososvarform} 
a corresponding result to Theorem \ref{soltheofosvarformpartsolfull}
can be proved as well. It reads as follows:

\begin{theo}
\label{soltheososvarformpartsolfull}
Let additionally $R(\Az)$ be closed.
The partial solutions $x_{f}=\hat{x}_{f}=\ti{x}_{f}\in D(\cAts\cAt)$, $x_{g}=\ti{x}_{g}\in D(\cAos)$,
and their potentials $y_{f}\in D(\cAts\cAt)$, $z_{g}\in D(\cAo)$
from Theorem \ref{soltheosos}, Theorem \ref{soltheososvarform}, and \eqref{varxfsostwoone}-\eqref{varzgsostwo}
can be found by the following six variational double saddle point formulations: 
\begin{itemize}
\item[\bf(i)]
There exists a unique tripple $(\hat{x}_{f},w,h)\in D(\Ats\At)\times D(\cAo)\times K_{2}$ such that
\begin{align}
\nonumber
\forall\,(\psi,\varphi,\kappa)&\in D(\Ats\At)\times D(\cAo)\times K_{2}
&
\scp{\Ats\At\hat{x}_{f}}{\Ats\At\psi}_{\Hit}
+\scp{\Ao w}{\psi}_{\Hit}
+\scp{h}{\psi}_{\Hit}
&=\scp{f}{\Ats\At\psi}_{\Hit},\\
\mylabel{doublesaddlepointpartsystemsoshatxfoneone}
&&
\scp{\hat{x}_{f}}{\Ao\varphi}_{\Hit}
&=0,\\
\nonumber
&&
\scp{\hat{x}_{f}}{\kappa}_{\Hit}
&=0.
\end{align}
It holds $w=0$ and $h=0$.
Moreover, $\At\hat{x}_{f}\in D(\Ats)$ and $\Ats\At\hat{x}_{f}=f$ if and only if $f\in R(\Ats)$.
The second equation of 
\eqref{doublesaddlepointpartsystemsoshatxfoneone} holds for all $\varphi\in D(\Ao)$
and thus $\hat{x}_{f}\in N(\Aos)$.
Furthermore, $\pit\hat{x}_{f}=0$. 
Finally, if $f\in R(\Ats)$, we have $\hat{x}_{f}=x_{f}$ from Theorem \ref{soltheosos},
see Theorem \ref{soltheososvarform} (i).
\item[\bf(i')]
There exists a unique tripple $(\ti{x}_{f},u,h)\in D(\At)\times D(\cAo)\times K_{2}$ such that
\begin{align}
\nonumber
\forall\,(\xi,\varphi,\kappa)&\in D(\At)\times D(\cAo)\times K_{2}
&
\scp{\At\ti{x}_{f}}{\At\xi}_{\Hith}
+\scp{\Ao u}{\xi}_{\Hit}
+\scp{h}{\xi}_{\Hit}
&=\scp{f}{\xi}_{\Hit},\\
\mylabel{doublesaddlepointpartsystemsostixfoneone}
&&
\scp{\ti{x}_{f}}{\Ao\varphi}_{\Hit}
&=0,\\
\nonumber
&&
\scp{\ti{x}_{f}}{\kappa}_{\Hit}
&=0.
\end{align}
It holds $u=0$ if and only if $f\bot_{\Hit}R(\Ao)$ if and only if $f\in N(\Aos)$.
$h=0$ if and only if $f\bot_{\Hit}K_{2}$.
Thus $u=0$ and $h=0$ if and only if $f\in N(\Aos)\cap K_{2}^{\bot_{\Hit}}=R(\Ats)$.
Moreover, $\At\hat{x}_{f}\in D(\Ats)$ and $\Ats\At\ti{x}_{f}=f$ if and only if $f\in R(\Ats)$.
The second equation of 
\eqref{doublesaddlepointpartsystemsostixfoneone} holds for all $\varphi\in D(\Ao)$
and hence $\ti{x}_{f}\in N(\Aos)$.
Furthermore, $\pit\ti{x}_{f}=0$. 
Finally, if $f\in R(\Ats)$, we have $\ti{x}_{f}=x_{f}$ from Theorem \ref{soltheosos},
see Theorem \ref{soltheososvarform} (i').
\item[\bf(i'')]
There exists a unique tripple $(y_{f},v,h)\in D(\Ats\At)\times D(\cAo)\times K_{2}$ such that
\begin{align}
\nonumber
\forall\,(\psi,\varphi,\kappa)&\in D(\Ats\At)\times D(\cAo)\times K_{2}
&
\scp{\Ats\At y_{f}}{\Ats\At\psi}_{\Hit}
+\scp{\Ao v}{\psi}_{\Hit}
+\scp{h}{\psi}_{\Hit}
&=\scp{f}{\psi}_{\Hit},\\
\mylabel{doublesaddlepointpartsystemsosyfoneone}
&&
\scp{y_{f}}{\Ao\varphi}_{\Hit}
&=0,\\
\nonumber
&&
\scp{y_{f}}{\kappa}_{\Hit}
&=0.
\end{align}
It holds $v=0$ if and only if $f\bot_{\Hit}R(\Ao)$ if and only if $f\in N(\Aos)$.
$h=0$ if and only if $f\bot_{\Hit}K_{2}$.
Thus $v=0$ and $h=0$ if and only if $f\in N(\Aos)\cap K_{2}^{\bot_{\Hit}}=R(\Ats)$.
Moreover, $\Ats\At y_{f}\in D(\Ats\At)$ and $(\Ats\At)^2y_{f}=f$ if and only if $f\in R(\Ats)$.
The second equation of 
\eqref{doublesaddlepointpartsystemsosyfoneone} holds for all $\varphi\in D(\Ao)$
and thus $y_{f}\in N(\Aos)$.
Furthermore, $\pit y_{f}=0$. 
Finally, if $f\in R(\Ats)$, we have $\Ats\At y_{f}=x_{f}$ from Theorem \ref{soltheosos},
see Theorem \ref{soltheososvarform} (i'').
\item[\bf(ii)]
There exists a unique triple $(\ti{x}_{g},p,h)\in D(\Aos)\times D(\cAts)\times K_{2}$ such that
\begin{align}
\nonumber
\forall\,(\zeta,\phi,\kappa)&\in D(\Aos)\times D(\cAts)\times K_{2}
&
\scp{\Aos\ti{x}_{g}}{\Aos\zeta}_{\Hio}
+\scp{\Ats p}{\zeta}_{\Hit}
+\scp{h}{\zeta}_{\Hit}
&=\scp{g}{\Aos\zeta}_{\Hio},\\
\mylabel{doublesaddlepointpartsystemsostixgtwoone}
&&
\scp{\ti{x}_{g}}{\Ats\phi}_{\Hit}
&=0,\\
\nonumber
&&
\scp{\ti{x}_{g}}{\kappa}_{\Hit}
&=0.
\end{align}
It holds $p=0$ and $h=0$.
Moreover, $\Aos\ti{x}_{g}=g$ if and only if $g\in R(\Aos)$.
The second equation of \eqref{doublesaddlepointpartsystemsostixgtwoone}
holds for all $\phi\in D(\Ats)$ and thus $\ti{x}_{g}\in N(\At)$.
Furthermore, $\pit\ti{x}_{g}=0$. 
Finally, if $g\in R(\Aos)$, we have $\ti{x}_{g}=x_{g}$ from Theorem \ref{soltheosos},
see Theorem \ref{soltheososvarform} (ii).
\item[\bf(ii')]
There exists a unique triple $(\hat{x}_{g},q,h)\in D(\Aos)\times D(\Ats\cAt)\times K_{2}$ such that
\begin{align}
\nonumber
\forall\,(\zeta,\psi,\kappa)&\in D(\Aos)\times D(\Ats\cAt)\times K_{2}
&
\scp{\Aos\hat{x}_{g}}{\Aos\zeta}_{\Hio}
+\scp{\Ats\At q}{\zeta}_{\Hit}
+\scp{h}{\zeta}_{\Hit}
&=\scp{g}{\Aos\zeta}_{\Hio},\\
\mylabel{doublesaddlepointpartsystemsoshatxgtwoone}
&&
\scp{\hat{x}_{g}}{\Ats\At\psi}_{\Hit}
&=0,\\
\nonumber
&&
\scp{\hat{x}_{g}}{\kappa}_{\Hit}
&=0.
\end{align}
It holds $q=0$ and $h=0$.
Moreover, $\Aos\hat{x}_{g}=g$ if and only if $g\in R(\Aos)$.
The second equation of \eqref{doublesaddlepointpartsystemsoshatxgtwoone}
holds for all $\psi\in D(\Ats\At)$ and thus $\hat{x}_{g}\in N(\At)$
as $\hat{x}_{g}\,\bot_{\Hit}\,R(\Ats\At)=R(\Ats)$.
Furthermore, $\pit\hat{x}_{g}=0$. 
Finally, if $g\in R(\Aos)$, we have $\hat{x}_{g}=x_{g}$ from Theorem \ref{soltheosos},
see Theorem \ref{soltheososvarform} (ii).
\item[\bf(ii'')]
There exists a unique triple $(z_{g},r,h)\in D(\Ao)\times D(\cAz)\times K_{1}$ such that
\begin{align}
\nonumber
\forall\,(\varphi,\vartheta,\kappa)&\in D(\Ao)\times D(\cAz)\times K_{1}
&
\scp{\Ao z_{g}}{\Ao\varphi}_{\Hit}
+\scp{\Az r}{\varphi}_{\Hio}
+\scp{h}{\varphi}_{\Hio}
&=\scp{g}{\varphi}_{\Hio},\\
\mylabel{doublesaddlepointpartsystemsoszgtwoone}
&&
\scp{z_{g}}{\Az\vartheta}_{\Hio}
&=0,\\
\nonumber
&&
\scp{z_{g}}{\kappa}_{\Hio}
&=0.
\end{align}
It holds $r=0$ if and only if $g\bot_{\Hio}R(\Az)$ if and only if $g\in N(\Azs)$.
$h=0$ if and only if $g\bot_{\Hio}K_{1}$.
Thus $r=0$ and $h=0$ if and only if $g\in N(\Azs)\cap K_{1}^{\bot_{\Hio}}=R(\Aos)$.
Moreover, $\Ao z_{g}\in D(\Aos)$ and $\Aos\Ao z_{g}=g$ if and only if $g\in R(\Aos)$.
The second equation of 
\eqref{doublesaddlepointpartsystemsoszgtwoone} holds for all $\vartheta\in D(\Az)$
and hence $z_{g}\in N(\Azs)$.
Furthermore, $\pio z_{g}=0$. 
Finally, if $g\in R(\Aos)$, we have $\Ao z_{g}=x_{g}$ from Theorem \ref{soltheosos},
see Theorem \ref{soltheososvarform} (ii').
\end{itemize}
\end{theo}

\begin{proof}
The proof utilizes the same techniques as used before.
\end{proof}

\begin{rem}
\label{soltheososvarformpartsolfullrem}
The formulations in Theorem \ref{soltheososvarformpartsolfull} (i') resp.
Theorem \ref{soltheososvarformpartsolfull} (ii') are the same
as in Theorem \ref{soltheososvarformtogetherwithK2} (i) resp. Theorem \ref{soltheososvarformtogetherwithK2} (ii)
except of the right hand sides. We note that $\ti{x}=x$
can also be found by the formulation presented in Theorem \ref{soltheososvarformpartsolfull} (i).
\end{rem}

\begin{rem}
\label{soltheososvarformpartsolfullremmatrix}
Again we have formal matrix representations:
\begin{itemize}
\item[\bf(i)]
\eqref{doublesaddlepointpartsystemsoshatxfoneone} is a weak formulation of 
$$(\Ats\At)^2\hat{x}_{f}+\Ao w+h=\Ats\At f,\qquad
\Aos\hat{x}_{f}=0,\qquad
\pit\hat{x}_{f}=0,$$
i.e., in formal matrix notation
$$\begin{bmatrix}
(\Ats\At)^2 & \Ao & \iota_{K_{2}}\\
\Aos & 0 & 0\\
\pit=\iota_{K_{2}}^{*} & 0 & 0
\end{bmatrix}
\begin{bmatrix}
\hat{x}_{f}\\
w\\
h
\end{bmatrix}
=\begin{bmatrix}
\Ats\At f\\
0\\
0
\end{bmatrix}.$$
Note $w=0$ and $h=0$.
\item[\bf(i')]
\eqref{doublesaddlepointpartsystemsostixfoneone} is a weak formulation of 
$$\Ats\At\ti{x}_{f}+\Ao u+h=f,\qquad
\Aos\ti{x}_{f}=0,\qquad
\pit\ti{x}_{f}=0,$$
i.e., in formal matrix notation
$$\begin{bmatrix}
\Ats\At & \Ao & \iota_{K_{2}}\\
\Aos & 0 & 0\\
\pit=\iota_{K_{2}}^{*} & 0 & 0
\end{bmatrix}
\begin{bmatrix}
\ti{x}_{f}\\
u\\
h
\end{bmatrix}
=\begin{bmatrix}
f\\
0\\
0
\end{bmatrix}.$$
Note $u=0$ and $h=0$.
\item[\bf(i'')]
\eqref{doublesaddlepointpartsystemsosyfoneone} is a weak formulation of
$$(\At\Ats)^2y_{f}+\Ao v+h=f,\qquad
\Aos y_{f}=0,\qquad
\pit y_{f}=0,$$
i.e., in formal matrix notation
$$\begin{bmatrix}
(\At\Ats)^2 & \Ao & \iota_{K_{2}}\\
\Aos & 0 & 0\\
\pit=\iota_{K_{2}}^{*} & 0 & 0
\end{bmatrix}
\begin{bmatrix}
y_{f}\\
v\\
h
\end{bmatrix}
=\begin{bmatrix}
f\\
0\\
0
\end{bmatrix}.$$
Note $v=0$ and $h=0$.
\item[\bf(ii)]
\eqref{doublesaddlepointpartsystemsostixgtwoone} is a weak formulation of
$$\Ao\Aos\ti{x}_{g}+\Ats p+h=\Ao g,\qquad
\At\ti{x}_{g}=0,\qquad
\pit\ti{x}_{g}=0,$$
i.e., in formal matrix notation
$$\begin{bmatrix}
\Ao\Aos & \Ats & \iota_{K_{2}}\\
\At & 0 & 0\\
\pit=\iota_{K_{2}}^{*} & 0 & 0
\end{bmatrix}
\begin{bmatrix}
\ti{x}_{g}\\
p\\
h
\end{bmatrix}
=\begin{bmatrix}
\Ao g\\
0\\
0
\end{bmatrix}.$$
Note $p=0$ and $h=0$.
\item[\bf(ii')]
\eqref{doublesaddlepointpartsystemsoshatxgtwoone} is a weak formulation of
$$\Ao\Aos\hat{x}_{g}+\Ats\At q+h=\Ao g,\qquad
\Ats\At\hat{x}_{g}=0,\qquad
\pit\hat{x}_{g}=0,$$
i.e., in formal matrix notation
$$\begin{bmatrix}
\Ao\Aos & \Ats\At & \iota_{K_{2}}\\
\Ats\At & 0 & 0\\
\pit=\iota_{K_{2}}^{*} & 0 & 0
\end{bmatrix}
\begin{bmatrix}
\hat{x}_{g}\\
q\\
h
\end{bmatrix}
=\begin{bmatrix}
\Ao g\\
0\\
0
\end{bmatrix}.$$
Note $q=0$ and $h=0$.
\item[\bf(ii'')]
\eqref{doublesaddlepointpartsystemsoszgtwoone} is a weak formulation of
$$\Aos\Ao z_{g}+\Az r+h=g,\qquad
\Azs z_{g}=0,\qquad
\pio z_{g}=0,$$
i.e., in formal matrix notation
$$\begin{bmatrix}
\Aos\Ao & \Az & \iota_{K_{1}}\\
\Azs & 0 & 0\\
\pio=\iota_{K_{1}}^{*} & 0 & 0
\end{bmatrix}
\begin{bmatrix}
z_{g}\\
r\\
h
\end{bmatrix}
=\begin{bmatrix}
g\\
0\\
0
\end{bmatrix}.$$
Note $r=0$ and $h=0$.
\end{itemize}
\end{rem}

There is also an analogon for the quadruple saddle point formulations
presented in Theorem \ref{soltheofosvarformtogetherwithK2andmore}.
Let us recall from Theorem \ref{soltheososvarformtogetherwithK2}
$z\in D(\cAo)$ and $y\in D(\Ats\cAt)$, i.e.,
\begin{align*}
z\in R(\Aos)
&=N(\Azs)\cap K_{1}^{\bot_{\Hio}}
=R(\Az)^{\bot_{\Hio}}\cap K_{1}^{\bot_{\Hio}},\\
y\in R(\Ats)
&=N(\Aos)\cap K_{2}^{\bot_{\Hit}}
=R(\Ao)^{\bot_{\Hit}}\cap K_{2}^{\bot_{\Hit}}.
\end{align*}

\begin{theo}
\label{soltheososvarformtogetherwithK2andmore}
Let additionally $R(\Az)$ be closed.
Moreover, let $f\in R(\Ats)$ and $g\in R(\Aos)$.
The unique solution $x=x_{f}+x_{g}+k\in\ti{D}_{2}$ in Theorem \ref{soltheosos} 
can be found by the following two variational quadruple saddle point formulations: 
\begin{itemize}
\item[\bf(i)]
There exists a unique five tuple $(\ti{x},z,u,h_{2},h_{1})\in D(\At)\times D(\Ao)\times D(\cAz)\times K_{2}\times K_{1}$ 
such that for all $(\xi,\varphi,\vartheta,\kappa,\lambda)\in D(\At)\times D(\Ao)\times D(\cAz)\times K_{2}\times K_{1}$
\begin{align}
\nonumber
\scp{\At\ti{x}}{\At\xi}_{\Hith}
+\scp{\Ao z}{\xi}_{\Hit}
+\scp{h_{2}}{\xi}_{\Hit}
&=\scp{f}{\xi}_{\Hit},\\
\nonumber
\scp{\ti{x}}{\Ao\varphi}_{\Hit}
+\scp{\Az u}{\varphi}_{\Hio}
+\scp{h_{1}}{\varphi}_{\Hio}
&=\scp{g}{\varphi}_{\Hio},\\
\mylabel{doublesaddlepointfullsystemsosoneoneandmore}
\scp{z}{\Az\vartheta}_{\Hio}
&=0,\\
\nonumber
\scp{\ti{x}}{\kappa}_{\Hit}
&=\scp{k}{\kappa}_{\Hit},\\
\nonumber
\scp{z}{\lambda}_{\Hio}
&=0.
\end{align}
The third equation of 
\eqref{doublesaddlepointfullsystemsosoneoneandmore} is valid for all $\vartheta\in D(\Az)$.
It holds $z=0$ and $h_{2}=0$ as well as $u=0$ and $h_{1}=0$.
Moreover, $\At\ti{x}\in D(\Ats)$ with $\Ats\At\ti{x}=f$ and $\ti{x}\in D(\Aos)$ with $\Aos\ti{x}=g$
as well as $\pit\ti{x}=k$.
Finally, $\ti{x}=x$ from Theorem \ref{soltheosos}.
\item[\bf(ii)]
There exists a unique five tuple $(\hat{x},y,v,h_{2},\hat{h}_{2})\in D(\Aos)\times D(\Ats\At)\times D(\cAo)\times K_{2}\times K_{2}$ 
such that for all $(\zeta,\phi,\psi,\kappa,\lambda)\in D(\Aos)\times D(\Ats\At)\times D(\cAo)\times K_{2}\times K_{2}$ 
\begin{align}
\nonumber
\scp{\Aos\hat{x}}{\Aos\zeta}_{\Hio}
+\scp{\Ats\At y}{\zeta}_{\Hit}
+\scp{h_{2}}{\zeta}_{\Hit}
&=\scp{g}{\Aos\zeta}_{\Hio},\\
\nonumber
\scp{\hat{x}}{\Ats\At\phi}_{\Hit}
+\scp{\Ao v}{\phi}_{\Hit}
+\scp{\hat{h}_{2}}{\phi}_{\Hit}
&=\scp{f}{\phi}_{\Hit},\\
\mylabel{doublesaddlepointfullsystemsostwooneandmore}
\scp{y}{\Ao\psi}_{\Hit}
&=0,\\
\nonumber
\scp{\hat{x}}{\kappa}_{\Hit}
&=\scp{k}{\kappa}_{\Hit},\\
\nonumber
\scp{y}{\lambda}_{\Hit}
&=0.
\end{align}
The third equation of 
\eqref{doublesaddlepointfullsystemsostwooneandmore} is valid for all $\psi\in D(\Ao)$.
It holds $y=0$ and $h_{2}=0$ as well as $v=0$ and $\hat{h}_{2}=0$.
Moreover, $\Aos\hat{x}=g$ and $\hat{x}\in D(\Ats\At)$ with $\Ats\At\hat{x}=f$
as well as $\pit\hat{x}=k$.
Finally, $\hat{x}=x$ from Theorem \ref{soltheosos}.
\item[\bf(ii')]
There is 
$(\hat{x},y,v,u,h_{2},\hat{h}_{2},h_{1})
\in D(\Aos)\times D(\Ats\At)\times D(\Ao)\times D(\cAz)\times K_{2}\times K_{2}\times K_{1}$,
a unique seven tuple, such that for all 
$(\zeta,\phi,\psi,\vartheta,\kappa,\lambda,\nu)
\in D(\Aos)\times D(\Ats\At)\times D(\Ao)\times D(\cAz)\times K_{2}\times K_{2}\times K_{1}$
\begin{align}
\nonumber
\scp{\Aos\hat{x}}{\Aos\zeta}_{\Hio}
+\scp{\Ats\At y}{\zeta}_{\Hit}
+\scp{h_{2}}{\zeta}_{\Hit}
&=\scp{g}{\Aos\zeta}_{\Hio},\\
\nonumber
\scp{\hat{x}}{\Ats\At\phi}_{\Hit}
+\scp{\Ao v}{\phi}_{\Hit}
+\scp{\hat{h}_{2}}{\phi}_{\Hit}
&=\scp{f}{\phi}_{\Hit},\\
\nonumber
\scp{y}{\Ao\psi}_{\Hit}
+\scp{\Az u}{\psi}_{\Hio}
+\scp{h_{1}}{\psi}_{\Hio}
&=0,\\
\mylabel{doublesaddlepointfullsystemsosthreeoneandmore}
\scp{v}{\Az\vartheta}_{\Hio}
&=0,\\
\nonumber
\scp{\hat{x}}{\kappa}_{\Hit}
&=\scp{k}{\kappa}_{\Hit},\\
\nonumber
\scp{y}{\lambda}_{\Hit}
&=0,\\
\nonumber
\scp{v}{\nu}_{\Hio}
&=0.
\end{align}
The fourth equation of 
\eqref{doublesaddlepointfullsystemsosthreeoneandmore} is valid for all $\vartheta\in D(\Az)$.
It holds $y=0$, $h_{2}=0$ and $v=0$, $\hat{h}_{2}=0$ as well as $u=0$ and $h_{1}=0$.
Moreover, $\Aos\hat{x}=g$ and $\hat{x}\in D(\Ats\At)$ with $\Ats\At\hat{x}=f$
as well as $\pit\hat{x}=k$.
Finally, $\hat{x}=x$ from Theorem \ref{soltheosos}.
\end{itemize}
\end{theo}

Theorem \ref{soltheososvarformpartsolfull} can be extended in the same way.

\begin{rem}
\mylabel{soltheososvarformtogetherwithK2andmorerem}
Let us note that generally the solution and test spaces look like
\begin{align*}
D(\Al)\times D(\Almo)&\times\cdots\times D(\A_{\ell-n+1})\times D(\cA_{\ell-n})
\times K_{\ell}\times K_{\ell-1}\times\cdots\times K_{\ell-n+1},\\
D(\Als)\times D(\Alpos\Alpo)&\times D(\Al)\times D(\Almo)\times\cdots\\
\cdots&\times D(\A_{\ell-n+1})\times D(\cA_{\ell-n})
\times K_{\ell+1}\times K_{\ell+1}\times K_{\ell}\times K_{\ell-1}\times\cdots\times K_{\ell-n+1}.
\end{align*}
Moreover:
\begin{itemize}
\item[\bf(i)]
\eqref{doublesaddlepointfullsystemsosoneoneandmore} is a weak formulation of 
$$\Ats\At\ti{x}+\Ao z+h_{2}=f,\quad
\Aos\ti{x}+\Az u+h_{1}=g,\quad
\Azs z=0,\quad
\pit\ti{x}=k,\quad
\pio z=0,$$
i.e., in formal matrix notation
$$\begin{bmatrix}
\Ats\At & \Ao & 0 & \iota_{K_{2}} & 0\\
\Aos & 0 & \Az & 0 & \iota_{K_{1}}\\
0 & \Azs & 0 & 0 & 0\\
\pit=\iota_{K_{2}}^{*} & 0 & 0 & 0 & 0\\
0 & \pio=\iota_{K_{1}}^{*} & 0 & 0 & 0
\end{bmatrix}
\begin{bmatrix}
\ti{x}\\
z\\
u\\
h_{2}\\
h_{1}
\end{bmatrix}
=\begin{bmatrix}
f\\
g\\
0\\
k\\
0
\end{bmatrix}.$$
Note $z=0$, $u=0$ and $h_{2}=0$, $h_{1}=0$.
\item[\bf(ii)]
\eqref{doublesaddlepointfullsystemsostwooneandmore} is a weak formulation of
$$\Ao\Aos\hat{x}+\Ats\At y+h_{2}=\Ao g,\quad
\Ats\At\hat{x}+\Ao v+\hat{h}_{2}=f,\quad
\Aos y=0,\quad
\pit\hat{x}=k,\quad
\pit y=0,$$
i.e., in formal matrix notation
$$\begin{bmatrix}
\Ao\Aos & \Ats\At & 0 & \iota_{K_{2}} & 0\\
\Ats\At & 0 & \Ao & 0 & \iota_{K_{2}}\\
0 & \Aos & 0 & 0 & 0\\
\pit=\iota_{K_{2}}^{*} & 0 & 0 & 0 & 0\\
0 & \pit=\iota_{K_{2}}^{*} & 0 & 0 & 0
\end{bmatrix}
\begin{bmatrix}
\hat{x}\\
y\\
v\\
h_{2}\\
\hat{h}_{2}
\end{bmatrix}
=\begin{bmatrix}
\Ao g\\
f\\
0\\
k\\
0
\end{bmatrix}.$$
Note $y=0$, $v=0$ and $h_{2}=0$, $\hat{h}_{2}=0$.
\item[\bf(ii')]
\eqref{doublesaddlepointfullsystemsosthreeoneandmore} is a weak formulation of
$$\Ao\Aos\hat{x}+\Ats\At y+h_{2}=\Ao g,\quad
\Ats\At\hat{x}+\Ao v+\hat{h}_{2}=f,\quad
\Aos y+\Az u+h_{1}=0,\quad
\Azs v=0,$$
and $\pit\hat{x}=k$, $\pit y=0$, $\pio v=0$,
i.e., in formal matrix notation
$$\begin{bmatrix}
\Ao\Aos & \Ats\At & 0 & 0 & \iota_{K_{2}} & 0 & 0\\
\Ats\At & 0 & \Ao & 0 & 0 & \iota_{K_{2}} & 0\\
0 & \Aos & 0 & \Az & 0 & 0 & \iota_{K_{1}}\\
0 & 0 & \Azs & 0 & 0 & 0 & 0\\
\pit=\iota_{K_{2}}^{*} & 0 & 0 & 0 & 0 & 0 & 0\\
0 & \pit=\iota_{K_{2}}^{*} & 0 & 0 & 0 & 0 & 0\\
0 & 0 & \pio=\iota_{K_{1}}^{*} & 0 & 0 & 0 & 0
\end{bmatrix}
\begin{bmatrix}
\hat{x}\\
y\\
v\\
u\\
h_{2}\\
\hat{h}_{2}\\
h_{1}
\end{bmatrix}
=\begin{bmatrix}
\Ao g\\
f\\
0\\
0\\
k\\
0\\
0
\end{bmatrix}.$$
Note $y=0$, $v=0$, $u=0$ and $h_{2}=0$, $\hat{h}_{2}=0$, $h_{1}=0$.
\end{itemize}
\end{rem}

\section{Functional A Posteriori Error Estimates}
\mylabel{secfuncapostest}

Having establishes a solution theory including suitable variational formulations,
we now turn to the so-called functional a posteriori error estimates.
Note that General Assumption \ref{genass} is supposed to hold.

\subsection{First Order Systems}

Let $x\in D_{2}$ be the exact solution of \eqref{Aprobsoltheo}
and $\ti{x}\in\Hit$, which may be considered as a non-conforming
approximation of $x$. Utilizing the notations from Theorem \ref{soltheofos} 
we define and decompose the error
\begin{align}
\label{edef}
\begin{split}
\Hit\ni e
&:=x-\ti{x}
=e_{\Ao}+e_{K_{2}}+e_{\Ats}
\in R(\Ao)\oplus_{\Hit}K_{2}\oplus_{\Hit}R(\Ats),\\
e_{\Ao}
&:=\pi_{\Ao}e
=x_{g}-\pi_{\Ao}\ti{x}
\in R(\Ao),\\
e_{\Ats}
&:=\pi_{\Ats}e
=x_{f}-\pi_{\Ats}\ti{x}
\in R(\Ats),\\
e_{K_{2}}
&:=\pit e
=k-\pit\ti{x}
\in K_{2}
\end{split}
\end{align}
using the Helmholtz type decompositions of Lemma \ref{lemtoolboxexseq}.
By orthogonality it holds
\begin{align}
\mylabel{eortho}
\norm{e}_{\Hit}^2
&=\norm{e_{\Ao}}_{\Hit}^2
+\norm{e_{K_{2}}}_{\Hit}^2
+\norm{e_{\Ats}}_{\Hit}^2.
\end{align}

\subsubsection{Upper Bounds}

Testing \eqref{edef} with $\Ao\varphi$ for $\varphi\in D(\cAo)$ we get 
for all $\zeta\in D(\Aos)$
by orthogonality and Corollary \ref{cortoolboxone} (i)
\begin{align}
\label{apostestone}
\begin{split}
\scp{e_{\Ao}}{\Ao\varphi}_{\Hit}
&=\scp{e}{\Ao\varphi}_{\Hit}
=\scp{\Aos x}{\varphi}_{\Hio}
-\scp{\ti{x}-\zeta+\zeta}{\Ao\varphi}_{\Hit}\\
&=\scp{g-\Aos\zeta}{\varphi}_{\Hio}
-\bscp{\pi_{\Ao}(\ti{x}-\zeta)}{\Ao\varphi}_{\Hit}\\
&\leq\norm{g-\Aos\zeta}_{\Hio}\norm{\varphi}_{\Hio}
+\bnorm{\pi_{\Ao}(\ti{x}-\zeta)}_{\Hit}\norm{\Ao\varphi}_{\Hit}\\
&\leq\Big(c_{1}\norm{g-\Aos\zeta}_{\Hio}
+\bnorm{\pi_{\Ao}(\ti{x}-\zeta)}_{\Hit}\Big)
\norm{\Ao\varphi}_{\Hit}.
\end{split}
\end{align}
As $e_{\Ao}\in R(\Ao)=R(\cAo)$, 
we have $e_{\Ao}=\Ao\varphi_{e}$ 
with $\varphi_{e}:=\cA_{1}^{-1}e_{\Ao}\in D(\cAo)$.
Choosing $\varphi:=\varphi_{e}$ in \eqref{apostestone} we obtain
\begin{align}
\label{apostesttwo}
\begin{split}
\forall\,\zeta\in D(\Aos)\qquad
\norm{e_{\Ao}}_{\Hit}
\leq c_{1}\norm{g-\Aos\zeta}_{\Hio}
+\bnorm{\pi_{\Ao}(\ti{x}-\zeta)}_{\Hit}
\leq c_{1}\norm{g-\Aos\zeta}_{\Hio}
+\norm{\ti{x}-\zeta}_{\Hit}.
\end{split}
\end{align}
Analogously, testing with $\Ats\phi$ for $\phi\in D(\cAts)$ we get 
for all $\xi\in D(\At)$
by orthogonality and Corollary \ref{cortoolboxone} (i)
\begin{align}
\label{apostestthree}
\begin{split}
\scp{e_{\Ats}}{\Ats\phi}_{\Hit}
&=\scp{e}{\Ats\phi}_{\Hit}
=\scp{\At x}{\phi}_{\Hith}
-\scp{\ti{x}-\xi+\xi}{\Ats\phi}_{\Hit}\\
&=\scp{f-\At\xi}{\phi}_{\Hith}
-\bscp{\pi_{\Ats}(\ti{x}-\xi)}{\Ats\phi}_{\Hit}\\
&\leq\norm{f-\At\xi}_{\Hith}\norm{\phi}_{\Hith}
+\bnorm{\pi_{\Ats}(\ti{x}-\xi)}_{\Hit}\norm{\Ats\phi}_{\Hit}\\
&\leq\Big(c_{2}\norm{f-\At\xi}_{\Hith}
+\bnorm{\pi_{\Ats}(\ti{x}-\xi)}_{\Hit}\Big)
\norm{\Ats\phi}_{\Hit}.
\end{split}
\end{align}
As $e_{\Ats}\in R(\Ats)=R(\cAts)$, 
we have $e_{\Ats}=\Ats\phi_{e}$ 
with $\phi_{e}:=(\cAts)^{-1}e_{\Ats}\in D(\cAts)$.
Choosing $\phi:=\phi_{e}$ in \eqref{apostestthree} we obtain
\begin{align}
\label{apostestfour}
\begin{split}
\forall\,\xi\in D(\At)\qquad
\norm{e_{\Ats}}_{\Hit}
\leq c_{2}\norm{f-\At\xi}_{\Hith}
+\bnorm{\pi_{\Ats}(\ti{x}-\xi)}_{\Hit}
\leq c_{2}\norm{f-\At\xi}_{\Hith}
+\norm{\ti{x}-\xi}_{\Hit}.
\end{split}
\end{align}
Finally, for all $\varphi\in D(\Ao)$
and all $\phi\in D(\Ats)$ we get by orthogonality
\begin{align}
\mylabel{apostestfive}
\norm{e_{K_{2}}}_{\Hit}^2
&=\scp{e_{K_{2}}}{k-\pit\ti{x}+\Ao\varphi+\Ats\phi}_{\Hit}
=\scp{e_{K_{2}}}{k-\ti{x}+\Ao\varphi+\Ats\phi}_{\Hit}
\end{align}
and thus
\begin{align}
\mylabel{apostestsix}
\forall\,\varphi\in D(\Ao)\quad
\forall\,\phi\in D(\Ats)\qquad
\norm{e_{K_{2}}}_{\Hit}
&\leq\norm{k-\ti{x}+\Ao\varphi+\Ats\phi}_{\Hit}.
\end{align}

Let us summarize:

\begin{theo}
\label{apostestfos}
Let $x\in D_{2}$ be the exact solution of \eqref{Aprobsoltheo}
and $\ti{x}\in\Hit$. Then the following estimates hold for the error 
$e=x-\ti{x}$ defined in \eqref{edef}:
\begin{itemize}
\item[\bf(i)]
The error decomposes according to \eqref{edef}-\eqref{eortho}, i.e.,
$$e=e_{\Ao}+e_{K_{2}}+e_{\Ats}
\in R(\Ao)\oplus_{\Hit}K_{2}\oplus_{\Hit}R(\Ats),\qquad
\norm{e}_{\Hit}^2
=\norm{e_{\Ao}}_{\Hit}^2
+\norm{e_{K_{2}}}_{\Hit}^2
+\norm{e_{\Ats}}_{\Hit}^2.$$
\item[\bf(ii)]
The projection $e_{\Ao}=\pi_{\Ao}e=x_{g}-\pi_{\Ao}\ti{x}\in R(\Ao)$ satisfies
$$\norm{e_{\Ao}}_{\Hit}
=\min_{\zeta\in D(\Aos)}
\big(c_{1}\norm{\Aos\zeta-g}_{\Hio}
+\norm{\zeta-\ti{x}}_{\Hit}\big)$$
and the minimum is attained at 
$$\hat{\zeta}
:=e_{\Ao}+\ti{x}
=\pi_{\Ao}e+\ti{x}
=-(1-\pi_{\Ao})e+x
=-\pi_{N(\Aos)}e+x
\in D(\Aos)$$
since $\Aos\hat{\zeta}=\Aos x=g$.
\item[\bf(iii)]
The projection $e_{\Ats}=\pi_{\Ats}e=x_{f}-\pi_{\Ats}\ti{x}\in R(\Ats)$ satisfies
$$\norm{e_{\Ats}}_{\Hit}
=\min_{\xi\in D(\At)}
\big(c_{2}\norm{\At\xi-f}_{\Hith}
+\norm{\xi-\ti{x}}_{\Hit}\big)$$
and the minimum is attained at 
$$\hat{\xi}
:=e_{\Ats}+\ti{x}
=\pi_{\Ats}e+\ti{x}
=-(1-\pi_{\Ats})e+x
=-\pi_{N(\At)}e+x
\in D(\At)$$
since $\At\hat{\xi}=\At x=f$.
\item[\bf(iv)]
The projection $e_{K_{2}}=\pit e=k-\pit\ti{x}\in K_{2}$ satisfies
$$\norm{e_{K_{2}}}_{\Hit}
=\min_{\varphi\in D(\Ao)}
\min_{\phi\in D(\Ats)}
\norm{k-\ti{x}+\Ao\varphi+\Ats\phi}_{\Hit}$$
and the minimum is attained at any
$\hat{\varphi}\in D(\Ao)$ and $\hat{\phi}\in D(\Ats)$
solving $\Ao\hat{\varphi}=\pi_{\Ao}\ti{x}$
and $\Ats\hat{\phi}=\pi_{\Ats}\ti{x}$ since 
$(\pi_{\Ao}+\pi_{\Ats})\ti{x}
=(1-\pit)\ti{x}$, especially at
$$\hat{\varphi}
:=\cA_{1}^{-1}\pi_{\Ao}\ti{x}
\in D(\cAo),\qquad
\hat{\phi}
:=(\cAts)^{-1}\pi_{\Ats}\ti{x}
\in D(\cAts).$$
\end{itemize}
\end{theo}

For conforming approximations we get:

\begin{cor}
\label{apostestfosconfcor}
Let the assumptions of Theorem \ref{apostestfos} be satisfied.
\begin{itemize}
\item[\bf(i)]
If $\ti{x}\in D(\Aos)$, then $e\in D(\Aos)$ 
and hence $e_{\Ao}=\pi_{\Ao}e\in D(\cAos)$ with
$\Aos e_{\Ao}=\Aos e$ and
$$\norm{e_{\Ao}}_{\Hit}
\leq c_{1}\norm{\Aos\ti{x}-g}_{\Hio}
=c_{1}\norm{\Aos e}_{\Hio}$$
by setting $\zeta:=\ti{x}$, which also follows directly by the Friedrichs/Poincar\'e type estimate.
\item[\bf(ii)]
If $\ti{x}\in D(\At)$, then $e\in D(\At)$ 
and hence $e_{\Ats}=\pi_{\Ats}e\in D(\cAt)$ with
$\At e_{\Ats}=\At e$ and
$$\norm{e_{\Ats}}_{\Hit}
\leq c_{2}\norm{\At\ti{x}-f}_{\Hith}
=c_{2}\norm{\At e}_{\Hith}$$
by setting $\xi:=\ti{x}$, which also follows directly by the Friedrichs/Poincar\'e type estimate.
\item[\bf(iii)]
If $\ti{x}\in D_{2}$, then $e\in D_{2}$ and 
\begin{align*}
\norm{e}_{D_{2}}^2
&=\norm{e_{\Ao}}_{\Hit}^2
+\norm{e_{K_{2}}}_{\Hit}^2
+\norm{e_{\Ats}}_{\Hit}^2
+\norm{\At e}_{\Hith}^2
+\norm{\Aos e}_{\Hio}^2\\
&\leq\norm{e_{K_{2}}}_{\Hit}^2
+(1+c_{2}^2)\norm{\At e}_{\Hith}^2
+(1+c_{1}^2)\norm{\Aos e}_{\Hio}^2
\end{align*}
with
$$e_{K_{2}}=k-\pit\ti{x},\qquad
\At e=f-\At\ti{x},\qquad
\Aos e=g-\Aos\ti{x},$$
which again also follows immediately by the Friedrichs/Poincar\'e type estimates.
\end{itemize}
\end{cor}

\begin{rem}
\label{apostestfosconfcorrem}
Corollary \ref{apostestfosconfcor} (iii) shows,
that for very conforming $\ti{x}\in D_{2}$ 
the weighted least squares functional
$$\calF(\ti{x})
:=\norm{k-\pit\ti{x}}_{\Hit}^2
+(1+c_{2}^2)\norm{\At\ti{x}-f}_{\Hith}^2
+(1+c_{1}^2)\norm{\Aos\ti{x}-g}_{\Hio}^2$$
is equivalent to the conforming error, i.e.,
$$\norm{e}_{D_{2}}^2
\leq\calF(\ti{x})
\leq(1+\max\{c_{1},c_{2}\}^2)\norm{e}_{D_{2}}^2.$$
\end{rem}

Recalling the variational resp. saddle point formulations
\eqref{varyftwo}-\eqref{varzgtwo} resp.
\eqref{saddleyfone}-\eqref{saddlezgone} 
and that the partial solutions are given by
$$x_{f}=\Ats y_{f}\in D(\cAt),\qquad
x_{g}=\Ao z_{g}\in D(\cAos),$$
a possible numerical method, 
using these variational formulations
in some finite dimensional subspaces
to find $\ti{y}_{f}\in D(\Ats)$ and $\ti{z}_{g}\in D(\Ao)$, 
such as the finite element method,
will always ensure
$$\ti{x}_{f}:=\Ats\ti{y}_{f}\in R(\Ats)=N(\At)^{\bot_{\Hit}}\subset N(\Aos),\quad
\ti{x}_{g}:=\Ao\ti{z}_{g}\in R(\Ao)=N(\Aos)^{\bot_{\Hit}}\subset N(\At)$$
and thus 
$$\ti{x}_{\bot}:=\ti{x}_{f}+\ti{x}_{g}\in R(\Ats)\oplus_{\Hit}R(\Ao)=K_{2}^{\bot_{\Hit}},$$
but maybe not $\ti{x}_{f}\in D(\At)$ or $\ti{x}_{g}\in D(\Aos)$.
Therefore, a reasonable assumption for our non-conforming approximations is
$$\ti{x}=\ti{x}_{\bot}+k,\qquad
\ti{x}_{\bot}\in K_{2}^{\bot_{\Hit}},$$
with $e_{K_{2}}=\pit e=\pit(x-\ti{x})=-\pit\ti{x}_{\bot}=0$.

\begin{cor}
\label{apostestfosKbot}
Let $x\in D_{2}$ be the exact solution of \eqref{Aprobsoltheo}
and $\ti{x}:=k+\ti{x}_{\bot}$ with some 
$\ti{x}_{\bot}\in K_{2}^{\bot_{\Hit}}$. 
Then for the error $e$ defined in \eqref{edef} it holds:
\begin{itemize}
\item[\bf(i)]
According to \eqref{edef}-\eqref{eortho} the error decomposes, i.e.,
$$e
=x-\ti{x}
=x_{f}+x_{g}-\ti{x}_{\bot}
=e_{\Ao}+e_{\Ats}
\in R(\Ao)\oplus_{\Hit}R(\Ats)
=K_{2}^{\bot_{\Hit}},\qquad
e_{K_{2}}=0,$$
and
$\norm{e}_{\Hit}^2
=\norm{e_{\Ao}}_{\Hit}^2
+\norm{e_{\Ats}}_{\Hit}^2$.
Hence there is no error in the ``kernel'' part.
\item[\bf(ii)]
The projection $e_{\Ao}=\pi_{\Ao}e=x_{g}-\pi_{\Ao}\ti{x}=x_{g}-\pi_{\Ao}\ti{x}_{\bot}\in R(\Ao)$ satisfies
\begin{align*}
\norm{e_{\Ao}}_{\Hit}
&=\min_{\zeta\in D(\Aos)}
\big(c_{1}\norm{\Aos\zeta-g}_{\Hio}
+\norm{\zeta-\ti{x}}_{\Hit}\big)\\
&=\min_{\zeta\in D(\Aos)}
\big(c_{1}\norm{\Aos\zeta-g}_{\Hio}
+\norm{\zeta-\ti{x}_{\bot}}_{\Hit}\big)
\end{align*}
(exchanging $\zeta$ by $\zeta+k$) and the minima are attained at 
\begin{align*}
\hat{\zeta}
&:=e_{\Ao}+\ti{x}
=\pi_{\Ao}e+\ti{x}
=-(1-\pi_{\Ao})e+x
=-\pi_{N(\Aos)}e+x
\in D(\Aos),\\
\hat{\zeta}_{\bot}
&:=e_{\Ao}+\ti{x}_{\bot}
=\pi_{\Ao}e+\ti{x}_{\bot}
=-(1-\pi_{\Ao})e+x-k
=-\pi_{N(\Aos)}e+x-k
\in D(\Aos)
\end{align*}
since $\Aos\hat{\zeta}_{\bot}=\Aos\hat{\zeta}=\Aos x=g$.
\item[\bf(iii)]
The projection $e_{\Ats}=\pi_{\Ats}e=x_{f}-\pi_{\Ats}\ti{x}=x_{f}-\pi_{\Ats}\ti{x}_{\bot}\in R(\Ats)$ satisfies
\begin{align*}
\norm{e_{\Ats}}_{\Hit}
&=\min_{\xi\in D(\At)}
\big(c_{2}\norm{\At\xi-f}_{\Hith}
+\norm{\xi-\ti{x}}_{\Hit}\big)\\
&=\min_{\xi\in D(\At)}
\big(c_{2}\norm{\At\xi-f}_{\Hith}
+\norm{\xi-\ti{x}_{\bot}}_{\Hit}\big)
\end{align*}
(exchanging $\xi$ by $\xi+k$) and the minima are attained at 
\begin{align*}
\hat{\xi}
&:=e_{\Ats}+\ti{x}
=\pi_{\Ats}e+\ti{x}
=-(1-\pi_{\Ats})e+x
=-\pi_{N(\At)}e+x
\in D(\At),\\
\hat{\xi}_{\bot}
&:=e_{\Ats}+\ti{x}_{\bot}
=\pi_{\Ats}e+\ti{x}_{\bot}
=-(1-\pi_{\Ats})e+x-k
=-\pi_{N(\At)}e+x-k
\in D(\At)
\end{align*}
since $\At\hat{\xi}_{\bot}=\At\hat{\xi}=\At x=f$.
\end{itemize}
\end{cor}

\subsubsection{Lower Bounds}

In any Hilbert space $\Hilbert$ we have
\begin{align}
\mylabel{hilbertest}
\forall\,\hat{h}\in\Hilbert\qquad
\norm{\hat{h}}_{\Hilbert}^2
=\max_{h\in\Hilbert}
\big(2\scp{\hat{h}}{h}_{\Hilbert}
-\norm{h}_{\Hilbert}^2\big)
\end{align}
and the maximum is attained at $\hat{h}$.
We recall \eqref{edef} and \eqref{eortho}, especially 
$$\norm{e}_{\Hit}^2
=\norm{e_{\Ao}}_{\Hit}^2
+\norm{e_{K_{2}}}_{\Hit}^2
+\norm{e_{\Ats}}_{\Hit}^2.$$
Using \eqref{hilbertest} for $\Hilbert=R(\Ao)$ and orthogonality we get
\begin{align*}
\norm{e_{\Ao}}_{\Hit}^2
&=\max_{\varphi\in D(\Ao)}
\big(2\scp{e_{\Ao}}{\Ao\varphi}_{\Hit}
-\norm{\Ao\varphi}_{\Hit}^2\big)\\
&=\max_{\varphi\in D(\Ao)}
\big(2\scp{e}{\Ao\varphi}_{\Hit}
-\norm{\Ao\varphi}_{\Hit}^2\big)\\
&=\max_{\varphi\in D(\Ao)}
\big(2\scp{g}{\varphi}_{\Hio}
-\scp{2\ti{x}+\Ao\varphi}{\Ao\varphi}_{\Hit}\big)
\intertext{and the maximum is attained at any $\hat{\varphi}\in D(\Ao)$ with $\Ao\hat{\varphi}=e_{\Ao}$. 
Analogously for $\Hilbert=R(\Ats)$}
\norm{e_{\Ats}}_{\Hit}^2
&=\max_{\phi\in D(\Ats)}
\big(2\scp{f}{\phi}_{\Hith}
-\scp{2\ti{x}+\Ats\phi}{\Ats\phi}_{\Hit}\big)
\intertext{and the maximum is attained at any $\hat{\phi}\in D(\Ats)$ with $\Ats\hat{\phi}=e_{\Ats}$. 
Finally for $\Hilbert=K_{2}$ and by orthogonality}
\norm{e_{K_{2}}}_{\Hit}^2
&=\max_{\theta\in K_{2}}
\big(2\scp{e_{K_{2}}}{\theta}_{\Hit}
-\norm{\theta}_{\Hit}^2\big)
=\max_{\theta\in K_{2}}
\bscp{2(k-\ti{x})-\theta}{\theta}_{\Hit}
\end{align*}
and the maximum is attained at $\hat{\theta}=e_{K_{2}}$. 

\begin{theo}
\label{apostestfoslowbd}
Let $x\in D_{2}$ be the exact solution of \eqref{Aprobsoltheo}
and $\ti{x}\in\Hit$. Then the following estimates hold for the error 
$e=x-\ti{x}$ defined in \eqref{edef}:
\begin{itemize}
\item[\bf(i)]
The error decomposes according to \eqref{edef}-\eqref{eortho}, i.e.,
$$e=e_{\Ao}+e_{K_{2}}+e_{\Ats}
\in R(\Ao)\oplus_{\Hit}K_{2}\oplus_{\Hit}R(\Ats),\qquad
\norm{e}_{\Hit}^2
=\norm{e_{\Ao}}_{\Hit}^2
+\norm{e_{K_{2}}}_{\Hit}^2
+\norm{e_{\Ats}}_{\Hit}^2.$$
\item[\bf(ii)]
The projection $e_{\Ao}=\pi_{\Ao}e=x_{g}-\pi_{\Ao}\ti{x}\in R(\Ao)$ satisfies
$$\norm{e_{\Ao}}_{\Hit}^2
=\max_{\varphi\in D(\Ao)}
\big(2\scp{g}{\varphi}_{\Hio}
-\scp{2\ti{x}+\Ao\varphi}{\Ao\varphi}_{\Hit}\big)$$
and the maximum is attained at any $\hat{\varphi}\in D(\Ao)$ with $\Ao\hat{\varphi}=e_{\Ao}$, 
e.g., at $\hat{\varphi}:=\cA_{1}^{-1}e_{\Ao}\in D(\cAo)$.
\item[\bf(iii)]
The projection $e_{\Ats}=\pi_{\Ats}e=x_{f}-\pi_{\Ats}\ti{x}\in R(\Ats)$ satisfies
$$\norm{e_{\Ats}}_{\Hit}^2
=\max_{\phi\in D(\Ats)}
\big(2\scp{f}{\phi}_{\Hith}
-\scp{2\ti{x}+\Ats\phi}{\Ats\phi}_{\Hit}\big)$$
and the maximum is attained at any $\hat{\phi}\in D(\Ats)$ with $\Ats\hat{\phi}=e_{\Ats}$, 
e.g.,  $\hat{\phi}:=(\cAts)^{-1}e_{\Ats}\in D(\cAts)$.
\item[\bf(iv)]
The projection $e_{K_{2}}=\pit e=k-\pit\ti{x}\in K_{2}$ satisfies
$$\norm{e_{K_{2}}}_{\Hit}^2
=\max_{\theta\in K_{2}}
\bscp{2(k-\ti{x})-\theta}{\theta}_{\Hit}$$
and the maximum is attained at 
$\hat{\theta}
:=e_{K_{2}}
\in K_{2}$.
\end{itemize}
If $\ti{x}:=k+\ti{x}_{\bot}$ with some $\ti{x}_{\bot}\in K_{2}^{\bot_{\Hit}}$,
see Corollary \ref{apostestfosKbot}, then $e_{K_{2}}=0$, and 
in (ii) and (iii) $\ti{x}$ can be replaced by $\ti{x}_{\bot}$ as 
$k\,\bot_{\Hit}\,R(\Ao)\oplus_{\Hit}R(\Ats)$.
\end{theo}

\subsubsection{Two-Sided Bounds}

We summarize our results from the latter sections.

\begin{cor}
\label{apostestfoscortsb}
Let $x\in D_{2}$ be the exact solution of \eqref{Aprobsoltheo}
and $\ti{x}\in\Hit$. Then the following estimates hold for the error 
$e=x-\ti{x}$ defined in \eqref{edef}:
\begin{itemize}
\item[\bf(i)]
The error decomposes according to \eqref{edef}-\eqref{eortho}, i.e.,
$$e=e_{\Ao}+e_{K_{2}}+e_{\Ats}
\in R(\Ao)\oplus_{\Hit}K_{2}\oplus_{\Hit}R(\Ats),\qquad
\norm{e}_{\Hit}^2
=\norm{e_{\Ao}}_{\Hit}^2
+\norm{e_{K_{2}}}_{\Hit}^2
+\norm{e_{\Ats}}_{\Hit}^2.$$
\item[\bf(ii)]
The projection $e_{\Ao}=\pi_{\Ao}e=x_{g}-\pi_{\Ao}\ti{x}\in R(\Ao)$ satisfies
\begin{align*}
\norm{e_{\Ao}}_{\Hit}^2
&=\min_{\zeta\in D(\Aos)}
\big(c_{1}\norm{\Aos\zeta-g}_{\Hio}
+\norm{\zeta-\ti{x}}_{\Hit}\big)^2\\
&=\max_{\varphi\in D(\Ao)}
\big(2\scp{g}{\varphi}_{\Hio}
-\scp{2\ti{x}+\Ao\varphi}{\Ao\varphi}_{\Hit}\big)
\end{align*}
and the minimum resp. maximum is attained at 
$$\hat{\zeta}
:=e_{\Ao}+\ti{x}
\in D(\Aos),\qquad
\hat{\varphi}
:=\cA_{1}^{-1}e_{\Ao}
\in D(\cAo)$$
with $\Aos\hat{\zeta}=\Aos x=g$,
and at any $\hat{\varphi}\in D(\Ao)$ with $\Ao\hat{\varphi}=e_{\Ao}$.
\item[\bf(iii)]
The projection $e_{\Ats}=\pi_{\Ats}e=x_{f}-\pi_{\Ats}\ti{x}\in R(\Ats)$ satisfies
\begin{align*}
\norm{e_{\Ats}}_{\Hit}^2
&=\min_{\xi\in D(\At)}
\big(c_{2}\norm{\At\xi-f}_{\Hith}
+\norm{\xi-\ti{x}}_{\Hit}\big)^2\\
&=\max_{\phi\in D(\Ats)}
\big(2\scp{f}{\phi}_{\Hith}
-\scp{2\ti{x}+\Ats\phi}{\Ats\phi}_{\Hit}\big)
\end{align*}
and the minimum resp. maximum is attained at 
$$\hat{\xi}
:=e_{\Ats}+\ti{x}
\in D(\At),\qquad
\hat{\phi}
:=(\cAts)^{-1}e_{\Ats}
\in D(\cAts)$$
with $\At\hat{\xi}=\At x=f$,
and at any $\hat{\phi}\in D(\Ats)$ with $\Ats\hat{\phi}=e_{\Ats}$.
\item[\bf(iv)]
The projection $e_{K_{2}}=\pit e=k-\pit\ti{x}\in K_{2}$ satisfies
\begin{align*}
\norm{e_{K_{2}}}_{\Hit}^2
&=\min_{\varphi\in D(\Ao)}
\min_{\phi\in D(\Ats)}
\norm{k-\ti{x}+\Ao\varphi+\Ats\phi}_{\Hit}^2\\
&=\max_{\theta\in K_{2}}
\bscp{2(k-\ti{x})-\theta}{\theta}_{\Hit}
\end{align*}
and the minimum resp. maximum is attained at 
$$\hat{\varphi}
:=\cA_{1}^{-1}\pi_{\Ao}\ti{x}
\in D(\cAo),\qquad
\hat{\phi}
:=(\cAts)^{-1}\pi_{\Ats}\ti{x}
\in D(\cAts),\qquad
\hat{\theta}
:=e_{K_{2}}
\in K_{2},$$
and at any
$\hat{\varphi}\in D(\Ao)$ and $\hat{\phi}\in D(\Ats)$
with $\Ao\hat{\varphi}=\pi_{\Ao}\ti{x}$
and $\Ats\hat{\phi}=\pi_{\Ats}\ti{x}$.
\end{itemize}
If $\ti{x}:=k+\ti{x}_{\bot}$ with some $\ti{x}_{\bot}\in K_{2}^{\bot_{\Hit}}$,
see Corollary \ref{apostestfosKbot}, then $e_{K_{2}}=0$, and 
in (ii) and (iii) $\ti{x}$ can be replaced by $\ti{x}_{\bot}$.
In this case, for the attaining minima it holds
$$\hat{\zeta}_{\bot}:=e_{\Ao}+\ti{x}_{\bot}\in D(\Aos),\qquad
\hat{\xi}_{\bot}:=e_{\Ats}+\ti{x}_{\bot}\in D(\At).$$
\end{cor}

\subsection{Second Order Systems}

Let $x\in\ti{D}_{2}$ be the exact solution of \eqref{AsAprobsoltheo}.
Recalling Remark \ref{soltheosecorderrem} we introduce the additional quantity $y:=\At x\in D(\cAts)$.
Then \eqref{AsAprobsoltheo} decomposes
into two first order systems of shape \eqref{Aprob} resp. \eqref{Aprobsoltheo}, i.e.,
\begin{align*}
\At x&=y,
&
\Ath y&=0,\\
\Aos x&=g,
&
\Ats y&=f,\\
\pit x&=k,
&
\pith y&=0
\end{align*}
for the pair $(x,y)\in D_{2}\times D_{3}$.
Hence, we can immediately apply our results for the first order systems.
Let $\ti{x}\in\Hit$ and $\ti{y}\in\Hith$, which may be considered as non-conforming
approximations of $x$ and $y$, respectively. 
Utilizing the notations from Theorem \ref{soltheosos} 
we define and decompose the errors
\begin{align}
\label{ehdef}
\begin{split}
\Hit\ni e
&:=x-\ti{x}
=e_{\Ao}+e_{K_{2}}+e_{\Ats}
\in R(\Ao)\oplus_{\Hit}K_{2}\oplus_{\Hit}R(\Ats),\\
\Hith\ni h
&:=y-\ti{y}
=h_{\At}+h_{K_{3}}+h_{\Aths}
\in R(\At)\oplus_{\Hith}K_{3}\oplus_{\Hith}R(\Aths),
\end{split}
\end{align}
\begin{align*}
e_{\Ao}
&:=\pi_{\Ao}e
=x_{g}-\pi_{\Ao}\ti{x}
\in R(\Ao),
&
h_{\At}
&:=\pi_{\At}h
=y-\pi_{\At}\ti{y}
\in R(\At),\\
e_{\Ats}
&:=\pi_{\Ats}e
=x_{y}-\pi_{\Ats}\ti{x}
\in R(\Ats),
&
h_{\Aths}
&:=\pi_{\Aths}h
=-\pi_{\Aths}\ti{y}
\in R(\Aths),\\
e_{K_{2}}
&:=\pit e
=k-\pit\ti{x}
\in K_{2},
&
h_{K_{3}}
&:=\pith e
=-\pith\ti{y}
\in K_{3}
\end{align*}
using the Helmholtz type decompositions of Lemma \ref{lemtoolboxexseq}
and noting $\pi_{\At}y=y$ as $y\in R(\At)$.
By orthogonality it holds
\begin{align}
\mylabel{ehortho}
\norm{e}_{\Hit}^2
&=\norm{e_{\Ao}}_{\Hit}^2
+\norm{e_{K_{2}}}_{\Hit}^2
+\norm{e_{\Ats}}_{\Hit}^2,
&
\norm{h}_{\Hith}^2
&=\norm{h_{\At}}_{\Hith}^2
+\norm{h_{K_{3}}}_{\Hith}^2
+\norm{h_{\Aths}}_{\Hith}^2.
\end{align}
Therefore, the results of the latter section can be applied to
$e_{\Ao}$, $e_{K_{2}}$, $e_{\Ats}$, $h_{\At}$, $h_{K_{3}}$, $h_{\Aths}$. 
Especially, by Corollary \ref{apostestfoscortsb} we obtain
\begin{align}
\label{secordeAoest}
\norm{e_{\Ao}}_{\Hit}^2
=\min_{\zeta\in D(\Aos)}
\big(c_{1}\norm{\Aos\zeta-g}_{\Hio}
+\norm{\zeta-\ti{x}}_{\Hit}\big)^2
=\max_{\varphi\in D(\Ao)}
\big(2\scp{g}{\varphi}_{\Hio}
-\scp{2\ti{x}+\Ao\varphi}{\Ao\varphi}_{\Hit}\big)
\end{align}
and the minimum resp. maximum is attained at 
$\hat{\zeta}
=e_{\Ao}+\ti{x}
\in D(\Aos)$ and
$\hat{\varphi}
=\cA_{1}^{-1}e_{\Ao}
\in D(\cAo)$
with $\Aos\hat{\zeta}=\Aos x=g$,
\begin{align}
\label{secordeAtsest}
\norm{e_{\Ats}}_{\Hit}^2
=\min_{\xi\in D(\At)}
\big(c_{2}\norm{\At\xi-y}_{\Hith}
+\norm{\xi-\ti{x}}_{\Hit}\big)^2
=\max_{\phi\in D(\Ats)}
\big(2\scp{y}{\phi}_{\Hith}
-\scp{2\ti{x}+\Ats\phi}{\Ats\phi}_{\Hit}\big)
\end{align}
and the minimum resp. maximum is attained at 
$\hat{\xi}
=e_{\Ats}+\ti{x}
\in D(\At)$ and
$\hat{\phi}
=(\cAts)^{-1}e_{\Ats}
\in D(\cAts)$
with $\At\hat{\xi}=\At x=y$,
\begin{align}
\label{secordeKest}
\norm{e_{K_{2}}}_{\Hit}^2
=\min_{\varphi\in D(\Ao)}
\min_{\phi\in D(\Ats)}
\norm{k-\ti{x}+\Ao\varphi+\Ats\phi}_{\Hit}^2
=\max_{\theta\in K_{2}}
\bscp{2(k-\ti{x})-\theta}{\theta}_{\Hit}
\end{align}
and the minimum resp. maximum is attained at 
$\hat{\varphi}
=\cA_{1}^{-1}\pi_{\Ao}\ti{x}
\in D(\cAo)$,
$\hat{\phi}
=(\cAts)^{-1}\pi_{\Ats}\ti{x}
\in D(\cAts)$, and
$\hat{\theta}
=e_{K_{2}}
\in K_{2}$
with $\Ao\hat{\varphi}+\Ats\hat{\phi}=(\pi_{\Ao}+\pi_{\Ats})\ti{x}=(1-\pit)\ti{x}$.
If $\ti{x}=k+\ti{x}_{\bot}$ with some $\ti{x}_{\bot}\in K_{2}^{\bot_{\Hit}}$,
then $e_{K_{2}}=0$, and $\ti{x}$ can be replaced by $\ti{x}_{\bot}$.
If the General Assumption \ref{genass} holds also for $\Ath$, i.e., 
$R(\Ath)$ is closed and (not neccessarily) $K_{3}$ is finite dimensional,
we get the corresponding results for $h_{\At}$, $h_{K_{3}}$, $h_{\Aths}$ as well. 
Replacing $\Ao$ by $\At$ and $\At$ by $\Ath$, Corollary \ref{apostestfoscortsb} yields
\begin{align}
\label{secordhAtest}
\norm{h_{\At}}_{\Hith}^2
=\min_{\zeta\in D(\Ats)}
\big(c_{2}\norm{\Ats\zeta-f}_{\Hit}
+\norm{\zeta-\ti{y}}_{\Hith}\big)^2
=\max_{\varphi\in D(\At)}
\big(2\scp{f}{\varphi}_{\Hit}
-\scp{2\ti{y}+\At\varphi}{\At\varphi}_{\Hith}\big)
\end{align}
and the minimum resp. maximum is attained at 
$\hat{\zeta}
=h_{\At}+\ti{y}
\in D(\Ats)$ and
$\hat{\varphi}
=\cA_{2}^{-1}h_{\At}
\in D(\cAt)$
with $\Ats\hat{\zeta}=\Ats y=f$,
\begin{align}
\label{secordhAthsest}
\norm{h_{\Aths}}_{\Hith}^2
=\min_{\xi\in D(\Ath)}
\big(c_{3}\norm{\Ath\xi}_{\Hif}
+\norm{\xi-\ti{y}}_{\Hith}\big)^2
=\max_{\phi\in D(\Aths)}
\big(-\scp{2\ti{y}+\Aths\phi}{\Aths\phi}_{\Hith}\big)
\end{align}
and the minimum resp. maximum is attained at 
$\hat{\xi}
=h_{\Aths}+\ti{y}
\in D(\Ath)$ and
$\hat{\phi}
=(\cAths)^{-1}h_{\Aths}
\in D(\cAths)$
with $\Ath\hat{\xi}=\Ath y=0$, i.e., $\hat{\xi}\in N(\Ath)$,
\begin{align}
\label{secordhKest}
\norm{h_{K_{3}}}_{\Hith}^2
=\min_{\varphi\in D(\At)}
\min_{\phi\in D(\Aths)}
\norm{-\ti{y}+\At\varphi+\Aths\phi}_{\Hith}^2
=\max_{\theta\in K_{3}}
\big(-\scp{2\ti{y}+\theta}{\theta}_{\Hith}\big)
\end{align}
and the minimum resp. maximum is attained at 
$\hat{\varphi}
=\cA_{2}^{-1}\pi_{\At}\ti{y}
\in D(\cAt)$,
$\hat{\phi}
=(\cAths)^{-1}\pi_{\Aths}\ti{y}
\in D(\cAths)$, and
$\hat{\theta}
=h_{K_{3}}
\in K_{3}$
with $\At\hat{\varphi}+\Aths\hat{\phi}=(\pi_{\At}+\pi_{\Aths})\ti{y}=(1-\pith)\ti{y}$.
If $\ti{y}=\ti{y}_{\bot}\in K_{3}^{\bot_{\Hith}}$,
then $h_{K_{3}}=0$, and $\ti{y}$ can be replaced by $\ti{y}_{\bot}$.
The upper bound for $\norm{h_{\Aths}}_{\Hith}$ in \eqref{secordhAthsest} equals
$$\norm{h_{\Aths}}_{\Hith}
=\min_{\xi\in N(\Ath)}
\norm{\xi-\ti{y}}_{\Hith}
=\norm{\hat{\xi}-\ti{y}}_{\Hith},\qquad
\hat{\xi}=h_{\Aths}+\ti{y}\in N(\Ath),$$
and so the constant $c_{3}$ does not play a role.
In \eqref{secordeAtsest} the unknown exact solution $y$ still appears 
in the upper and in the lower bound.
The term $\At\xi-y\in R(\At)$ of the upper bound in \eqref{secordeAtsest}
can be handled as an error $h_{\xi}=y-\ti{y}_{\xi}$ with $\ti{y}_{\xi}=\At\xi$.
As $h_{\xi}=\pi_{\At}h_{\xi}=h_{\xi,\At}$ we get by \eqref{secordhAtest}
\begin{align*}
\norm{\At\xi-y}_{\Hith}
=\norm{h_{\xi}}_{\Hith}
=\min_{\zeta\in D(\Ats)}
\big(c_{2}\norm{\Ats\zeta-f}_{\Hit}
+\norm{\zeta-\At\xi}_{\Hith}\big).
\end{align*}
Another option to compute an upper bound in \eqref{secordeAtsest} is the following one:
As $y\in D(\cAts)$ we observe $\At\xi-y\in D(\cAts)$ if $\xi\in D(\Ats\At)$.
The minimum in \eqref{secordeAtsest} is attained at 
$\hat{\xi}=e_{\Ats}+\ti{x}\in D(\At)$ with $\At\hat{\xi}=\At x=y$.
Since $\hat{\xi}\in D(\Ats\At)$ and $\Ats\At\hat{\xi}=\Ats y=f$ we obtain
$$\norm{e_{\Ats}}_{\Hit}
=\min_{\xi\in D(\Ats\At)}
\big(c_{2}\norm{\At\xi-y}_{\Hith}
+\norm{\xi-\ti{x}}_{\Hit}\big)
=\min_{\xi\in D(\Ats\At)}
\big(c_{2}^2\norm{\Ats\At\xi-f}_{\Hit}
+\norm{\xi-\ti{x}}_{\Hit}\big),$$
where the latter equality follows by the Friedrichs/Poincar\'e inequality.
To get a lower bound for $\norm{e_{\Ats}}_{\Hit}^2$ in \eqref{secordeAtsest} we
observe $e_{\Ats}\in R(\Ats)=R(\Ats\At)$ and derive
\begin{align*}
\norm{e_{\Ats}}_{\Hit}^2
&=\max_{\phi\in D(\Ats\At)}
\big(2\scp{e_{\Ats}}{\Ats\At\phi}_{\Hit}
-\norm{\Ats\At\phi}_{\Hit}^2\big)\\
&=\max_{\phi\in D(\Ats\At)}
\big(2\scp{e}{\Ats\At\phi}_{\Hit}
-\norm{\Ats\At\phi}_{\Hit}^2\big)\\
&=\max_{\phi\in D(\Ats\At)}
\big(2\scp{f}{\phi}_{\Hit}
-\scp{2\ti{x}+\Ats\At\phi}{\Ats\At\phi}_{\Hit}\big).
\end{align*}

We summarize the two sided bounds:

\begin{theo}
\label{apostestsostheo}
Additionally to the General Assumption \ref{genass}, suppose that $R(\Ath)$ is closed.
Let $x\in\ti{D}_{2}$ be the exact solution of \eqref{AsAprobsoltheo}, $y:=\At x$,
and let $(\ti{x},\ti{y})\in\Hit\times\Hith$. Then the following estimates hold for the errors 
$e=x-\ti{x}$ and $h=y-\ti{y}$ defined in \eqref{ehdef}:
\begin{itemize}
\item[\bf(i)]
The errors decompose, i.e.,
\begin{align*}
e=e_{\Ao}+e_{K_{2}}+e_{\Ats}
&\in R(\Ao)\oplus_{\Hit}K_{2}\oplus_{\Hit}R(\Ats),
&
\norm{e}_{\Hit}^2
&=\norm{e_{\Ao}}_{\Hit}^2
+\norm{e_{K_{2}}}_{\Hit}^2
+\norm{e_{\Ats}}_{\Hit}^2,\\
h=h_{\At}+h_{K_{3}}+h_{\Aths}
&\in R(\At)\oplus_{\Hith}K_{3}\oplus_{\Hith}R(\Aths),
&
\norm{h}_{\Hith}^2
&=\norm{h_{\At}}_{\Hith}^2
+\norm{h_{K_{3}}}_{\Hith}^2
+\norm{h_{\Aths}}_{\Hith}^2.
\end{align*}
\item[\bf(ii)]
The projection $e_{\Ao}=\pi_{\Ao}e=x_{g}-\pi_{\Ao}\ti{x}\in R(\Ao)$ satisfies
\begin{align*}
\norm{e_{\Ao}}_{\Hit}^2
&=\min_{\zeta\in D(\Aos)}
\big(c_{1}\norm{\Aos\zeta-g}_{\Hio}
+\norm{\zeta-\ti{x}}_{\Hit}\big)^2\\
&=\max_{\varphi\in D(\Ao)}
\big(2\scp{g}{\varphi}_{\Hio}
-\scp{2\ti{x}+\Ao\varphi}{\Ao\varphi}_{\Hit}\big)
\end{align*}
and the minimum resp. maximum is attained at 
$$\hat{\zeta}
:=e_{\Ao}+\ti{x}
\in D(\Aos),\qquad
\hat{\varphi}
:=\cA_{1}^{-1}e_{\Ao}
\in D(\cAo)$$
with $\Aos\hat{\zeta}=\Aos x=g$.
\item[\bf(iii)]
The projection $e_{\Ats}=\pi_{\Ats}e=x_{y}-\pi_{\Ats}\ti{x}\in R(\Ats)$ satisfies
\begin{align*}
\norm{e_{\Ats}}_{\Hit}^2
&=\min_{\xi\in D(\At)}
\min_{\zeta\in D(\Ats)}
\big(c_{2}^2\norm{\Ats\zeta-f}_{\Hit}
+c_{2}\norm{\zeta-\At\xi}_{\Hith}
+\norm{\xi-\ti{x}}_{\Hit}\big)^2\\
&=\min_{\xi\in D(\Ats\At)}
\big(c_{2}^2\norm{\Ats\At\xi-f}_{\Hit}
+\norm{\xi-\ti{x}}_{\Hit}\big)^2\\
&=\max_{\phi\in D(\Ats\At)}
\big(2\scp{f}{\phi}_{\Hit}
-\scp{2\ti{x}+\Ats\At\phi}{\Ats\At\phi}_{\Hit}\big)
\end{align*}
and the minima resp. maximum are attained at 
$$\hat{\xi}
:=e_{\Ats}+\ti{x}
\in D(\Ats\At),\quad
\hat{\zeta}
:=h_{\xi}+\At\xi=y
\in D(\Ats),\quad
\hat{\phi}
:=\cA_{2}^{-1}(\cAts)^{-1}e_{\Ats}
\in D(\cAts\cAt)$$
with $\At\hat{\xi}=\At x=y$ and $\Ats\At\hat{\xi}=\Ats y=f$
as well as $\Ats\hat{\zeta}=\Ats y=f$.
\item[\bf(iv)]
The projection $e_{K_{2}}=\pit e=k-\pit\ti{x}\in K_{2}$ satisfies
\begin{align*}
\norm{e_{K_{2}}}_{\Hit}^2
&=\min_{\varphi\in D(\Ao)}
\min_{\phi\in D(\Ats)}
\norm{k-\ti{x}+\Ao\varphi+\Ats\phi}_{\Hit}^2\\
&=\max_{\theta\in K_{2}}
\bscp{2(k-\ti{x})-\theta}{\theta}_{\Hit}
\end{align*}
and the minimum resp. maximum is attained at 
$$\hat{\varphi}
:=\cA_{1}^{-1}\pi_{\Ao}\ti{x}
\in D(\cAo),\qquad
\hat{\phi}
:=(\cAts)^{-1}\pi_{\Ats}\ti{x}
\in D(\cAts),\qquad
\hat{\theta}
:=e_{K_{2}}
\in K_{2}$$
with $\Ao\hat{\varphi}+\Ats\hat{\phi}=(\pi_{\Ao}+\pi_{\Ats})\ti{x}=(1-\pit)\ti{x}$.
\item[\bf(v)]
The projection $h_{\At}=\pi_{\At}h=y-\pi_{\At}\ti{y}\in R(\At)$ satisfies
\begin{align*}
\norm{h_{\At}}_{\Hith}^2
&=\min_{\zeta\in D(\Ats)}
\big(c_{2}\norm{\Ats\zeta-f}_{\Hit}
+\norm{\zeta-\ti{y}}_{\Hith}\big)^2\\
&=\max_{\varphi\in D(\At)}
\big(2\scp{f}{\varphi}_{\Hit}
-\scp{2\ti{y}+\At\varphi}{\At\varphi}_{\Hith}\big)
\end{align*}
and the minimum resp. maximum is attained at 
$$\hat{\zeta}
:=h_{\At}+\ti{y}
\in D(\Ats),\qquad
\hat{\varphi}
:=\cA_{2}^{-1}h_{\At}
\in D(\cAt)$$
with $\Ats\hat{\zeta}=\Ats y=f$.
\item[\bf(vi)]
The projection $h_{\Aths}=\pi_{\Aths}h=-\pi_{\Aths}\ti{y}\in R(\Aths)$ satisfies
\begin{align*}
\norm{h_{\Aths}}_{\Hith}^2
&=\min_{\xi\in D(\Ath)}
\big(c_{3}\norm{\Ath\xi}_{\Hif}
+\norm{\xi-\ti{y}}_{\Hith}\big)^2
=\min_{\xi\in N(\Ath)}
\norm{\xi-\ti{y}}_{\Hith}^2\\
&=\max_{\phi\in D(\Aths)}
\big(-\scp{2\ti{y}+\Aths\phi}{\Aths\phi}_{\Hith}\big)\end{align*}
and the minimum resp. maximum is attained at 
$$\hat{\xi}
:=h_{\Aths}+\ti{y}
\in N(\Ath),\qquad
\hat{\phi}
:=(\cAths)^{-1}h_{\Aths}
\in D(\cAths)$$
with $\Ath\hat{\xi}=\Ath y=0$.
\item[\bf(vii)]
The projection $h_{K_{3}}=\pith e=-\pith\ti{y}\in K_{3}$ satisfies
\begin{align*}
\norm{h_{K_{3}}}_{\Hith}^2
&=\min_{\varphi\in D(\At)}
\min_{\phi\in D(\Aths)}
\norm{-\ti{y}+\At\varphi+\Aths\phi}_{\Hith}^2\\
&=\max_{\theta\in K_{3}}
\big(-\scp{2\ti{y}+\theta}{\theta}_{\Hith}\big)
\end{align*}
and the minimum resp. maximum is attained at 
$$\hat{\varphi}
:=\cA_{2}^{-1}\pi_{\At}\ti{y}
\in D(\cAt),\qquad
\hat{\phi}
:=(\cAths)^{-1}\pi_{\Aths}\ti{y}
\in D(\cAths),\qquad
\hat{\theta}
:=h_{K_{3}}
\in K_{3}$$
with $\At\hat{\varphi}+\Aths\hat{\phi}=(\pi_{\At}+\pi_{\Aths})\ti{y}=(1-\pith)\ti{y}$.
\end{itemize}
If $\ti{x}=k+\ti{x}_{\bot}$ with some $\ti{x}_{\bot}\in K_{2}^{\bot_{\Hit}}$,
then $e_{K_{2}}=0$, and in (ii) and (iii) $\ti{x}$ can be replaced by $\ti{x}_{\bot}$.
If $\ti{y}=\ti{y}_{\bot}\in K_{3}^{\bot_{\Hith}}$,
then $h_{K_{3}}=0$, and in (v) and (vi) $\ti{y}$ can be replaced by $\ti{y}_{\bot}$.
\end{theo}

\begin{rem}
A reasonable assumption provided by standard numerical methods is $\ti{y}\in R(\At)$.
Hence it often holds $h_{\Aths}=h_{K_{3}}=0$.
\end{rem}

\subsection{Computing the Error Functionals}
\mylabel{seccomperrfunc}

We propose suitable ways to compute the most important error functionals
in Theorem \ref{apostestfos}, Corollary \ref{apostestfosKbot}, and Corollary \ref{apostestfoscortsb}.
For example, let us focus on Corollary \ref{apostestfoscortsb} (ii), i.e.,
for $\ti{x}\in\Hit$ on the error estimates
\begin{align}
\mylabel{comperrorfuncmaxmin}
\max_{\varphi\in D(\Ao)}
\big(2\scp{g}{\varphi}_{\Hio}
-\scp{2\ti{x}+\Ao\varphi}{\Ao\varphi}_{\Hit}\big)
=\norm{e_{\Ao}}_{\Hit}^2
=\min_{\zeta\in D(\Aos)}
\big(c_{1}\norm{\Aos\zeta-g}_{\Hio}
+\norm{\zeta-\ti{x}}_{\Hit}\big)^2.
\end{align}
Before proceeding, let us note that instead of computing the maximum resp. minimum of the lower resp. upper bound
we can simply und cheaply choose any $\varphi\in D(\Ao)$ and any $\zeta\in D(\Aos)$
given by any method or guess and we obtain the guaranteed error bounds
$$2\scp{g}{\varphi}_{\Hio}
-\scp{2\ti{x}+\Ao\varphi}{\Ao\varphi}_{\Hit}
\leq\norm{e_{\Ao}}_{\Hit}^2
\leq\big(c_{1}\norm{\Aos\zeta-g}_{\Hio}
+\norm{\zeta-\ti{x}}_{\Hit}\big)^2.$$

\subsubsection{Lower Bounds}

Considering the maximum on the left hand side of \eqref{comperrorfuncmaxmin} we differentiate the lower bound
$\Phi(\varphi):=2\scp{g}{\varphi}_{\Hio}-\scp{2\ti{x}+\Ao\varphi}{\Ao\varphi}_{\Hit}$ 
with respect to $\varphi$.
Hence a maximizer $\hat{\varphi}\in D(\Ao)$ solves the variational formulation
\begin{align}
\mylabel{varformMmAo}
\forall\,\varphi\in D(\Ao)\qquad
0=
-\foh\Phi'(\hat{\varphi})\varphi=
\scp{\Ao\hat{\varphi}}{\Ao\varphi}_{\Hit}
+\scp{\ti{x}}{\Ao\varphi}_{\Hit}
-\scp{g}{\varphi}_{\Hio},
\end{align}
which implies $\Ao\hat{\varphi}+\ti{x}\in D(\Aos)$ with 
$$\Aos(\Ao\hat{\varphi}+\ti{x})=g=\Aos x$$
and presents a weak formulation of
$\Aos\Ao\hat{\varphi}
=g-\Aos\ti{x}
=\Aos e
=\Aos e_{\Ao}$. By 
$$0=\Aos(\Ao\hat{\varphi}+\ti{x}-x)=\Aos(\Ao\hat{\varphi}-e)=\Aos(\Ao\hat{\varphi}-e_{\Ao})$$
we observe $\Ao\hat{\varphi}-e_{\Ao}\in N(\Aos)\cap R(\Ao)=N(\Aos)\cap N(\Aos)^{\bot_{\Hit}}=\{0\}$,
i.e., $\hat{\varphi}$ solves $\Ao\hat{\varphi}=e_{\Ao}$,
see Corollary \ref{apostestfoscortsb} (ii).
As $\Ao$ is strictly positive over $D(\cAo)=D(\Ao)\cap R(\Aos)=D(\Ao)\cap N(\Ao)^{\bot_{\Hio}}$,
\eqref{varformMmAo} admits a unique solution $\hat{\varphi}\in D(\cAo)$.
A particularly simple case is given if $N(\Ao)$ is finite dimensional or even trivial,
which occurs in many applications. Otherwise one has to work with the
saddle point or double saddle point formulations as we have discussed earlier. 
The previous considerations show that the unique maximizer $\hat{\varphi}\in D(\cAo)$ is given by
$$\hat{\varphi}=\cA_{1}^{-1}e_{\Ao},$$ 
which is already written down in
Corollary \ref{apostestfoscortsb} (ii). Moroever, we finally note
$$\hat{\varphi}
=\cA_{1}^{-1}e_{\Ao}
=\cA_{1}^{-1}\pi_{\Ao}e
=\cA_{1}^{-1}\pi_{\Ao}(x-\ti{x})
=\cA_{1}^{-1}(x_{g}-\pi_{\Ao}\ti{x})
=\cA_{1}^{-1}\big((\cAos)^{-1}g-\pi_{\Ao}\ti{x}\big).$$
If $\ti{x}\in D(\Aos)$ then $\pi_{\Ao}\ti{x}\in D(\cAos)$ 
with $\Aos\pi_{\Ao}\ti{x}=\Aos\ti{x}$ and 
$\hat{\varphi}=\cA_{1}^{-1}(\cAos)^{-1}(g-\Aos\ti{x})$.

\begin{rem}
\label{remlowbdcomprem}
The maximum in \eqref{comperrorfuncmaxmin} is attained at any 
$\hat{\varphi}\in D(\Ao)$ with $\Ao\hat{\varphi}=e_{\Ao}$,
especially at $\hat{\varphi}=\cA_{1}^{-1}e_{\Ao}\in D(\cAo)$.
$\hat{\varphi}\in D(\Ao)$ can be found by the variational formulation
$$\forall\,\varphi\in D(\Ao)\qquad
\scp{\Ao\hat{\varphi}}{\Ao\varphi}_{\Hit}
=\scp{g}{\varphi}_{\Hio}
-\scp{\ti{x}}{\Ao\varphi}_{\Hit},$$
which is coercive (positive) over $D(\cAo)$.
\end{rem}

\subsubsection{Upper Bounds}

For the minimum on the right hand side of \eqref{comperrorfuncmaxmin} we can roughly estimate the upper bound by
$\Psi(\zeta):=2c_{1}^2\norm{\Aos\zeta-g}_{\Hio}^2+2\norm{\zeta-\ti{x}}_{\Hit}^2$.
Differentiating $\Psi$ shows that the minimizer $\check{\zeta}\in D(\Aos)$
of $\min_{\zeta\in D(\Aos)}\Psi(\zeta)$ solves the variational formulation
\begin{align}
\mylabel{varformMpAo}
\begin{split}
\forall\,\zeta\in D(\Aos)\qquad
0=
\frac{1}{4}\Psi'(\check{\zeta})\zeta
&=c_{1}^2\scp{\Aos\check{\zeta}-g}{\Aos\zeta}_{\Hio}
+\scp{\check{\zeta}-\ti{x}}{\zeta}_{\Hit}\\
&=c_{1}^2\scp{\Aos\check{\zeta}}{\Aos\zeta}_{\Hio}
+\scp{\check{\zeta}}{\zeta}_{\Hit}
-c_{1}^2\scp{g}{\Aos\zeta}_{\Hio}
-\scp{\ti{x}}{\zeta}_{\Hit},
\end{split}
\end{align}
which implies $\Aos\check{\zeta}-g\in D(\Ao)$
and $c_{1}^2\Ao(\Aos\check{\zeta}-g)=(\ti{x}-\check{\zeta})$
and presents a weak formulation of
$$c_{1}^2\Ao\Aos\check{\zeta}+\check{\zeta}
=c_{1}^2\Ao g+\ti{x}.$$
Unique solvability of \eqref{varformMpAo} in $D(\Aos)$ is trivial,
as the variational formulation reproduces a graph inner product of $D(\Aos)$,
and we have $\check{\zeta}=(c_{1}^2\Ao\Aos+1)^{-1}(c_{1}^2\Ao g+\ti{x})$.
Moreover, as $g\in R(\Aos)$ it even holds $\Aos\check{\zeta}-g\in D(\cAo)$
and hence by the Friedrichs/Poincar\'e estimate, the equation for $\check{\zeta}$,
and inserting $\zeta=\check{\zeta}$ into Corollary \ref{apostestfoscortsb} (ii)
\begin{align}
\label{ubdsimple}
\begin{split}
\norm{e_{\Ao}}_{\Hit}
\leq c_{1}\norm{\Aos\check{\zeta}-g}_{\Hio}
+\norm{\check{\zeta}-\ti{x}}_{\Hit}
&\leq c_{1}^2\bnorm{\Ao(\Aos\check{\zeta}-g)}_{\Hio}
+\norm{\check{\zeta}-\ti{x}}_{\Hit}\\
&=2\begin{cases}
\norm{\check{\zeta}-\ti{x}}_{\Hit},\\
c_{1}^2\bnorm{\Ao(\Aos\check{\zeta}-g)}_{\Hio}.
\end{cases}
\end{split}
\end{align}

This rough minimization process can be improved by using a bit more careful estimate
for the square term in \eqref{comperrorfuncmaxmin}.
For this we observe for all $\zeta\in D(\Aos)$ and all $t>0$
$$\norm{e_{\Ao}}_{\Hit}^2
\leq(1+t^{-1})\,c_{1}^2\,\norm{\Aos\zeta-g}_{\Hio}^2
+(1+t)\norm{\zeta-\ti{x}}_{\Hit}^2
=:\Psi(\ti{x};\zeta,t)$$
and obtain by choosing $\zeta=\hat{\zeta}=e_{\Ao}+\ti{x}\in D(\Aos)$ from Theorem \ref{apostestfos}, 
Corollary \ref{apostestfosKbot} or Corollary \ref{apostestfoscortsb}
$$\norm{e_{\Ats}}_{\Hit}^2
\leq\inf_{t\in(0,\infty)}
\inf_{\zeta\in D(\Aos)}
\Psi(\ti{x};\zeta,t)
\leq\inf_{t\in(0,\infty)}
\Psi(\ti{x};\hat{\zeta},t)
=\inf_{t\in(0,\infty)}(1+t)
\norm{e_{\Ao}}_{\Hit}^2
=\norm{e_{\Ao}}_{\Hit}^2.$$
Thus
\begin{align}
\label{funccompmintwo}
\norm{e_{\Ao}}_{\Hit}^2
=\min_{\substack{t\in[0,\infty],\\\zeta\in D(\Aos)}}
\Psi(\ti{x};\zeta,t)
=\min_{\substack{t\in[0,\infty],\\\zeta\in D(\Aos)}}
\big((1+t^{-1})\,c_{1}^2\,\norm{\Aos\zeta-g}_{\Hio}^2
+(1+t)\norm{\zeta-\ti{x}}_{\Hit}^2\big)
\end{align}
and the minimum is attained at $(t,\zeta)=(0,\hat{\zeta})$.
For fixed $\zeta\in D(\Aos)$ the minimal $t_{\zeta}\in[0,\infty]$ is given by
$$t_{\zeta}=
\begin{cases}
\displaystyle
c_{1}\frac{\norm{\Aos\zeta-g}_{\Hio}}{\norm{\zeta-\ti{x}}_{\Hit}}
&\text{, if }\zeta\neq\ti{x},\\
\infty
&\text{, if }\zeta=\ti{x}.
\end{cases}$$
We note that the case $t_{\zeta}=\infty$ can only happen if $\ti{x}\in D(\Aos)$.
In any case, inserting $t_{\zeta}$ into \eqref{funccompmintwo} we get back
the right hand side of  \eqref{comperrorfuncmaxmin}, i.e.,
$$\norm{e_{\Ats}}_{\Hit}^2
\leq\min_{\zeta\in D(\Aos)}
\big(c_{1}\norm{\Aos\zeta-g}_{\Hio}
+\norm{\zeta-\ti{x}}_{\Hit}\big)^2
=\norm{e_{\Ao}}_{\Hit}^2.$$
On the other hand, for fixed $0<t<\infty$ the minimization of $\Psi_{t}(\zeta):=\Psi(\ti{x};\zeta,t)$ 
over $\zeta\in D(\Aos)$ is equivalent to find $\zeta_{t}\in D(\Aos)$, such that
\begin{align}
\mylabel{vfcompone}
\forall\,\zeta\in D(\Aos)\qquad
\frac{t}{2c_{1}^2(1+t)}\Psi_{t}'(\zeta_{t})(\zeta)
=\scp{\Aos\zeta_{t}-g}{\Aos\zeta}_{\Hio}
+\frac{t}{c_{1}^2}\scp{\zeta_{t}-\ti{x}}{\zeta}_{\Hit}
=0.
\end{align}
Especially $\Aos\zeta_{t}-g\in D(\Ao)$ with 
\begin{align}
\mylabel{algoubdeqAoAos}
\Ao(\Aos\zeta_{t}-g)
=\frac{t}{c_{1}^2}(\ti{x}-\zeta_{t})
\in R(\Ao)
\end{align}
and hence \eqref{vfcompone} is a standard weak formulation of the coercive problem 
(in formally strong form)
$(\Ao\Aos+\frac{t}{c_{1}^2})\zeta_{t}=\Ao g+\frac{t}{c_{1}^2}\ti{x}$, i.e.,
\begin{align}
\mylabel{vfcomptwo}
\forall\,\zeta\in D(\Aos)\qquad
\scp{\Aos\zeta_{t}}{\Aos\zeta}_{\Hio}
+\frac{t}{c_{1}^2}\scp{\zeta_{t}}{\zeta}_{\Hit}
=\scp{g}{\Aos\zeta}_{\Hio}
+\frac{t}{c_{1}^2}\scp{\ti{x}}{\zeta}_{\Hit}.
\end{align}
Moreover, as $g\in R(\Aos)$ we even have 
$\Aos\zeta_{t}-g\in D(\cAo)$
and the strong form holds rigorously if $g\in D(\Ao)$.
Furthermore, inserting $\zeta_{t}$ into \eqref{funccompmintwo} 
and using the Friedrichs/Poincar\'e type estimate shows
\begin{align*}
\begin{split}
\norm{e_{\Ao}}_{\Hit}^2
&\leq\min_{t\in[0,\infty]}
\big((1+t^{-1})\,c_{1}^2\,\norm{\Aos\zeta_{t}-g}_{\Hio}^2
+(1+t)\norm{\zeta_{t}-\ti{x}}_{\Hit}^2\big)\\
&\leq\min_{t\in[0,\infty]}
\big((1+t^{-1})\,c_{1}^4\,\bnorm{\Ao(\Aos\zeta_{t}-g)}_{\Hit}^2
+(1+t)\norm{\zeta_{t}-\ti{x}}_{\Hit}^2\big)\\
&=
\begin{cases}
\displaystyle\min_{t\in[0,\infty]}
(1+t)^2\norm{\zeta_{t}-\ti{x}}_{\Hit}^2,\\
\displaystyle\min_{t\in[0,\infty]}
(1+t^{-1})^2\,c_{1}^4\,\bnorm{\Ao(\Aos\zeta_{t}-g)}_{\Hit}^2,
\end{cases}
\end{split}
\end{align*}
compare to \eqref{ubdsimple}. Hence the overestimation by the factor $2$ is removed
as long as $t$ is close to $0$ or $\infty$.
A suitable algorithm for computing a good pair $(t,\zeta)$ 
for approximately minimizing \eqref{funccompmintwo} is the following:

%

\begin{algo}
\label{algotwo}
Computing a minimizer $(t,\zeta)$ in \eqref{funccompmintwo}, i.e., an upper bound for $\norm{e_{\Ao}}_{\Hit}$:
\begin{itemize}
\item{\sf initialization:} 
Set $n:=0$. Pick $\zeta_{0}\in D(\Aos)$ with $\zeta_{0}\neq\ti{x}$.
\item{\sf loop:} Set $n:=n+1$. Compute 
$\displaystyle t_{n}=c_{1}\frac{\norm{\Aos\zeta_{n-1}-g}_{\Hio}}{\norm{\zeta_{n-1}-\ti{x}}_{\Hit}}$
and then $\zeta_{n}\in D(\Aos)$ by solving
$$\forall\,\zeta\in D(\Aos)\qquad
c_{1}^2\scp{\Aos\zeta_{n}}{\Aos\zeta}_{\Hio}
+t_{n}\scp{\zeta_{n}}{\zeta}_{\Hit}
=c_{1}^2\scp{g}{\Aos\zeta}_{\Hio}
+t_{n}\scp{\ti{x}}{\zeta}_{\Hit}.$$
Compute 
$\Psi_{\Aos}(\ti{x};\zeta_{n},t_{n})
:=(1+t_{n}^{-1})\,c_{1}^2\,\norm{\Aos\zeta_{n}-g}_{\Hio}^2+(1+t_{n})\norm{\zeta_{n}-\ti{x}}_{\Hit}^2$.
\item{\sf stop} if $\Psi_{\Aos}(\ti{x};\zeta_{n},t_{n})-\Psi_{\Aos}(\ti{x};\zeta_{n-1},t_{n-1})$ is small.
\end{itemize}
\end{algo}

\begin{rem}
\label{algorem}
\eqref{vfcomptwo} shows for $\zeta=\zeta_{t}$
$$c_{1}^2\norm{\Aos\zeta_{t}}_{\Hio}^2
+t\norm{\zeta_{t}}_{\Hit}^2
=c_{1}^2\scp{g}{\Aos\zeta_{t}}_{\Hio}
+t\scp{\ti{x}}{\zeta_{t}}_{\Hit}
\leq\big(c_{1}^2\norm{g}_{\Hio}^2
+t\norm{\ti{x}}_{\Hit}^2\big)^{\oh}
\big(c_{1}^2\norm{\Aos\zeta_{t}}_{\Hio}^2
+t\norm{\zeta_{t}}_{\Hit}^2\big)^{\oh}$$
and thus
$$c_{1}^2\norm{\Aos\zeta_{t}}_{\Hio}^2
+t\norm{\zeta_{t}}_{\Hit}^2
\leq c_{1}^2\norm{g}_{\Hio}^2
+t\norm{\ti{x}}_{\Hit}^2.$$
By \eqref{algoubdeqAoAos} and since $\Aos\zeta_{t}-g\in D(\cAo)$ we get
$$\Aos\zeta_{t}-g=\frac{t}{c_{1}^2}\cA_{1}^{-1}(\ti{x}-\zeta_{t})$$
and hence
$$\norm{\Aos\zeta_{t}-g}_{\Hio}
\leq c\,t^{\oh}\big(\norm{g}_{\Hio}+t^{\oh}\norm{\ti{x}}_{\Hit}\big)$$
with $c>0$ independent of $t$ and $\zeta_{t}$.
Let us assume $t_{n}\to0$ in Algorithm \ref{algotwo}.
Then by the latter considerations $(\Aos\zeta_{n})$ 
and $(t_{n}^{\oh}\zeta_{n})$ are bounded
and $\Aos\zeta_{n}\to g$ with the minimal rate $t_{n}^{\oh}$.
Moreover, the projected sequence $(\pi_{\Ao}\zeta_{n})\subset D(\cAos)$ is bounded in $D(\Aos)$
by $\Aos\pi_{\Ao}\zeta_{n}=\Aos\zeta_{n}$ and the Friedrichs/Poincar\'e estimate 
$\norm{\pi_{\Ao}\zeta_{n}}_{\Hit}\leq c_{1}\norm{\Aos\pi_{\Ao}\zeta_{n}}_{\Hio}$.
If $D(\cAos)\dhookrightarrow\Hit$ is compact, then we can extract a subsequence, again denoted by $(t_{n})$,
such that $\pi_{\Ao}\zeta_{n}\to\check{\zeta}$ in $\Hit$.
Thus $\check{\zeta}\in D(\cAos)$ and $\Aos\check{\zeta}=g$ as $\cAos$ is closed,
which shows $\check{\zeta}=(\cAos)^{-1}g=x_{g}=\pi_{\Ao}x$, see Theorem \ref{soltheofos}.
As the limit $x_{g}$ is unique, even the whole sequence $\pi_{\Ao}\zeta_{n}$ converges to $x_{g}$.
For the other part $(1-\pi_{\Ao})\zeta_{n}\subset N(\Aos)$ 
we apply the projector $1-\pi_{\Ao}$ to \eqref{algoubdeqAoAos} and obtain $(1-\pi_{\Ao})(\ti{x}-\zeta_{n})=0$, i.e.,
$(1-\pi_{\Ao})\zeta_{n}=(1-\pi_{\Ao})\ti{x}$ is constant. Hence
$$\zeta_{n}
=\pi_{\Ao}\zeta_{n}+(1-\pi_{\Ao})\zeta_{n}
\to\pi_{\Ao}x+(1-\pi_{\Ao})\ti{x}
=e_{\Ao}+\ti{x}
=\hat{\zeta},$$
where $\hat{\zeta}\in D(\Aos)$ is the unique minimizer from Corollary \ref{apostestfoscortsb} (ii).
Finally, Algorithm \ref{algotwo} defines a sequence $(\zeta_{n})$ converging in $D(\Aos)$ to $\hat{\zeta}$
provided that $D(\cAos)\dhookrightarrow\Hit$ is compact and $t_{n}\to0$.
\end{rem}

\section{Applications}
\label{secappl}

\subsection{Prototype First Order System: Electro-Magneto Statics}
\label{emssec}

As a prototypical example for a first order system we will discuss
the system of electro-magneto statics with mixed boundary conditions.
Let $\om\subset\rt$ be a bounded weak Lipschitz domain,
see \cite[Definition 2.3]{bauerpaulyschomburgmcpweaklip}, and
let $\ga:=\p\om$ denote its boundary (Lipschitz manifold), which is supposed to be decomposed
into two relatively open weak Lipschitz subdomains (Lipschitz submanifolds) $\gat$ and $\gan:=\ga\setminus\ovl{\gat}$
see \cite[Definition 2.5]{bauerpaulyschomburgmcpweaklip}.
Let us consider the linear first order system (in classical strong formulation)
for a vector field $E:\om\to\rt$
\begin{align}
\nonumber
\rot E&=F
&
\text{in }&\om,
&
n\times E&=0
&
\text{at }&\gat,\\
\mylabel{elmgstat}
-\div\eps E&=g
&
\text{in }&\om,
&
n\cdot\eps E&=0
&
\text{at }&\gan,\\
\nonumber
\pi_{\harm}E&=K
&
\text{in }&\om.
\end{align}
Here, $\eps:\om\to\rttt$ is a symmetric, uniformly positive definite $\li$-matrix field
and $n$ denotes the outer unit normal at $\ga$. Let us put $\mu:=\eps^{-1}$.
The usual Lebesgue and Sobolev (Hilbert) spaces will be denoted by $\ltom$, $\hlom$, $\ell\in\n_{0}$, 
and (in the distributional sense) we introduce
\begin{align*}
\rom:=\setb{E\in\ltom}{\rot E\in\ltom},\qquad 
\dom:=\setb{E\in\ltom}{\div E\in\ltom}.
\end{align*}
Let us also define
$$\ltbotom:=\ltom\cap\reals^{\bot_{\ltom}},\qquad
\hobotom:=\hoom\cap\ltbotom.$$
With the test functions or test vector fields
\begin{align*}
\cgen{}{\infty}{\gat}(\om):=\setb{\varphi|_\om}{\varphi\in\ci(\rt),
\,\supp\varphi\text{ compact in }\rt,\,\dist(\supp\varphi,\gat)>0},\qquad
\cgen{}{\infty}{\emptyset}(\om)=\ciomb,
\end{align*}
we define as closures of test functions resp. test fields
\begin{align*}
\hogatom:=\overline{\cgen{}{\infty}{\gat}(\om)}^{\hoom},\qquad 
\rgatom :=\overline{\cgen{}{\infty}{\gat}(\om)}^{\rom},\qquad
\dganom :=\overline{\cgen{}{\infty}{\gan}(\om)}^{\dom},
\end{align*}
generalizing homogeneous scalar, tangential, and normal traces on $\gat$ and $\gan$, respectively.
Moreover, we introduce the closed subspaces
\begin{align*}
\rzom&:=\set{E\in\rom}{\rot E=0},
&
\dzom&:=\set{E\in\dom}{\div E=0},\\
\rgatzom&:=\rgatom\cap\rzom,
&
\dganzom&:=\dganom\cap\dzom,
\end{align*}
and the Dirichlet-Neumann fields including the corresponding orthonormal projector
$$\harmgatnepsom:=\rgatzom\cap\mu\,\dganzom,\qquad
\pi_{\harm}:\ltepsom\to\harmgatnepsom.$$
Here, $\ltepsom$ denotes $\ltom$ equipped with the inner product 
$\scp{\,\cdot\,}{\,\cdot\,}_{\ltepsom}:=\scpltom{\eps\,\cdot\,}{\,\cdot\,}$. 
Let $\Hio:=\ltom$, $\Hif:=\ltom$ (both scalar valued) 
and $\Hit:=\ltepsom$, $\Hith:=\ltom$ (both vector valued) as well as
\begin{align*}
\Ao:=\grad_{\gat}:D(\Ao):=\hogatom\subset\ltom&\to\ltepsom,\\
\At:=\rot_{\gat}:D(\At):=\rgatom\subset\ltepsom&\to\ltom,\\
\Ath:=\div_{\gat}:D(\Ath):=\dgatom\subset\ltom&\to\ltom.
\intertext{In \cite{bauerpaulyschomburgmcpweaklip} it has been shown that the adjoints are}
\Aos=\grad_{\gat}^{*}=-\div_{\gan}\eps:D(\Aos)=\mu\,\dganom\subset\ltepsom&\to\ltom,\\
\Ats=\rot_{\gat}^{*}=\mu\rot_{\gan}:D(\Ats)=\rganom\subset\ltom&\to\ltepsom,\\
\Aths=\div_{\gat}^{*}=-\grad_{\gan}:D(\Aths)=\hoganom\subset\ltom&\to\ltom.
\end{align*}
As $\Ao$, $\At$, $\Ath$ define the well known de Rham complex, see, e.g., \cite[Lemma 2.2]{bauerpaulyschomburgmcpweaklip},
so do their adjoints, i.e., for\footnote{For $\gat=\emptyset$ 
we have
$$\begin{CD}
\reals @> \Az=\iota_{\reals} >>
\hoom @> \Ao=\grad >>
\rom @> \At=\rot >>
\dom @> \Ath=\div >>
\ltom @> \A_{4}=\pi_{\{0\}} >>
\{0\},
\end{CD}$$
$$\begin{CD}
\reals @< \Azs=\pi_{\reals} <<
\ltom @< \Aos=-\div_{\ga}\eps <<
\mu\,\dgaom @< \Ats=\mu\rot_{\ga} <<
\rgaom @< \Aths=-\grad_{\ga} <<
\hogaom @< \A_{4}^{*}=\iota_{\{0\}} <<
\{0\},
\end{CD}$$
which also shows the case $\gat=\ga$ by interchanging $\gat$ and $\gan$ and shifting $\eps$.
More precisely, for $\gat=\ga$ it holds
$$\begin{CD}
\{0\} @> \Az=\iota_{\{0\}} >>
\hogaom @> \Ao=\grad_{\ga} >>
\rgaom @> \At=\rot_{\ga} >>
\dgaom @> \Ath=\div_{\ga} >>
\ltom @> \A_{4}=\pi_{\reals} >>
\reals,
\end{CD}$$
$$\begin{CD}
\{0\} @< \Azs=\pi_{\{0\}} <<
\ltom @< \Aos=-\div\eps <<
\mu\,\dom @< \Ats=\mu\rot <<
\rom @< \Aths=-\grad <<
\hoom @< \A_{4}^{*}=\iota_{\reals} <<
\reals.
\end{CD}$$
} 
$\emptyset\neq\gat\neq\ga$
$$\begin{CD}
\{0\} @> \Az=\iota_{\{0\}} >>
\hogatom @> \Ao=\grad_{\gat} >>
\rgatom @> \At=\rot_{\gat} >>
\dgatom @> \Ath=\div_{\gat} >>
\ltom @> \A_{4}=\pi_{\{0\}} >>
\{0\},
\end{CD}$$
$$\begin{CD}
\{0\} @< \A_{0}^{*}=\pi_{\{0\}} <<
\ltom @< \Aos=-\div_{\gan}\eps <<
\mu\,\dganom @< \Ats=\mu\rot_{\gan} <<
\rganom @< \Aths=-\grad_{\gan} <<
\hoganom @< \A_{4}^{*}=\iota_{\{0\}} <<
\{0\},
\end{CD}$$
where we have introduced the additional canonical embedding and projection
operators $\Az$, $\Azs$, $\A_{4}$, $\A_{4}^{*}$ by
\begin{align*}
\Az=\begin{cases}
\iota_{\{0\}}:\hsymbol_{0}=\{0\}&\text{, if }\gat\neq\emptyset\\
\iota_{\reals}:\hsymbol_{0}=\reals&\text{, if }\gat=\emptyset
\end{cases}\,&\to\ltom,
&
\A_{4}=\begin{cases}
\pi_{\{0\}}&\text{, if }\gat\neq\ga\\
\pi_{\reals}&\text{, if }\gat=\ga
\end{cases}\,:\ltom&\to\begin{cases}
\{0\}&\text{, if }\gat\neq\ga\\
\reals&\text{, if }\gat=\ga
\end{cases},\\
\A_{4}^{*}=\begin{cases}
\iota_{\{0\}}:\hsymbol_{5}=\{0\}&\text{, if }\gat\neq\ga\\
\iota_{\reals}:\hsymbol_{5}=\reals&\text{, if }\gat=\ga
\end{cases}\,&\to\ltom,
&
\Azs=\begin{cases}
\pi_{\{0\}}&\text{, if }\gat\neq\emptyset\\
\pi_{\reals}&\text{, if }\gat=\emptyset
\end{cases}\,:\ltom&\to\begin{cases}
\{0\}&\text{, if }\gat\neq\emptyset\\
\reals&\text{, if }\gat=\emptyset
\end{cases}.
\end{align*}
For the kernels we have
\begin{align*}
N(\Az)&=\{0\}
&
N(\Azs)&=\begin{cases}\ltom&\text{, if }\gat\neq\emptyset,\\\ltbotom&\text{, if }\gat=\emptyset,\end{cases},\\
N(\Ao)&=\begin{cases}\{0\}&\text{, if }\gat\neq\emptyset,\\\reals&\text{, if }\gat=\emptyset,\end{cases}
&
N(\Aos)&=\mu\,\dganzom,\\
N(\At)&=\rgatzom,
&
N(\Ats)&=\rganzom,\\
N(\Ath)&=\dgatzom,
&
N(\Aths)&=\begin{cases}\{0\}&\text{, if }\gat\neq\ga,\\\reals&\text{, if }\gat=\ga,\end{cases},\\
N(\A_{4})&=\begin{cases}\ltom&\text{, if }\gat\neq\ga,\\\ltbotom&\text{, if }\gat=\ga,\end{cases},
&
N(\A_{4}^{*})&=\{0\}
\end{align*}
and for the cohomology groups
\begin{align*}
K_{0}&=N(\Az)=\{0\},\\
K_{1}&=N(\Ao)\cap N(\Azs)=\{0\},\\
K_{2}&=N(\At)\cap N(\Aos)=\rgatzom\cap\mu\,\dganzom=\harmgatnepsom,\\
K_{3}&=N(\Ath)\cap N(\Ats)=\dgatzom\cap\rganzom=:\harmgantom,\\
K_{4}&=N(\A_{4})\cap N(\Aths)=\{0\},\\
K_{5}&=N(\A_{4}^{*})=\{0\}.
\end{align*}
Using the latter operators $\At=\rot_{\gat}$ and $\Aos=-\div_{\gan}\eps$,
the linear first order system \eqref{elmgstat} 
(in weak formulation) has the form of \eqref{Aprob} resp. \eqref{Aprobsoltheo}, i.e.,
find a vector field 
$$E\in D_{2}=D(\At)\cap D(\Aos)=\rgatom\cap\mu\,\dganom,$$
such that
\begin{align}
\mylabel{elmgstatpo}
\begin{split}
\rot_{\gat}E&=F,\\
-\div_{\gan}\eps E&=g,\\
\pi_{\harm}E&=K,
\end{split}
\end{align}
where $K_{2}=\harmgatnepsom$.
In \cite[Theorem 5.1]{bauerpaulyschomburgmcpweaklip}
the embedding $D_{2}\dhookrightarrow\Hit$, i.e.,
$$\rgatom\cap\mu\,\dganom\dhookrightarrow\ltepsom,$$
was shown to be compact. Hence also the embedding $D_{3}=D(\Ath)\cap D(\Ats)\dhookrightarrow\Hith$, i.e.,
$$\dgatom\cap\rganom\dhookrightarrow\ltom,$$
is compact. Thus, by the results of the functional analysis toolbox Section \ref{sectoolbox}, 
all occurring ranges are closed, certain Helmholtz type decompositions hold,
corresponding Friedrichs/Poincar\'e type estimates are valid, 
and the respective inverse operators are continuous resp. compact.
Especially, the reduced operators are
\begin{align*}
\cAo:=\widetilde\grad_{\gat}:D(\cAo)=\hogatom\cap\ltom\subset\ltom&\to\grad\hogatom,\\
\cAt:=\widetilde\rot_{\gat}:D(\cAt)=\rgatom\cap\mu\rot\rganom\subset\mu\rot\rganom&\to\rot\rgatom,\\
\cAth:=\widetilde\div_{\gat}:D(\cAth)=\dgatom\cap\grad\hoganom\subset\grad\hoganom&\to\ltom,
\intertext{where $\grad\hogatom$ and $\mu\rot\rganom$ have to be understood as closed subspaces of $\ltepsom$,
and $\ltom$ has to be replaced by $\ltbotom$ 
in $\cAo$, if $\gat=\emptyset$, 
and in $\cAth$, if $\gat=\ga$, with adjoints}
\cAos=\widetilde\grad_{\gat}^{*}=-\widetilde\div_{\gan}\eps:D(\cAos)=\mu\,\dganom\cap\grad\hogatom\subset\grad\hogatom&\to\ltom,\\
\cAts=\widetilde\rot_{\gat}^{*}=\mu\,\widetilde\rot_{\gan}:D(\cAts)=\rganom\cap\rot\rgatom\subset\rot\rgatom&\to\mu\rot\rganom,\\
\cAths=\widetilde\div_{\gat}^{*}=-\widetilde\grad_{\gan}:D(\cAths)=\hoganom\cap\ltom\subset\ltom&\to\grad\hoganom,
\end{align*}
where $\ltom$ has to be replaced by $\ltbotom$ 
in $\cAos$, if $\gat=\emptyset$, 
and in $\cAths$, if $\gat=\ga$.
Note that the reduced operators possess bounded resp. compact inverse operators.
For the ranges we have
\begin{align*}
R(\Ao)&=R(\cAo)\subset N(\At),\text{ i.e.,}
&
\grad\hogatom&=\grad\big(\hogatom\cap\ltom\big)\subset\rgatzom,\\
R(\At)&=R(\cAt)\subset N(\Ath),\text{ i.e.,}
&
\rot\rgatom&=\rot\big(\rgatom\cap\mu\rot\rganom\big)\subset\dgatzom,\\
R(\Ath)&=R(\cAth),\text{ i.e.,}
&
\div\dgatom&=\div\big(\dgatom\cap\grad\hoganom\big),\\
R(\Aos)&=R(\cAos),\text{ i.e.,}
&
\div\dganom&=\div\big(\dganom\cap\eps\grad\hogatom\big),\\
R(\Ats)&=R(\cAts)\subset N(\Aos),\text{ i.e.,}
&
\mu\rot\rganom&=\mu\rot\big(\rganom\cap\rot\rgatom\big)\subset\mu\,\dganzom,\\
R(\Aths)&=R(\cAths)\subset N(\Ats),\text{ i.e.,}
&
\grad\hoganom&=\grad\big(\hoganom\cap\ltom\big)\subset\rganzom,
\end{align*}
where $\ltom$ has to be replaced by $\ltbotom$ for $\gat=\emptyset$ resp. $\gat=\ga$.
Note that the assertions of $R(\Ath)$, $R(\Ats)$, $R(\Aths)$ are already included in those of
$R(\Ao)$, $R(\At)$, $R(\Aos)$ by interchanging $\gat$ and $\gan$ and setting $\eps:=\id$.
Furthermore, the following Friedrichs/Poincar\'e type estimates hold:
\begin{align*}
\forall\,u&\in D(\cAo)=\hogatom\cap\ltom
&
\normltom{u}&\leq c_{\mathsf{fp}}\,\normltepsom{\grad u},\\
\forall\,E&\in D(\cAos)=\mu\,\dganom\cap\grad\hogatom,
&
\normltepsom{E}&\leq c_{\mathsf{fp}}\,\normltom{\div\eps E},\\
\forall\,E&\in D(\cAt)=\rgatom\cap\mu\rot\rganom,
&
\normltepsom{E}&\leq c_{\mathsf{m}}\,\normltom{\rot E},\\
\forall\,E&\in D(\cAts)=\rganom\cap\rot\rgatom,
&
\normltom{E}&\leq c_{\mathsf{m}}\,\normltmuom{\rot E},\\
\forall\,E&\in D(\cAth)=\dgatom\cap\grad\hoganom,
&
\normltom{E}&\leq\ti{c}_{\mathsf{fp}}\,\normltom{\div E},\\
\forall\,u&\in D(\cAths)=\hoganom\cap\ltom
&
\normltom{u}&\leq\ti{c}_{\mathsf{fp}}\,\normltom{\grad u},
\end{align*}
where the Friedrichs/Poincar\'e and Maxwell constants 
$c_{\mathsf{fp}}$, $c_{\mathsf{m}}$, $\ti{c}_{\mathsf{fp}}$, 
are given by the respective Raleigh quotients, and
$\ltom$ has to be replaced by $\ltbotom$ for $\gat=\emptyset$ resp. $\gat=\ga$.
Again note that the latter two assertions are already included in 
the first two inequalities by interchanging $\gat$ and $\gan$ and setting $\eps:=\id$.

\begin{rem}
\label{remconvex}
Let the Friedrichs and the Poincar\'e constants $c_{\mathsf{f}}$, $c_{\mathsf{p}}$ 
as well as upper and lower bounds for the matrix field $\eps$ be defined by
\begin{align*}
\frac{1}{c_{\mathsf{p}}}
&:=\inf_{0\neq\varphi\in\hobotom}\frac{\normltom{\grad\varphi}}{\normltom{\varphi}},
&
\frac{1}{\ovl{\eps}}
&:=\inf_{0\neq\Phi\in\ltom}\frac{\normltom{\Phi}}{\normltepsom{\Phi}},\\
\frac{1}{c_{\mathsf{f}}}
&:=\inf_{0\neq\varphi\in\hogaom}\frac{\normltom{\grad\varphi}}{\normltom{\varphi}},
&
\frac{1}{\unl{\eps}}
&:=\inf_{0\neq\Phi\in\ltom}\frac{\normltepsom{\Phi}}{\normltom{\Phi}}.
\end{align*}
In \cite{paulymaxconst1}, see also \cite{paulymaxconst0,paulymaxconst2}, the following has been proved
for bounded and convex $\om$:
\begin{itemize}
\item[\bf(i)]
If $\gat=\emptyset$ or $\gat=\ga$, then 
$c_{\mathsf{m}}\leq\ovl{\eps}\,c_{\mathsf{p}}\leq\ovl{\eps}\diam\om/\pi$.
\item[\bf(ii)]
If $\gat=\emptyset$ we have
$c_{\mathsf{p}}/\ovl{\eps}\leq c_{\mathsf{fp}}\leq\unl{\eps}c_{\mathsf{p}}$
and $\ti{c}_{\mathsf{fp}}=c_{\mathsf{f}}<c_{\mathsf{p}}$.
\item[\bf(iii)]
If $\gat=\ga$ it holds
$c_{\mathsf{f}}/\ovl{\eps}\leq c_{\mathsf{fp}}\leq\unl{\eps}c_{\mathsf{f}}$ 
and $c_{\mathsf{f}}<c_{\mathsf{p}}=\ti{c}_{\mathsf{fp}}$.
\end{itemize}
\end{rem}

Finally, the following Helmholtz decompositions hold:
\begin{align*}
\Hio=\ltom
&=\begin{cases}
\{0\}&\text{, if }\gat\neq\emptyset,\\
\reals&\text{, if }\gat=\emptyset,
\end{cases}\,
\oplus_{\ltom}
\begin{cases}
\ltom&\text{, if }\gat\neq\emptyset,\\
\ltbotom&\text{, if }\gat=\emptyset,
\end{cases}
&
\big(\Hio&=R(\Az)\oplus_{\Hio}N(\Azs)\big)\\
&=\begin{cases}
\{0\}&\text{, if }\gat\neq\emptyset,\\
\reals&\text{, if }\gat=\emptyset,
\end{cases}\,
\oplus_{\ltom}\div\dganom,
&
\big(\Hio&=N(\Ao)\oplus_{\Hio}R(\Aos)\big)\\
\Hit=\ltepsom&=\grad\hogatom\oplus_{\ltepsom}\mu\,\dganzom
&
\big(\Hit&=R(\Ao)\oplus_{\Hit}N(\Aos)\big)\\
&=\rgatzom\oplus_{\ltepsom}\mu\rot\rganom
&
\big(\Hit&=N(\At)\oplus_{\Hit}R(\Ats)\big)\\
&=\grad\hogatom\oplus_{\ltepsom}\harmgatnepsom\oplus_{\ltepsom}\mu\rot\rganom,
&
\big(\Hit&=R(\Ao)\oplus_{\Hit}K_{2}\oplus_{\Hit}R(\Aos)\big)\\
\Hith=\ltom&=\grad\hoganom\oplus_{\ltom}\dgatzom
&
\big(\Hith&=R(\Aths)\oplus_{\Hith}N(\Ath)\big)\\
&=\rganzom\oplus_{\ltom}\rot\rgatom
&
\big(\Hith&=N(\Ats)\oplus_{\Hith}R(\At)\big)\\
&=\grad\hoganom\oplus_{\ltom}\harmgantom\oplus_{\ltom}\rot\rgatom,
&
\big(\Hith&=R(\Aths)\oplus_{\Hith}K_{3}\oplus_{\Hith}R(\At)\big)\\
\Hif=\ltom
&=\begin{cases}
\{0\}&\text{, if }\gat\neq\ga,\\
\reals&\text{, if }\gat=\ga,
\end{cases}\,
\oplus_{\ltom}
\begin{cases}
\ltom&\text{, if }\gat\neq\ga,\\
\ltbotom&\text{, if }\gat=\ga,
\end{cases}
&
\big(\Hif&=R(\A_{4}^{*})\oplus_{\Hif}N(\A_{4})\big)\\
&=\begin{cases}
\{0\}&\text{, if }\gat\neq\ga,\\
\reals&\text{, if }\gat=\ga,
\end{cases}\,
\oplus_{\ltom}\div\dgatom.
&
\big(\Hif&=N(\Aths)\oplus_{\Hif}R(\Ath)\big)
\end{align*}
The latter two decompositions are already given by the first two ones
by interchanging $\gat$ and $\gan$ and setting $\eps:=\id$.
Especially, it holds
\begin{align*}
\grad\hogatom&=\rgatzom\ominus_{\ltepsom}\harmgatnepsom,
&
\mu\rot\rganom&=\mu\,\dganzom\ominus_{\ltepsom}\harmgatnepsom,\\
\grad\hoganom&=\rganzom\ominus_{\ltom}\harmgantom,
&
\rot\rgatom&=\dgatzom\ominus_{\ltom}\harmgantom.
\end{align*}
If $\gat=\ga$ and $\ga$ is connected, then the Dirichlet fields are trivial, i.e., 
$$\harmgatnepsom=\rgazom\cap\mu\,\dzom=\{0\}.$$
If $\gat=\emptyset$ and $\om$ is simply connected, then the Neumann fields are trivial, i.e., 
$$\harmgatnepsom=\rzom\cap\mu\,\dgazom=\{0\}.$$

Now we can apply the general results of Section \ref{secsoltheovarform} and Section \ref{secfuncapostest}.

\begin{theo}[Theorem \ref{soltheofos}]
\label{soltheoemsys}
\eqref{elmgstat} resp. \eqref{elmgstatpo} is uniquely solvable, if and only if
$$F\in\rot\rgatom=\dgatzom\ominus_{\ltom}\harmgantom,\quad
g\in\ltom,\quad
K\in\harmgatnepsom,$$ 
where $\ltom$ has to be replaced by $\ltbotom$ if $\gat=\emptyset$. 
The unique solution $E\in\rgatom\cap\mu\,\dganom$ is given by 
\begin{align*}
E:=E_{F}+E_{g}+K
&\in\big(\rgatom\cap\mu\rot\rganom\big)\oplus_{\ltepsom}\big(\mu\,\dganom\cap\grad\hogatom\big)\oplus_{\ltepsom}\harmgatnepsom\\
&=\rgatom\cap\mu\,\dganom,\\
E_{F}:=(\widetilde\rot_{\gat})^{-1}F
&\in\rgatom\cap\mu\rot\rganom
=\rgatom\cap\mu\,\dganzom\cap\harmgatnepsom^{\bot_{\ltepsom}},\\
E_{g}:=-(\widetilde\div_{\gan}\eps)^{-1}g
&\in\mu\,\dganom\cap\grad\hogatom
=\mu\,\dganom\cap\rgatzom\cap\harmgatnepsom^{\bot_{\ltepsom}}
\end{align*}
and depends continuously on the data, i.e.,
$\norm{E}_{\ltepsom}
\leq c_{\mathsf{m}}\,\norm{F}_{\ltom}
+c_{\mathsf{fp}}\,\norm{g}_{\ltom}
+\norm{K}_{\ltom}$,
as 
$$\norm{E_{F}}_{\ltepsom}\leq c_{\mathsf{m}}\,\norm{F}_{\ltom},\qquad
\norm{E_{g}}_{\ltepsom}\leq c_{\mathsf{fp}}\,\norm{g}_{\ltom}.$$
Moreover, $\norm{E}_{\ltepsom}^2=\norm{E_{F}}_{\ltepsom}^2+\norm{E_{g}}_{\ltepsom}^2+\norm{K}_{\ltepsom}^2$.
\end{theo}

The partial solutions $E_{F}$ and $E_{g}$, solving 
\begin{align*}
\rot_{\gat}E_{F}&=F,
&
\rot_{\gat}E_{g}&=0,\\
-\div_{\gan}\eps E_{F}&=0,
&
-\div_{\gan}\eps E_{g}&=g,\\
\pi_{\harm}E_{F}&=0,
&
\pi_{\harm}E_{g}&=0,
\end{align*}
can be found and computed by the following 
four variational formulations: 

\begin{theo}[Theorem \ref{soltheofosvarform}]
\label{soltheofosvarformappfos}
The partial solutions $E_{F}$ and $E_{g}$ in Theorem \ref{soltheoemsys}
can be found by the following four variational formulations: 
\begin{itemize}
\item[\bf(i)]
There exists a unique 
$\ti{E}_{F}\in\rgatom\cap\mu\rot\rganom$ such that
\begin{align}
\mylabel{varxfappfos}
\forall\,\Phi\in\rgatom\cap\mu\rot\rganom\qquad
\scpltom{\rot\ti{E}_{F}}{\rot\Phi}
=\scpltom{F}{\rot\Phi}.
\end{align}
\eqref{varxfappfos} is even satisfied for all $\Phi\in\rgatom$.
Moreover, the equation $\rot\ti{E}_{F}=F$ holds if and only if $F\in\rot\rgatom$.
In this case $\ti{E}_{F}=E_{F}$.
\item[\bf(i')]
There exists a unique potential 
$H_{F}\in\rganom\cap\rot\rgatom$ such that
\begin{align}
\mylabel{varyfappfos}
\forall\,\Psi\in\rganom\cap\rot\rgatom\qquad
\scpltom{\mu\rot H_{F}}{\rot\Psi}
=\scpltom{F}{\Psi}.
\end{align}
\eqref{varyfappfos} even holds for all $\Psi\in\rganom$ 
if and only if $F\in\rot\rgatom$.
In this case we have 
$$\mu\rot H_{F}\in\rgatom\cap\mu\rot\rganom$$ 
with $\rot\mu\rot H_{F}=F$ and hence $\mu\rot H_{F}=E_{F}$.
\item[\bf(ii)]
Let $\gat\neq\emptyset$.
There is a unique 
$\ti{E}_{g}\in\mu\,\dganom\cap\grad\hogatom$ such that
\begin{align}
\mylabel{varxgappfos}
\forall\,\Theta\in\mu\,\dganom\cap\grad\hogatom\qquad
\scpltom{\div\eps\ti{E}_{g}}{\div\eps\Theta}
=-\scpltom{g}{\div\eps\Theta}.
\end{align}
\eqref{varxgappfos} is even satisfied for all $\Theta\in\mu\,\dganom$.
Moreover, $-\div\eps\ti{E}_{g}=g$ and $\ti{E}_{g}=E_{g}$.
In the case $\gat=\emptyset$ the condition 
$g\in\ltbotom$ has to be added, i.e., 
$-\div\eps\ti{E}_{g}=g$ if and only if $g\in\ltbotom$ 
and in this case $\ti{E}_{g}=E_{g}$.
\item[\bf(ii')]
Let $\gat\neq\emptyset$.
There exists a unique potential $u_{g}\in\hogatom$ such that
\begin{align}
\mylabel{varzgappfos}
\forall\,\psi\in\hogatom\qquad
\scpltom{\eps\grad u_{g}}{\grad\psi}
=\scpltom{g}{\psi}.
\end{align}
It holds 
$$\grad u_{g}\in\mu\,\dganom\cap\grad\hogatom$$ 
with $-\div\eps\grad u_{g}=g$ and thus $\grad u_{g}=E_{g}$.
In the case $\gat=\emptyset$ we replace $\hogatom$
by $\hobotom$. Then
\eqref{varzgappfos} even holds for all $\psi\in\hoom$
if and only if $g\in\ltbotom$.
In this case the other assertions hold as stated before. 
\end{itemize}
\end{theo}

\begin{rem}[Remark \ref{soltheofosvarformremmatrix}]
\label{soltheofosvarformappfosremone}
Let us note the following:
\begin{itemize}
\item[\bf(i)]
It holds $\grad\hogatom=\rgatzom\cap\harmgatnepsom^{\bot_{\ltepsom}}$ and
\begin{align*}
\mu\rot\rganom
=\mu\,\dganzom\cap\harmgatnepsom^{\bot_{\ltepsom}},\qquad
\rot\rgatom
=\dgatzom\cap\harmgantom^{\bot_{\ltom}}.
\end{align*}
\item[\bf(ii)]
We have
\begin{align*}
E_{F}
=(\widetilde\rot_{\gat})^{-1}F
&\in D(\widetilde\rot_{\gat})
=\rgatom\cap\mu\rot\rganom,\\
H_{F}
=(\mu\,\widetilde\rot_{\gan})^{-1}E_{F}
=(\mu\,\widetilde\rot_{\gan})^{-1}(\widetilde\rot_{\gat})^{-1}F
&\in D(\widetilde\rot_{\gat}\mu\,\widetilde\rot_{\gan})
\subset\rganom\cap\rot\rgatom,\\
E_{g}
=-(\widetilde\div_{\gan}\eps)^{-1}g
&\in D(\widetilde\div_{\gan}\eps)=\mu\,\dganom\cap\grad\hogatom,\\
u_{g}
=(\widetilde\grad_{\gat})^{-1}E_{g}
=-(\widetilde\grad_{\gat})^{-1}(\widetilde\div_{\gan}\eps)^{-1}g
&\in D(\widetilde\div_{\gan}\eps\,\widetilde\grad_{\gat})
\subset\hogatom,
\end{align*}
and these vector fields and functions solve
\begin{align*}
\rot_{\gat}E_{F}&=F,
&
\rot_{\gat}\mu\rot_{\gan}H_{F}&=F,
&
-\div_{\gan}\eps E_{g}&=g,
&
-\div_{\gan}\eps\grad_{\gat}u_{g}&=g,\\
-\div_{\gan}\eps E_{F}&=0,
&
\div_{\gat}H_{F}&=0,
&
\rot_{\gat}E_{g}&=0,
&
\pi_{\{0\}/\reals}u_{g}&=0,\\
\pi_{\harm}E_{F}&=0,
&
\pi_{\ti{\mathcal{H}}}H_{F}&=0,
&
\pi_{\harm}E_{g}&=0,
\end{align*}
where $\pi_{\ti{\mathcal{H}}}:\ltom\to\harmgantom$
is the Neumann-Dirichlet orthonormal projector
and $\pi_{\{0\}/\reals}$ denotes $\pi_{\{0\}}$
or $\pi_{\reals}$ if $\gat=\emptyset$.
Mooreover, \eqref{varxfappfos}-\eqref{varzgappfos} are weak formulations of
\begin{align*}
\mu\rot_{\gan}\rot_{\gat}\ti{E}_{F}&=\mu\rot_{\gan}F,
&
-\div_{\gan}\eps\ti{E}_{F}&=0,
&
\pi_{\harm}\ti{E}_{F}&=0,\\
\rot_{\gat}\mu\rot_{\gan}H_{F}&=F,
&
\div_{\gat}H_{F}&=0,
&
\pi_{\ti{\mathcal{H}}}H_{F}&=0,\\
-\grad_{\gat}\div_{\gan}\eps\ti{E}_{g}&=\grad_{\gat}g,
&
\rot_{\gat}\ti{E}_{g}&=0,
&
\pi_{\harm}\ti{E}_{g}&=0,\\
-\div_{\gan}\eps\grad_{\gat}u_{g}&=g,
&
\pi_{\{0\}/\reals} u_{g}&=0,
\end{align*}
i.e., in formal matrix notation
\begin{align*}
\begin{bmatrix}
\mu\rot_{\gan}\rot_{\gat}\\
-\div_{\gan}\eps\\
\pi_{\harm}
\end{bmatrix}
\begin{bmatrix}
\ti{E}_{F}
\end{bmatrix}
&=\begin{bmatrix}
\mu\rot_{\gan}F\\
0\\
0
\end{bmatrix},
&
\begin{bmatrix}
\rot_{\gat}\mu\rot_{\gan}\\
\div_{\gat}\\
\pi_{\ti{\mathcal{H}}}
\end{bmatrix}
\begin{bmatrix}
H_{F}
\end{bmatrix}
&=\begin{bmatrix}
F\\
0\\
0
\end{bmatrix},\\
\begin{bmatrix}
-\grad_{\gat}\div_{\gan}\eps\\
\rot_{\gat}\\
\pi_{\harm}
\end{bmatrix}
\begin{bmatrix}
\ti{E}_{g}
\end{bmatrix}
&=\begin{bmatrix}
\grad_{\gat}g\\
0\\
0
\end{bmatrix},
&
\begin{bmatrix}
-\div_{\gan}\eps\grad_{\gat}\\
\pi_{\{0\}/\reals}
\end{bmatrix}
\begin{bmatrix}
u_{g}
\end{bmatrix}
&=\begin{bmatrix}
g\\
0
\end{bmatrix}.
\end{align*}
\end{itemize}
\end{rem}

\begin{rem}[Remark \ref{solremfosvarform}]
\label{soltheofosvarformappfosremtwo}
Let us note the following, especially for possible numerical purposes and applications.
\begin{itemize}
\item[\bf(i)]
Using the variational formulation in Theorem \ref{soltheofosvarformappfos} (i)
corresponding to $E_{F}=\ti{E}_{F}\in\rgatom$
for finding a numerical (discrete) approximation $E_{F,h}$ of
$E_{F}$ proposes a $\rgatom$-conforming method 
in some finite dimensional (discrete) subspace $\rsymbol_{\gat,h}(\om)$ of $\rgatom$
giving also a $\rgatom$-conforming discrete solution $E_{F,h}\in\rsymbol_{\gat,h}(\om)\subset\rgatom$.
\item[\bf(i')]
Utilizing the variational formulation in Theorem \ref{soltheofosvarformappfos} (i')
for $E_{F}=\mu\rot H_{F}\in\mu\rot\rganom$ 
to find a discrete approximation $E_{F,h}=\mu\rot H_{F,h}$ of $E_{F}$
proposes a $\rganom$-conforming method
in some discrete subspace $\rsymbol_{\gan,h}(\om)$ of $\rganom$
giving then a $\rganom$-conforming discrete potential $H_{F,h}\in\rsymbol_{\gan,h}(\om)\subset\rganom$,
but yielding a $\mu\,\dganom$-conforming solution as
$$E_{F,h}=\mu\rot H_{F,h}\in
\mu\rot\rganom
=\mu\,\dganzom\cap\harmgatnepsom^{\bot_{\ltepsom}}
\subset\mu\,\dganom.$$
\item[\bf(ii)]
Using the variational formulation in Theorem \ref{soltheofosvarformappfos} (ii)
corresponding to $E_{g}=\ti{E}_{g}\in\mu\,\dganom$
for finding a discrete approximation $E_{g,h}$ of
$E_{g}$ proposes a $\mu\,\dganom$-conforming method 
in some discrete subspace $\mu\,\dsymbol_{\gan,h}(\om)$ of $\mu\,\dganom$
giving also a $\mu\,\dganom$-conforming discrete solution 
$E_{g,h}\in\mu\,\dsymbol_{\gan,h}(\om)\subset\mu\,\dganom$.
\item[\bf(ii')]
Utilizing the variational formulation in Theorem \ref{soltheofosvarformappfos} (ii')
for $E_{g}=\grad u_{g}\in\grad\hogatom$ 
to find a discrete approximation $E_{g,h}=\grad u_{g,h}$ of $E_{g}$
proposes a $\hogatom$-conforming method
in some discrete subspace $\hsymbol^{1}_{\gat,h}(\om)$ of $\hogatom$
giving then a $\hogatom$-conforming discrete potential $u_{g,h}\in\hsymbol^{1}_{\gat,h}(\om)\subset\hogatom$,
but yielding a $\rgatom$-conforming solution as
$$E_{g,h}=\grad u_{g,h}\in
\grad\hogatom
=\rgatzom\cap\harmgatnepsom^{\bot_{\ltepsom}}
\subset\rgatom.$$
\item[\bf(iii)]
A possible discrete solution $E_{F,h}=\mu\rot H_{F,h}$ from (ii') satisfies 
automatically the side conditions 
$$-\div_{\gan}\eps E_{F,h}=0,\qquad
\pi_{\harm}E_{F,h}=0,$$
i.e., even on the discrete level there is no error in the side conditions.
The other option from (ii) yields a discrete solution $E_{F,h}$, which
most probably has got errors in the side conditions.
\item[\bf(iii')]
A possible discrete solution $E_{g,h}=\grad u_{g,h}$ from (iii') satisfies 
automatically the side conditions 
$$\rot_{\gat}E_{g,h}=0,\qquad
\pi_{\harm}E_{g,h}=0,$$
i.e., even on the discrete level there is no error in the side conditions.
The other option from (iii) yields a discrete solution $E_{g,h}$, which
most probably has got errors in the side conditions.
\end{itemize}
\end{rem}

\begin{theo}[Theorem \ref{soltheofosvarformtogetherwithK2}]
\label{soltheofosvarformtogetherwithK2appfos}
The unique solution $E=E_{F}+E_{g}+K\in\rgatom\cap\mu\,\dganom$ in Theorem \ref{soltheoemsys}
can be found by the following two variational double saddle point formulations: 
\begin{itemize}
\item[\bf(i)]
Let $\gat\neq\emptyset$.
There exists a unique tripple
$(\ti{E},u,H)\in\rgatom\times\hogatom\times\harmgatnepsom$ such that for all
$(\Phi,\varphi,\Theta)\in\rgatom\times\hogatom\times\harmgatnepsom$
\begin{align}
\nonumber
\scpltom{\rot\ti{E}}{\rot\Phi}
+\scpltom{\eps\grad u}{\Phi}
+\scpltom{\eps H}{\Phi}
&=\scpltom{F}{\rot\Phi},\\
\mylabel{vartiEappfos}
\scpltom{\eps\ti{E}}{\grad\varphi}
&=\scpltom{g}{\varphi},\\
\nonumber
\scpltom{\eps\ti{E}}{\Theta}
&=\scpltom{\eps K}{\Theta}.
\end{align}
It holds $u=0$ and $H=0$. $\rot\ti{E}=F$ if and only if $F\in\rot\rgatom$.
Moreover, $\eps\ti{E}\in\dganom$ and $-\div\eps\ti{E}=g$ as well as $\pi_{\harm}\ti{E}=K$.
In this case, i.e., $F\in\rot\rgatom$, we have $\ti{E}=E$ from Theorem \ref{soltheoemsys}.
If $\gat=\emptyset$, we have to replace $\hogatom$ by $\hobotom$. 
Then \eqref{vartiEappfos} even holds for all $\varphi\in\hoom$
if and only if $g\in\ltbotom$ if and only if
$\eps\ti{E}\in\dgaom$ and $-\div\eps\ti{E}=g$. Furthermore, $\pi_{\harm}\ti{E}=K$.
In this case, i.e., $F\in\rot\rom$ and $g\in\ltbotom$, we have $\ti{E}=E$ from Theorem \ref{soltheoemsys}.
\item[\bf(ii)]
Let $\gat\neq\emptyset$.
There exists a unique tripple
$(\hat{E},U,H)\in\mu\,\dganom\times\big(\rganom\cap\rot\rgatom\big)\times\harmgatnepsom$ such that for all
$(\Psi,\Phi,\Theta)\in\mu\,\dganom\times\big(\rganom\cap\rot\rgatom\big)\times\harmgatnepsom$
\begin{align}
\nonumber
\scpltom{\div\eps\hat{E}}{\div\eps\Psi}
+\scpltom{\rot U}{\Psi}
+\scpltom{\eps H}{\Psi}
&=-\scpltom{g}{\div\eps\Psi},\\
\mylabel{varhatEappfos}
\scpltom{\hat{E}}{\rot\Phi}
&=\scpltom{F}{\Phi},\\
\nonumber
\scpltom{\eps\hat{E}}{\Theta}
&=\scpltom{\eps K}{\Theta}.
\end{align}
It holds $U=0$ and $H=0$ as well as $-\div\eps\hat{E}=g$.
\eqref{varhatEappfos} holds for all $\Phi\in\rganom$ 
if and only if $F\in\rot\rgatom$
if and only if $\hat{E}\in\rgatom$ with $\rot\hat{E}=F$.
Moreover, $\pi_{\harm}\hat{E}=K$.
In this case, i.e., $F\in\rot\rgatom$, we have $\hat{E}=E$ from Theorem \ref{soltheoemsys}.
If $\gat=\emptyset$, the condition 
$g\in\ltbotom$ has to be added, i.e., 
$-\div\eps\hat{E}=g$ if and only if $g\in\ltbotom$.
In this case, i.e., $F\in\rot\rom$ and $g\in\ltbotom$, we have $\ti{E}=E$ from Theorem \ref{soltheoemsys}.
\end{itemize}
\end{theo}

\begin{rem}[Remark \ref{soltheofosvarformtogetherrem}]
\label{soltheofosvarformtogetherwithK2appfosrem}
Let us note the following:
\begin{itemize}
\item[\bf(i)]
Using the saddle point formulation in Theorem \ref{soltheofosvarformtogetherwithK2appfos} (i)
for finding a numerical approximation $E_{h}$ of $E$
provides a $\rgatom$-conforming approximation $E_{h}\in\rgatom$ of \eqref{elmgstat} or \eqref{elmgstatpo},
whereas using the saddle point formulation in Theorem \ref{soltheofosvarformtogetherwithK2appfos} (ii)
for finding a numerical approximation $E_{h}$ of $E$
provides a $\mu\,\dganom$-conforming approximation $E_{h}\in\mu\,\dganom$ of of \eqref{elmgstat} or \eqref{elmgstatpo}.
\item[\bf(ii)]
Related variational formulations to those presented in Theorem \ref{soltheofosvarformtogetherwithK2appfos}
have recently been announced and proposed in \cite{alonsorodriguezbertolazzivalli2017}.
\item[\bf(iii)]
\eqref{vartiEappfos} and \eqref{varhatEappfos} are weak formulations of
\begin{align*}
\mu\rot_{\gan}\rot_{\gat}\ti{E}
+\grad_{\gat}u
+H
&=\mu\rot_{\gan}F,
&
-\div_{\gan}\eps\ti{E}
&=g,
&
\pi_{\harm}\ti{E}
&=K,\\
-\grad_{\gat}\div_{\gan}\eps\hat{E}
+\mu\rot_{\gan}U
+H
&=\grad_{\gat}g,
&
\rot_{\gat}\hat{E}
&=F,
&
\pi_{\harm}\hat{E}
&=K,
\end{align*}
i.e., in formal matrix notation
\begin{align*}
\begin{bmatrix}
\mu\rot_{\gan}\rot_{\gat} & \grad_{\gat} & \iota_{\harm}\\
-\div_{\gan}\eps & 0 & 0\\
\pi_{\harm} & 0 & 0
\end{bmatrix}
\begin{bmatrix}
\ti{E}\\
u\\
H
\end{bmatrix}
&=\begin{bmatrix}
\mu\rot_{\gan}F\\
g\\
K
\end{bmatrix},
&
\begin{bmatrix}
-\grad_{\gat}\div_{\gan}\eps & \mu\rot_{\gan} & \iota_{\harm}\\
\rot_{\gat} & 0 & 0\\
\pi_{\harm} & 0 & 0
\end{bmatrix}
\begin{bmatrix}
\hat{E}\\
U\\
H
\end{bmatrix}
&=\begin{bmatrix}
\grad_{\gat}g\\
F\\
K
\end{bmatrix}.
\end{align*}
\end{itemize}
\end{rem}

\begin{theo}[Theorem \ref{soltheofosvarformpartsolfull}]
\label{soltheofosvarformpartsolfullappfos}
The partial solution vector fields $E_{F}=\ti{E}_{F}\in\rgatom\cap\mu\rot\rganom$
and $E_{g}=\ti{E}_{g}\in\mu\,\dganom\cap\grad\hogatom$ together with their potentials
$H_{F}\in\rganom\cap\rot\rgatom$, $u_{g}\in\hogatom$ resp. $u_{g}\in\hobotom$
from Theorem \ref{soltheoemsys} and Theorem \ref{soltheofosvarformappfos}
can be found by the following four variational double saddle point formulations: 
\begin{itemize}
\item[\bf(i)]
Let $\gat\neq\emptyset$.
There exists a unique tripple
$(\ti{E}_{F},u,H)\in\rgatom\times\hogatom\times\harmgatnepsom$ such that for all
$(\Phi,\varphi,\Theta)\in\rgatom\times\hogatom\times\harmgatnepsom$
\begin{align}
\nonumber
\scpltom{\rot\ti{E}_{F}}{\rot\Phi}
+\scpltom{\eps\grad u}{\Phi}
+\scpltom{\eps H}{\Phi}
&=\scpltom{F}{\rot\Phi},\\
\mylabel{vartiEFappfos}
\scpltom{\eps\ti{E}_{F}}{\grad\varphi}
&=0,\\
\nonumber
\scpltom{\eps\ti{E}_{F}}{\Theta}
&=0.
\end{align}
It holds $u=0$ and $H=0$. $\rot\ti{E}_{F}=F$ if and only if $F\in\rot\rgatom$.
Moreover, $\eps\ti{E}_{F}\in\dganzom$ and $\pi_{\harm}\ti{E}_{F}=0$.
Hence, if $F\in\rot\rgatom$, we have $\ti{E}_{F}=E_{F}$ 
from Theorem \ref{soltheoemsys}, see Theorem \ref{soltheofosvarformappfos} (i).
If $\gat=\emptyset$, we have to replace $\hogatom$ by $\hobotom$. 
Then \eqref{vartiEFappfos} even holds for all $\varphi\in\hoom$
and thus $\eps\ti{E}_{F}\in\dgazom$. Furthermore, $\pi_{\harm}\ti{E}_{F}=0$.
Again, if $F\in\rot\rom$, we have $\ti{E}_{F}=E_{F}$ from Theorem \ref{soltheoemsys}.
\item[\bf(i')]
Let $\gat\neq\ga$.
There exists a unique tripple
$(H_{F},v,H)\in\rganom\times\hoganom\times\harmgantom$ such that for all
$(\Psi,\psi,\Theta)\in\rganom\times\hoganom\times\harmgantom$
\begin{align}
\nonumber
\scpltom{\mu\rot H_{F}}{\rot\Psi}
-\scpltom{\grad v}{\Psi}
+\scpltom{H}{\Psi}
&=\scpltom{F}{\Psi},\\
\mylabel{varHFappfos}
\scpltom{H_{F}}{\grad\psi}
&=0,\\
\nonumber
\scpltom{H_{F}}{\Theta}
&=0.
\end{align}
It holds $v=0$ if and only if $F\bot\grad\hoganom$
if and only if $F\in\dgatzom$. 
$H=0$ if and only if $F\bot\harmgantom$.
Thus $v=0$ and $H=0$ if and only if $F\in\dgatzom\cap\harmgantom^{\bot_{\ltom}}=\rot\rgatom$.
Moreover, $\mu\rot H_{F}\in\rgatom$ and $\rot\mu\rot H_{F}=F$
if and only if $F\in\rot\rgatom$. Furthermore, $H_{F}\in\dgatzom$ and $\pi_{\ti{\mathcal{H}}}H_{F}=0$.
Hence, if $F\in\rot\rgatom$, we have $\mu\rot H_{F}=E_{F}$
from Theorem \ref{soltheoemsys}, see Theorem \ref{soltheofosvarformappfos} (i').
If $\gat=\ga$, we have to replace $\hoganom$ by $\hobotom$. 
Then \eqref{varHFappfos} even holds for all $\psi\in\hoom$
and thus $H_{F}\in\dgazom$. Furthermore, $\pi_{\ti{\mathcal{H}}}H_{F}=0$.
Again, if $F\in\rot\rgaom$, we have $\mu\rot H_{F}=E_{F}$ from Theorem \ref{soltheoemsys}.
\item[\bf(ii)]
Let $\gat\neq\emptyset$.
There exists a unique tripple
$(\ti{E}_{g},U,H)\in\mu\,\dganom\times\big(\rganom\cap\rot\rgatom\big)\times\harmgatnepsom$ such that for all
$(\Psi,\Phi,\Theta)\in\mu\,\dganom\times\big(\rganom\cap\rot\rgatom\big)\times\harmgatnepsom$
\begin{align}
\nonumber
\scpltom{\div\eps\ti{E}_{g}}{\div\eps\Psi}
+\scpltom{\rot U}{\Psi}
+\scpltom{\eps H}{\Psi}
&=-\scpltom{g}{\div\eps\Psi},\\
\mylabel{vartiEgappfos}
\scpltom{\ti{E}_{g}}{\rot\Phi}
&=0,\\
\nonumber
\scpltom{\eps\ti{E}_{g}}{\Theta}
&=0.
\end{align}
It holds $U=0$ and $H=0$ as well as $-\div\eps\ti{E}_{g}=g$.
\eqref{vartiEgappfos} holds for all $\Phi\in\rganom$ 
and hence $\ti{E}_{g}\in\rgatzom$.
Moreover, $\pi_{\harm}\ti{E}_{g}=0$.
Finally, we have $\ti{E}_{g}=E_{g}$ from Theorem \ref{soltheoemsys}, see Theorem \ref{soltheofosvarformappfos} (ii).
If $\gat=\emptyset$, the condition 
$g\in\ltbotom$ has to be added, i.e., 
$-\div\eps\ti{E}_{g}=g$ if and only if $g\in\ltbotom$.
Again, \eqref{vartiEgappfos} holds for all $\Phi\in\rgaom$, 
$\ti{E}_{g}\in\rzom$, and $\pi_{\harm}\ti{E}_{g}=0$.
Finally, if $g\in\ltbotom$, we have $\ti{E}_{g}=E_{g}$ from Theorem \ref{soltheoemsys}.
\item[\bf(ii')]
For $\gat\neq\emptyset$ see Theorem \ref{soltheofosvarformappfos} (ii').
Let $\gat=\emptyset$.
There exists a unique pair $(u_{g},r)\in\hoom\times\reals$ such that
\begin{align}
\mylabel{varugappfos}
\begin{split}
\forall\,(\psi,\varrho)\in\hoom\times\reals\qquad
\scpltom{\eps\grad u_{g}}{\grad\psi}
+\scpltom{\iota_{\reals}r}{\psi}
&=\scpltom{g}{\psi},\\
\scpltom{u_{g}}{\iota_{\reals}\varrho}
&=0.
\end{split}
\end{align}
It holds $r=0$ if and only if $g\bot_{\ltom}\iota_{\reals}\reals=\reals$
if and only if $g\in\ltbotom=\div\dgaom$.
Moreover, $\grad u_{g}\in\mu\,\dgaom$
with $-\div\eps\grad u_{g}=g$ 
if and only if $g\in\ltbotom$.
The second equation of \eqref{varugappfos} shows $u_{g}\in\ltbotom$,
i.e., $u_{g}\in\hobotom$.
Finally, if $g\in\ltbotom$, we have $\grad u_{g}=E_{g}$
from Theorem \ref{soltheoemsys}, see Theorem \ref{soltheofosvarformappfos} (ii').
\end{itemize}
\end{theo}

\begin{rem}[Remark \ref{soltheofosvarformpartsolfullremmatrix}]
\label{soltheofosvarformpartsolfullappfosrem}
\eqref{vartiEFappfos}-\eqref{varugappfos} are weak formulations of
\begin{align*}
\mu\rot_{\gan}\rot_{\gat}\ti{E}_{F}
+\grad_{\gat}u
+H
&=\mu\rot_{\gan}F,
&
-\div_{\gan}\eps\ti{E}_{F}
&=0,
&
\pi_{\harm}\ti{E}_{F}
&=0,\\
\rot_{\gat}\mu\rot_{\gan}H_{F}
-\grad_{\gan}v
+H
&=F,
&
\div_{\gat}H_{F}
&=0,
&
\pi_{\ti{\mathcal{H}}}H_{F}
&=0,\\
-\grad_{\gat}\div_{\gan}\eps\ti{E}_{g}
+\mu\rot_{\gan}U
+H
&=\grad_{\gat}g,
&
\rot_{\gat}\ti{E}_{g}
&=0,
&
\pi_{\harm}\ti{E}_{g}
&=0,\\
-\div_{\ga}\eps\grad_{\emptyset}u_{g}
+\iota_{\reals}r
&=g,
&
\pi_{\reals}u_{g}
&=0,
\end{align*}
i.e., in formal matrix notation
\begin{align*}
\begin{bmatrix}
\mu\rot_{\gan}\rot_{\gat} & \grad_{\gat} & \iota_{\harm}\\
-\div_{\gan}\eps & 0 & 0\\
\pi_{\harm} & 0 & 0
\end{bmatrix}
\begin{bmatrix}
\ti{E}_{F}\\
u\\
H
\end{bmatrix}
&=\begin{bmatrix}
\mu\rot_{\gan}F\\
0\\
0
\end{bmatrix},
&
\begin{bmatrix}
\rot_{\gat}\mu\rot_{\gan} & -\grad_{\gat} & \iota_{\ti{\mathcal{H}}}\\
\div_{\gat} & 0 & 0\\
\pi_{\ti{\mathcal{H}}} & 0 & 0
\end{bmatrix}
\begin{bmatrix}
H_{F}\\
v\\
H
\end{bmatrix}
&=\begin{bmatrix}
F\\
0\\
0
\end{bmatrix},\\
\begin{bmatrix}
-\grad_{\gat}\div_{\gan}\eps & \mu\rot_{\gan} & \iota_{\harm}\\
\rot_{\gat} & 0 & 0\\
\pi_{\harm} & 0 & 0
\end{bmatrix}
\begin{bmatrix}
\ti{E}_{g}\\
U\\
H
\end{bmatrix}
&=\begin{bmatrix}
\grad_{\gat}g\\
0\\
0
\end{bmatrix},
&
\begin{bmatrix}
-\div_{\ga}\eps\grad_{\emptyset} & \iota_{\reals}\\
\pi_{\reals} & 0
\end{bmatrix}
\begin{bmatrix}
u_{g}\\
r
\end{bmatrix}
&=\begin{bmatrix}
g\\
0
\end{bmatrix}.
\end{align*}
\end{rem}

\begin{theo}[Theorem \ref{soltheofosvarformtogetherwithK2andmore}]
\label{soltheofosvarformtogetherwithK2andmoreappfos}
Let $F\in\rot\rgatom$ and $g\in\ltom$. If $\gat=\emptyset$, let $g\in\ltbotom$.
The unique solution $E=E_{F}+E_{g}+K\in\rgatom\cap\mu\,\dganom$ in Theorem \ref{soltheoemsys}
can be found by the following three variational multiple saddle point formulations: 
\begin{itemize}
\item[\bf(i)]
For $\gat\neq\emptyset$ see Theorem \ref{soltheofosvarformtogetherwithK2appfos} (i).
Let $\gat=\emptyset$.
There is
$(\ti{E},u,r,H)\in\rom\times\hoom\times\reals\times\harmgatnepsom$,
a unique quadruple, such that for all
$(\Phi,\varphi,\varrho,\Theta)\in\rom\times\hoom\times\reals\times\harmgatnepsom$
\begin{align}
\mylabel{vartiEurappfos}
\begin{split}
\scpltom{\rot\ti{E}}{\rot\Phi}
+\scpltom{\eps\grad u}{\Phi}
+\scpltom{\eps H}{\Phi}
&=\scpltom{F}{\rot\Phi},\\
\scpltom{\eps\ti{E}}{\grad\varphi}
+\scpltom{\iota_{\reals}r}{\varphi}
&=\scpltom{g}{\varphi},\\
\scpltom{u}{\iota_{\reals}\varrho}
&=0,\\
\scpltom{\eps\ti{E}}{\Theta}
&=\scpltom{\eps K}{\Theta}.
\end{split}
\end{align}
It holds $u=0$, $H=0$, and $r=0$.
Moreover, $\rot\ti{E}=F$ and $\eps\ti{E}\in\dgaom$ with 
$-\div\eps\ti{E}=g$ as well as $\pi_{\harm}\ti{E}=K$.
Finally, $\ti{E}=E$ from Theorem \ref{soltheoemsys}.
\item[\bf(ii)]
Let $\gat\neq\ga$.
There is
$(\hat{E},U,v,H,\ti{H})\in\mu\,\dganom\times\rganom\times\hoganom\times\harmgatnepsom\times\harmgantom$,
a unique five tuple, such that for all
$(\Psi,\Phi,\psi,\Theta,\ti{\Theta})\in\mu\,\dganom\times\rganom\times\hoganom\times\harmgatnepsom\times\harmgantom$
\begin{align}
\nonumber
\scpltom{\div\eps\hat{E}}{\div\eps\Psi}
+\scpltom{\rot U}{\Psi}
+\scpltom{\eps H}{\Psi}
&=-\scpltom{g}{\div\eps\Psi},\\
\nonumber
\scpltom{\hat{E}}{\rot\Phi}
-\scpltom{\grad v}{\Phi}
+\scpltom{\ti{H}}{\Phi}
&=\scpltom{F}{\Phi},\\
\mylabel{varhatEUvappfos}
-\scpltom{U}{\grad\psi}
&=0,\\
\nonumber
\scpltom{\eps\hat{E}}{\Theta}
&=\scpltom{\eps K}{\Theta},\\
\nonumber
\scpltom{U}{\ti{\Theta}}
&=0.
\end{align}
It holds $U=0$, $H=0$ and $v=0$, $\ti{H}=0$.
Moreover, $-\div\eps\hat{E}=g$ and $\hat{E}\in\rgatom$ with $\rot\hat{E}=F$
as well as $\pi_{\harm}\hat{E}=K$.
Finally, $\hat{E}=E$ from Theorem \ref{soltheoemsys}.
If $\gat=\ga$, we have to replace $\hoganom$ by $\hobotom$
and the assertions hold as before. 
\item[\bf(ii')]
Let $\gat=\ga$.
There is
$(\hat{E},U,v,r,H,\ti{H})\in\mu\,\dom\times\rom\times\hoom\times\reals\times\harmgatnepsom\times\harmgantom$,
a unique six tuple, such that for all
$(\Psi,\Phi,\psi,\varrho\Theta,\ti{\Theta})\in\mu\,\dom\times\rom\times\hoom\times\reals\times\harmgatnepsom\times\harmgantom$
\begin{align}
\mylabel{varhatEUvrappfos}
\begin{split}
\scpltom{\div\eps\hat{E}}{\div\eps\Psi}
+\scpltom{\rot U}{\Psi}
+\scpltom{\eps H}{\Psi}
&=-\scpltom{g}{\div\eps\Psi},\\
\scpltom{\hat{E}}{\rot\Phi}
-\scpltom{\grad v}{\Phi}
+\scpltom{\ti{H}}{\Phi}
&=\scpltom{F}{\Phi},\\
-\scpltom{U}{\grad\psi}
+\scpltom{\iota_{\reals}r}{\psi}
&=0,\\
\scpltom{v}{\iota_{\reals}\varrho}
&=0,\\
\scpltom{\eps\hat{E}}{\Theta}
&=\scpltom{\eps K}{\Theta},\\
\scpltom{U}{\ti{\Theta}}
&=0.
\end{split}
\end{align}
It holds $U=0$, $H=0$ and $v=0$, $\ti{H}=0$
as well as $r=0$.
Moreover, $-\div\eps\hat{E}=g$ and $\hat{E}\in\rgaom$ with $\rot\hat{E}=F$
as well as $\pi_{\harm}\hat{E}=K$.
Finally, $\hat{E}=E$ from Theorem \ref{soltheoemsys}.
\end{itemize}
\end{theo}

Theorem \ref{soltheofosvarformpartsolfullappfos}
can be extended in the same way.

\begin{rem}[Remark \ref{soltheofosvarformtogetherwithK2andmorerem}]
\label{soltheofosvarformtogetherwithK2andmoreappfosrem}
\eqref{vartiEurappfos}-\eqref{varhatEUvrappfos} are weak formulations of
{\small\begin{align*}
\mu\rot_{\ga}\rot\ti{E}
+\grad u
+H
&=\mu\rot_{\ga}F,
&
-\div_{\ga}\eps\ti{E}
+\iota_{\reals}r
&=g,
&
\pi_{\reals}u
&=0,\\
-\grad_{\gat}\div_{\gan}\eps\hat{E}
+\mu\rot_{\gan}U
+H
&=\grad_{\gat}g,
&
\rot_{\gat}\hat{E}
-\grad_{\gan}v
+\ti{H}
&=F,
&
\div_{\gat}U
&=0,\\
-\grad_{\ga}\div\eps\hat{E}
+\mu\rot U
+H
&=\grad_{\ga}g,
&
\rot_{\ga}\hat{E}
-\grad v
+\ti{H}
&=F,
&
\div_{\ga}U
+\iota_{\reals}r
&=0,
&
\pi_{\reals}v
&=0,
\end{align*}}
and $\pi_{\harm}\ti{E}=K$ as well as $\pi_{\harm}\hat{E}=K$, $\pi_{\ti{\mathcal{H}}}U=0$,
resp. $\pi_{\harm}\hat{E}=K$, $\pi_{\ti{\mathcal{H}}}U=0$,
i.e., in formal matrix notation
\begin{align*}
\begin{bmatrix}
\mu\rot_{\ga}\rot & \grad & 0 & \iota_{\harm}\\
-\div_{\ga}\eps & 0 & \iota_{\reals} & 0\\
0 & \pi_{\reals} & 0 & 0\\
\pi_{\harm} & 0 & 0 & 0
\end{bmatrix}
\begin{bmatrix}
\ti{E}\\
u\\
r\\
H
\end{bmatrix}
&=\begin{bmatrix}
\mu\rot_{\ga}F\\
g\\
0\\
K
\end{bmatrix},\\
\begin{bmatrix}
-\grad_{\gat}\div_{\gan}\eps & \mu\rot_{\gan} & 0 & \iota_{\harm} & 0\\
\rot_{\gat} & 0 & -\grad_{\gan} & 0 & \iota_{\ti{\mathcal{H}}}\\
0 & \div_{\gat} & 0 & 0 & 0\\
\pi_{\harm} & 0 & 0 & 0 & 0\\
0 & \pi_{\ti{\mathcal{H}}} & 0 & 0 & 0
\end{bmatrix}
\begin{bmatrix}
\hat{E}\\
U\\
v\\
H\\
\ti{H}
\end{bmatrix}
&=\begin{bmatrix}
\grad_{\gat}g\\
F\\
0\\
K\\
0
\end{bmatrix},\\
\begin{bmatrix}
-\grad_{\ga}\div\eps & \mu\rot & 0 & 0 & \iota_{\harm} & 0\\
\rot_{\ga} & 0 & -\grad & 0 & 0 & \iota_{\ti{\mathcal{H}}}\\
0 & \div_{\ga} & 0 & \iota_{\reals} & 0 & 0\\
0 & 0 & \pi_{\reals} & 0 & 0 & 0\\
\pi_{\harm} & 0 & 0 & 0 & 0 & 0\\
0 & \pi_{\ti{\mathcal{H}}} & 0 & 0 & 0 & 0
\end{bmatrix}
\begin{bmatrix}
\hat{E}\\
U\\
v\\
r\\
H\\
\ti{H}
\end{bmatrix}
&=\begin{bmatrix}
\grad_{\ga}g\\
F\\
0\\
0\\
K\\
0
\end{bmatrix}.
\end{align*}
\end{rem}

We can apply the main functional a posteriori error estimate Corollary \ref{apostestfoscortsb}
to \eqref{elmgstat} resp. \eqref{elmgstatpo}.

\begin{theo}
\label{apostestemsys}
Let $E\in\rgatom\cap\mu\,\dganom$ be the exact solution of \eqref{elmgstat} resp. \eqref{elmgstatpo}
and $\ti{E}\in\ltepsom$. Then the following estimates hold for the error 
$e=E-\ti{E}$ defined in \eqref{edef}:
\begin{itemize}
\item[\bf(i)]
The error decomposes, i.e.,
$e=e_{\grad}+e_{\harm}+e_{\rot}
\in\grad\hogatom\oplus_{\ltepsom}\harmgatnepsom\oplus_{\ltepsom}\mu\rot\rganom$
and
$$\norm{e}_{\ltepsom}^2
=\norm{e_{\grad}}_{\ltepsom}^2
+\norm{e_{\harm}}_{\ltepsom}^2
+\norm{e_{\rot}}_{\ltepsom}^2.$$
\item[\bf(ii)]
The projection $e_{\grad}=\pi_{\grad}e=E_{g}-\pi_{\grad}\ti{E}\in\grad\hogatom$ satisfies
\begin{align*}
\norm{e_{\grad}}_{\ltepsom}^2
&=\min_{\Phi\in\mu\,\dganom}
\big(c_{\mathsf{fp}}\norm{\div\eps\,\Phi+g}_{\ltom}
+\norm{\Phi-\ti{E}}_{\ltepsom}\big)^2\\
&=\max_{\varphi\in\hogatom}
\big(2\scp{g}{\varphi}_{\ltom}
-\scp{2\ti{E}+\grad\varphi}{\eps\grad\varphi}_{\ltom}\big)
\end{align*}
and the minimum resp. maximum is attained at 
$$\hat{\Phi}
:=e_{\grad}+\ti{E}
\in\mu\,\dganom,\qquad
\hat{\varphi}
:=(\widetilde\grad_{\gat})^{-1}e_{\grad}
\in\hogatom$$
with $-\div\eps\,\hat{\Phi}=-\div\eps\,E=g$,
where $\hogatom$ has to be replaced by $\hobotom$, if $\gat=\emptyset$.
In the latter case $\hat{\varphi}$ is unique only up to a constant.
\item[\bf(iii)]
The projection $e_{\rot}=\pi_{\rot}e=E_{F}-\pi_{\rot}\ti{E}\in\mu\rot\rganom$ satisfies
\begin{align*}
\norm{e_{\rot}}_{\ltepsom}^2
&=\min_{\Phi\in\rgatom}
\big(c_{\mathsf{m}}\norm{\rot\Phi-F}_{\ltom}
+\norm{\Phi-\ti{E}}_{\ltepsom}\big)^2\\
&=\max_{\Psi\in\rganom}
\big(2\scp{F}{\Psi}_{\ltom}
-\scp{2\ti{E}+\mu\rot\Psi}{\rot\Psi}_{\ltom}\big)
\end{align*}
and the minimum resp. maximum is attained at 
$$\hat{\Phi}
:=e_{\rot}+\ti{E}
\in\rgatom,\qquad
\hat{\Psi}
:=(\mu\,\widetilde\rot_{\gan})^{-1}e_{\rot}
\in\rganom\cap\rot\rgatom$$
with $\rot\hat{\Phi}=\rot E=F$,
and at any $\hat{\Psi}\in\rganom$ with $\mu\rot\hat{\Psi}=e_{\rot}$.
\item[\bf(iv)]
The projection $e_{\harm}=\pi_{\harm}e=H-\pi_{\harm}\ti{E}\in\harmgatnepsom$ satisfies
\begin{align*}
\norm{e_{\harm}}_{\ltepsom}^2
&=\min_{\varphi\in\hogatom}
\min_{\Phi\in\rganom}
\norm{H-\ti{E}+\grad\varphi+\mu\rot\Phi}_{\ltepsom}^2\\
&=\max_{\Psi\in\harmgatnepsom}
\bscp{2(H-\ti{E})-\Psi}{\Psi}_{\ltepsom}
\end{align*}
and the minimum resp. maximum is attained at 
$$\hat{\varphi}
:=(\widetilde\grad_{\gat})^{-1}\pi_{\grad}\ti{E}
\in\hogatom,\qquad
\hat{\Phi}
:=(\mu\,\widetilde\rot_{\gan})^{-1}\pi_{\rot}\ti{E}
\in\rganom\cap\rot\rgatom$$
resp. $\hat{\Psi}:=e_{\harm}\in\harmgatnepsom$
with $\grad\hat{\varphi}+\mu\rot\hat{\phi}=(\pi_{\grad}+\pi_{\rot})\ti{E}=(1-\pi_{\harm})\ti{E}$,
and at any $\hat{\Phi}\in\rganom$ with $\mu\rot\hat{\Phi}=\pi_{\rot}\ti{E}$,
where $\hogatom$ has to be replaced by $\hobotom$, if $\gat=\emptyset$.
In the latter case $\hat{\varphi}$ is unique only up to a constant.
\end{itemize}
If $\ti{E}:=H+\ti{E}_{\bot}$ with some $\ti{E}_{\bot}\in\harmgatnepsom^{\bot_{\ltepsom}}$,
then $e_{\harm}=0$, and 
in (ii) and (iii) $\ti{E}$ can be replaced by $\ti{E}_{\bot}$.
In this case, for the attaining minima it holds
$$\hat{\Phi}_{\bot}:=e_{\grad}+\ti{E}_{\bot}\in\mu\,\dganom,\qquad
\hat{\Phi}_{\bot}:=e_{\rot}+\ti{E}_{\bot}\in\rgatom.$$
\end{theo}

\begin{rem}
\label{apostestemsysrem}
For conforming approximations Corollary \ref{apostestfosconfcor}
and Remark \ref{apostestfosconfcorrem} yield the following:
\begin{itemize}
\item[\bf(i)]
If $\ti{E}\in\mu\,\dganom$, then $e\in\mu\,\dganom$ and 
$$\norm{e_{\grad}}_{\ltepsom}
\leq c_{\mathsf{fp}}\norm{\div\eps\,\ti{E}+g}_{\ltom}
=c_{\mathsf{fp}}\norm{\div\eps\,e}_{\ltom}.$$
\item[\bf(ii)]
If $\ti{E}\in\rgatom$, then $e\in\rgatom$ and 
$$\norm{e_{\rot}}_{\ltepsom}
\leq c_{\mathsf{m}}\norm{\rot\ti{E}-F}_{\ltom}
=c_{\mathsf{m}}\norm{\rot e}_{\ltom}.$$
\item[\bf(iii)]
If $\ti{E}\in\rgatom\cap\mu\,\dganom$, then $e\in\rgatom\cap\mu\,\dganom$ and 
this very conforming error is equivalent to 
the weighted least squares functional
$$\calF(\ti{E})
:=\norm{H-\pi_{\harm}\ti{E}}_{\ltepsom}^2
+(1+c_{\mathsf{m}}^2)\norm{\rot\ti{E}-F}_{\ltom}^2
+(1+c_{\mathsf{fp}}^2)\norm{\div\eps\,\ti{E}+g}_{\ltom}^2,$$
i.e., $\norm{e}_{\rgatom\,\cap\,\mu\,\dganom}^2
\leq\calF(\ti{E})
\leq(1+\max\{c_{\mathsf{fp}},c_{\mathsf{m}}\}^2)\norm{e}_{\rgatom\,\cap\,\mu\,\dganom}^2$.
\end{itemize}
\end{rem}

\subsection{Prototype Second Order Systems: Laplacian and $\rot\rot$}

As prototypical examples for second order systems we will discuss
the Laplacian and the $\rot\rot$-system, both with mixed boundary conditions.
Suppose the assumptions of Section \ref{emssec} are valid and recall the notations.
For simplicity and to avoid case studies we assume $\emptyset\neq\gat\neq\ga$.

\subsubsection{The Laplacian}

Suppose $g\in\ltom$.
Let us consider the linear second order equation (in classical strong formulation)
of the perturbed negative Laplacian with mixed boundary conditions
for a function $u:\om\to\reals$
\begin{align}
\mylabel{Lapsos}
-\div\eps\grad u&=g\text{ in }\om,
&
u&=0\text{ at }\gat,
&
n\cdot\eps\grad u&=0\text{ at }\gan.
\end{align}
The corresponding variational formulation, which is uniquely solvable 
by Lax-Milgram's lemma, is the following: Find $u\in\hogatom$, such that
$$\forall\,\varphi\in\hogatom\qquad
\scp{\grad u}{\grad\varphi}_{\ltepsom}=\scpltom{g}{\varphi}.$$
Then, by definition and the results of \cite{bauerpaulyschomburgmcpweaklip}, 
we get $\eps\grad u\in\dganom$ with $-\div\eps\grad u=g$.
Hence, by setting 
$$E:=\grad u\in\mu\,\dganom\cap\grad\hogatom=\mu\,\dganom\cap\rgatzom\cap\harmgatnepsom^{\bot_{\ltepsom}}$$ 
we see that the pair $(u,E)$ 
solves the linear first order system (in classical strong formulation)
of electro-magneto statics type with mixed boundary conditions
\begin{align}
\nonumber
\grad u&=E,
&
\rot E&=0
&
\text{in }&\om,
&
u&=0,
&
n\times E&=0
&
\text{at }&\gat,\\
\mylabel{Lapsossys}
&
&
-\div\eps E&=g
&
\text{in }&\om,
&
&
&
n\cdot\eps E&=0
&
\text{at }&\gan,\\
\nonumber
&
&
\pi_{\harm}E&=0
&
\text{in }&\om.
\end{align}
Similar to the latter subsection we define the operators $\Ao$, $\At$, $\Ath$ and also $\Az$, $\A_{4}$
together with the respective adjoints and reduced operators by the de Rham complexes
$$\begin{CD}
\{0\} @> \Az=\iota_{\{0\}} >>
\hogatom @> \Ao=\grad_{\gat} >>
\rgatom @> \At=\rot_{\gat} >>
\dgatom @> \Ath=\div_{\gat} >>
\ltom @> \A_{4}=\pi_{\{0\}} >>
\{0\},
\end{CD}$$
$$\begin{CD}
\{0\} @< \Azs=\pi_{\{0\}} <<
\ltom @< \Aos=-\div_{\gan}\eps <<
\mu\,\dganom @< \Ats=\mu\rot_{\gan} <<
\rganom @< \Aths=-\grad_{\gan} <<
\hoganom @< \A_{4}^*=\iota_{\{0\}} <<
\{0\}.
\end{CD}$$
As before, all basic Hilbert spaces are $\ltom$ except of $\Hit=\ltepsom$.
Then \eqref{Lapsos} turns to 
\begin{align*}
\Aos\Ao u&=g,\\
\Azs u=\pi_{\{0\}}u&=0,\\
\pio u=\pi_{\{0\}}u&=0
\end{align*}
and this system is (again) uniquely solvable by Theorem \ref{soltheosos}
as $g\in\ltom=R(\Aos)$
with solution $u$ depending continuously on the data.
\eqref{Lapsossys} reads
\begin{align*}
\Ao u=\grad_{\gat}u&=E,
&
\At E=\rot_{\gat}E&=0,\\
\Azs u=\pi_{\{0\}}u&=0,
&
\Aos E=-\div_{\gan}\eps\,E&=g,\\
\pio u=\pi_{\{0\}}u&=0,
&
\pit E=\pi_{\harm}E&=0.
\end{align*}
We can apply the main functional a posteriori error estimates from Theorem \ref{apostestsostheo}.

\begin{theo}
\label{apostestsostheoemsys}
Let $u\in\hogatom$ be the exact solution of \eqref{Lapsos}, $E:=\grad u$, and 
$(\ti{u},\ti{E})\in\ltom\times\ltepsom$. Then the following estimates hold for the errors
$e_{u}:=u-\ti{u}$ and $e_{E}:=E-\ti{E}$:
\begin{itemize}
\item[\bf(i)]
The error $e_{E}$ decomposes, i.e.,
$$e_{E}=e_{E,\grad}+e_{E,\harm}+e_{E,\rot}
\in\grad\hogatom\oplus_{\ltepsom}\harmgatnepsom\oplus_{\ltepsom}\mu\rot\rganom$$
and
$$\norm{e_{E}}_{\ltepsom}^2
=\norm{e_{E,\grad}}_{\ltepsom}^2
+\norm{e_{E,\harm}}_{\ltepsom}^2
+\norm{e_{E,\rot}}_{\ltepsom}^2.$$
\item[\bf(ii)]
$e_{u}=\pi_{\div}e_{u}\in\div\dganom=\ltom$ and 
\begin{align*}
\normltom{e_{u}}^2
&=\min_{\varphi\in\hogatom}
\min_{\Phi\in\mu\,\dganom}
\big(c_{\mathsf{fp}}^2\normltom{\div\eps\,\Phi+g}
+c_{\mathsf{fp}}\norm{\Phi-\grad\varphi}_{\ltepsom}
+\normltom{\varphi-\ti{u}}\big)^2\\
&=\min_{\substack{\varphi\in\hogatom,\\\grad\varphi\in\mu\,\dganom}}
\big(c_{\mathsf{fp}}^2\normltom{\div\eps\grad\varphi+g}
+\normltom{\varphi-\ti{u}}\big)^2\\
&=\max_{\substack{\phi\in\hogatom,\\\grad\phi\in\mu\,\dganom}}
\big(2\scpltom{g}{\phi}+\scpltom{2\ti{u}-\div\eps\grad\phi}{\div\eps\grad\phi}\big)
\end{align*}
and the minima resp. maximum are attained at
$$\hat\varphi:=e_{u}+\ti{u}\in\hogatom,\qquad
\hat\Phi:=E\in\mu\,\dganom,\qquad
\hat\phi:=(\widetilde\grad_{\gat})^{-1}(-\widetilde\div_{\gan}\eps)^{-1}\in\hogatom$$
with $\grad\hat\varphi,\grad\hat\phi\in\mu\,\dganom$ and
$\grad\hat\varphi=\grad u=E$ and $-\div\eps\grad\hat\varphi=-\div\eps\,E=g$ as well as
$-\div\eps\,\hat\Phi=-\div\eps\,E=g$.
\item[\bf(iii)]
The projection $e_{E,\grad}=\pi_{\grad}e_{E}=E-\pi_{\grad}\ti{E}\in\grad\hogatom$ satisfies
\begin{align*}
\norm{e_{E,\grad}}_{\ltepsom}^2
&=\min_{\Phi\in\mu\,\dganom}
\big(c_{\mathsf{fp}}\norm{\div\eps\,\Phi+g}_{\ltom}
+\norm{\Phi-\ti{E}}_{\ltepsom}\big)^2\\
&=\max_{\varphi\in\hogatom}
\big(2\scp{g}{\varphi}_{\ltom}
-\scp{2\ti{E}+\grad\varphi}{\grad\varphi}_{\ltepsom}\big)
\end{align*}
and the minimum resp. maximum is attained at 
$$\hat{\Phi}
:=e_{E,\grad}+\ti{E}
\in\mu\,\dganom,\qquad
\hat{\varphi}
:=(\widetilde\grad_{\gat})^{-1}e_{E,\grad}
\in\hogatom$$
with $-\div\eps\,\hat{\Phi}=-\div\eps\,E=g$.
\item[\bf(iv)]
The projection $e_{E,\rot}=\pi_{\rot}e_{E}=-\pi_{\rot}\ti{E}\in\mu\rot\rganom$ satisfies
\begin{align*}
\norm{e_{E,\rot}}_{\ltepsom}^2
&=\min_{\Phi\in\rgatom}
\big(c_{\mathsf{m}}\norm{\rot\Phi}_{\ltom}
+\norm{\Phi-\ti{E}}_{\ltepsom}\big)^2
=\min_{\Phi\in\rgatzom}
\norm{\Phi-\ti{E}}_{\ltepsom}^2\\
&=\max_{\Psi\in\rganom}
\big(-\scp{2\ti{E}+\mu\rot\Psi}{\mu\rot\Psi}_{\ltepsom}\big)
\end{align*}
and the minimum resp. maximum is attained at 
$$\hat{\Phi}
:=e_{E,\rot}+\ti{E}
\in\rgatzom,\qquad
\hat{\Psi}
:=(\mu\,\widetilde\rot_{\gan})^{-1}e_{E,\rot}
\in\rganom\cap\rot\rgatom$$
with $\rot\hat{\Phi}=\rot E=0$.
\item[\bf(v)]
The projection $e_{E,\harm}=\pi_{\harm}e_{E}=-\pi_{\harm}\ti{E}\in\harmgatnepsom$ satisfies
\begin{align*}
\norm{e_{E,\harm}}_{\ltepsom}^2
&=\min_{\varphi\in\hogatom}
\min_{\Phi\in\rganom}
\norm{-\ti{E}+\grad\varphi+\mu\rot\Phi}_{\ltepsom}^2\\
&=\max_{\Psi\in\harmgatnepsom}
\big(-\bscp{2\ti{E}+\Psi}{\Psi}_{\ltepsom}\big)
\end{align*}
and the minimum resp. maximum is attained at 
$$\hat{\varphi}
:=(\widetilde\grad_{\gat})^{-1}\pi_{\grad}\ti{E}
\in\hogatom,\qquad
\hat{\Phi}
:=(\mu\,\widetilde\rot_{\gan})^{-1}\pi_{\rot}\ti{E}
\in\rganom\cap\rot\rgatom$$
resp. $\hat{\Psi}:=e_{E,\harm}\in\harmgatnepsom$
with $\grad\hat{\varphi}+\mu\rot\hat{\phi}=(\pi_{\grad}+\pi_{\rot})\ti{E}=(1-\pi_{\harm})\ti{E}$.
\end{itemize}
If $\ti{E}:=\ti{E}_{\bot}$ with some $\ti{E}_{\bot}\in\harmgatnepsom^{\bot_{\ltepsom}}$,
then $e_{E,\harm}=0$, and 
in (iii) and (iv) $\ti{E}$ can be replaced by $\ti{E}_{\bot}$.
In this case, for the attaining minima it holds
$$\hat{\Phi}_{\bot}:=e_{E,\grad}+\ti{E}_{\bot}\in\mu\,\dganom,\qquad
\hat{\Phi}_{\bot}:=e_{E,\rot}+\ti{E}_{\bot}\in\rgatzom.$$
\end{theo}

For conforming approximations $\ti{E}\in\grad\hogatom$ we have $e_{E,\rot}=e_{E,\harm}=0$
and $e_{E}=e_{E,\grad}$. Especially, if $\ti{u}\in\hogatom$ and $\ti{E}:=\grad\ti{u}$
with a conforming approximation $\ti{u}\in\hogatom$, the estimates of the latter theorem simplify.
More precisely, (ii) turns to the following result:
If $\ti{u}\in\hogatom$, then $e_{u}\in\hogatom$ and we can choose, e.g., $\varphi:=\ti{u}$ yielding, e.g.,
$$\normltom{e_{u}}
\leq\min_{\Phi\in\mu\,\dganom}
\big(c_{\mathsf{fp}}^2\normltom{\div\eps\,\Phi+g}
+c_{\mathsf{fp}}\norm{\Phi-\grad\ti{u}}_{\ltepsom}\big),$$
which might not be sharp anymore. Similarly, the results of (iii) read as follows:
If $\ti{u}$ belongs to $\hogatom$, then $\ti{E}:=\grad\ti{u}\in\grad\hogatom$
and $\grad(u-\ti{u})=e_{E}=e_{E,\grad}\in\grad\hogatom$ as well as
\begin{align}
\label{apostlapegrad}
\begin{split}
\norm{e_{E}}_{\ltepsom}^2
&=\min_{\Phi\in\mu\,\dganom}
\big(c_{\mathsf{fp}}\norm{\div\eps\,\Phi+g}_{\ltom}
+\norm{\Phi-\grad\ti{u}}_{\ltepsom}\big)^2\\
&=\max_{\varphi\in\hogatom}
\big(2\scp{g}{\varphi}_{\ltom}
-\scp{\grad(2\ti{u}+\varphi)}{\grad\varphi}_{\ltepsom}\big)
\end{split}
\end{align}
and the minimum resp. maximum is attained at 
$$\hat{\Phi}
:=e_{E}+\grad\ti{u}=\grad u
\in\mu\,\dganom,\qquad
\hat{\varphi}
:=(\widetilde\grad_{\gat})^{-1}e_{E}
\in\hogatom$$
with $-\div\eps\,\hat{\Phi}=-\div\eps\,E=g$.
Note that \eqref{apostlapegrad} are the well known functional a posteriori error estimates 
for the energy norm associated to the Laplacian, see, e.g., \cite{repinbookone}.

\subsubsection{The $\rot\rot$-operator}

Suppose $F\in\rot\rgatom=\dgatzom\cap\harmgantom^{\bot_{\ltom}}$ and $g\in\ltom$ as well as $H\in\harmgantom$.
Let us consider the linear second order equation (in classical strong formulation)
of the perturbed $\rot\rot$-operator with mixed boundary conditions for a vector field $B:\om\to\rt$
\begin{align}
\nonumber
\rot\mu\rot B&=F
&
&\text{in }\om,
&
n\times B&=0
&&&
&\text{at }\gan,\\
\mylabel{rotrotsosone}
\div\nu B&=g
&
&\text{in }\om,
&
n\cdot\nu B&=0,
&
n\times\mu\rot B&=0
&
&\text{at }\gat,\\
\nonumber
\pi_{\ti\harmsymbol}B&=H
&
&\text{in }\om.
\end{align}
Here $\pi_{\ti\harmsymbol}:\ltom\to\harmgantom$ and 
for simplicity we set $\nu:=\id$ for the matrix field $\nu$.
The partial solution $B_{g}$ can be computed by solving a Laplace problem.
The corresponding variational formulation, which is uniquely solvable 
by Lax-Milgram's lemma, to find the partial solution $B_{F}$ of
\begin{align*}
\rot\mu\rot B_{F}&=F
&
&\text{in }\om,
&
n\times B_{F}&=0
&&&
&\text{at }\gan,\\
\div B_{F}&=0
&
&\text{in }\om,
&
n\cdot B_{F}&=0,
&
n\times\mu\rot B_{F}&=0
&
&\text{at }\gat,\\
\pi_{\ti\harmsymbol}B_{F}&=0
&
&\text{in }\om,
\end{align*}
is the following: Find $B_{F}\in\rganom\cap\rot\rgatom$, 
such that\footnote{Note that \eqref{rotrotvarform}
holds for all $\Phi\in\rganom\cap\rot\rgatom$ if and only if 
it holds for all $\Phi\in\rganom$ since $F\in\rot\rgatom$.}
\begin{align}
\mylabel{rotrotvarform}
\forall\,\Phi\in\rganom\qquad
\scp{\rot B_{F}}{\rot\Phi}_{\ltmuom}=\scpltom{F}{\Phi}.
\end{align}
Then, by definition and the results of \cite{bauerpaulyschomburgmcpweaklip}, 
we get $\mu\rot B_{F}\in\rgatom$ with $\rot\mu\rot B_{F}=F$.
Hence, by setting 
$$E:=\mu\rot B_{F}\in\rgatom\cap\mu\rot\rganom=\rgatom\cap\mu\,\dganzom\cap\harmgatnepsom^{\bot_{\ltepsom}}$$ 
we see that the pair $(B,E)$ 
solves the linear first order system (in classical strong formulation)
of electro-magneto statics type with mixed boundary conditions
\begin{align}
\nonumber
\mu\rot B&=\mu\rot B_{F}=E,
&
\rot E&=F
&
\text{in }&\om,
&
n\times B&=0,
&
n\cdot\eps E&=0
&
\text{at }&\gan,\\
\mylabel{rotrotsossys}
\div B&=g,
&
\div\eps E&=0
&
\text{in }&\om,
&
n\cdot B&=0,
&
n\times E&=0
&
\text{at }&\gat,\\
\nonumber
\pi_{\ti\harmsymbol}B&=H,
&
\pi_{\harm}E&=0
&
\text{in }&\om.
\end{align}
Let us define operators $\T_{1}$, $\T_{2}$, $\T_{3}$ using $\Ao$, $\At$, $\Ath$
together with the respective adjoints and reduced operators by the complexes
{\footnotesize
$$\begin{CD}
\{0\} @> \T_{4}^{*}:=\iota_{\{0\}} >>
\hogatom @> \T_{3}^{*}:=\Ao=\grad_{\gat} >>
\rgatom @> \T_{2}^{*}:=\At=\rot_{\gat} >>
\dgatom @> \T_{1}^{*}:=\Ath=\div_{\gat} >>
\ltom @> \T_{0}^{*}:=\pi_{\{0\}} >>
\{0\},
\end{CD}$$
$$\begin{CD}
\{0\} @< \T_{4}:=\pi_{\{0\}} <<
\ltom @< \T_{3}:=\Aos=-\div_{\gan}\eps <<
\mu\,\dganom @< \T_{2}:=\Ats=\mu\rot_{\gan} <<
\rganom @< \T_{1}:=\Aths=-\grad_{\gan} <<
\hoganom @< \T_{0}:=\iota_{\{0\}} <<
\{0\}.
\end{CD}$$}
As before, all basic Hilbert spaces are $\ltom$ except of $\Hith=\ltepsom$,
corresponding to the domain of definition of $\T_{3}$.
Then \eqref{rotrotsosone} turns to 
\begin{align*}
\T_{2}^{*}\T_{2}B=\rot_{\gat}\mu\rot_{\gan}B&=F,\\
\T_{1}^{*}B=\div_{\gat}B&=g,\\
\pit B=\pi_{\ti\harmsymbol}B&=H
\end{align*}
and this system is uniquely solvable by Theorem \ref{soltheosos}
as $F\in R(\T_{2}^{*})$, $g\in R(\T_{1}^{*})$, and $H\in K_{2}$
with solution $B$ depending continuously on the data.
\eqref{rotrotsossys} reads
\begin{align*}
\T_{2}B=\mu\rot_{\gan}B&=E,\
&
\T_{3}E=-\div_{\gan}\eps\,E&=0,\\
\T_{1}^{*}B=\div_{\gat}B&=g,
&
\T_{2}^{*}E=\rot_{\gat}E&=F,\\
\pit B=\pi_{\ti\harmsymbol}B&=H,
&
\pith E=\pi_{\harm}E&=0.
\end{align*}
Again, we can apply the main functional a posteriori error estimates from Theorem \ref{apostestsostheo}.

\begin{theo}
\label{apostestsostheoemsysrotrot}
Let $B\in\rganom\cap\dgatom$ be the exact solution of \eqref{rotrotsosone}, 
$E:=\mu\rot B\in\rgatom$, and 
$(\ti{B},\ti{E})\in\ltom\times\ltepsom$. Then the following estimates hold for the errors
$e_{B}:=B-\ti{B}$ and $e_{E}:=E-\ti{E}$:
\begin{itemize}
\item[\bf(i)]
The errors $e_{B}$ and $e_{E}$ decompose, i.e.,
\begin{align*}
e_{B}=e_{B,\grad}+e_{B,\ti\harmsymbol}+e_{B,\rot}
&\in\grad\hoganom\oplus_{\ltom}\harmgantom\oplus_{\ltom}\rot\rgatom,\\
e_{E}=e_{E,\grad}+e_{E,\harm}+e_{E,\rot}
&\in\grad\hogatom\oplus_{\ltepsom}\harmgatnepsom\oplus_{\ltepsom}\mu\rot\rganom
\end{align*}
and
\begin{align*}
\norm{e_{B}}_{\ltom}^2
&=\norm{e_{B,\grad}}_{\ltom}^2
+\norm{e_{B,\ti\harmsymbol}}_{\ltom}^2
+\norm{e_{B,\rot}}_{\ltom}^2,\\
\norm{e_{E}}_{\ltepsom}^2
&=\norm{e_{E,\grad}}_{\ltepsom}^2
+\norm{e_{E,\harm}}_{\ltepsom}^2
+\norm{e_{E,\rot}}_{\ltepsom}^2.
\end{align*}
\item[\bf(ii)]
The projection $e_{B,\grad}=\pi_{\grad}e_{B}=B_{g}-\pi_{\grad}\ti{B}\in\grad\hoganom$ satisfies
\begin{align*}
\norm{e_{B,\grad}}_{\ltom}^2
&=\min_{\Phi\in\dgatom}
\big(\ti{c}_{\mathsf{fp}}\norm{\div\Phi-g}_{\ltom}
+\norm{\Phi-\ti{B}}_{\ltom}\big)^2\\
&=\max_{\varphi\in\hoganom}
\big(2\scpltom{g}{\varphi}+\scp{2\ti{B}-\grad\varphi}{\grad\varphi}_{\ltom}\big)
\end{align*}
and the minimum resp. maximum is attained at 
$$\hat{\Phi}
:=e_{B,\grad}+\ti{B}
\in\dgatom,\qquad
\hat{\varphi}
:=-(\widetilde\grad_{\gan})^{-1}e_{B,\grad}
\in\hoganom$$
with $\div\hat{\Phi}=\div B=g$.
\item[\bf(iii)]
The projection $e_{B,\rot}=\pi_{\rot}e_{B}=B_{E}-\pi_{\rot}\ti{B}\in\rot\rgatom$ satisfies
\begin{align*}
\norm{e_{B,\rot}}_{\ltom}^2
&=\min_{\Psi\in\rganom}
\min_{\Phi\in\rgatom}
\big(c_{\mathsf{m}}^2\norm{\rot\Phi-F}_{\ltom}
+c_{\mathsf{m}}\norm{\Phi-\mu\rot\Psi}_{\ltepsom}
+\norm{\Psi-\ti{B}}_{\ltom}\big)^2\\
&=\min_{\substack{\Psi\in\rganom,\\\mu\rot\Psi\in\rgatom}}
\big(c_{\mathsf{m}}^2\norm{\rot\mu\rot\Psi-F}_{\ltom}
+\norm{\Psi-\ti{B}}_{\ltom}\big)^2\\
&=\max_{\substack{\Theta\in\rganom,\\\mu\rot\Theta\in\rgatom}}
\big(2\scpltom{F}{\Theta}-\scp{2\ti{E}+\rot\mu\rot\Theta}{\rot\mu\rot\Theta}_{\ltom}\big)
\end{align*}
and the minima resp. maximum is attained at 
$$\hat{\Psi}
:=e_{B,\rot}+\ti{B}
\in\rganom,\qquad
\hat{\Phi}
:=E\in\rgatom,$$
and $\hat{\Theta}
:=(\mu\,\widetilde\rot_{\gan})^{-1}
(\widetilde\rot_{\gat})^{-1}e_{B,\rot}
\in\rganom\cap\rot\rgatom$
with $\mu\rot\hat\Psi,\mu\rot\hat\Theta,\in\rgatom$
and $\mu\rot\hat\Psi=\mu\rot B=E$ and $\rot\mu\rot\hat\Psi=\rot E=F$
as well as $\rot\hat\Phi=\rot E=F$.
\item[\bf(iv)]
The projection $e_{B,\ti\harmsymbol}=\pi_{\ti\harmsymbol}e_{B}=H-\pi_{\ti\harmsymbol}\ti{B}\in\harmgantom$ satisfies
\begin{align*}
\norm{e_{B,\ti\harmsymbol}}_{\ltom}^2
&=\min_{\varphi\in\hoganom}
\min_{\Phi\in\rgatom}
\norm{H-\ti{B}-\grad\varphi+\rot\Phi}_{\ltom}^2\\
&=\max_{\Psi\in\harmgantom}
\bscp{2(H-\ti{B})-\Psi}{\Psi}_{\ltom}
\end{align*}
and the minimum resp. maximum is attained at 
$$\hat{\varphi}
:=-(\widetilde\grad_{\gan})^{-1}\pi_{\grad}\ti{B}
\in\hoganom,\qquad
\hat{\Phi}
:=(\widetilde\rot_{\gat})^{-1}\pi_{\rot}\ti{B}
\in\rgatom\cap\mu\rot\rganom$$
resp. $\hat{\Psi}:=e_{B,\ti\harmsymbol}\in\harmgantom$
with $-\grad\hat{\varphi}+\rot\hat{\phi}=(\pi_{\grad}+\pi_{\rot})\ti{B}=(1-\pi_{\ti\harmsymbol})\ti{B}$.
\item[\bf(v)]
The projection $e_{E,\grad}=\pi_{\grad}e_{E}=-\pi_{\grad}\ti{E}\in\grad\hogatom$ satisfies
\begin{align*}
\norm{e_{E,\grad}}_{\ltepsom}^2
&=\min_{\Phi\in\mu\,\dganom}
\big(c_{\mathsf{fp}}\norm{\div\eps\,\Phi}_{\ltom}
+\norm{\Phi-\ti{E}}_{\ltepsom}\big)^2
=\min_{\Phi\in\mu\,\dganzom}
\norm{\Phi-\ti{E}}_{\ltepsom}^2\\
&=\max_{\varphi\in\hogatom}
\big(-\scp{2\ti{E}+\grad\varphi}{\grad\varphi}_{\ltepsom}\big)
\end{align*}
and the minimum resp. maximum is attained at 
$$\hat{\Phi}
:=e_{E,\grad}+\ti{E}
\in\mu\,\dganzom,\qquad
\hat{\varphi}
:=(\widetilde\grad_{\gat})^{-1}e_{E,\grad}
\in\hogatom$$
with $-\div\eps\,\hat{\Phi}=-\div\eps\,E=0$.
\item[\bf(vi)]
The projection $e_{E,\rot}=\pi_{\rot}e_{E}=E-\pi_{\rot}\ti{E}\in\mu\rot\rganom$ satisfies
\begin{align*}
\norm{e_{E,\rot}}_{\ltepsom}^2
&=\min_{\Phi\in\rgatom}
\big(c_{\mathsf{m}}\norm{\rot\Phi-F}_{\ltom}
+\norm{\Phi-\ti{E}}_{\ltepsom}\big)^2\\
&=\max_{\Psi\in\rganom}
\big(2\scpltom{F}{\Psi}-\scp{2\ti{E}+\mu\rot\Psi}{\mu\rot\Psi}_{\ltepsom}\big)
\end{align*}
and the minimum resp. maximum is attained at 
$$\hat{\Phi}
:=e_{E,\rot}+\ti{E}
\in\rgatom,\qquad
\hat{\Psi}
:=(\mu\,\widetilde\rot_{\gan})^{-1}e_{E,\rot}
\in\rganom\cap\rot\rgatom$$
with $\rot\hat{\Phi}=\rot E=F$.
\item[\bf(vii)]
The projection $e_{E,\harm}=\pi_{\harm}e_{E}=-\pi_{\harm}\ti{E}\in\harmgatnepsom$ satisfies
\begin{align*}
\norm{e_{E,\harm}}_{\ltepsom}^2
&=\min_{\varphi\in\hogatom}
\min_{\Phi\in\rganom}
\norm{-\ti{E}+\grad\varphi+\mu\rot\Phi}_{\ltepsom}^2\\
&=\max_{\Psi\in\harmgatnepsom}
\big(-\bscp{2\ti{E}+\Psi}{\Psi}_{\ltepsom}\big)
\end{align*}
and the minimum resp. maximum is attained at 
$$\hat{\varphi}
:=(\widetilde\grad_{\gat})^{-1}\pi_{\grad}\ti{E}
\in\hogatom,\qquad
\hat{\Phi}
:=(\mu\,\widetilde\rot_{\gan})^{-1}\pi_{\rot}\ti{E}
\in\rganom\cap\rot\rgatom$$
resp. $\hat{\Psi}:=e_{E,\harm}\in\harmgatnepsom$
with $\grad\hat{\varphi}+\mu\rot\hat{\phi}=(\pi_{\grad}+\pi_{\rot})\ti{E}=(1-\pi_{\harm})\ti{E}$.
\end{itemize}
If $\ti{B}=H+\ti{B}_{\bot}$ with some $\ti{B}_{\bot}\in\harmgantom^{\bot_{\ltom}}$,
then $e_{B,\ti\harmsymbol}=0$, and 
in (ii) and (iii) $\ti{B}$ can be replaced by $\ti{B}_{\bot}$.
If $\ti{E}=\ti{E}_{\bot}$ with some $\ti{E}_{\bot}\in\harmgatnepsom^{\bot_{\ltepsom}}$,
then $e_{E,\harm}=0$, and 
in (v) and (vi) $\ti{E}$ can be replaced by $\ti{E}_{\bot}$.
\end{theo}

A reasonable assumption is, that we have conforming approximations 
$$\ti{B}_{g}\in\grad\hoganom=\rganzom\cap\harmgantom^{\bot},\qquad
\ti{B}_{F}\in\rganom$$
of $B_{g}\in\dgatom\cap\grad\hoganom$ and $B_{F}\in\rganom\cap\rot\rgatom$ 
and hence a conforming approximation 
$$\ti{E}:=\mu\rot\ti{B}_{F}\in\mu\rot\rganom$$
of $E\in\rgatom\cap\mu\rot\rganom$,
which implies $e_{E}=e_{E,\rot}\in\mu\rot\rganom$ and $e_{E,\grad}=e_{E,\harm}=0$ 
as well as $\ti{B}-H=\ti{B}_{F}+\ti{B}_{g}\in\rganom$ and $e_{B}\in\rganom$. 
In this case the estimates of the latter theorem simplify.
More precisely, e.g., (iii) turns to the following result:
If $\ti{B}_{F},\ti{B}_{g}\in\rganom$, then $\ti{B},e_{B}\in\rganom$ 
and we can choose, e.g., $\Psi:=\ti{B}$ yielding, e.g.,
$$\norm{e_{B,\rot}}_{\ltom}
\leq\min_{\Phi\in\rgatom}
\big(c_{\mathsf{m}}^2\norm{\rot\Phi-F}_{\ltom}
+c_{\mathsf{m}}\norm{\Phi-\mu\rot\ti{B}}_{\ltepsom}\big),$$
which might not be sharp anymore. Similarly, the results of (vi) read as follows:
If $\ti{B}_{F}\in\rganom$, then $\ti{E}:=\mu\rot\ti{B}_{F}\in\mu\rot\rganom$
and $\mu\rot(B-\ti{B}_{F})=e_{E}=e_{E,\rot}\in\mu\rot\rganom$ as well as
\begin{align}
\label{apostrotroterot}
\begin{split}
\norm{e_{E}}_{\ltepsom}^2
&=\min_{\Phi\in\rgatom}
\big(c_{\mathsf{m}}\norm{\rot\Phi-F}_{\ltom}
+\norm{\Phi-\mu\rot\ti{B}_{F}}_{\ltepsom}\big)^2\\
&=\max_{\Psi\in\rganom}
\big(2\scpltom{F}{\Psi}-\scp{\mu\rot(2\ti{B}_{F}+\Psi)}{\mu\rot\Psi}_{\ltepsom}\big)
\end{split}
\end{align}
and the minimum resp. maximum is attained at 
$$\hat{\Phi}
:=e_{E}+\mu\rot\ti{B}_{F}
\in\rgatom,\qquad
\hat{\Psi}
:=(\mu\,\widetilde\rot_{\gan})^{-1}e_{E}
\in\rganom\cap\rot\rgatom$$
with $\rot\hat{\Phi}=\rot E=F$.
Note that \eqref{apostrotroterot} are in principle the functional a posteriori error estimates 
for the energy norm associated to the $\rot\rot$-operator,
which have been proved in \cite{paulyrepinmaxst}.

\subsection{More Applications}
\mylabel{moreappsec}

There are plenty more applications fitting our general theory 
for the systems \eqref{Aprob}, \eqref{AsAprob}, \eqref{AsAAsAprob}, i.e.,
\begin{align*}
\At x&=f,
&
\Ats\At x&=f,
&
\Ats\At x&=f,\\
\Aos x&=g,
&
\Aos x&=g,
&
\Ao\Aos x&=g,\\
\pit x&=k,
&
\pit x&=k,
&
\pit x&=k.
\end{align*}
E.g., if we denote the exterior derivative and the co-derivative 
associated with some Riemannian manifold having compact closure by $\ed$ and $\cd$,
we can discuss problems like
\begin{align*}
\ed E&=F,
&
-\cd\mu\ed E&=F,
&
-\cd\mu\ed E&=F,\\
-\cd\eps E&=G,
&
-\cd\eps E&=G,
&
-\ed\cd\eps E&=G,\\
\pi E&=H,
&
\pi E&=H,
&
\pi E&=H
\end{align*}
for mixed tangential and normal boundary conditions for some differential form $E$.
Moreover, problems in linear elasticity, Stokes equations, biharmonic theory, 
general relativity, $\rot\rot\rot\rot$-operators, to mention just a few examples,
fit into our general framework. Note that all these problems feature the underlying complexes
\eqref{seqAl}-\eqref{seqAls}. More precisely, 
let $\om\subset\rt$ or $\om\subset\rN$, $N\geq2$, be a bounded weak Lipschitz domain
with weak Lipschitz interface, and, for simplicity, let us just present
homogeneous material parameters with $\eps=\id$, $\mu=\id$ and skip the cohomology  projector $\pi$.
Then we have the following complexes and linear systems:
\begin{itemize}
\item
electro-magnetics (as already extensively discussed before)
$$\begin{CD}
\cdots @> \Az=\iota_{\cdots} >>
\hogatom @> \Ao=\grad_{\gat} >>
\rgatom @> \At=\rot_{\gat} >>
\dgatom @> \Ath=\div_{\gat} >>
\ltom @> \A_{4}=\pi_{\cdots} >>
\cdots
\end{CD}$$
$$\begin{CD}
\cdots @< \Azs=\pi_{\cdots} <<
\ltom @< \Aos=-\div_{\gan} <<
\dganom @< \Ats=\rot_{\gan} <<
\rganom @< \Aths=-\grad_{\gan} <<
\hoganom @< \A_{4}^{*}=\iota_{\cdots} <<
\cdots
\end{CD}$$
E.g., we can handle the systems
\begin{align*}
\rot_{\gat}E&=F,
&
\rot_{\gan}\rot_{\gat}E&=F,
&
-\grad_{\gan}\div_{\gat}E&=F,\\
-\div_{\gan}E&=g,
&
-\div_{\gan}E&=g,
&
\rot_{\gan}E&=G,
\end{align*}
or 
\begin{align*}
-\div_{\gan}\grad_{\gat}u&=f,
&
\rot_{\gan}\rot_{\gat}E&=F,\\
&&
-\grad_{\gat}\div_{\gan}E&=G.
\end{align*}
\item
generalized electro-magnetics (differential forms)
{\tiny
$$\begin{CD}
\cdots @> \Az=\iota_{\cdots} >>
\dgen{}{0}{\gat}(\om) @> \A_{1}=\ed_{\gat} >>
\cdots @> \A_{q-1}=\ed_{\gat} >>
\dgen{}{q-1}{\gat}(\om) @> \A_{q}=\ed_{\gat} >>
\dgen{}{q}{\gat}(\om) @> \A_{q+1}=\ed_{\gat} >>
\cdots @> \A_{N}=\ed_{\gat} >>
\lgen{}{2,N}{}(\om) @> \A_{N+1}=\pi_{\cdots} >>
\cdots
\end{CD}$$
$$\begin{CD}
\cdots @< \Azs=\pi_{\cdots} <<
\lgen{}{2,0}{}(\om) @< \A_{1}^{*}=-\cd_{\gan} <<
\cdots @< \A_{q-1}^{*}=-\cd_{\gan} <<
\Delta^{q-1}_{\gan}(\om) @< \A_{q}^{*}=-\cd_{\gan} <<
\Delta^{q}_{\gan}(\om) @< \A_{q+1}^{*}=-\cd_{\gan} <<
\cdots @< \A_{N}^{*}=-\cd_{\gan} <<
\dgen{}{N}{\gan}(\om) @< \A_{N+1}^{*}=\iota_{\cdots} <<
\cdots
\end{CD}$$
}
E.g., we can handle the systems
\begin{align*}
\ed_{\gat}E&=F,
&
-\cd_{\gan}\ed_{\gat}E&=F,
&
\ed_{\gat}E&=F,
&
-\cd_{\gan}\ed_{\gat}E&=F,\\
-\cd_{\gan}E&=G,
&
-\cd_{\gan}E&=G,
&
-\ed_{\gat}\cd_{\gan}E&=G,
&
-\ed_{\gat}\cd_{\gan}E&=G.
\end{align*}
\item
biharmonic problems, Stokes problems, and general relativity
{\tiny
$$\begin{CD}
\cdots @> \Az=\iota_{\cdots} >>
\hgen{}{2}{\gat}(\om) @> \Ao=\mathrm{Grad}\grad_{\gat} >>
\rgen{}{}{\gat}(\om;\mathbb{S}) @> \At=\mathrm{Rot}_{\mathbb{S},\gat} >>
\dgen{}{}{\gat}(\om;\mathbb{T}) @> \Ath=\mathrm{Div}_{\mathbb{T},\gat} >>
\ltom @> \A_{4}=\pi_{\cdots} >>
\cdots
\end{CD}$$
$$\begin{CD}
\cdots @< \Azs=\pi_{\cdots} <<
\ltom @< \Aos=\div\mathrm{Div}_{\mathbb{S},\gan} <<
\dgen{}{}{}\dgen{}{}{\gan}(\om;\mathbb{S}) @< \Ats=\sym\mathrm{Rot}_{\mathbb{T},\gan} <<
\rgen{}{}{\sym,\gan}(\om;\mathbb{T}) @< \Aths=-\dev\mathrm{Grad}_{\gan} <<
\hoganom @< \A_{4}^{*}=\iota_{\cdots} <<
\cdots
\end{CD}$$
}
E.g., we can handle the systems
\begin{align*}
\mathrm{Rot}_{\mathbb{S},\gat}S&=F,
&
\mathrm{Div}_{\mathbb{T},\gat}T&=F,\\
\div\mathrm{Div}_{\mathbb{S},\gan}S&=g,
&
\sym\mathrm{Rot}_{\mathbb{T},\gan}T&=G,
\end{align*}
or 
\begin{align*}
\sym\mathrm{Rot}_{\mathbb{T},\gan}\mathrm{Rot}_{\mathbb{S},\gat}S&=F,
&
-\dev\mathrm{Grad}_{\gan}\mathrm{Div}_{\mathbb{T},\gat}T&=F,\\
\div\mathrm{Div}_{\mathbb{S},\gan}S&=g,
&
\sym\mathrm{Rot}_{\mathbb{T},\gan}T&=G,
\end{align*}
or 
\begin{align*}
\mathrm{Rot}_{\mathbb{S},\gat}S&=F,
&
\mathrm{Div}_{\mathbb{T},\gat}T&=F,\\
\mathrm{Grad}\grad_{\gat}\div\mathrm{Div}_{\mathbb{S},\gan}S&=G,
&
\mathrm{Rot}_{\mathbb{S},\gat}\sym\mathrm{Rot}_{\mathbb{T},\gan}T&=G,
\end{align*}
or 
\begin{align*}
\div\mathrm{Div}_{\mathbb{S},\gan}\mathrm{Grad}\grad_{\gat}u&=f,
&
-\mathrm{Div}_{\mathbb{T},\gat}\dev\mathrm{Grad}_{\gan}E&=F.
\end{align*}
\item
linear elasticity
{\tiny
$$\begin{CD}
\cdots @> \Az=\iota_{\cdots} >>
\hogatom @> \Ao=\sym\mathrm{Grad}_{\gat} >>
\rgen{}{}{}\rgen{}{\top}{\gat}(\om;\mathbb{S}) @> \At=\mathrm{Rot}\mathrm{Rot}^{\top}_{\mathbb{S},\gat} >>
\dgen{}{}{\gat}(\om;\mathbb{S}) @> \Ath=\mathrm{Div}_{\mathbb{S},\gat} >>
\ltom @> \A_{4}=\pi_{\cdots} >>
\cdots
\end{CD}$$
$$\begin{CD}
\cdots @< \Azs=\pi_{\cdots} <<
\ltom @< \Aos=-\mathrm{Div}_{\mathbb{S},\gan} <<
\dgen{}{}{\gan}(\om;\mathbb{S}) @< \Ats=\mathrm{Rot}\mathrm{Rot}^{\top}_{\mathbb{S},\gan} <<
\rgen{}{}{}\rgen{}{\top}{\gan}(\om;\mathbb{S}) @< \Aths=-\sym\mathrm{Grad}_{\gan} <<
\hoganom @< \A_{4}^{*}=\iota_{\cdots} <<
\cdots
\end{CD}$$
}
E.g., we can handle the systems
\begin{align*}
\mathrm{Rot}\mathrm{Rot}^{\top}_{\mathbb{S},\gat}S&=F,
&
\mathrm{Rot}\mathrm{Rot}^{\top}_{\mathbb{S},\gan}\mathrm{Rot}\mathrm{Rot}^{\top}_{\mathbb{S},\gat}S&=F,
&
\mathrm{Rot}\mathrm{Rot}^{\top}_{\mathbb{S},\gat}S&=F,\\
-\mathrm{Div}_{\mathbb{S},\gan}S&=G,
&
-\mathrm{Div}_{\mathbb{S},\gan}S&=G,
&
-\sym\mathrm{Grad}_{\gat}\mathrm{Div}_{\mathbb{S},\gan}S&=G,
\end{align*}
or 
\begin{align*}
\mathrm{Rot}\mathrm{Rot}^{\top}_{\mathbb{S},\gan}\mathrm{Rot}\mathrm{Rot}^{\top}_{\mathbb{S},\gat}S&=F,
&
-\mathrm{Div}_{\mathbb{S},\gan}\sym\mathrm{Grad}_{\gat}E&=G,\\
-\sym\mathrm{Grad}_{\gat}\mathrm{Div}_{\mathbb{S},\gan}S&=G.
\end{align*}
\end{itemize}


\bibliographystyle{plain} 
\bibliography{paule}

\end{document}